\documentclass[10pt]{report}

\usepackage{amsmath,amssymb,amsfonts}
\usepackage{fncychap}
\usepackage{geometry}

\usepackage{helvet}
\usepackage{courier}
\usepackage{type1cm}         
\usepackage{url}
\usepackage{makeidx}         
\usepackage{graphicx}        
\usepackage{multicol}        
\usepackage[bottom]{footmisc}
\usepackage{caption} 
\makeindex             
\usepackage{listings}
\usepackage{graphics}
\usepackage{subfigure}
\usepackage{color}
\usepackage{latexsym}
\usepackage{epsfig}
\usepackage{verbatim}
\usepackage[latin1,utf8x]{inputenc}
\usepackage[greek,english]{babel}

\usepackage{longtable}

  \usepackage{chngcntr}

\usepackage{makeidx}
\makeindex

\newcommand{\csf}{\mathsf {cf}}
\newcommand{\sfv}{\mathsf v}
\newcommand{\bbsR}{\bbR_{\rm sym}}
\newcommand{\Babs}{\B_{\rm abs}}
\newcommand{\cN}{N}

\newcommand{\sigp}{\sigma_{\rm {post}}}

\newcommand{\calM}{\mathcal M}
\newcommand{\calA}{\mathcal A}
\newcommand{\calB}{\mathcal B}
\newcommand{\hmu}{\widehat{\mu}}
\newcommand{\hC}{\widehat{C}}

\newcommand{\mue}{\mu^{\epsilon}}
\newcommand{\Ve}{V^{\epsilon}}

\newcommand{\pd}{p^{\dagger}}
\newcommand{\qd}{q^{\dagger}}

\definecolor{darkred}{rgb}{.7,0,0}
\definecolor{darkgreen}{rgb}{0,0.5,0}
\definecolor{darkblue}{rgb}{0,0,0.7}
\definecolor{darkyellow}{rgb}{0.0,0.704,0.7}

\newcommand{\bbR}{\mathbb{R}}
\newcommand{\bbC}{\mathbb{C}}
\newcommand{\cf}{{\mathcal F}}

\newcommand{\bF}{{\overline f}}

\newcommand{\cP}{\mathcal{P}}

\newcommand{\dt}{\tau}
\newcommand{\fm}{f_{\rm max}}
\newcommand{\bbaP}{\bbP_{\rm approx}}
\newcommand{\va}{v_{\rm approx}}
\newcommand{\Jj}{\mathsf J}
\newcommand{\Ii}{\mathsf I}
\newcommand{\Ibi}{\overline {\Ii}}
\newcommand{\Jjd}{{\mathsf J}_{\rm det}}
\newcommand{\Iid}{{\mathsf I}_{\rm det}}
\newcommand{\Ibid}{\overline {\Ibid}}

\newcommand{\Ir}{{\mathsf I}_{\mbox r}}
\newcommand{\Phr}{\PPhi_{\mbox r}}
\newcommand{\Jr}{{\mathsf J}_{\mbox r}}

\newcommand{\Jtd}{{\mathsf I}_{\rm filter}}
\newcommand{\Jtdn}{{\mathsf I}_{{\rm filter},n}}
\newcommand{\dhh}{d_{\mbox {\tiny{\rm Hell}}}}
\newcommand{\dtv}{d_{\mbox {\tiny{\rm TV}}}}
\newcommand{\dkl}{D_{\mbox {\tiny{\rm KL}}}}
\newcommand{\MAT}{{\sc {matlab}}\,\,}
\newcommand{\hx}{h_{{\tiny{\rm max}}}}
\newcommand{\ppx}{\Phi_{{\tiny{\rm max}}}}
\newcommand{\us}{\mathsf{z}}

\newcommand{\N}{\mathbb{N}}
\newcommand{\rp}{\mathbb{P}}
\newcommand{\rqq}{\mathbb{Q}}
\newcommand{\bbP}{\mathbb{P}}
\newcommand{\bbE}{\mathbb{E}}

\newcommand{\J}{\mathbb{J}}

\newcommand{\G}{N}

\newcommand{\B}{\mathcal{B}}

\newcommand{\Z}{\mathbb{Z}}
\newcommand{\bbI}{\mathbb{I}}

\newcommand{\norm}[1]{\big\|#1\big\|}

\newcommand{\s}{\sigma}

\newcommand{\iid}{\text {i.i.d.}}

\newcommand{\R}{\mathbb{R}}
\newcommand{\C}{C}

\newcommand{\bbN}{\mathbb{N}}

\newcommand{\bbQ}{\mathbb{Q}}

\newcommand{\hc}{\widehat{C}}
\newcommand{\hm}{\widehat{m}}
\newcommand{\hv}{\widehat{v}}
\newcommand{\hw}{\widehat{w}}
\newcommand{\tw}{\widetilde{w}}

\newcommand{\PPsi}{\mathrm{\Psi}}
\newcommand{\PPhi}{\mathrm{\Phi}}

\newcommand{\Phid}{{\PPhi}_{\rm det}}
\newcommand{\mdet}{{m}_{\rm det}}
\newcommand{\Cdet}{{C}_{\rm det}}
\newcommand{\Ldet}{{L}_{\rm det}}

\newcommand{\pst}{\mathrm{\Psi}_{\rm true}}
\newcommand{\vd}{v^{\dagger}}
\newcommand{\yd}{y^{\dagger}}

\newtheorem{theorem}{Theorem}[chapter]{\bfseries}{\itshape}

\newtheorem{remark}[theorem]{Remark}{\bfseries}{\itshape}
\newtheorem{example}[theorem]{Example}{\bfseries}{}
\let\Oldendexample\endexample%
\def\endexample{ \hfill $\spadesuit$ \Oldendexample}%

\newtheorem{lemma}[theorem]{Lemma}{\bfseries}{\itshape}
\newtheorem{definition}[theorem]{Definition}{\bfseries}{\itshape}

\newtheorem{corollary}[theorem]{Corollary}{\bfseries}{\itshape}
{\bfseries}{\itshape}

\newenvironment{proof}{\emph{Proof} }{\hfill$\square$}

\usepackage{titlesec}
\titleformat{\subsection}
  {\normalfont\large}{\thesubsection.}{0.4em}{}

\titleformat{\chapter}[display]
  {\normalfont\bfseries\LARGE}
  {\chaptertitlename~\thechapter}{1pc}
  {{\titlerule[1pt]}\vspace{1pc}}

\begin{document}

\author{K.J.H. Law, A.M. Stuart and K.C. Zygalakis}
\title{Data Assimilation:\\ A Mathematical Introduction}

\maketitle



\tableofcontents



\chapter{Mathematical Background}\label{sec:mathbackground}
\label{sec:mb}

The purpose of this chapter is to briefly overview the key 
mathematical ways of thinking
that underpin our presentation of the subject of data assimilation \index{data assimilation}. In
particular we touch on the subjects of probability, dynamical systems,
probability metrics and dynamical systems for probability measures, in
sections \ref{ssec:p}, \ref{ssec:ds}, \ref{ssec:pm} and \ref{ssec:pds}
respectively. Our treatment is necessarily terse and very selective and
the bibliography section \ref{ssec:backbib} provides references to the
literature. 
We conclude with exercises in section \ref{ex:intro}.

We highlight here the fact that, throughout this book,
all probability measures on $\bbR^{\ell}$ will be assumed to possess
a density with respect to Lebesgue measure and, furthermore, this density
will be assumed to be strictly
positive everywhere in $\bbR^{\ell}$. This assumption simplifies greatly 
our subsequent probabilistic calculations.

\section{Probability}

We describe here some basic notation and facts from probability theory, all of
which will be fundamental to formulating data assimilation \index{data assimilation} from a probabilistic
perspective.

\label{ssec:p}
\subsection{Random Variables on $\bbR^{\ell}$}
We consider {\em random variables}\index{random variable} $z$ on $\bbR^{\ell}.$ To 
define a probability
measure $\mu$ on $\bbR^{\ell}$ we need to work with a sufficiently
rich collection of subsets of $\bbR^{\ell}$, to each of which we can assign
the probability that $z$ is contained in it; this collection of
subsets is termed a {\em $\sigma$-algebra}\index{$\sigma$-algebra}.
Throughout these  notes we work with $\calB(\bbR^{\ell})$, the 
Borel $\sigma$-algebra\index{$\sigma$-algebra!Borel}\index{Borel!$\sigma$-algebra}
generated by the open sets; we will abbreviate 
this $\sigma$-algebra\index{$\sigma$-algebra!Borel}
by $\calB$, when the set $\bbR^{\ell}$ is clear. 
The Borel $\sigma$-algebra\index{$\sigma$-algebra!Borel}
is the natural collection of subsets
available on $\bbR^{\ell}$; an element in $\calB$ will be termed
a {\em Borel set}\index{Borel!set}.
From a practical viewpoint the reader of this
book does not need to understand the finer properties of the
Borel $\sigma$-algebra.

We have defined a {\em probability triple}\index{probability!triple} $\bigl(\bbR^{\ell},
\calB,\mu\bigr).$
For simplicity we assume throughout the book
that $z$ has a strictly positive {\em probability density function}\index{probability density function}  (pdf) $\rho$ with respect to Lebesgue measure. Then, for any Borel set\index{Borel!set} $A\subset \bbR^{\ell},$
\begin{equation*}
\bbP(A)=\bbP (z\in A)=\int_A \rho(x)\,dx,
\end{equation*}
where $\rho:\, \bbR^{\ell}\to\bbR^+$ satisfies
$$\int_{\bbR^{\ell}} \rho(x)\,dx =1.$$
A Borel set \index{Borel!set} $A\subset \bbR^{\ell}$ is sometimes
termed an {\em event}\index{event} and the event is said to occur
{\em almost surely}\index{almost surely} if $\bbP(A)=1.$ Since $\rho$
integrates to $1$ over $\bbR^{\ell}$ and is strictly positive,
this implies that the Lebesgue
measure\index{measure!Lebesgue} of the complement of $A$, the set $A^{c}$, is zero.

We write $z\sim \mu$ as shorthand for the statement that
$z$ is {\em distributed}\index{distributed} according to probability measure $\mu$ on $\bbR^{\ell}$. 
Note that here $\mu: \calB(\bbR^{\ell})\to[0,1]$ denotes a probability measure 
and $\rho:\bbR^{\ell} \to \bbR^+$ the corresponding density. However, we will 
sometimes use the letter $\bbP$ to denote both the measure and its 
corresponding pdf. This should create no confusion: $\bbP(\cdot)$ will be a probability measure whenever its argument is a Borel set\index{Borel!set}, 
and a density whenever its argument is a point in $\bbR^{\ell}.$

For any function $f:\bbR^{\ell} \to \bbR^{p\times q}$ 
we denote by $\bbE f(z)$ the {\em expected
value}\index{expected value} 
of the random variable $f(z)$ on $\bbR^{p\times q};$ this
expectation is given by
\begin{equation*}
\bbE f(z)=\int_{\bbR^{\ell}}f(x)\mu(dx), \quad \mu(dx)=\rho(x)dx.
\end{equation*}
We also sometimes write $\mu(f)$ for $\bbE f(z)$.
The case where the function $f$ is vector valued corresponds to $q=1$ so 
that $\bbR^{p\times q}=\bbR^{p\times 1}\equiv \bbR^{p}.$
We will sometimes write $\bbE^{\mu}$ if we wish to differentiate 
between different measures with respect to which the expectation is to be understood.

The {\em characteristic function} of the random variable $z$ on $\bbR^\ell$ 
is $\cf:\bbR^\ell \to \bbC$ defined by
$$\csf(h)=\bbE \exp\bigl(i \langle h,z \rangle\bigr).$$

\begin{example}
Let $\ell =1$ and set $\rho(x)=\frac{1}{\pi (1+x^2)}.$ Note that $\rho(x)>0$ for every $x\in\bbR.$ Also, using the change of variables $x=\tan \theta,$  
\begin{equation*}
\int_{-\infty}^{\infty} \frac{dx}{\pi(1+x^2)} =
2 \int_{0}^{\infty} \frac{dx}{\pi(1+x^2)} =\int_{\arctan(0)}^{\arctan(\infty)} \frac{2\sec^2 \theta\,d\theta}{\pi(1+\tan^2\theta)}=\frac{2}{\pi}\int_{0}^{\pi/2}
d\theta= 1,
\end{equation*} 
and therefore $\rho$ is the pdf of a random variable $z$ on $\bbR.$ We say that such random variable has the {\em Cauchy distribution}. 
\end{example}

The {\em pushforward}\index{pushforward} of a pdf $\rho$ on $\bbR^l$ under a map
$G: \bbR^\ell \to \bbR^\ell$ is denoted $G\star\rho.$
It may be calculated explicitly by means of the change of variable
formula under an integral.
Indeed if $G$ is invertible then
$$G \star \rho(v) := \rho(G^{-1}(v))|DG^{-1}(v)|.$$

We will occasionally use the {\em Markov inequality}\index{Markov inequality}
which states that, for a random variable $z$ on 
$\bbR^{\ell}$, and any $R>0$,
\begin{equation}
\label{eq:mie}
\bbP(|z|\ge R) \le R^{-1}\bbE |z|.
\end{equation}
As a consequence
\begin{equation}
\label{eq:mie2}
\bbP(|z|<R) \ge 1- R^{-1}\bbE |z|.
\end{equation}
In particular, if $\bbE|z|<\infty$, then choosing $R$ sufficiently large
shows that $\bbP(|z|<R)>0.$ In our setting this last inequality follows
in any case, by assumption on the strict positivity of $\rho(\cdot)$ everywhere
on $\bbR^{\ell}$. 

Finally we say that a sequence of probability measures $\mu^{(n)}$ on $\bbR^{\ell}$ is
said to {\em converge weakly}\index{weak convergence}
to a limiting probability measures $\mu$ on $\bbR^{\ell}$
if, for all continuous bounded functions $\varphi: \bbR^{\ell} \to \bbR$,
$$\bbE^{\mu^{(n)}} \varphi(u) \to  \bbE^{\mu} \varphi(u)$$
as $n \to \infty.$

\subsection{Gaussian Random Variables}
We work in finite dimensions, but all the ideas can be generalized to
infinite dimensional contexts, such as the Hilbert space setting, for example.
A {\em Gaussian}\index{Gaussian}\index{random variable!Gaussian}
random variable\footnote{Sometimes
also called  {\em normal} random variable\index{random variable!normal}} on $\bbR^{\ell}$ is characterized by:
\begin{itemize}
\item Mean: $m\in\bbR^{\ell}.$
\item Covariance: $C\in \bbsR^{\ell\times\ell},\quad C\ge 0.$
\end{itemize}
We write $z\sim N(m,C)$ and call the Gaussian random variable {\em centred}
if $m=0.$ If $C>0$ then $z$ has strictly positive pdf on $\bbR^{\ell},$ given by
\begin{subequations}
\label{gaussiandensity}
\begin{eqnarray}
\rho(x)
&=&\frac{1}{(2\pi)^{\ell /2} (\det C)^{1/2}}  \exp\Bigl(-\frac{1}{2}\bigl|C^{-\frac12}(x-m)\bigr|^2\Bigr)\\
&=&\frac{1}{(2\pi)^{\ell /2} (\det C)^{1/2}}  \exp\Bigl(-\frac{1}{2}|x-m|_{C}^2\Bigr).
\end{eqnarray}
\end{subequations}

It can be shown that indeed $\rho$ given by \eqref{gaussiandensity}  satisfies 
\begin{equation}\label{integral1}
\int_{\bbR^{\ell}} \rho(x)\,dx =1. 
\end{equation}

\begin{lemma}
Let $z\sim N(m,C),$ $C>0.$ Then
\begin{enumerate}
\item $\bbE z=m.$
\item $\bbE (z-m)(z-m)^T =C.$
\end{enumerate}  
\end{lemma}
\begin{proof}
For the first item 
\begin{align*}
\bbE z &=  \frac{1}{(2\pi)^{\ell /2} (\det C)^{1/2}} \int_{\bbR^{\ell}} x \exp\bigl(-\frac{1}{2}|x-m|_{C}^2\bigr)\, dx  \\
&= \frac{1}{(2\pi)^{\ell /2} (\det C)^{1/2}} \int_{\bbR^{\ell}} (y+m) \exp\bigl(-\frac{1}{2}|y|_{C}^2\bigr)\, dy \\
&= \frac{1}{(2\pi)^{\ell /2} (\det C)^{1/2}} \int_{\bbR^{\ell}} y \exp\bigl(-\frac{1}{2}|y|_{C}^2\bigr)\, dy +\frac{m}{(2\pi)^{\ell /2} (\det C)^{1/2}} \int_{\bbR^{\ell}}  \exp\bigl(-\frac{1}{2}|y|_{C}^2\bigr)\, dy \\
&= 0+m\\
&= m,
\end{align*}
where we used in the last line that the function $y\mapsto y \exp\bigl(-\frac{1}{2}|y|_{C}^2\bigr)$ is even and the fact that,
by \eqref{integral1},
$$\frac{1}{(2\pi)^{\ell /2} (\det C)^{1/2}} \int_{\bbR^{\ell}}  \exp\bigl(-\frac{1}{2}|y|_{C}^2\bigr)=1.$$ 
For the second item
\begin{align*}
\bbE (z-m)(z-m)^T &=  \frac{1}{(2\pi)^{\ell /2} (\det C)^{1/2}} \int_{\bbR^{\ell}} (x-m)(x-m)^T \exp\bigl(-\frac{1}{2}|x-m|_{C}^2\bigr)\, dx \\
&= \frac{1}{(2\pi)^{\ell /2} (\det C)^{1/2}} \int_{\bbR^{\ell}} yy^T \exp\bigl(-\frac{1}{2}|C^{-1/2}y|^2\bigr)\, dy \\
&= \frac{1}{(2\pi)^{\ell /2} (\det C)^{1/2}} \int_{\bbR^{\ell}} C^{1/2} ww^T C^{1/2} \exp\bigl(-\frac{1}{2}|w|^2\bigr)\det(C^{1/2})\, dw  \\
&= C^{1/2} J C^{1/2}
\end{align*}
where
\begin{equation*}
J= \frac{1}{(2\pi)^{\ell /2}} \int_{\bbR^{\ell}} ww^T \exp\bigl(-\frac{1}{2}|w|^2\bigr)\, dw \in \bbR^{\ell}\times\bbR^{\ell}
\end{equation*}
and so
\begin{equation*}
 J_{ij}= \frac{1}{(2\pi)^{\ell /2}} \int_{\bbR^{\ell}} w_i w_j \exp\bigl(-\frac{1}{2}\sum_{k=1}^{\ell} w_k^2\bigr) \prod_{k=1}^{\ell} \,dw_k.
\end{equation*}
To complete the proof we need to show that $J$ is the identity matrix $I$
on $\bbR^{\ell}\times \bbR^{\ell}.$
Indeed, for $i\neq j$
\begin{equation*}
J_{ij}\propto \int_{\bbR} w_i \exp\bigl(-\frac{1}{2} w_i^2\bigr)\,dw_i \int_{\bbR} w_j \exp\bigl(-\frac{1}{2} w_j^2\bigr)\,dw_j =0,
\end{equation*}
by symmetry; and for $i=j$
\begin{align*}
J_{jj} &= \frac{1}{(2\pi)^{\frac12}}   \int_{\bbR} w_j^2 \exp\bigl(-\frac{1}{2} w_j^2\bigr)\,dw_j \left( \frac{1}{(2\pi)^{\frac12}}  \int_{\bbR}  \exp\bigl(-\frac{1}{2} w_k^2\bigr)\,dw_k \right)^{\ell -1} \\
&= \frac{1}{(2\pi)^{\frac12}}   \int_{\bbR} w_j^2 \exp\bigl(-\frac{1}{2} w_j^2\bigr)\,dw_j \\
&= -\frac{1}{(2\pi)^{\frac12}}w_j \exp\bigl(-\frac{1}{2}w_j^2\bigr) \Big|_{-\infty}^\infty  + \frac{1}{(2\pi)^{\frac12}}   \int_{\bbR}  \exp\bigl(-\frac{1}{2} w_j^2\bigr)\,dw_j =1,
\end{align*}
where we again used \eqref{integral1} in the first and last lines.
Thus $J=I$, the identity in $\bbR^{\ell},$ and $\bbE (z-m)(z-m)^T = C^{1/2}C^{1/2}=C.$
\end{proof}

The following characterization of Gaussians is often useful.

\begin{lemma}
The characteristic function of the Gaussian $N(m,C)$ is given by
$$\csf(h)=\exp\bigl(i\langle h,m\rangle-\frac12 \langle Ch,h\rangle \bigr).$$
\end{lemma}
\begin{proof}
This follows from noting that
$$\frac12|x-m|_{C}^2-i \langle h,x \rangle=
\frac12|x-(m+iCh)|_{C}^2-i\langle h,m \rangle+\frac12 \langle Ch,h \rangle.$$
\end{proof}

\begin{remark}
Note that the pdf for the Gaussian random variable that we wrote
down in equation \eqref{gaussiandensity} is defined only for $C>0$ since
it involves $C^{-1}.$ 
The characteristic function appearing in the preceding lemma can
be used to {\em define} a Gaussian with mean $m$ and covariance $C$, including
the case where $C \ge 0$ so that
the Gaussian covariance $C$ is only positive semi-definite, since
it is defined in terms of $C$ and not $C^{-1}$. For example if we
let $z\sim N(m,C)$ with $C=0$ then $z$ is a {\em Dirac mass}\index{Dirac mass}
at $m,$ i.e. $z=m$ almost surely and for any continuous function $f$
\begin{equation*}
\bbE f(z) = f(m). 
\end{equation*} 
This Dirac mass \index{Dirac mass} may be viewed as a particular case of a Gaussian
random variable. We will write $\delta_m$ for $N(m,0).$
\end{remark}

\begin{lemma}
\label{lem:to}
The following hold for Gaussian random variables:
\begin{itemize}
\item If $z=a_1 z_1+a_2 z_2$ where $z_1, z_2$ are independent
Gaussians with distributions $N(m_1,C_1)$ and $N(m_2,C_2)$
respectively then $z$ is Gaussian with distribution $N(a_1 m_1+a_2 m_2,
a_1^2C_1+a_2^2 C_2).$
\item If $z\sim N(m,C)$ and $w=Lz+a$ then $w\sim N(Lm+a,LCL^T)$.
\end{itemize}
\end{lemma}
\begin{proof}
The first result follows from computing the characteristic function
of $z.$ By independence this is the product of the characteristic
functions of $a_1 z_1$ and of $a_2 z_2$. The characteristic
function of $a_i z_i$ has logarithm equal to
$$i\langle h, a_i m_i \rangle-\frac12\langle a_i^2 C h,h \rangle.$$
Adding this for $i=1,2$ gives the logarithm of the characteristic
function of $z$.

For the second result we note that the characteristic function
of $a+Lz$ is the expectation of the exponential of
$$i\langle h,a+Lz \rangle= i \langle h,a \rangle+i \langle L^T h, z \rangle.$$
Using the properties of the characteristic functions of $z$ we deduce
that the logarithm of the characteristic function of $a+Lz$  is equal to
$$i\langle h,a \rangle+i \langle L^T h, m \rangle-\frac 12 \langle CL^Th,
L^T h\rangle.$$
This may be re-written as
$$i \langle h,a+Lm \rangle-\frac12 \langle LCL^T h,h \rangle$$
which is the logarithm of the characteristic function
of $N(a+Lm,LCL^T)$ as required.
\end{proof}

We finish by stating a lemma whose proof is straightforward, given the
foregoing material in this section, and left as an exercise.

\begin{lemma}
\label{l:qf}
Define
$$I(v):=\frac12 \bigl\langle (v-m),L(v-m)\bigr\rangle$$ 
with $L\in \bbsR^{\ell \times \ell}$
satisfying $L>0$ and $m \in \bbR^{\ell}.$
Then $\exp\bigl(-I(v)\bigr)$ can be normalized to produce the pdf of
the Gaussian random variable $N(m,L^{-1})$ on $\bbR^{\ell}$. The matrix
$L$ is known as the {\em precision matrix}\index{precision}
of the Gaussian random variable.
\end{lemma}

\subsection{Conditional and Marginal Distributions}
Let $(a,b)\in \bbR^{\ell} \times \bbR^m$ denote a jointly varying random variable.
\begin{definition}
The {\em marginal pdf}\index{marginal distribution} of $a,$ $\bbP(a),$ is given in terms of the pdf of $(a,b),$ $\bbP (a,b),$ by 
\begin{equation*}
\bbP(a) = \int_{\bbR^m} \bbP(a,b) \,db.
\end{equation*} 
\end{definition}
\begin{remark}
With this definition, for $A \subset \calB(\bbR^{\ell}),$ 
\begin{align*}
\bbP(a\in A) &= \bbP \bigg( (a,b) \in A\times \bbR^m\bigg) = \int_A \int_{\bbR^m} \bbP(a,b) \,da \,db\\
&=\int_A \left(\int_{\bbR^m} \bbP(a,b)\,db\right) da = \int_A \bbP(a)\,da.
\end{align*}
Thus the marginal pdf $\bbP(a)$\index{probability!marginal} 
is indeed the pdf for $a$ in situations
where we have no information about the random variable $b$, other than that it
is in $\bbR^m$.
\end{remark}

We now consider the situation which is the extreme opposite of the marginal
situation. To be precise, we assume that we know {\em everything} about the 
random variable $b$: we have observed it and know what value it takes. 
This leads to
consideration of the random variable $a$ {\em given} that we know the value 
taken by $b$; we write $a|b$ for $a$ given $b$. The following definition is 
then natural:

\begin{definition}\label{definitionconditional}
The {\em conditional pdf}\index{conditional distribution} of $a|b,$ $\bbP(a|b),$ is defined by 
\begin{equation}
\label{eq:cpeq}
\bbP(a|b) = \frac{\bbP(a,b)}{\bbP(b)}.
\end{equation}
\end{definition}
\begin{remark}
Conditioning a jointly varying random variable can be 
useful when computing probabilities\index{probability!conditional}, 
as the following calculation
demonstrates.
\begin{align*}
\bbP \bigg( (a,b)\in A\times B\bigg) &= \int_A \int_B \bbP(a,b) \,da\,db \\
&=\int_A \int_B \bbP(a|b)\bbP(b) \,da\,db  \\
&= \int_B \underbrace{\left( \int_A \bbP(a|b)\,da \right)}_{=:I_1} \underbrace{\bbP(b) \,db}_{=:I_2}.
\end{align*}
Given $b,$ $I_1$ computes the probability that $a$ is in $A.$ $I_2$ then denotes averaging over given outcomes of $b$ in $B$. 
\end{remark}
\subsection{Bayes' Formula \index{Bayes' formula}}
\label{ssec:BF}
By Definition \ref{definitionconditional} we have 
\begin{subequations}
\label{eq:preB}
\begin{eqnarray}
\bbP(a,b) = \bbP (a|b)\bbP(b),\\
\bbP(a,b) = \bbP (b|a)\bbP(a).
\end{eqnarray}
\end{subequations}
Equating and rearranging we obtain {\bf Bayes' formula}\index{Bayes' formula}
which states that 
\begin{equation}
\label{eq:bayes}
\bbP(a|b)=\frac{1}{\bbP(b)}\bbP(b|a)\bbP(a).
\end{equation}
The beauty of this formula is apparent
in situations where $\bbP(a)$ and $\bbP(b|a)$ are individually easy to write down. Then $\bbP(a|b)$ may be identified easily too.
\begin{example}
Let $(a,b)\in \bbR \times \bbR$ be a jointly varying random variable specified via
\begin{subequations}
\begin{align*}
a&\sim N(m,\sigma^2), \quad\bbP(a);\\
b|a &\sim N(f(a),\gamma^2), \quad\bbP(b|a).
\end{align*}
\end{subequations}
Notice that, by using equation \eqref{eq:cpeq}, 
$\bbP(a,b)$ is defined via two Gaussian distributions.
In fact we have 
\begin{equation*}
\bbP(a,b) = \frac{1}{2\pi\gamma\sigma} \exp\left( -\frac{1}{2\gamma^2}|b-f(a)|^2 - \frac{1}{2\sigma^2}|a-m|^2\right). 
\end{equation*}
Unless $f(\cdot)$ is linear this is not the pdf of a Gaussian distribution.
Integrating over $a$ we obtain, from the definition of the marginal pdf of $b$,
\begin{equation*}
\bbP(b)=\frac{1}{2\pi\gamma\sigma}\int_{\bbR} \exp\left( -\frac{1}{2\gamma^2}|b-f(a)|^2 - \frac{1}{2\sigma^2}|a-m|^2\right) \, da.
\end{equation*}
Using equation \eqref{eq:preB} then shows that
\begin{equation*}
\bbP(a|b) = \frac{1}{\bbP(b)}\times \frac{1}{2\pi\pi\gamma\sigma}\exp\left( -\frac{1}{2\gamma^2}|b-f(a)|^2 - \frac{1}{2\sigma^2}|a-m|^2\right).
\end{equation*}
Note that $a|b$, like $(a,b)$, is not Gaussian. Thus, for both $(a,b)$
and $a|b$, we have constructed
a non-Gaussian pdf in a simple fashion from the knowledge of the two Gaussians
and $a$ and $b|a$.
\end{example}

When Bayes' formula\index{Bayes' formula}  
\eqref{eq:bayes} is used in statistics
then typically $b$ is observed data and $a$ is the unknown about which
we wish to find information, using the data. In this context
we refer to $\bbP(a)$ as the {\bf prior}, to $\bbP(b|a)$ as the
{\bf likelihood \index{likelihood}} and to $\bbP(a|b)$ as the {\bf posterior}.
The beauty of Bayes's formula\index{Bayes' formula}
as a tool in applied mathematics is that the likelihood \index{likelihood} is
often easy to determine explicitly, given reasonable assumptions
on the observational noise, whilst there is considerable
flexibility inherent in modelling prior knowledge via
probabilities, to give the prior. Combining the prior and
likelihood \index{likelihood} as in \eqref{eq:bayes}
gives the posterior, which is the random variable of interest;
whilst the probability distributions used to define
the likelihood \index{likelihood} $\bbP(b|a)$ (via a probability density on the
data space) and prior $\bbP(a)$ (via a probability on the space
of unknowns) may be quite simple, the resulting posterior
probability distribution can be very complicated.
A second key
point to note about Bayes' formula\index{Bayes' formula} in this context is that $\bbP(b)$,
which normalizes the posterior to a pdf, may be hard to determine
explicitly, but algorithms exist to find information from
the posterior without knowing this normalization constant\index{normalization
constant}.
We return to this point in subsequent chapters.

\subsection{Independence}
Consider the jointly varying random variable 
$(a,b)\in \bbR^{\ell}\times \bbR^m$. The random variables $a$ and $b$ are 
said to be {\em independent} if
\begin{equation*}
\bbP(a,b)=\bbP(a)\bbP(b).
\end{equation*}
In this case, for $f:\bbR^\ell \to \bbR^{\ell'}$ and $g:\bbR^m \to \bbR^{m'}$,
\begin{equation*}
\bbE f(a)g(b)^T= \left(\bbE f(a)\right) \times \left (\bbE g(b)^T\right)
\end{equation*}
as 
\begin{equation*}
\bbE f(a)g(b)^T = \int_{\bbR^{\ell}\times\bbR^m} f(a)g(b)^T\bbP(a)\bbP(b)\,da\,db = \left(\int_{\bbR^{\ell}} f(a)\bbP(a)da\right) \left(\int_{\bbR^{m}} g(b)^T\bbP(b)db\right).
\end{equation*}

An $\iid$ (independent, identically distributed)\index{i.i.d.} 
sequence $\{\xi_j\}_{j\in \N}$ 
is one for which:\footnote{This discussion is easily generalized to $j \in \Z^+$.}
\begin{itemize}
\item each $\xi_j$ is distributed according to the same pdf $\rho;$
\item $\xi_j$ is independent of $\xi_k$ for $j\neq k.$
\end{itemize}
If $\J$ is a subset of $\N$ with finite cardinality then this i.i.d. \index{i.i.d.}
sequence satisfies
\begin{equation*}
\bbP\left(\{\xi_j\}_{j\in\J} \right) = \prod_{j\in\J} \rho(\xi_j)
\end{equation*}

\section{Dynamical Systems}

We will discuss data assimilation \index{data assimilation} in the context of both discrete-time\index{dynamical system!discrete-time}
and continuous-time dynamical systems\index{dynamical system!continuous-time}. 
In this section we introduce
some basic facts about such dynamical systems\index{dynamical system}.

\label{ssec:ds}
\subsection{Iterated Maps}
Let $ \in C(\bbR^{\ell},\bbR^{\ell}).$ We will frequently be interested in the iterated map, or discrete-time dynamical system\index{dynamical system!discrete-time}, defined by 
\begin{equation*}
v_{j+1}=\PPsi(v_j),\quad  v_0=u,
\end{equation*}
and in studying properties of the sequence $\{v_j\}_{j\in \Z^+}.$
A {\em fixed point}\index{fixed point} 
of the map is a point $v_{\infty}$ which satisfies
$v_{\infty}=\PPsi(v_{\infty})$; initializing the map at $u=v_{\infty}$ will
result in a sequence satisfying $v_j=v_{\infty}$ for all $j\in \Z^+$.

\begin{example}
Let 
\begin{equation*}
\PPsi(v)=\lambda  v+a.
\end{equation*}
Then 
\begin{equation*}
v_{j+1}=\lambda v_j + a, \quad v_0 = u.
\end{equation*}
By induction we see that, for $\lambda\neq 1,$
\begin{equation*}
v_j= \lambda^j u + a\sum_{i=0}^{j-1} \lambda^i = \lambda^j u + a \frac{1-\lambda^j}{1-\lambda}.
\end{equation*}
Thus if $|\lambda|<1$ then 
\begin{equation*}
v_j\to \frac{a}{1-\lambda} \quad {\text{as}} \quad j\to\infty.
\end{equation*}
The limiting value $\frac{a}{1-\lambda}$ is a fixed point \index{fixed point} of the map.
\end{example}

\begin{remark}
In the preceding example the long-term dynamics of the map, for
$|\lambda|<1$, is described by convergence to a fixed point  \index{fixed point}. Far
more complex behaviour is, of course, possible; we will explore such
complex behaviour in the next chapter.
\end{remark}

The following result is known as the (discrete time) Gronwall 
lemma\index{Gronwall lemma!discrete time}.

\begin{lemma}
\label{lem:GL}
Let $\{v_j\}_{j\in \Z^+}$ be a positive sequence and $(\lambda,a)$ a pair of reals with $\lambda>0.$ Then if 
\begin{equation*}
v_{j+1}\le \lambda v_j +a, \quad \quad j=0,1,\ldots
\end{equation*}
it follows that 
\begin{equation*}
v_j\le \lambda^j v_0 + a \frac{1-\lambda^j}{1-\lambda}, \quad \quad \lambda\neq 1,
\end{equation*}
and 
\begin{equation*}
v_j\le v_0+ja, \quad \quad \lambda=1.
\end{equation*}
\end{lemma}
\begin{proof}
We prove the case $\lambda\neq 1$ as the case $\lambda=1$ may be proved
similarly. We proceed by induction. The result clearly holds for $j=0.$ Assume that the result is true for $j=J.$ Then 
\begin{align*}
v_{J+1}&\le \lambda v_J+a \\
&\le \lambda\left(\lambda^J v_0 + a \frac{1-\lambda^J}{1-\lambda}\right) +a\\
&= \lambda^{J+1}v_0 + a \frac{\lambda-\lambda^{J+1}}{1-\lambda} + a \frac{1-\lambda}{1-\lambda}\\
&= \lambda^{J+1}v_0 + a \frac{1-\lambda^{J+1}}{1-\lambda}.
\end{align*}
This establishes the inductive step and the proof is complete.
\end{proof}

We will also be interested in stochastic dynamical systems \index{dynamical system!stochastic} of the form 
\begin{equation*}
v_{j+1}=\PPsi(v_j) + \xi_j, \quad v_0=u,
\end{equation*}
where $\xi=\{\xi_j\}_{j\in\N}$ is an $\iid$  sequence of random variables on $\bbR^{\ell},$ and $u$ is a random variable on $\bbR^{\ell},$ independent of $\xi.$
\begin{example}\label{examplegauss}
This is a simple but important one dimensional (i.e. $\ell=1.$) example. Let $|\lambda|<1$ and let 
\begin{subequations}
\begin{align*}
v_{j+1}&=\lambda v_j+ \xi_j, \quad \quad \xi_j\sim N(0,\sigma^2)\,\,\, \iid,\\
v_0&\sim N(m_0,\sigma_0^2).
\end{align*}
\end{subequations}
By induction
\begin{equation*}
v_j= \lambda^j v_0 + \sum_{i=0}^{j-1} \lambda^{j-i-1} \xi_i.
\end{equation*}
Thus $v_j$ is Gaussian, as a linear transformation of Gaussians -- see
Lemma \ref{lem:to}. Furthermore,
using independence of the initial condition from the sequence $\xi$, we obtain
\begin{align*}
m_j&:= \bbE v_j = \lambda^j m_0 \\
\sigma_j^2 &:= \bbE (v_j-m_j)^2 = \lambda^{2j} \bbE (v_0-m_0)^2 + \sum_{i=0}^{j-1} \lambda^{2j-2i-2}\sigma^2\\
&=\lambda^{2j}\sigma_0^2 +\sigma^2 \sum_{i=0}^{j-1} \lambda^{2i} = \lambda^{2j}\sigma_0^2 + \sigma^2 \frac{1-\lambda^{2j}}{1-\lambda^2}.
\end{align*}
Since $|\lambda|<1$ we deduce that $m_j \to 0$ and $\sigma_j^2 \to \sigma^2
(1-\lambda^2)^{-1}.$ Thus the sequence of Gaussians generated by this
stochastic dynamical system \index{dynamical system!stochastic} has a limit, which is a centred Gaussian
with variance larger than the variance of $\xi_1$, unless $\lambda=0.$
\end{example}

\subsection{Differential Equations}
Let $f\in \C^1(\bbR^{\ell},\bbR^{\ell})$ and consider the 
ordinary differential equation (ODE) 
\begin{equation*}
\frac{dv}{dt}=f(v), \quad\quad v(0)=u.
\end{equation*}
Assume a solution exists for all $u\in \bbR^{\ell},$ $t\in\bbR^+;$ for any given $u$ this solution
is then an element of the space $C^1(\bbR^+;\bbR^{\ell}).$ In this situation, the
ODE generates a continuous-time dynamical system\index{dynamical system!continuous-time}. We are interested in properties of the function $v.$ 
An {\em equilibrium point}\index{equilibrium point} $v_{\infty} \in \bbR^{\ell}$ is a point for which
$f(v_{\infty})=0.$ Initializing the equation at $u=v_{\infty}$ results
in a solution $v(t)=v_{\infty}$ for all $t \ge 0.$

\begin{example}\label{exode}
Let $f(v)=-\alpha v +\beta.$  Then 
\begin{equation*}
e^{\alpha t} \left(\frac{dv}{dt}+\alpha v\right) = \beta e^{\alpha t}
\end{equation*}
and so 
\begin{equation*}
\frac{d}{dt}\Bigl(e^{\alpha t} v\Bigr)  = \frac{d}{dt}\left( \frac{\beta}{\alpha} e^{\alpha t} \right).
\end{equation*}
Thus 
\begin{equation*}
e^{\alpha t} v(t) - u = \frac{\beta}{\alpha} (e^{\alpha t} -1),
\end{equation*}
so that
\begin{equation*}
v(t)= e^{-\alpha t} u + \frac{\beta}{\alpha} (1- e^{-\alpha t} ).
\end{equation*}
If $\alpha>0$ then 
\begin{equation*}
v(t)\to \frac{\beta}{\alpha} \quad \quad {\text{as}} \,\,\, t\to \infty.
\end{equation*}
Note that $v_{\infty}:=\frac{\beta}{\alpha}$ is a the unique equilibrium
point of the equation.
\end{example} 

\begin{remark}
In the preceding example the long-term dynamics of the ODE, for
$\alpha>0$, is described by convergence to an equilibrium point \index{equilibrium point}. As in discrete
time, far
more complex behaviour is, of course, possible; we will explore this
possibility in the next chapter.
\end{remark}

If the differential equation has a solution for every $u\in \bbR^{\ell}$ and 
every $t\in \bbR^+$ then 
there is a one-parameter semigroup\index{semigroup} of operators
$\PPsi(\cdot;t)$, parametrized by time $t \ge 0$, with
the properties that
\begin{subequations}
\label{eq:sg}
\begin{eqnarray}
v(t)&=&\PPsi(u;t), \;t \in (0,\infty),\\
\PPsi(u;t+s)&=&\PPsi\bigl(\PPsi(u;s);t\bigr), \;t,s \in \bbR^+, u\in\bbR^\ell,\\
\PPsi(u;0)&=&u \in \bbR^\ell.
\end{eqnarray}
\end{subequations}
We call $\PPsi(\cdot;\cdot)$ the solution operator\index{solution operator}
for the ODE.
In this scenario we can consider the iterated map defined by $\PPsi(\cdot)= \PPsi(\cdot;h),$ for some fixed $h>0$, thereby linking the discrete time iterated
maps with continuous time ODEs.

\begin{example}{(Example \ref{exode} continued)}
Let 
\begin{equation*}
\PPsi(u;t)= e^{-\alpha t} u + \frac{\beta}{\alpha} \Bigl(1-e^{-\alpha t}\Bigr)
\end{equation*}
which is the solution operator\index{solution operator} for the equation in that $v(t)=\PPsi(u;t)$.
Clearly $\PPsi(u;0)=u.$ Also 
\begin{align*}
\PPsi(u;t+s) &= e^{-\alpha t} e^{-\alpha s} u + \frac{\beta}{\alpha} \Bigl(1- e^{-\alpha t} e^{-\alpha s}\Bigr)\\
&= e^{-\alpha t} \left( e^{-\alpha s} u + \frac{\beta}{\alpha} 
\bigl(1- e^{-\alpha s}\bigr) \right) + \frac{\beta}{\alpha} \Bigl(1-e^{-\alpha t}\Bigr) \\
&= \PPsi\left(\PPsi(u;s);t\right).
\end{align*}
\end{example}

The following result is known as the (continuous time) Gronwall lemma\index{Gronwall lemma!continuous time}.

\begin{lemma}\label{gronwallcts}
Let $z\in \C^1 (\bbR^+, \bbR)$ satisfy 
\begin{equation*}
\frac{dz}{dt}\le az + b, \quad z(0)= z_0,
\end{equation*}
for some $a,b\in\bbR.$ Then $$z(t)\le e^{at}z_0 + \frac{b}{a}\bigl(e^{at} -1\bigr).$$
\end{lemma}
\begin{proof}
Multiplying both sides of the given identity by $e^{-at}$ we obtain
\begin{equation*}
e^{-at}\left(\frac{dz}{dt} - az\right) \le b e^{-at}
\end{equation*}
which implies that
\begin{equation*}
\frac{d}{dt}\left(e^{-at} z\right) \le be^{-at}.
\end{equation*}
Therefore,
\begin{equation*}
e^{-at}z(t) - z(0) \le \frac{b}{a} \bigl(1-e^{-at}\bigr)
\end{equation*}
so that
\begin{equation*}
z(t)\le e^{at} z_0 + \frac{b}{a} \bigl(e^{at} - 1\bigr).
\end{equation*}
\end{proof}

\subsection{Long-Time Behaviour}
We consider the long-time behaviour of discrete-time dynamical systems \index{dynamical system!discrete}. The
ideas are easily generalized to continuous-time dynamical systems -- ODEs -- and
indeed our example will demonstrate such a generalization. To facilitate
our definitions we now extend $\PPsi$ to act on Borel\index{Borel!set} subsets of $\bbR^{\ell}$. Note that
currently $\PPsi: \bbR^{\ell} \to \bbR^{\ell}$; we extend to 
$\PPsi:\calB(\bbR^{\ell}) \to \calB(\bbR^{\ell})$ via
$$\PPsi(A)=\bigcup_{u \in A} \PPsi(u), \quad A \in \calB(\bbR^{\ell}).$$ 
For both $\PPsi:\bbR^{\ell} \to \bbR^{\ell}$ and 
$\PPsi:\calB(\bbR^{\ell}) \to \calB(\bbR^{\ell})$ we denote by
$$\PPsi^{(j)} = \PPsi \circ \dots \circ \PPsi$$ 
the $j-$fold composition of $\PPsi$ with itself.
In the following, let $B(0,R)$ denote the ball of radius $R$ in $\bbR^{\ell}$, in the
Euclidean norm, centred at the origin.
\begin{definition}
\label{def:dabs}
A discrete time dynamical system \index{dynamical system!discrete} has a bounded {\em absorbing set}\index{absorbing set} $\Babs\subset \bbR^{\ell}$ if, for every $R>0$, there exists $J= J(R)$ such that 
\begin{equation*}
\PPsi^{(J)} \left(B(0,R)\right) \subset \Babs, \quad \forall j\ge J.
\end{equation*}
\end{definition}

\begin{remark}
\label{rem:cabz}
The definition of absorbing set is readily generalized to
continuous time dynamical systems; this is left as an exercise
for the reader.
\end{remark}

\begin{example}
\label{ex:diss}
Consider an ODE for which there exist $\alpha,\beta >0$ such that 
\begin{equation*}
\langle f(v),v\rangle \le \alpha - \beta |v|^2, \quad \forall v\in \bbR^{\ell}.
\end{equation*}
Then 
\begin{equation*}
\frac{1}{2}\frac{d}{dt}|v|^2 = \Bigl \langle v,\frac{dv}{dt} \Bigr \rangle = 
\langle v, f(v) \rangle \le \alpha - \beta |v|^2.
\end{equation*}
Applying the Gronwall\index{Gronwall lemma} Lemma \ref{gronwallcts} gives 
\begin{equation*}
|v(t)|^2 \le e^{-2\beta t} |v(0)|^2 + \frac{\alpha}{\beta} \bigl(1-e^{-2\beta t}\big).
\end{equation*}
Hence, if $|v(0)|^2\le R$ then 
\begin{equation*}
|v(t)|^2 \le 2 \frac{\alpha}{\beta} \quad \forall t\ge T: \, e^{-2\beta t} R^2 \le \frac{\alpha}{\beta}.
\end{equation*} 
Therefore the set $\Babs= B\left(0, \sqrt{\frac{2\alpha}{R}} \right)$ is 
absorbing for the ODE (with the generalization of the above
definition of absorbing set to continuous time, as in Remark
\ref{rem:cabz}).

If $v_j= v(jh)$ so that $\PPsi(\cdot)= \PPsi(\cdot; h)$ and $v_{j+1}=\PPsi(v_j)$ then
\begin{equation*}
|v_j|^2 \le 2 \frac{\alpha}{\beta} \quad \forall J\ge \frac{T}{h}, 
\end{equation*}
where $T$ is as for the ODE case. Hence $\Babs= B\left(0, \sqrt{\frac{2\alpha}{R}} \right)$ is also an absorbing set \index{absorbing set} for the iterated map associated with the ODE.
\end{example}

\begin{definition}
\label{def:ga}
When the discrete time dynamical system \index{dynamical system!discrete} has a bounded absorbing set \index{absorbing set} $\Babs$ we define the global attractor\index{global attractor} $\calA$ to be 
\begin{equation*}
\calA= \bigcap_{k\ge 0} \overline{\bigcup_{j\ge k} \PPsi^{(j)}(\Babs).}
\end{equation*}
\end{definition}
This object captures all the long-time dynamics of the dynamical system.\index{dynamical system} As for the absorbing\index{absorbing set} set itself
this definition is readily generalized to continuous time.

\subsection{Controlled Dynamical Systems}
It is frequently of interest to add a {\em controller} $w=\{w_j\}_{j=0}^\infty$ to the discrete time dynamical system to obtain
\begin{equation*}
v_{j+1}=\PPsi(v_j) + w_j.
\end{equation*}
The aim of the controller is to ``steer" the dynamical system to achieve some objective. Interesting examples include:
\begin{itemize}
\item given point $v^* \in \bbR^{\ell}$ and time $J\in \Z^+,$ choose $w$ so that $v_J=v^*;$
\item given open set $B$ and  time $J\in \Z^+,$ choose $w$ so that $v_j\in B$ for all $j\ge J;$
\item given $y=\{y_j\}_{j\in\N},$ where $y_j\in \bbR^m,$ and given a function $h:\bbR^{\ell} \to \bbR^m,$ choose $w$ to keep $|y_j-h(v_j)|$ small in some sense.
\end{itemize}

The third option is most relevant in the context of data assimilation \index{data assimilation},
and so we focus on it. In this context we will consider controllers of the
form $w_j=K\bigl(y_j-h(v_j)\bigr)$ so that
\begin{equation}
\label{eq:ctrl}
v_{j+1}=\PPsi(v_j) + K\bigl(y_j-h(v_j)\bigr).
\end{equation}
A key question is then how to choose $K$ to ensure the desired property.
We present a simple example which illustrates this.
\begin{example}
Let $\ell=m=1$, $\PPsi(v)=\lambda v$ and $h(v)=v.$ We assume that the data 
$\{y_j\}_{j \in \bbN}$ is given by $y_{j+1}= \vd_{j+1}$ where $\vd_{j+1}= \lambda \vd_j.$ Thus the
data is itself generated by the uncontrolled dynamical system. We wish to use
the controller to ensure that the solution of the
controlled system is close to the data $\{y_j\}_{j \in \N}$
generated by the uncontrolled dynamical system, and hence to the solution of
the uncontrolled dynamical system itself.

Consider the controlled dynamical system
\begin{align*}
v_{j+1} &=\PPsi(v_j) + K\left(y_j - h(v_j)\right)\\
 &=\lambda v_j + \underbrace{K(y_j -v_j)}_{w_j},  \quad \quad j\ge 1.
\end{align*}
and assume that $v_{0}\neq \vd_{0}$. We are interested in whether $v_j$ 
approaches $\vd_j$ as $j \to \infty.$ 

To this end suppose that $K$ is chosen so that $|\lambda-K|<1.$ Then note that 
\begin{equation*}
\vd_{j+1} = \lambda \vd_j + K\underbrace{(y_j-\vd_j)}_{=0}.
\end{equation*}
Hence $e_j=v_j-\vd_j$ satisfies
\begin{equation*}
e_{j+1} = (\lambda-K)e_j
\end{equation*}
and 
\begin{equation*}
|e_{j+1}| = |\lambda-K| |e_j|.
\end{equation*}
Since we have chosen $K$ so that $|\lambda-K|<1$ 
then we have $|e_j|\to 0$ as $j\to\infty.$ Thus 
the controlled dynamical system approaches the solution of the uncontrolled 
dynamical system as $j\to\infty.$ This is prototypical of certain data
assimilation algorithms that we will study in Chapter \ref{sec:dtfa}.
\end{example}

It is also of interest to consider continuous time controllers $\{w(t)\}_{t\ge 0}$ for differential equations 
\begin{equation*}
\frac{dv}{dt} = f(v) + w.
\end{equation*}
Again, the goal is to choose $w$ to achieve some objective analogous to those
described in discrete time.

\section{Probability Metrics}
\label{ssec:pm}

Since we will frame data assimilation \index{data assimilation} in terms of probability, natural measures
of robustness of the problem will require the idea of distance \index{distance} between
probability measures\index{probability measure}. 
Here we introduce basic metric properties,
and then some specific distances on probability measures, and their
properties.
 
\subsection{Metric Properties}
\begin{definition}
A metric\index{metric} on a set $X$ is a function $d: X\times X \to \bbR^+$ (distance)\index{distance} satisfying the following properties:
\begin{itemize}
\item coincidence: $d(x,y) = 0$ iff $x=y;$ 
\item symmetry: $d(x,y)=d(y,x);$
\item triangle: $d(x,z)\le d(x,y)+d(y,z).$
\end{itemize}
\end{definition}
\begin{example}
Let $X=\bbR^{\ell}$ viewed as a normed vector space with 
norm $\|\cdot\|$; for example we might take $\|\cdot\|=|\cdot|$, the
Euclidean norm. Then the function $d:\bbR^{\ell}\times \bbR^{\ell} \to \bbR^+$ given by $d(x,y):=\|x-y\|$ defines a metric. Indeed
\begin{itemize}
\item $\|x-y\|=0$ iff $x=y.$
\item $\|x-y\|= \|y-x\|.$
\item $\|x-z\|=\|x-y+y-z\|\le \|x-y\| + \|y-z\|.$
\end{itemize}
from properties of norms.
\end{example}

\subsection{Metrics on Spaces of Probability Measures}
Let $\calM$ denote the space of probability measures on $\bbR^{\ell}$ with strictly positive Lebesgue density on $\bbR^{\ell}.$
Throughout this section we let $\mu$ and $\mu'$ be two probability measures on $\calM$, and let $\rho$ and $\rho'$ denote the corresponding densities;
recall that we assume that these densities are positive everywhere,
in order to simplify the presentation.
We define two useful metrics on probability measures.
\begin{definition} The {\em total variation distance}\index{distance!total variation}\index{metric!total variation} on $\calM$ is
defined by
\begin{align*}
\dtv(\mu,\mu')&=\frac{1}{2}\int_{\bbR^{\ell}} |\rho(u)-\rho'(u)|\,du\\
&=\frac{1}{2}\bbE^\mu \left|1-\frac{\rho'(u)}{\rho(u)}\right|.
\end{align*}
\label{def:tvd}
\end{definition}

Thus the total variation distance is half of the $L^1$ norm of the difference
of the two pdfs.
Note that clearly $\dtv(\mu,\mu')\ge 0.$ Also
\begin{align*}
\dtv(\mu,\mu')&\le \frac{1}{2} \int_{\bbR^{\ell}} |\rho(u)|\,du + \frac{1}{2} \int_{\bbR^{\ell}} |\rho'(u)|\,du\\
&= \frac{1}{2} \int_{\bbR^{\ell}} \rho(u)\,du + \frac{1}{2} \int_{\bbR^{\ell}} \rho'(u)\,du\\
&=1.
\end{align*}
Note also that $\dtv$ may be characterized as
\begin{equation}
\label{eq:vtd}
\dtv(\mu,\mu') = \frac12 {\rm sup}_{|f|_{\infty} \leq 1} |\bbE^\mu(f)-\bbE^{\mu'}(f)|=
\frac12{\rm sup}_{|f|_{\infty} \leq 1} |\mu(f)-\mu'(f)|
\end{equation}
where we have used the convention that $\mu(f)=\bbE^\mu(f)=\int_{\bbR^\ell} f(v)\mu(dv)$ and $|f|_{\infty}=\sup_{u} |f(u)|.$

\begin{definition} The {\em Hellinger distance}\index{distance!Hellinger}\index{metric!Hellinger} on $\calM$ is defined by
\begin{align*}
\dhh(\mu,\mu') &= \left(  \frac12 \int_{\bbR^{\ell}} \left(\sqrt{\rho(u)} - \sqrt{\rho'(u)}\right)^2 \,du \right)^{1/2} \\
&= \left( \frac12 \bbE^{\mu} \left( 1- \sqrt{\frac{\rho'(u)}{\rho(u)}} \right)^2 \right)^{1/2}.
\end{align*}
\end{definition}

Thus the Hellinger distance is a multiple of the $L^2$ distance between
the square-roots of the two pdfs. Again clearly $\dhh(\mu,\mu')\ge 0.$ Also 
\begin{equation*}
\dhh(\mu,\mu')^2 \le \frac12 \int_{\bbR^{\ell}} \left( \rho(u) + \rho'(u) \right)\, du = 1.
\end{equation*}
We also note that the Hellinger and TV distances can be written in a symmetric
way and satisfy the triangle inequality -- they are indeed valid distance metrics on the space of probability
measures\index{probability measure}.

\begin{lemma}
\label{lem:tvh}
The total variation and Hellinger distances satisfy
\begin{equation*}
0\le \frac{1}{\sqrt{2}} \dtv(\mu,\mu') \le \dhh(\mu,\mu')\le \dtv(\mu,\mu')^{1/2} \le 1
\end{equation*}
\end{lemma}
\begin{proof}
The upper and lower bounds of, respectively, $0$ and $1$ are proved above. 
We show first that $\frac{1}{\sqrt{2}} \dtv(\mu,\mu') \le \dhh(\mu,\mu').$ Indeed, by the Cauchy-Schwarz inequality,
\begin{align*}
\dtv(\mu,\mu') &= \frac12 \int_{\bbR^{\ell}} \left| 1- \sqrt{\frac{\rho'(u)}{\rho(u)}} \right| \left| 1+ \sqrt{\frac{\rho'(u)}{\rho(u)}} \right| \rho(u)\, du \\
&\le \left(  \frac12 \int_{\bbR^{\ell}} \left| 1- \sqrt{\frac{\rho'(u)}{\rho(u)}} \right|^2\rho(u) \, du  \right)^{1/2}     
\left( \frac12 \int_{\bbR^{\ell}} \left| 1+ \sqrt{\frac{\rho'(u)}{\rho(u)}} \right|^2\rho(u) \, du  \right)^{1/2}  \\   
& \le \dhh(\mu,\mu') \left( \int_{\bbR^{\ell}} \left| 1 + \frac{\rho'(u)}{\rho(u)} \right| \rho(u)\,du \right)^{1/2} \\
&= \sqrt{2} \dhh(\mu,\mu').
\end{align*}
Finally, for the inequality $\dhh(\mu,\mu')\le \dtv(\mu,\mu')^{1/2}$ note that 
\begin{equation*}
|\sqrt{a}-\sqrt{b}|\le \sqrt{a} + \sqrt{b} \quad \forall a,b>0.
\end{equation*} 
Therefore,
\begin{align*}
\dhh(\mu,\mu')^2 &= \frac12 \int_{\bbR^{\ell}} \left|1- \sqrt{\frac{\rho'(u)}{\rho(u)}} \right| \left|1- \sqrt{\frac{\rho'(u)}{\rho(u)}} \right| \rho(u) \,du \\
&\le  \frac12 \int_{\bbR^{\ell}} \left|1- \sqrt{\frac{\rho'(u)}{\rho(u)}} \right| \left|1+ \sqrt{\frac{\rho'(u)}{\rho(u)}} \right| \rho(u) \,du \\
&= \frac12 \int_{\bbR^{\ell}} \left|1-\frac{\rho'(u)}{\rho(u)} \right| \rho(u)\, du \\
&= \dtv(\mu,\mu').
\end{align*}
\end{proof}

Why do we bother to introduce the Hellinger\index{distance!Hellinger} 
distance, rather than
working with the more familiar total variation\index{distance!total variation}?
The answer stems from the following two lemmas.

\begin{lemma}
\label{lem:1h}
Let $f:\R^\ell\to\R^p$ be such that
$$(\bbE^{\mu}|f(u)|^2+\bbE^{\mu'}|f(u)|^2)<\infty.$$
Then 
\begin{equation}\label{eq:dtf8}|\bbE^{\mu}f(u)-\bbE^{\mu'}f(u)|\leq2(\bbE^{\mu}|f(u)|^2+\bbE^{\mu'}|f(u)|^2)^\frac12d_{{\rm{\tiny Hell}}}(\mu, \mu').
\end{equation}
As a consequence 
\begin{equation}\label{eq:dtf808}|\bbE^{\mu}f(u)-\bbE^{\mu'}f(u)|\leq2(\bbE^{\mu}|f(u)|^2+\bbE^{\mu'}|f(u)|^2)^\frac12d_{{\rm{\tiny tv}}}(\mu, \mu')^{\frac12}.
\end{equation}
\end{lemma}
\begin{proof}
In the following all integrals are over $\bbR^{\ell}$. Now 
\begin{align*}|\bbE^{\mu}f(u)-\bbE^{\mu'}f(u)|&\leq\int|f(u)||\rho(u)-\rho'(u)|du\\
&=\int\sqrt{2}|f(u)||\sqrt{\rho(u)}+\sqrt{\rho'(u)}|\cdot\frac1{\sqrt{2}}|\sqrt{\rho(u)}-\sqrt{\rho'(u)}|du\\
&\leq\left(\int2|f(u)|^2|\sqrt{\rho(u)}+\sqrt{\rho'(u)}|^2du\right)^\frac12\left(\frac12\int|\sqrt{\rho(u)}-\sqrt{\rho'(u)}|^2du\right)^{\frac12}\\
&\leq\left(\int4|f(u)|^2(\rho(u)+\rho'(u))du\right)^\frac12\left(\frac12\int\left(1-\frac{\sqrt{\rho'(u)}}{\sqrt{\rho(u)}}\right)^2\rho(u)du\right)^{\frac12}\\
&=2(\bbE^{\mu}|f(u)|^2+\bbE^{\mu'}|f(u)|^2)^\frac12d_{{\rm{\tiny Hell}}}(\mu, \mu').
\end{align*}
Thus \eqref{eq:dtf8} follows.
The bound \eqref{eq:dtf808} follows from Lemma \ref{lem:tvh}.
\end{proof}

\begin{remark}
\label{lem:thel}
The preceding lemma shows that,
if two measures $\mu$ and $\mu'$ are ${\mathcal O}(\epsilon)$
close in the Hellinger metric, and if the function $f(u)$ is
square integrable with respect to $u$ distributed according to
$\mu$ and $\mu'$, then expectations of $f(u)$ with
respect to $\mu$ and $\mu'$ are also ${\mathcal O}(\epsilon)$ close.
It also shows that, under the same assumptions on $f$, 
if two measures $\mu$ and $\mu'$ are ${\mathcal O}(\epsilon)$
close in the total variation metric, then expectations of $f(u)$ with
respect to $\mu$ and $\mu'$ are only ${\mathcal O}(\epsilon^{\frac12})$ 
close. This second result is sharp and
to get ${\mathcal O}(\epsilon)$ closeness of expectations using
${\mathcal O}(\epsilon)$ closeness in the TV metric requires a
stronger assumption on $f$, as we now show.
\end{remark}

\begin{lemma}
\label{lem:1tv}
Assume that $|f|$ is finite almost surely with respect to both $\mu$ and 
$\mu'$ and denote the almost
sure upper bound on $|f|$ by $\fm.$ Then
$$|\bbE^{\mu}f(u)-\bbE^{\mu'}f(u)| \le 2\fm \dtv(\mu, \mu').$$
\end{lemma}  
\begin{proof}
Under the given assumption on $f$,
\begin{align*}|\bbE^{\mu}f(u)-\bbE^{\mu'}f(u)|&\leq\int|f(u)||\rho(u)-\rho'(u)|du\\
&\leq 2\fm\left(\frac12\int|\rho(u)-\rho'(u)|du\right)\\
&\leq 2\fm\left(\frac12\int\left|1-\frac{\rho'(u)}{\rho(u)}\right|\rho(u)du\right)\\
&=2\fm \dtv(\mu, \mu').
\end{align*}
\end{proof}

The implication of the preceding two lemmas and remark
is that it is natural to work with the Hellinger\index{metric!Hellinger} 
metric, rather than the total variation metric, whenever considering
the effect of perturbations of the measure on expectations of
functions which are square integrable, but not bounded.

\section{Probabilistic View of Dynamical Systems}
\label{ssec:pds}

Here we look at the natural connection between dynamical systems,
and the underlying dynamical system that they generate on probability
measures\index{probability measure}. 
The key idea here is that the Markovian propagation of
probability measures is linear, even when the underlying dynamical
system is nonlinear. This advantage of linearity is partially offset
by the fact that the underlying dynamics on probability distributions
is infinite dimensional, but it is nonetheless a powerful perspective
on dynamical systems. Example \ref{examplegauss} provides a nice introductory
example demonstrating the probability distributions carried by
a stochastic dynamical system; in that case the probability
distributions are Gaussian and we explicitly characterize their evolution
through the mean and covariance.  The idea of mapping probability
measures under the dynamical system can be generalized, but
is typically more complicated because the probability distributions
are typically not Gaussian and not characterized by a finite number of
parameters.

\subsection{Markov\index{Markov kernel} Kernel}
\label{ssec:mk}
\begin{definition}
$p: \bbR^{\ell} \times \calB(\bbR^{\ell}) \to \bbR^+$ is a {\em Markov kernel}\index{Markov kernel} if: 
\begin{itemize}
\item for each $x\in\bbR^{\ell},$ $p(x,\cdot)$ is a probability measure on $\bigl(\bbR^{\ell}, \calB(\bbR^{\ell})\bigr);$
\item $x\mapsto p(x, A)$ is $\calB(\bbR^{\ell})$-measurable for all $A \in \calB(\bbR^{\ell}).$
\end{itemize}
\end{definition}
The first condition is the key one for the material in this book:
the Markov kernel at fixed $x$ describes the probability distribution
of a new point $y \sim p(x,\cdot).$ By iterating on this we may
generate a sequence of points which constitute a sample from 
the distribution of the Markov chain\index{Markov chain}, as described
below, defined by the Markov kernel.\index{Markov kernel} 
The second measurability\index{measurable} condition ensures
an appropriate mathematical setting for the problem, but an in-depth
understanding of this condition is not essential for the reader of this book. 
In the same way that we use $\bbP$ to denote both the probability
measure and its pdf, we sometimes use $p(x,\cdot):\bbR^{\ell} \to \bbR^+$, for each
fixed $x \in \bbR^{\ell}$, to denote the corresponding pdf of 
the Markov kernel\index{Markov kernel} from the preceding definition. 

Consider the stochastic dynamical system 
\begin{equation*}
v_{j+1}=\PPsi(v_j) + \xi_j,
\end{equation*}
where $\xi=\{\xi_j\}_{j\in\Z^+}$ is an $\iid$ sequence distributed according 
to probability measure\index{probability measure}
on $\bbR^{\ell}$ with density $\rho(\cdot).$ We assume that the initial
condition $v_0$ is possibly random, but independent of $\xi.$
Under these assumptions on the probabilistic structure,
we say that $\{v_j\}_{j \in \Z^+}$ is a {\em Markov chain}\index{Markov chain}.
For this Markov chain\index{Markov chain} we have
\begin{equation*}
\bbP(v_{j+1}|v_j) = \rho\bigl(v_{j+1}-\PPsi(v_j)\bigr);
\end{equation*}
thus
\begin{equation*}
\bbP(v_{j+1}\in A|v_j) =\int_A \rho\bigl(v_{j+1}-\PPsi(v_j)\bigr) \,dv.
\end{equation*}
In fact we can define a Markov Kernel 
\begin{equation*}
p(u,A)= \int_A \rho\bigl(v-\PPsi(u)\bigr) \,dv,
\end{equation*}
with the associated pdf  
\begin{equation*}
p(u,v)=\rho\left(v-\PPsi(u)\right).
\end{equation*}
If $v_j\sim \mu_j$ with pdf $\rho_j$ then 
\begin{align*}
\mu_{j+1} &= \bbP(v_{j+1}\in A) \\
&=\int_{\bbR^{\ell}} \bbP(v_{j+1}\in A|v_j) \bbP(v_j) \, dv_j \\
&= \int_{\bbR^{\ell}} p(u,A) \rho_j(u) \, du.
\end{align*}
And then 
\begin{align*}
\rho_{j+1}(v) &= \int_{\bbR^{\ell}} p(u,v) \rho_j(u) \, du\\
&= \int_{\bbR^{\ell}}  \rho\bigl(v - \PPsi(u) \bigr) \rho_j(u) \, du.
\end{align*}
Furthermore we have a linear dynamical system for the evolution of the pdf
\begin{equation}
\rho_{j+1} = P \rho_j,
\label{eq:rd}
\end{equation}
where $P$ is the integral operator 
\begin{equation*}
(P\pi)(v) = \int_{\bbR^{\ell}}  \rho\bigl(v-\PPsi(u)\bigr) \pi(u) \, du.
\end{equation*}

\begin{example}
Let  $\PPsi: \bbR^{\ell} \to \bbR^{\ell}.$ Assume that $\xi_1 \sim N(0,\sigma^2 I).$ Then 
\begin{equation*}
\rho_{j+1}(v) = \int_{\bbR^{\ell}}  \frac{1}{(2\pi)^{\ell/2} \sigma^{{\ell}}} \exp\left( -\frac{1}{2\sigma^2} |v-\PPsi(u)|^2 \right) \rho_j(u) \, du.
\end{equation*}
As $\sigma\to\infty$ we obtain the deterministic model
\begin{equation*}
\rho_{j+1}(v) = \int_{\bbR^{\ell}} \delta\bigl(v- \PPsi(u)\bigr) \rho_j(u) \, du.
\end{equation*}
\end{example}

For each integer $n \in \bbN$, we use the notation $p^n(u,\cdot)$ to
denote the Markov kernel arising from $n$ steps of the Markov chain;
thus $p^1(u,\cdot)=p(u,\cdot).$
Furthermore, $p^n(u,A)=\rp(u^{(n)} \in A|u^{(0)}=u).$

\subsection{Ergodicity}\index{ergodic}
In many situations we will appeal to {\em ergodic
theorems}\index{ergodic!theorem} to extract information
from sequences $\{v_j\}_{j \in \Z^+}$ generated by a
(possibly stochastic) dynamical systems. Assume that this dynamical
systems is invariant with respect to probability measure $\mu_{\infty}$.
Then, roughly speaking, an  ergodic \index{dynamical system!ergodic}\index{ergodic!dynamical system}
dynamical systems is one for which, 
for a suitable class of test functions $\varphi: \bbR^{\ell} \to \bbR$, and $v_0$ almost surely with respect to the 
{\em invariant measure}\index{invariant measure} $\mu_{\infty},$ the 
Markov chain\index{Markov chain} from the
previous subsection satisfies 
\begin{equation}\label{eq:erg}
\frac{1}{J}\sum_{j=1}^J \varphi(v_j) \to \int_{\bbR^{\ell}} \varphi(v)\mu_{\infty}(dv) = \bbE^{\mu_{\infty}} \varphi(v).
\end{equation}
We say that {\em the time average equals the space average.} 
The preceding identity encodes the idea that the histogram formed by
a single trajectory  $\{v_j\}$ of the Markov chain \index{Markov chain} looks more and more
like the pdf of the underlying invariant measure\index{invariant measure}.
Since the convergence is almost sure with respect to the initial condition,
this implies that the statistics of where the trajectory spends time is, 
asymptotically, independent of the initial condition; this is a very
powerful property.

If the Markov chain \index{Markov chain} has a unique 
{\em invariant density}\index{invariant density} $\rho_{\infty}$, which is
a fixed point  \index{fixed point} of the linear dynamical 
system \eqref{eq:rd}, then it will satisfy 
\begin{equation}
\label{eq:comeon}
\rho_{\infty} = P\rho_{\infty},
\end{equation}
or equivalently
\begin{equation}
\label{eq:comeon2}
\rho_{\infty}(v)=\int_{\bbR^{\ell}} p(u,v)\rho_{\infty}(u) du.
\end{equation}
In the  ergodic\index{ergodic} setting, this equation will have a form of
uniqueness within the class of pdfs and, furthermore, it is often possible
to prove, in some norm, the convergence 
\begin{equation*}
\rho_j\to \rho_{\infty} \quad {\text{as}} \,\, j\to\infty.
\end{equation*}

\begin{example}
Example \ref{examplegauss} generates an ergodic
Markov chain \index{Markov chain}$\{v_j\}_{j \in \Z^+}$
carrying the sequence of pdfs $\rho_j$. Furthermore, each 
$\rho_j$ is the density of a Gaussian $N(m_j,\sigma_j^2).$
If $|\lambda|<1$ then $m_j\to 0$ and $\sigma_j^2 \to \sigma_{\infty}^2$ where
\begin{equation*}
\sigma_{\infty}^2 = \frac{\sigma^2}{1-\lambda^2}.
\end{equation*} 
Thus $\rho_{\infty}$ is the density of a Gaussian $N(0,\sigma_{\infty}^2).$
We then have, 
\begin{equation*}
\frac{1}{J}\sum_{j=1}^J \varphi(v_j)  = \frac{1}{J}\sum_{j=1}^J \varphi\left(\lambda^j v_0 + \sum_{i=1}^{J-1} \lambda^{j-i-1} \xi_i\right) \to \int_{\bbR} \rho_{\infty}(v) \varphi(v)\,dv.
\end{equation*}
\end{example}

\subsection{Bayes' Formula \index{Bayes' formula} as a Map}
\label{ssec:btm}
Recall that Bayes' formula \index{Bayes' formula} states that 
\begin{equation*}
\bbP(a|b) = \frac{1}{\bbP(b)}\bbP(b|a)\bbP(a).
\end{equation*}
This may be viewed as a map from $\bbP(a)$ (what we know about $a$ a priori,
the prior\index{prior}) to $\bbP(a|b)$ (what we know about $a$ once we
have observed the variable 
$b$, the posterior.\index{posterior}) Since 
\begin{equation*}
\bbP(b) = \int_{\bbR^{\ell}} \bbP(b|a) \bbP(a)\, da 
\end{equation*} 
we see that 
\begin{equation*}
\bbP(a|b) = \frac{\bbP(b|a)\bbP(a)}{\int_{\bbR^{\ell}} \bbP(b|a)\bbP(a)\,da} =:L\bbP(a).
\end{equation*}
$L$ is a {\em nonlinear} map which takes pdf $\bbP(a)$ into $\bbP(a|b).$
We use the letter $L$ to highlight the fact that the map is defined, in the
context of Bayesian statistics, by using the likelihood\index{likelihood}
to map prior to posterior.

\section{Bibliography}
\label{ssec:backbib}

\begin{itemize}

\item For background material on probability, as covered in section 
\ref{ssec:p}, the reader is directed to the
elementary textbook \cite{breiman1992probability}, 
and to the more advanced texts \cite{G&Z,williamsm} for further material
(for example the definition of measurable\index{measurable}.)
The book \cite{Nor97}, together with the references therein, provides an
excellent introduction to Markov chains. The book \cite{MT93}
is a comprehensive study of ergodicity\index{ergodic} for
Markov chains\index{Markov chain}; 
the central use of Lyapunov\index{Lyapunov function} functions will
make it particular accessible for readers with a background in dynamical systems.
Note also that Theorem \ref{th21} 
contains a basic ergodic result for Markov chains.

\item Section \ref{ssec:ds} concerns dynamical systems and stochastic
dynamical systems. The deterministic setting is over-viewed in numerous
textbooks, such as \cite{Guck,Wig90}, with more advanced material, related
to infinite dimensional problems, covered in \cite{book:Temam1997}.
The ergodicity\index{ergodic}
of stochastic dynamical systems is over-viewed in \cite{Arn} and targeted 
treatments based on the small noise scenario include \cite{FW,BerGen06}. 
The book \cite{SH96} contains elementary chapters on dynamical systems,
and the book chapter \cite{HS02} contains related material in the context
of stochastic dynamical systems. For the subject of control
theory the reader is directed to \cite{zabczyk2009mathematical}, which
has a particularly good exposition of the linear theory, and
\cite{son98} for the nonlinear setting.

\item Probability metrics are the subject of section \ref{ssec:pm}
and the survey paper \cite{GS02} provides a very readable introduction to
this subject, together with references to the wider literature.

\item Viewing (stochastic) dynamical systems as generating a dynamical system
on the probability measure\index{probability measure}
which they carry is an enormously powerful
way of thinking. The reader is directed to the books \cite{viana1997stochastic} 
and \cite{baladi2000positive} for
overviews of this subject, and further references. 

\end{itemize}

\section{Exercises}
\label{ex:intro}

\begin{enumerate}

\item Consider the ODE
$$\frac{dv}{dt}=v-v^3, \quad v(0)=v_0.$$
By finding the exact solution, determine the one-parameter
semigroup\index{semigroup} $\PPsi(\cdot;t)$ with properties \eqref{eq:sg}.

\item Consider a jointly varying random variable $(a,b) \in \bbR^2$ defined
as follows: $a \sim N(0,\sigma^2)$ and $b|a \sim N(a,\gamma^2).$ Find a
formula for the probability density function \index{probability density function} of $(a,b)$, using
(\ref{eq:cpeq}b), and demonstrate that the random variable is a Gaussian
with mean and covariance which you should specify. Using \eqref{eq:bayes},
find  a formula for the probability density function \index{probability density function} of $a|b$; again 
demonstrate that the random variable is a Gaussian
with mean and covariance which you should specify.

\item Consider two Gaussian densities on $\R$: $\cN(m_1,\sigma_1^2)$
and $\cN(m_2,\sigma_2^2)$.
Show that the Hellinger distance \index{distance!Hellinger} between them is given by
$$\dhh(\mu,\mu')^2=1-\sqrt{\exp\Bigl(-\frac{(m_1-m_2)^2}{2(\sigma_1^2
+\sigma_2^2)}\Bigr)\frac{2\sigma_1\sigma_2}{(\sigma_1^2+\sigma_2^2)}}.$$

\item Consider two Gaussian measures on $\R$: $\cN(m_1,\sigma_1^2)$
and $\cN(m_2,\sigma_2^2)$. Show that the total variation distance \index{distance!total variation}
between the measures tends to zero if $m_2 \to m_1$ and
$\sigma_2^2 \to \sigma_1^2.$ 

\item The Kullback-Leibler divergence\index{Kullback-Leibler
divergence} between
two measures $\mu'$ and $\mu$, with pdfs $\rho'$ and 
$\rho$ respectively, is
$$\dkl(\mu'||\mu)=\int \log\Bigl( \frac{\rho(x)'}{\rho(x)}\Bigr)
\rho'(x)dx.$$
Does $\dkl$ define a metric on probability measures? Justify your answer.
Consider two Gaussian densities on $\R$: $\cN(m_1,\sigma_1^2)$
and $\cN(m_2,\sigma_2^2)$.
Show that
the Kullback-Leibler divergence \index{Kullback-Leibler
divergence}between them is given by
$$\dkl(\mu_1||\mu_2)=\ln\Bigl(\frac{\sigma_2}{\sigma_1}\Bigr)+
\frac12\Bigl(\frac{\sigma_1^2}{\sigma_2^2}-1\Bigr)+\frac{(m_2-m_1)^2}{2\sigma_2^2}.$$

\item Assume that two measures $\mu$ and $\mu'$
have positive Lebesgue densities $\rho$ and $\rho'$ respectively. 
Prove the bounds:
$$\dhh(\mu,\mu')^2 \le \frac12 \dkl(\mu||\mu')\;,\qquad
\dtv(\mu,\mu')^2 \le \dkl(\mu||\mu'),$$
where the Kullback-Leibler divergence \index{Kullback-Leibler
divergence} $\dkl$ is defined in the preceding exercise.

\item Consider the stochastic dynamical system of Example \ref{examplegauss}.
Find explicit formulae for the maps $m_j \mapsto m_{j+1}$ and
$\sigma_j^2 \mapsto \sigma_{j+1}^2.$

\item Directly compute the mean and covariance of $w=a+Lz$ if $z$ 
is Gaussian $N(m,C)$, without using the characteristic function. 
Verify that you obtain the same result as in Lemma \ref{lem:to}. 

\item Prove Lemma \ref{l:qf}.

\item Generalize Definitions \ref{def:dabs} and \ref{def:ga}
to continuous time, as
suggested in Remark \ref{rem:cabz}.

\end{enumerate}

\graphicspath{{./figs/chapter1/}}
\chapter{Discrete Time: Formulation}\label{sec:dtf}


In this chapter we introduce the mathematical framework
for discrete-time data assimilation.
Section \ref{ssec:s} describes the mathematical
models we use for the underlying signal\index{signal}, 
which we wish to recover, and for the data\index{data}, 
which we use for the recovery.
In section \ref{ssec:ge} we introduce a number of
examples used throughout the text to illustrate the
theory. Sections \ref{ssec:sp} and \ref{ssec:fp}
respectively describe two key problems related to
the conditioning of the signal $v$ on the data $y$,
namely {\bf smoothing}\index{smoothing} and 
{\bf filtering}\index{filtering}; in
section \ref{ssec:fasar} we describe how these
two key problems are related.
Section \ref{ssec:wp} proves that the smoothing\index{smoothing} problem
is well-posed \index{well-posed} and, using the connection to filtering\index{filtering}
described in \ref{ssec:fasar}, that the filtering\index{filtering} problem
is well-posed\index{well-posed}; here well-posedness \index{well-posed} refers to continuity
of the desired conditioned probability distribution
with respect to the observed data.
Section \ref{ssec:qual} discusses approaches to evaluating
the quality of data assimilation algorithms. In
section \ref{ssec:ill} we describe various
illustrations of the foregoing theory and conclude the
chapter with section \ref{ssec:bib} devoted to
a bibliographical overview and section \ref{ex:first} containing
exercises.

\section{Set-Up}\label{ssec:s}
We assume throughout the book that
$\PPsi \in C(\R^n,\R^n)$ and consider the Markov chain  \index{Markov chain} $v=\{v_j\}_{j\in\Z^+}$ 
defined by the random map
\begin{subequations}
\label{eq:dtf1}
\begin{eqnarray}
&v_{j+1}=\PPsi(v_j)+\xi_j, \;j\in\Z^+, \\
&v_0\sim \G(m_0,C_0),
\end{eqnarray}
\end{subequations}
where $\xi=\{\xi_j\}_{j\in\Z^+}$ is an i.i.d. \index{i.i.d.} sequence, 
with $\xi_0\sim\G(0,\Sigma)$ and $\Sigma>0$. Because $(v_0,\xi)$ is a random variable,
so too is the solution sequence $\{v_j\}_{j\in\Z^+}$: the 
{\bf signal}\index{signal},
which determines the state of the system at each discrete time instance.
For simplicity we assume that $v_0$ and $\xi$ are independent.
The probability distribution of the random variable $v$ quantifies
the uncertainty in predictions arising from this
{\bf stochastic dynamics \index{stochastic dynamics}} model.

In many applications, models such as \eqref{eq:dtf1} are
supplemented by observations of the system as it evolves;
this information then changes the probability distribution on the
signal, typically reducing the uncertainty. To describe such
situations we assume that we are given {\bf data}\index{data}, or
{\bf observations}\index{observations}, $y=\{y_j\}_{j\in\N}$ defined as 
follows. At each discrete time instance we observe a (possibly nonlinear)
function of the signal, with additive noise \index{noise!additive}:
\begin{equation}\label{eq:dtf2}
y_{j+1}=h(v_{j+1})+\eta_{j+1}, \;j\in\Z^+,
\end{equation}
where $h \in C(\R^n,\R^m)$ and $\eta=\{\eta_j\}_{j\in\N}$ is an i.i.d. \index{i.i.d.} sequence,
independent of $(v_0,\xi)$, with $\eta_1\sim\G(0,\Gamma)$ and $\Gamma>0$. 
The function $h$ is known as the {\bf observation operator}\index{observation operator}.
The objective of {\bf data assimilation}\index{data assimilation} 
is to determine 
information about the signal\index{signal} $v$, given data\index{data} $y$.
Mathematically we wish to solve the problem of conditioning
the random variable $v$ on the observed data $y$, or problems closely
related to this. 
Note that we have assumed that both the model noise 
$\xi$ and the observational noise $\eta$ are Gaussian\index{Gaussian}; 
this assumption is made for convenience only, 
and could be easily generalized.

We will also be interested in the case where the
dynamics is deterministic and \eqref{eq:dtf1} becomes
\begin{subequations}
\label{eq:dtf11}
\begin{eqnarray}\label{eq:dtf11a}
&v_{j+1}=\PPsi(v_j), \;j\in\Z^+, \\
&v_0\sim\G(m_0,C_0).
\end{eqnarray}
\end{subequations}
In this case, which we refer to as 
{\bf deterministic dynamics},\index{deterministic dynamics}
we are interested in the random variable $v_0$, given
the observed data $y$; note that $v_0$ determines all subsequent values
of the signal $v$.

Finally we mention that in many applications the function
$\PPsi$ is the solution operator \index{solution operator} for an ordinary
differential equation (ODE) of the form
\footnote{Here the use of $v=\{v(t)\}_{t \ge 0}$ for the solution of this
equation should be distinguished
from our use of $v=\{v_j\}_{j=0}^{\infty}$ 
for the solution of \eqref{eq:dtf1}.} 
\begin{subequations}\label{eq:ode1}
\begin{eqnarray}
\frac{dv}{dt}&=&f(v), \;t \in (0,\infty),\\
v(0)&=&v_0.
\end{eqnarray}
\end{subequations}
Then, assuming the solution exists for all $t \ge 0$,
there is a one-parameter semi-group of operators
$\PPsi(\cdot;t)$, parametrized by time $t \ge 0$, with
properties defined in \eqref{eq:sg}. 
In this situation we assume that $\PPsi(u)=\PPsi(u;\dt)$, i.e. the solution
operator over $\dt$ time units, where $\dt$ is the time 
between observations;\index{observations}
thus we implicitly make the simplifying assumption that the observations
are made at equally spaced time-points, and note that the state $v_j=v(jh)$
evolves according to (\ref{eq:dtf11a}). 
We use the notation $\PPsi^{(j)}(\cdot)$
to denote the $j-$fold composition of $\PPsi$ with itself.
Thus, in the case of continuous time dynamics, 
$\PPsi(\cdot\,;j\tau)=\PPsi^{(j)}(\cdot)$.

\section{Guiding Examples}\label{ssec:ge}
Throughout these notes we will use the following examples to illustrate the 
theory and algorithms presented.

\begin{example} \label{ex:ex1}
We consider the case of one dimensional linear dynamics where
\begin{equation} \label{eq:ex1}
\PPsi(v)=\lambda v
\end{equation}
for some scalar $\lambda \in \bbR.$
Figure \ref{fig:ex1} compares the behaviour of 
the stochastic dynamics \index{stochastic dynamics} (\ref{eq:dtf1}) and
deterministic dynamics (\ref{eq:dtf11}) for 
the two values $\lambda=0.5$ and $\lambda=1.05$. We
set $\Sigma=\sigma^{2}$ and in
both cases $50$ iterations of the 
map are shown.
We observe that the presence of noise does not significantly alter the 
dynamics of the system for the case when $|\lambda|>1$, since for
both the stochastic and deterministic models  
$|v_{j}|\rightarrow \infty$ as $j \to \infty.$ The effects of stochasticity are
more pronounced when $|\lambda|<1$, since in this case the deterministic 
map satisfies $v_{j} \rightarrow 0$ whilst, for the stochastic model,
$v_{j}$ fluctuates randomly around 0. 

\begin{figure}[h]
\centering
\subfigure[$\lambda=0.5$]{\includegraphics[scale=0.365]{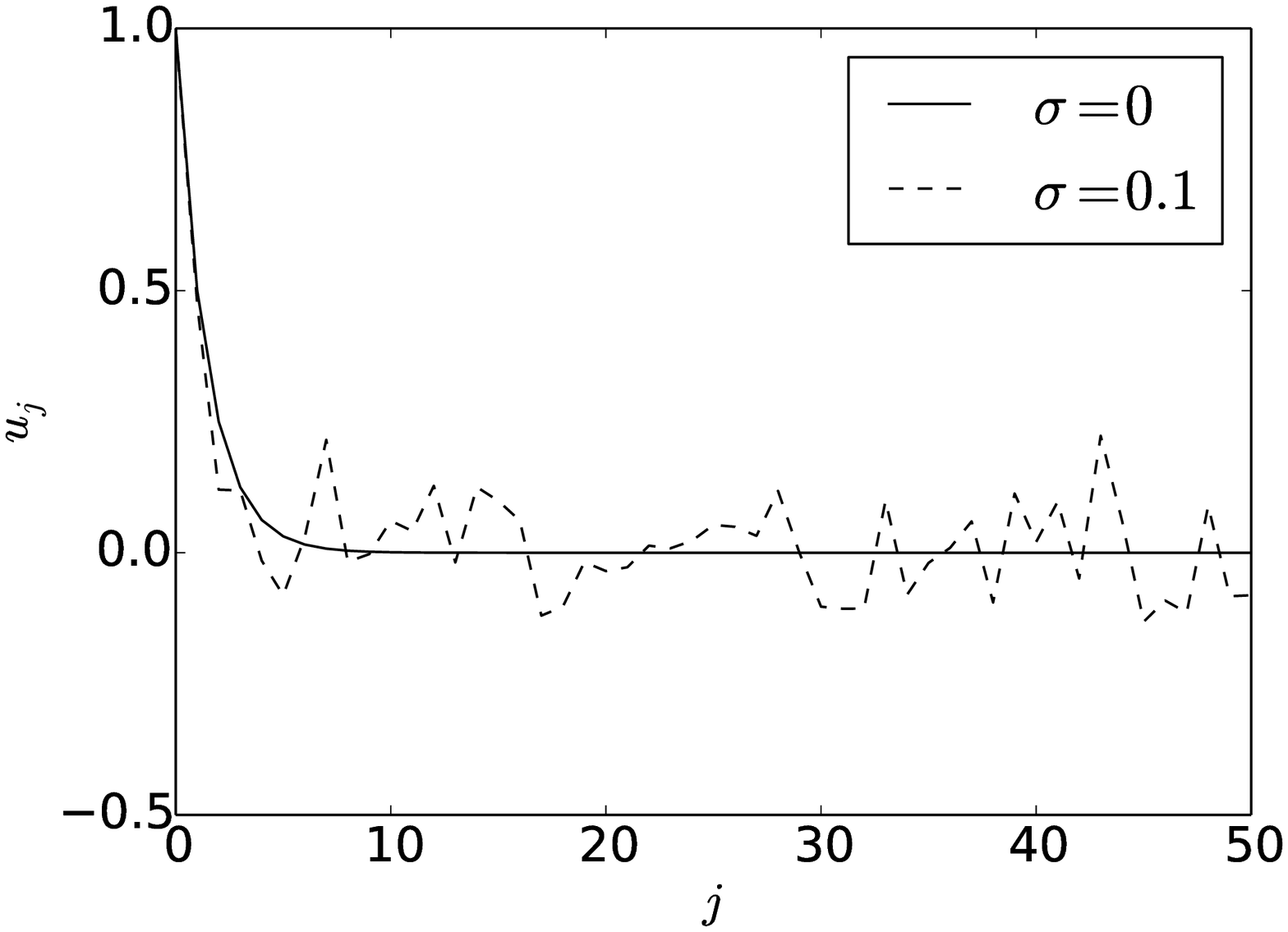}}
\subfigure[$\lambda=1.05$]{\includegraphics[scale=0.365]{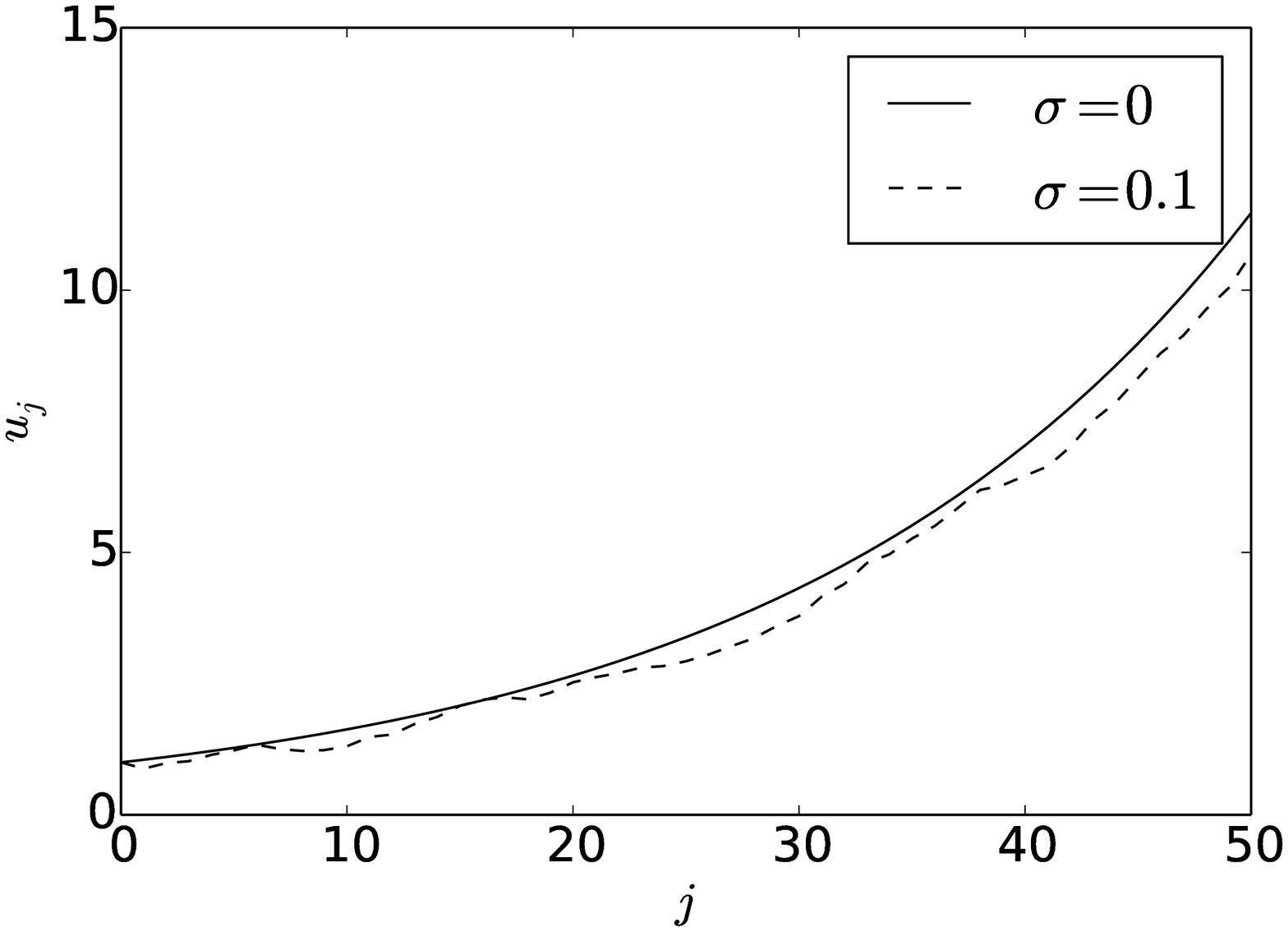}}
\caption{Behaviour of (\ref{eq:dtf1}) for $\PPsi$ given by \eqref{eq:ex1} for different values of $\lambda$ and $\Sigma=\sigma^{2}$.}
\label{fig:ex1}
\end{figure}

Using (\ref{eq:dtf1}a),
together with the linearity of $\PPsi$ and the Gaussianity
of the noise $\xi_{j}$,  we obtain
\[
\bbE(v_{j+1})=\lambda \bbE(v_{j}), \quad \bbE(v^{2}_{j+1})=\lambda^{2} \bbE(v^{2}_{j})+\sigma^{2}.
\] 
If $|\lambda|>1$ then the the second moment 
explodes as $j \to \infty$, as does the modulus of the first moment
if $\bbE(v_0) \ne 0.$
On the other hand, if $|\lambda|<1$, we see (Example \ref{examplegauss})
that $\bbE(v_j) \to 0$ and 
$\bbE (v^{2}_{j}) \to \sigma_{\infty}^2$ where
\begin{equation} 
\label{eq:OU_v}
\sigma^{2}_{\infty}=\frac{\sigma^{2}}{1-\lambda^{2}}.
\end{equation}    
Indeed, since $v_0$ is Gaussian, the model (\ref{eq:dtf1}a)
with linear $\PPsi$ and Gaussian noise $\xi_{j}$ gives rise
to a random variable $v_j$ which is also Gaussian.
Thus, from the convergence of the mean and the second moment of $v_j$,
we conclude that $v_{j}$ converges weakly to the random variable $N(0,\sigma^{2}_{\infty}).$ 
This is an example of ergodicity\index{ergodic} as expressed in 
\eqref{eq:erg}; the invariant measure \index{invariant measure} $\mu_{\infty}$ is the 
Gaussian $N(0,\sigma^{2}_{\infty})$ and the density $\rho_{\infty}$ is the 
Lebesgue density of this Gaussian.  
\end{example}

\begin{example} \label{ex:ex2}
Now consider the case of two dimensional linear dynamics. In this case 
\begin{equation} \label{eq:ex2}
\PPsi(v)=Av,
\end{equation}
with $A$ a $2\times 2$ dimensional matrix of one of the
following three forms $A_{\ell}$:
\[A_{1}=
\left( \begin{array}{cc}
\lambda_{1} & 0 \\
0 & \lambda_{2} \end{array} \right), \quad A_{2}=\left( \begin{array}{cc}
\lambda & \alpha \\
0 & \lambda \end{array} \right), \quad A_{3}=\left( \begin{array}{cc}
0 & 1 \\
-1 & 0 \end{array} \right)
\]
For $\ell=1,2$ the behaviour of (\ref{eq:dtf1}) for $\PPsi(u)=A_\ell u$ 
can be understood from the analysis underling the
previous Example \ref{ex:ex1} and the behaviour is similar, in each
coordinate, depending on whether the $\lambda$ value on the diagonal
is smaller than, or larger than, $1$. 
However, the picture is more interesting 
when we consider the third choice $\PPsi(u)=A_3u$ as, in this case, 
the matrix $A_3$ has purely imaginary eigenvalues and corresponds to a rotation by $\pi/2$ on the plane; this is illustrated in Figure \ref{fig:ex2}a. 
Addition of noise into the dynamics gives a qualitatively different picture: 
now the step $j$ to $j+1$ corresponds to a rotation by $\pi/2$, composed with a random shift of origin; this is illustrated in Figure \ref{fig:ex2}b.

\begin{figure}[h]
\centering
\subfigure[Deterministic dynamics]{\includegraphics[scale=0.365]{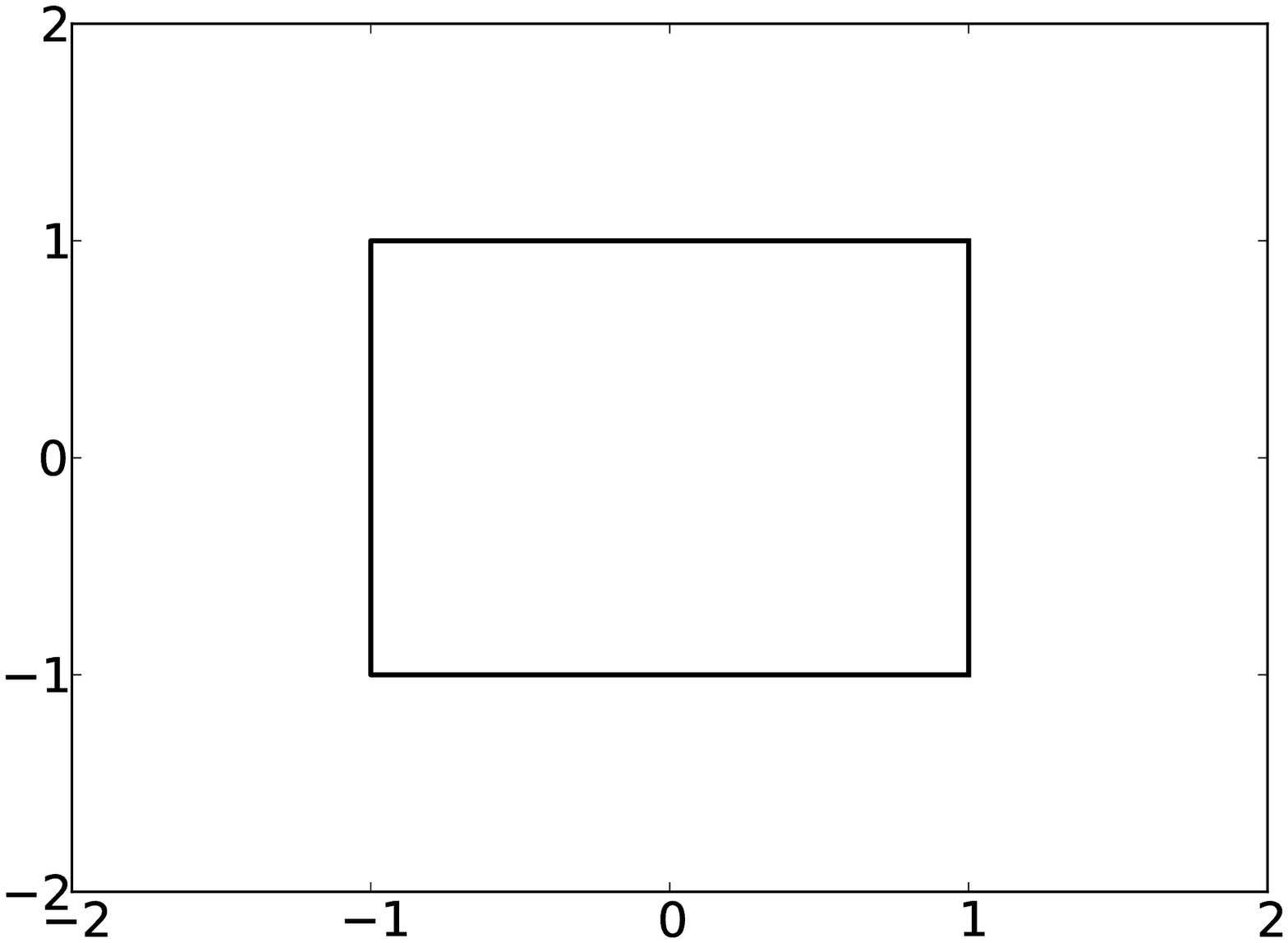}}
\subfigure[Stochastic dynamics, $\sigma=0.1$ ]{\includegraphics[scale=0.365]{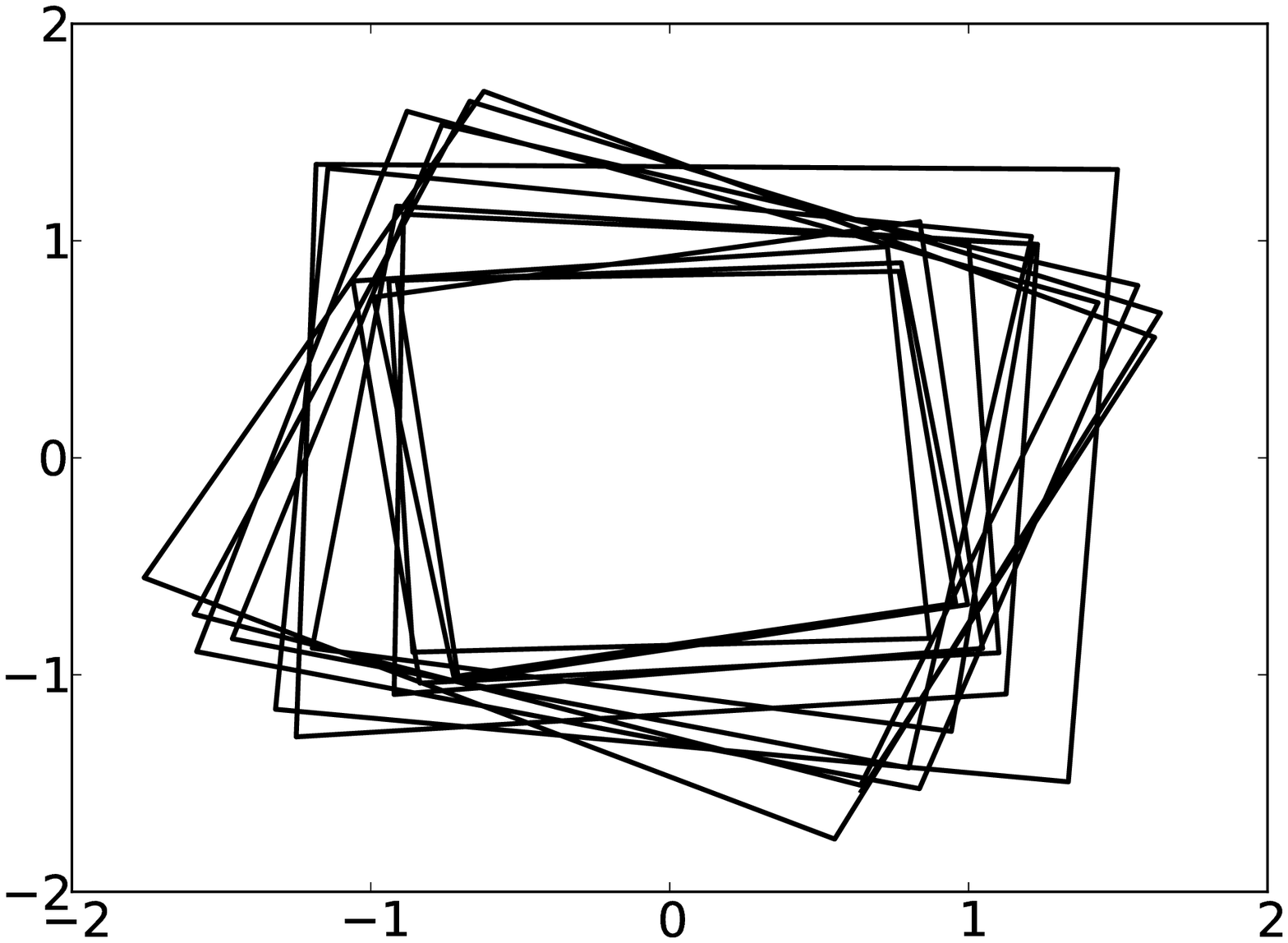}}
\caption{Behaviour of (\ref{eq:dtf1}) for $\PPsi$ given by \eqref{eq:ex2}, and $\Sigma=\sigma^{2}$.}
\label{fig:ex2}
\end{figure}
 \end{example}

\begin{example} \label{ex:ex3}
We now consider our first nonlinear example, namely the
one-dimensional dynamics for which 
\begin{equation} \label{eq:ex3}
\PPsi(v)=\alpha \sin{v}. 
\end{equation}
Figure \ref{fig:p1} illustrates the behaviour of  (\ref{eq:dtf1}) for this 
choice of $\PPsi$, and with $\alpha=2.5$, both for deterministic and stochastic dynamics \index{stochastic dynamics}. 
In the case of deterministic dynamics, Figure \ref{fig:p1}a, we see that eventually 
iterates of the discrete map converge to a period $2$ solution. Although
only one period $2$ solution is seen in  this single trajectory, we
can deduce that there will be another
period $2$ solution, related to this one by the symmetry $u \mapsto -u$.
This second solution is manifest when we consider stochastic dynamics \index{stochastic dynamics}.
\begin{figure}[h!]
\centering
\subfigure[Deterministic dynamics]{\includegraphics[scale=0.365]{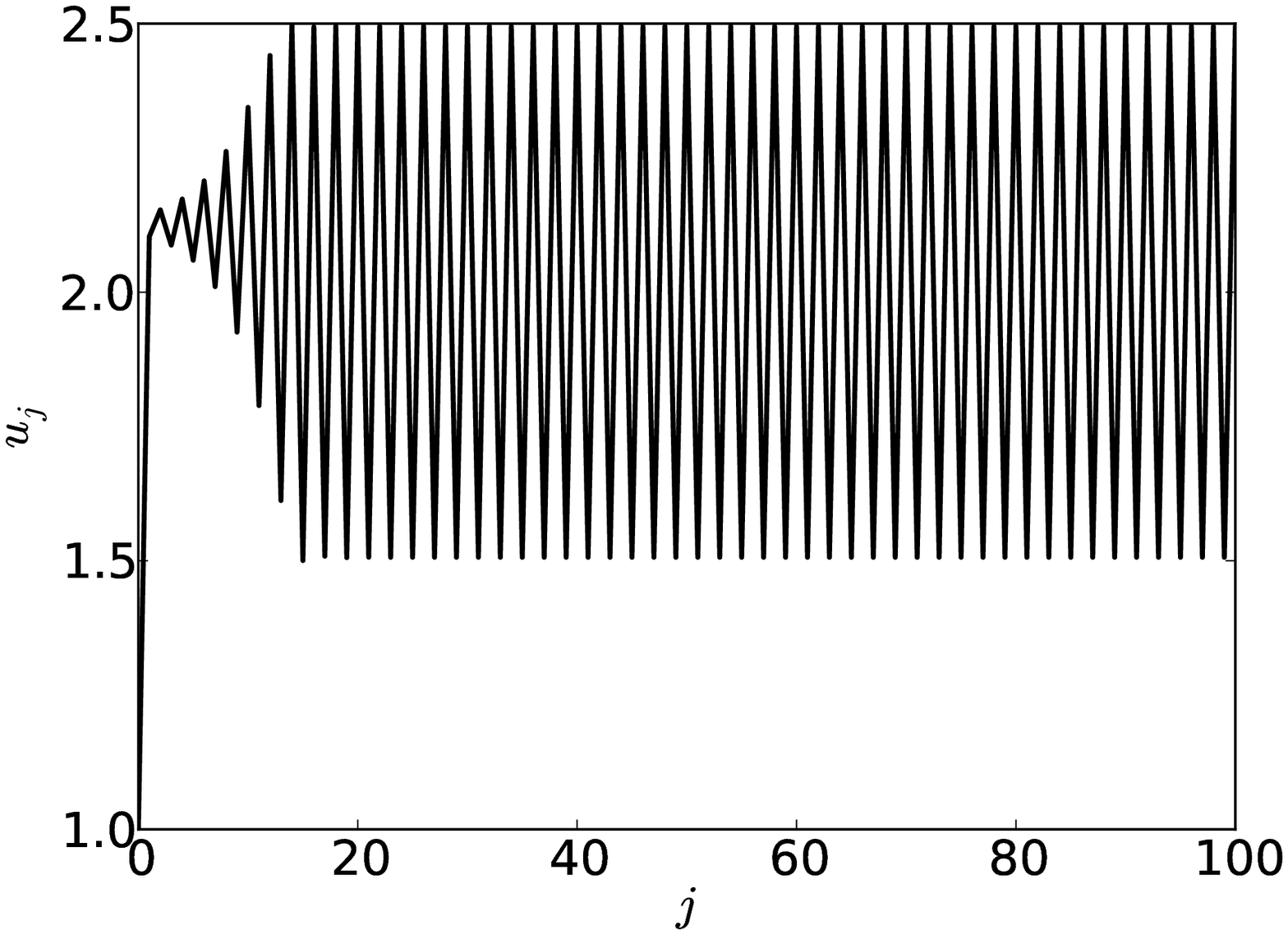}}
\subfigure[Stochastic dynamics, $\sigma=0.25$]{\includegraphics[scale=0.365]{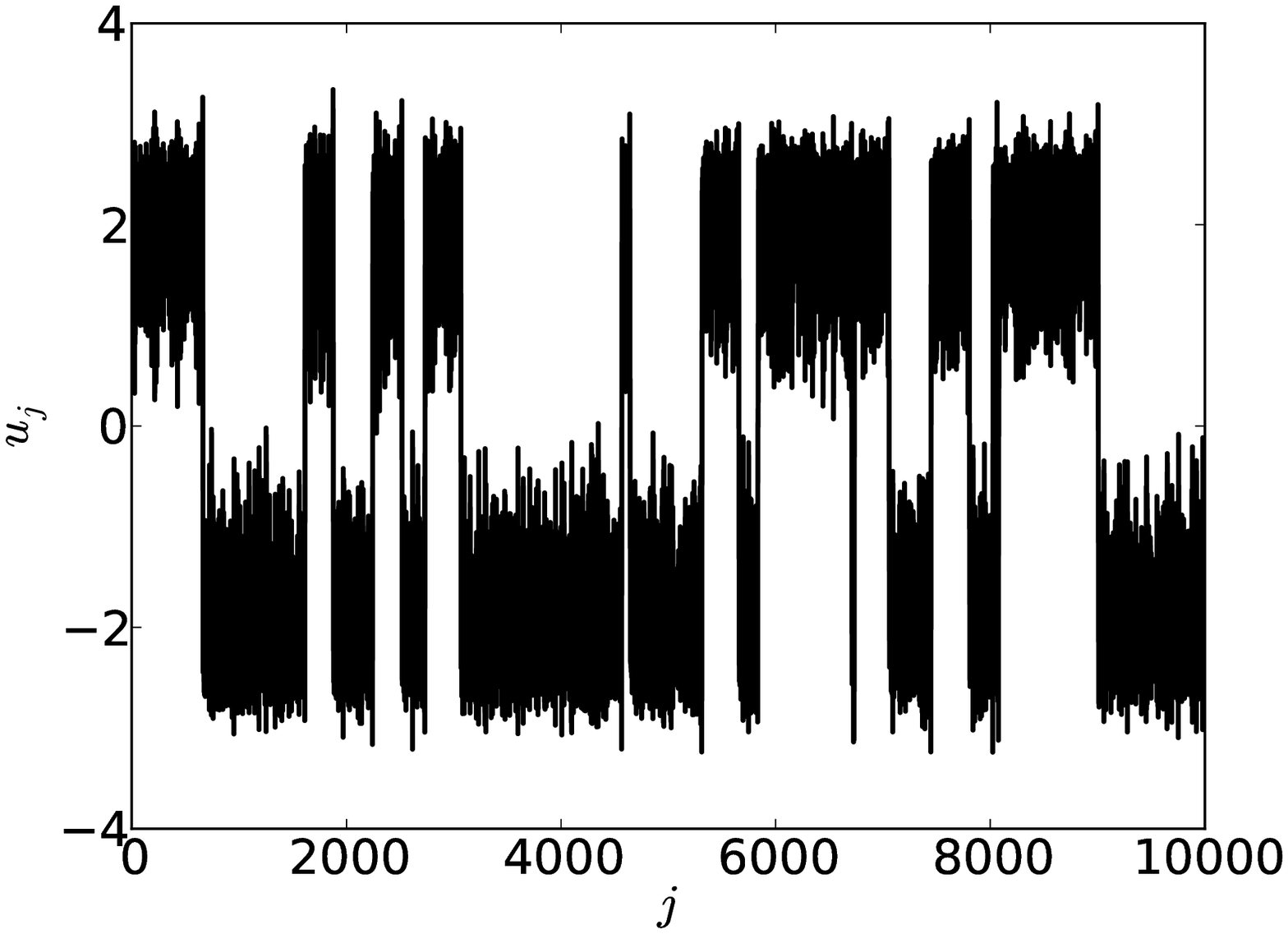}}
\caption{Behaviour of (\ref{eq:dtf1}) for $\PPsi$ given by \eqref{eq:ex3} for $\alpha=2.5$ and $\Sigma=\sigma^{2}$, see also { \tt p1.m} in section \ref{ssec:p1}.}
\label{fig:p1}
\end{figure}
Figure \ref{fig:p1}b
demonstrates that the inclusion of noise significantly changes the 
behaviour of the system. The signal now 
exhibits bistable behaviour and, within each mode of the behavioural
dynamics, vestiges of the period $2$ dynamics may be seen: the upper
mode of the dynamics is related to the period $2$ solution shown in
Figure \ref{fig:p1}a and the lower mode to the period $2$ solution
found from applying the symmetry $u \mapsto -u$ to obtain a second
period $2$ solution from that shown in Figure \ref{fig:p1}a. 

A good way of visualizing ergodicity\index{ergodic}
is via the {\em empirical\index{empirical measure} measure}
or {\em histogram}\index{histogram}, generated by a trajectory of the dynamical system.
Equation \eqref{eq:erg} formalizes the idea that the histogram\index{histogram}, 
in the large $J$ limit, converges to the probability density function \index{probability density function}  
of a random variable, independently of the starting point $v_0.$
Thinking in terms of pdfs of the signal, or functions of the signal, and
neglecting time-ordering information,
is a very useful viewpoint throughout these notes.

Histograms visualize complex dynamical behaviour
such as that seen in Figure \ref{fig:p1}b by ignoring
time-correlation in the signal and simply 
keeping track of
{\em where} the solution goes as time elapses, but not the {\em order}
in which places are visited.
This is illustrated in Figure 
\ref{fig:ex34}a, where we plot the 
histogram\index{histogram} 
corresponding to the dynamics shown in Figure \ref{fig:p1}b,
but calculated using a simulation of length $J=10^7$. 
We observe that the system quickly forgets its initial condition 
and spends an almost equal proportion of time around the 
positive and negative period $2$ solutions of the underlying 
deterministic map. The Figure \ref{fig:ex34}a would change 
very little if the system were started from a different initial condition,
reflecting ergodicity\index{ergodic} of the underlying map.

\end{example}

\begin{figure}[h!]
\centering
\subfigure[Example \ref{ex:ex3}, as in Figure \ref{fig:p1}b]{\includegraphics[scale=0.365]{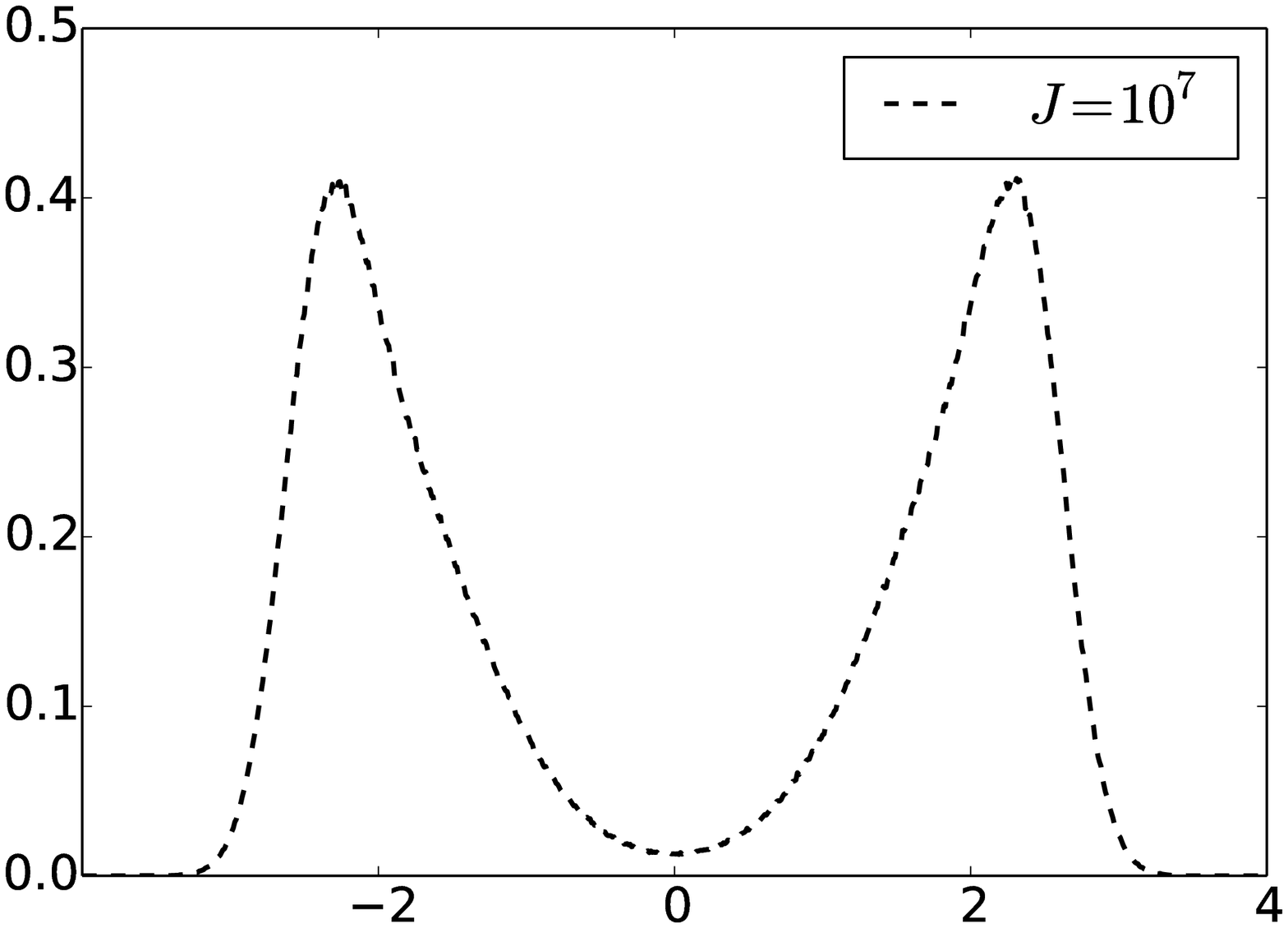}}
\subfigure[Example \ref{ex:ex4}, as in Figure \ref{fig:ex4}b]{\includegraphics[scale=0.365]{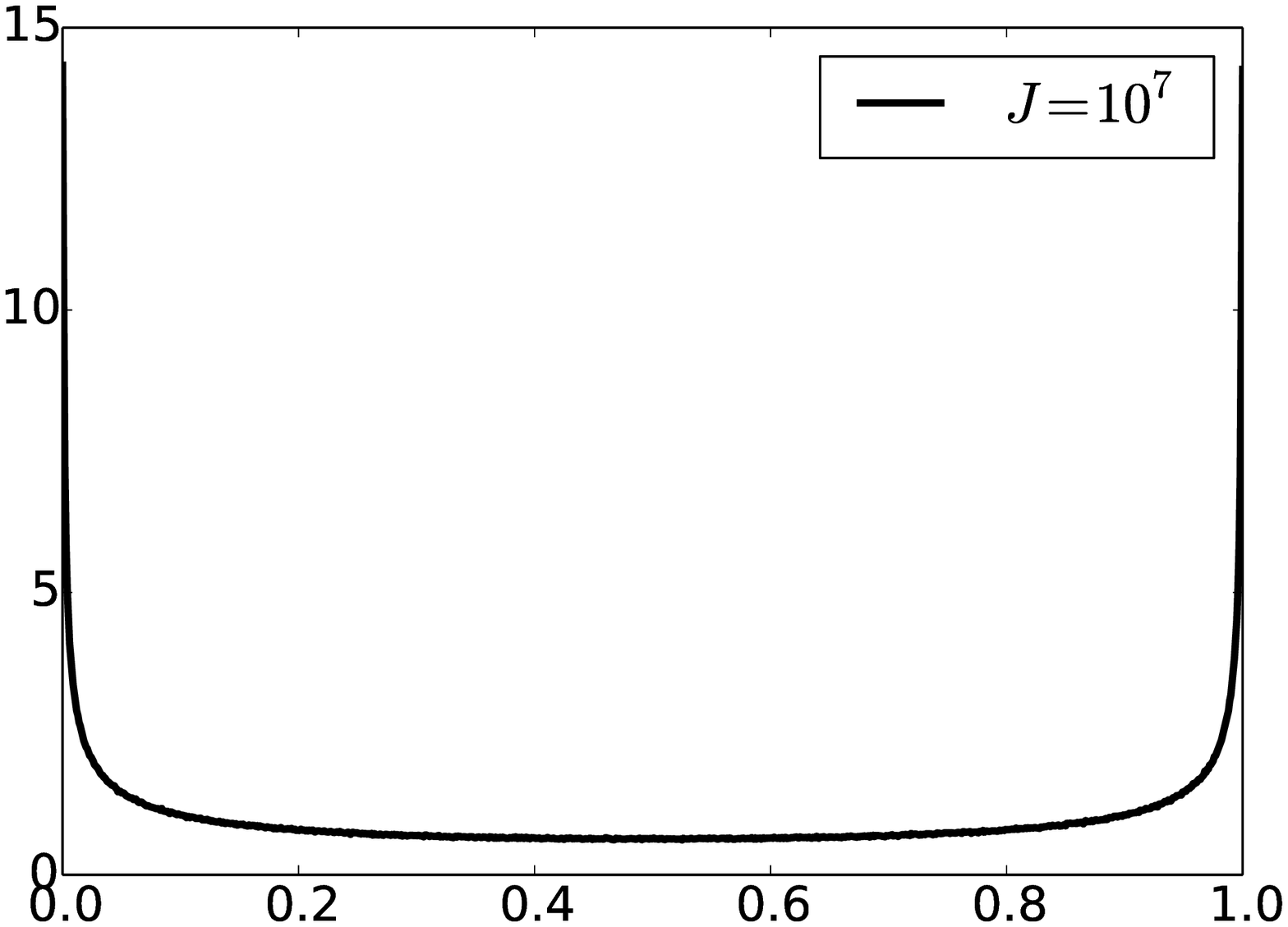}}
\caption{Probability density functions for $v_{j}, j=0, \cdots, J$, for $J=10^7$}
\label{fig:ex34}
\end{figure}

\begin{example} \label{ex:ex4}
We now consider a second one-dimensional and nonlinear map, for which 
\begin{equation} \label{eq:ex4}
\PPsi(v)=r v(1-v). 
\end{equation}
We consider initial data $v_0 \in [0,1]$ noting that, for $r \in [0,4]$,
the signal will then satisfy $v_j \in [0,1]$ for all $j$, in the case
of the deterministic dynamics \index{deterministic dynamics} \eqref{eq:dtf11}. We confine our discussion
here to the deterministic case which can itself exhibit quite rich behaviour.
In particular, the behaviour of (\ref{eq:dtf11}, \ref{eq:ex4}) can be seen in Figure \ref{fig:ex4} for the values of 
\begin{figure}
\centering
\subfigure[Deterministic dynamics, $r=2$]{\includegraphics[scale=0.365]{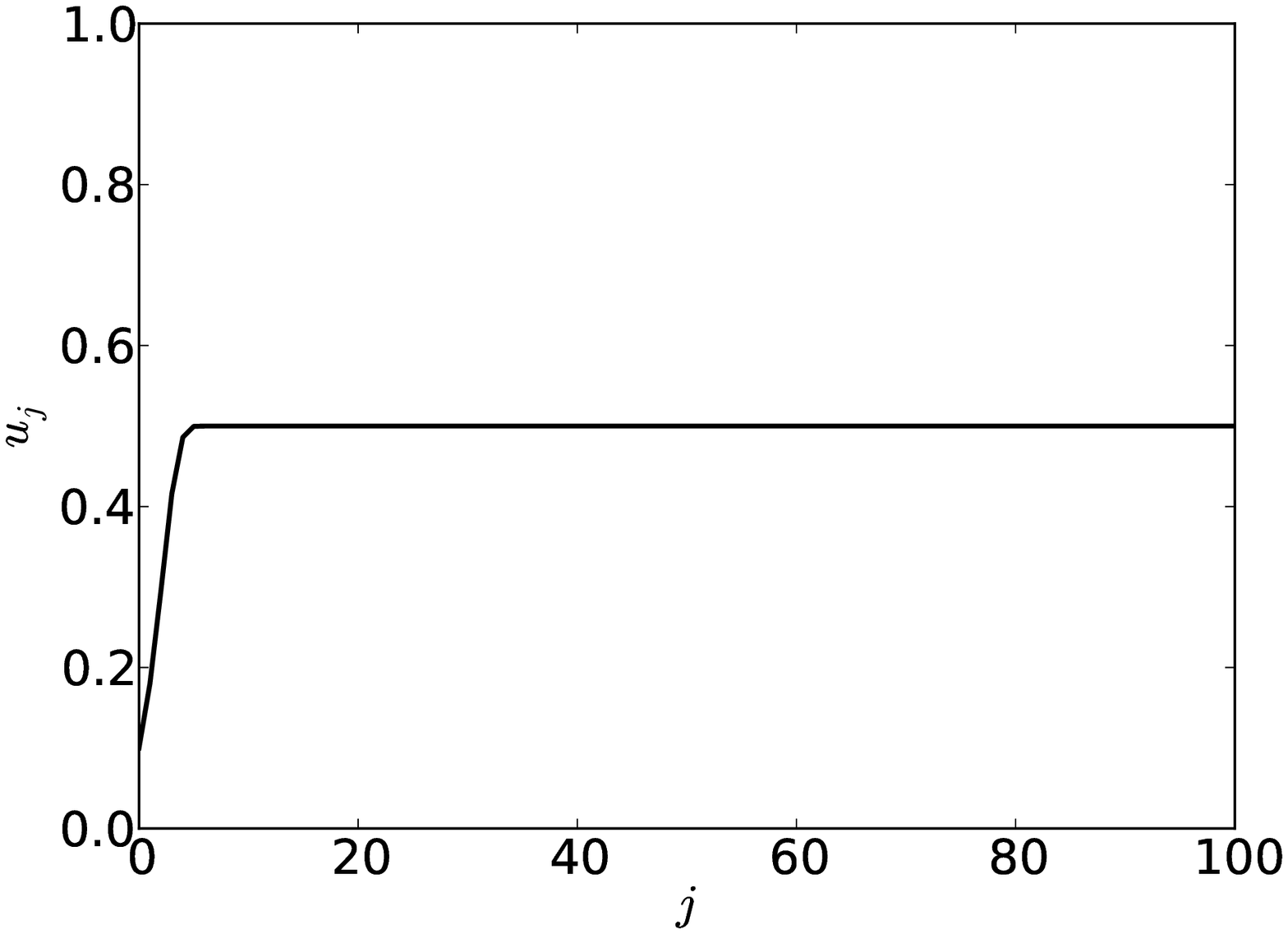}}
\subfigure[Deterministic dynamics, $r=4$]{\includegraphics[scale=0.365]{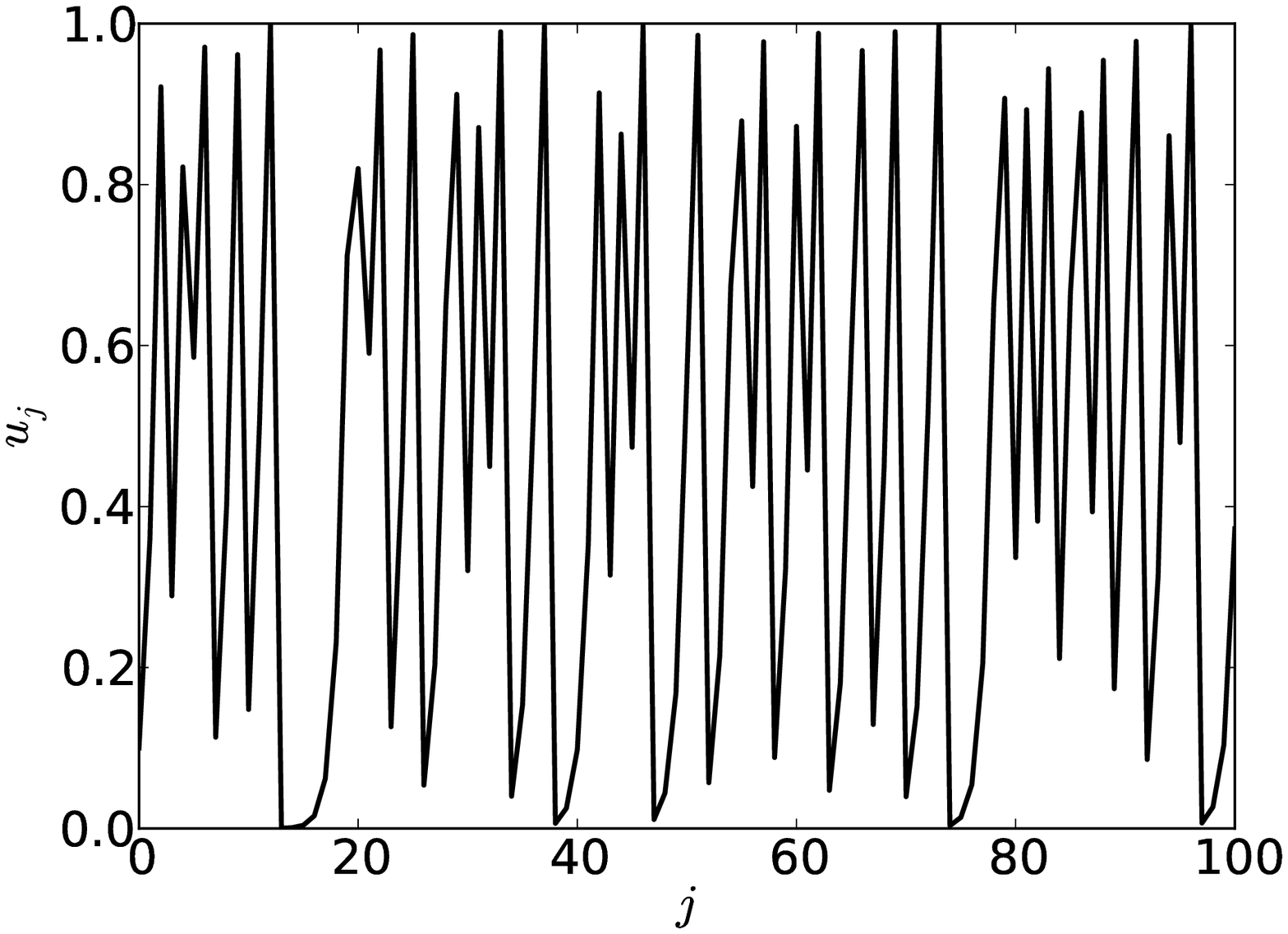}}
\caption{Behaviour of (\ref{eq:dtf1}) for $\PPsi$ given by \eqref{eq:ex4}.}
\label{fig:ex4}
\end{figure}
$r=2$ and $r=4$. These values of $r$ have the desirable property that
it is possible to determine the signal analytically. 
For $r=2$ one obtains
\begin{equation} \label{eq:solr2}
v_j = \frac{1}{2} - \frac{1}{2}(1-2v_0)^{2^{j}},
\end{equation}
which implies that, for any value of $v_{0} \neq 0,1$, $v_{j} \rightarrow 1/2$ as we can also see in Figure \ref{fig:ex4}a. For $v_0=0$ the
solution remains at the unstable fixed point  \index{fixed point} $0$, whilst for
$v_0=1$ the solution maps onto $0$ in one step, and then remains 
there.  In the case $r=4$ the solution is given by 
\begin{equation}
\label{eq:solr3}
v_{j}=4\sin^{2}(2^{j}\pi\theta),\quad \text{with} \quad  v_{0}=4\sin^{2}(\pi\theta)
\end{equation}
This solution can also be expressed in the form
\begin{equation}
\label{eq:solr4}
v_{j}=\sin^{2}(2 \pi z_{j}).
\end{equation}
where 
\[
z_{j+1}=\begin{cases}2z_j, & 0 \le z_j < \frac{1}{2}, \\2z_j -1, & \frac{1}{2} \le z_j < 1, \end{cases} 
\]
and using this formula it is possible to show that this map produces  chaotic dynamics for almost all initial conditions. This is illustrated in Figure \ref{fig:ex4}b, where we plot the first $100$ iterations of the map.  In addition,
 in Figure \ref{fig:ex34}b, we plot the pdf using a long trajectory of $v_{j}$ of length $J=10^7$, {demonstrating the ergodicity\index{ergodic} of the map.}
In fact there is an analytic formula for the steady state value of the pdf 
(the invariant density\index{invariant density}) found as
$J \to \infty$; it is given by 
\begin{equation} \label{eq:logistic_invariant}
\rho(x) =  \pi ^{-1}x^{-1/2}(1-x)^{-1/2}.
\end{equation}   
\end{example}

\begin{figure}
\centering
\subfigure[$u_{1}$ vs $u_{2}$]{\includegraphics[scale=0.365]{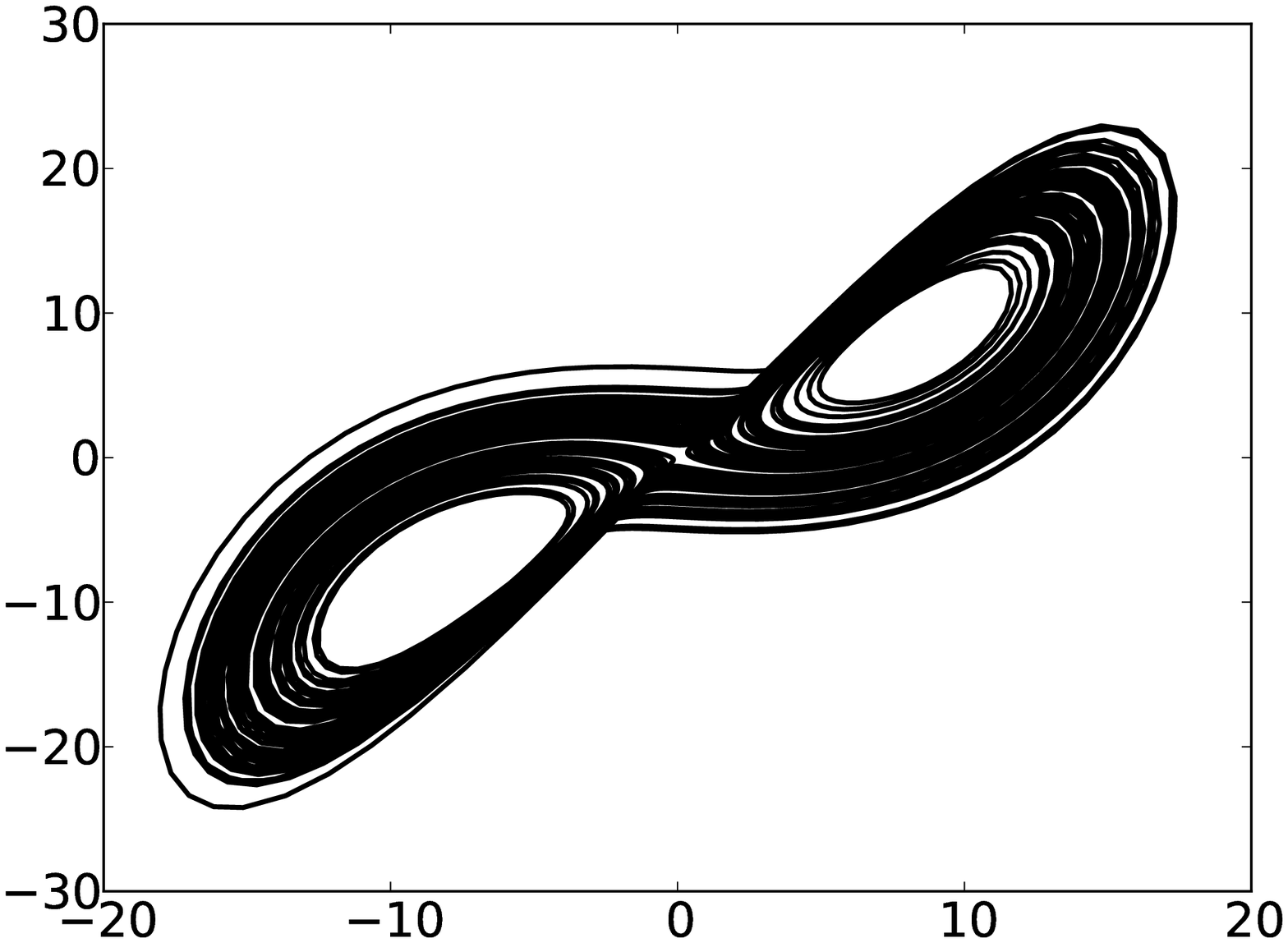}}
\subfigure[$u_{2}$ vs $u_{3}$]{\includegraphics[scale=0.365]{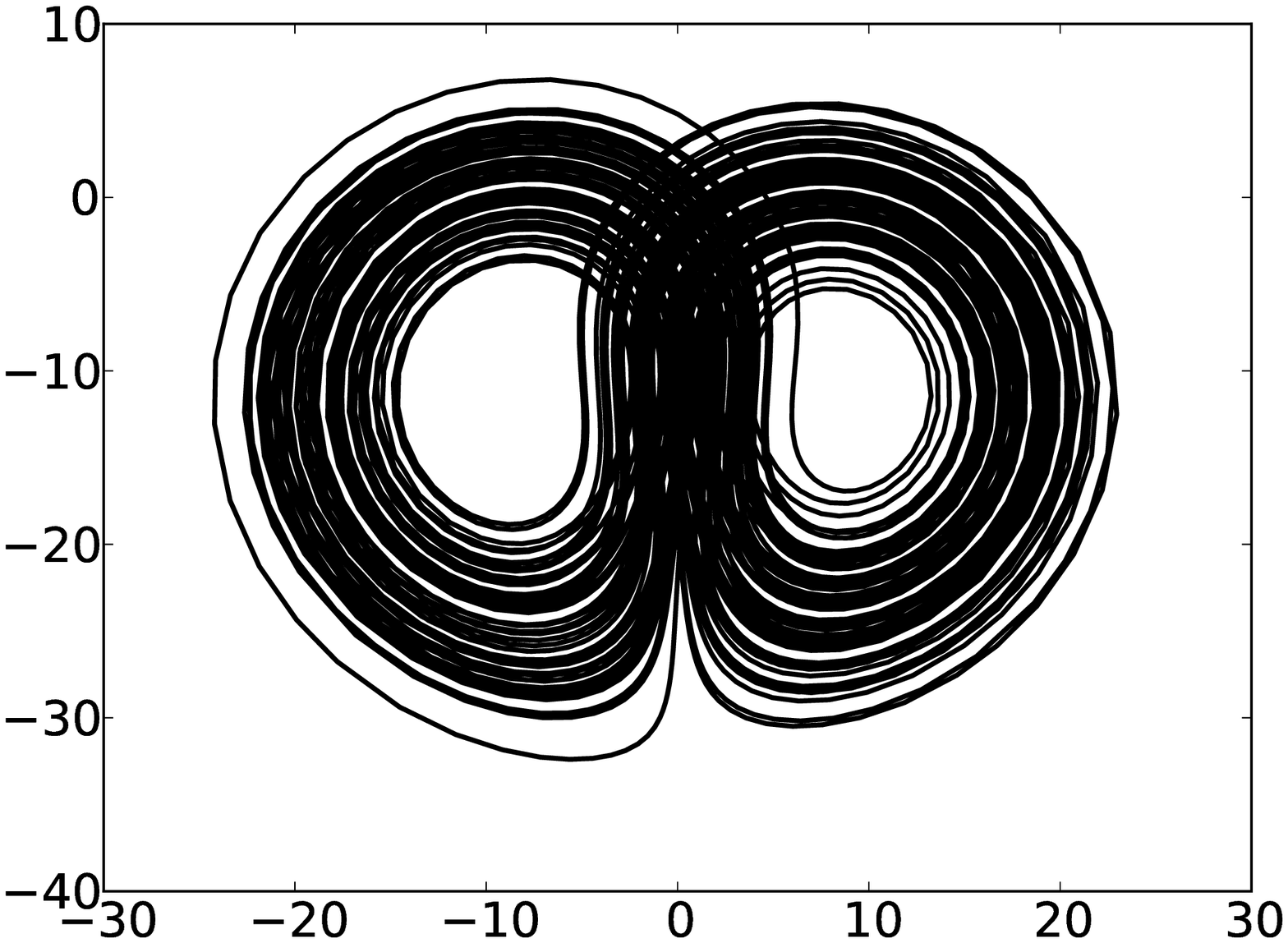}}
\caption{Projection of the Lorenz'63 attractor onto two different pairs of coordinates.}
\label{fig:ex6_1}
\end{figure}

\begin{example}
\label{ex:ex5}
Turning now to maps $\PPsi$ derived from differential equations,
the simplest case is to consider linear autonomous dynamical systems of the form
\begin{subequations}
\begin{eqnarray}
\frac{dv}{dt}&=& Lv,\\
v(0)&=&v_0.
\end{eqnarray}
\end{subequations}
Then $\PPsi(u)=Au$ with $A=\exp(L\tau).$
\end{example}

\begin{example}
\label{ex:ex6}
The Lorenz\index{Lorenz model!'63} '63 model is perhaps the simplest continuous-time system
to exhibit sensitivity to initial conditions and chaos. It is a system of
three coupled non-linear ordinary differential equations whose solution $v\in\mathbb{R}^3$, where $v=(v_1, v_2, v_3)$, 
satisfies\footnote{Here index denotes
components of the solution, not discrete time.}
\begin{subequations} \label{eq:ex6}
\begin{eqnarray}
\frac{dv_1}{dt}&=&a(v_2-v_1),
\label{32a}\\
\frac{dv_2}{dt}&=&-av_1-v_2-{v_1}v_3, \label{32b}\\
\frac{dv_3}{dt}&=&v_1{v_2}-b{v_3}-b(r+a). \label{32c}
\end{eqnarray}
\end{subequations}
Note that we have employed a coordinate system where the origin 
in the original version of the equations proposed by Lorenz is shifted. In the coordinate
system that we employ here we have equation \eqref{eq:ode1} with vector field
$f$ satisfying
\begin{equation}
\langle f(v), v \rangle \le \alpha-\beta|v|^2 \label{eq:diss}
\end{equation}
for some $\alpha,\beta>0$. As demonstrated in Example \ref{ex:diss},
this implies the existence of an absorbing
set\index{absorbing set}:
\begin{equation}
\limsup_{t \to \infty} |v(t)|^2<R \label{eq:ab}
\end{equation}
for any $R>\alpha/\beta.$ Mapping the ball $B(0,R)$ forward under the
dynamics gives the global attractor\index{global attractor} (see Definition
\ref{def:ga}) for the dynamics. In Figure \ref{fig:ex6_1}
we visualize this attractor, projected onto two different pairs
of coordinates {at the classical parameter
values $(a,b,r)=(10,\frac{8}{3},28)$.

Throughout these notes we will use the classical parameter 
values $(a,b,r)=(10,\frac{8}{3},28)$ in
all of our numerical experiments; at these values the
system is chaotic \index{chaos} and exhibits} sensitive dependence with respect to the
initial condition. A trajectory
of $v_{1}$ versus time can be found in Figure \ref{fig:ex6_2}a and
in Figure \ref{fig:ex6_2}b we illustrate the evolution of a small
perturbation to the initial condition which generated Figure \ref{fig:ex6_2}a;
to be explicit we plot the evolution of the error in the Euclidean
norm $|\cdot|$, for an initial perturbation of magnitude $10^{-4}$. 
{Figure \ref{fig:ex6_1} suggests that the  measure $\mu_{\infty}$
is supported on a strange set with Lebesgue measure zero, and this is
indeed the case; for this example there is no Lebesgue density $\rho_{\infty}$
for the invariant measure\index{invariant measure}, reflecting the fact that the attractor has a fractal
dimension less than three, the dimension of the space where the dynamical
system lies.}

\begin{figure}
\centering
\subfigure[$u_{1}$ as a function of time]{\includegraphics[scale=0.364]{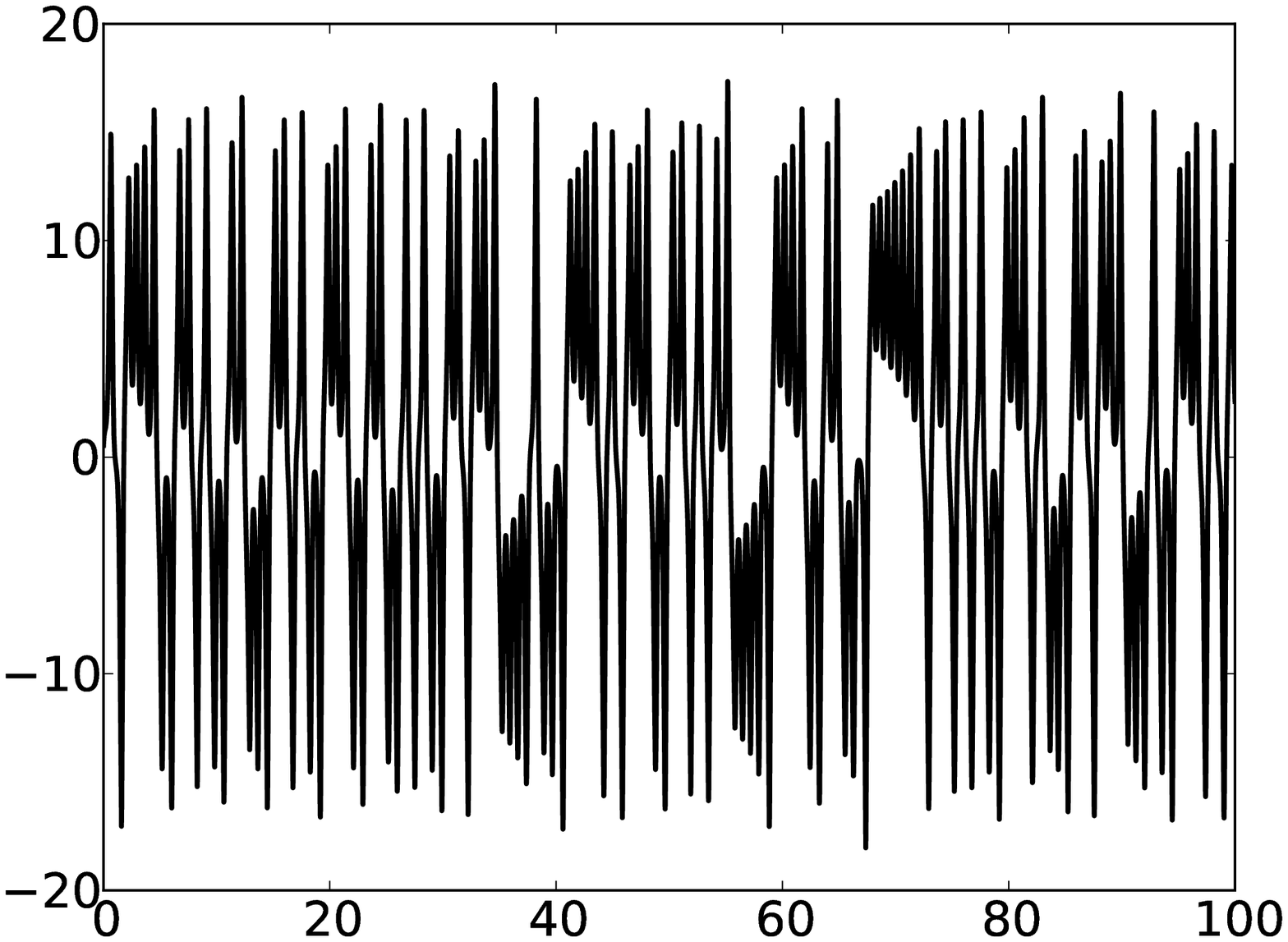}}
\subfigure[Evolution of error for a small perturbation]{\includegraphics[scale=0.364]{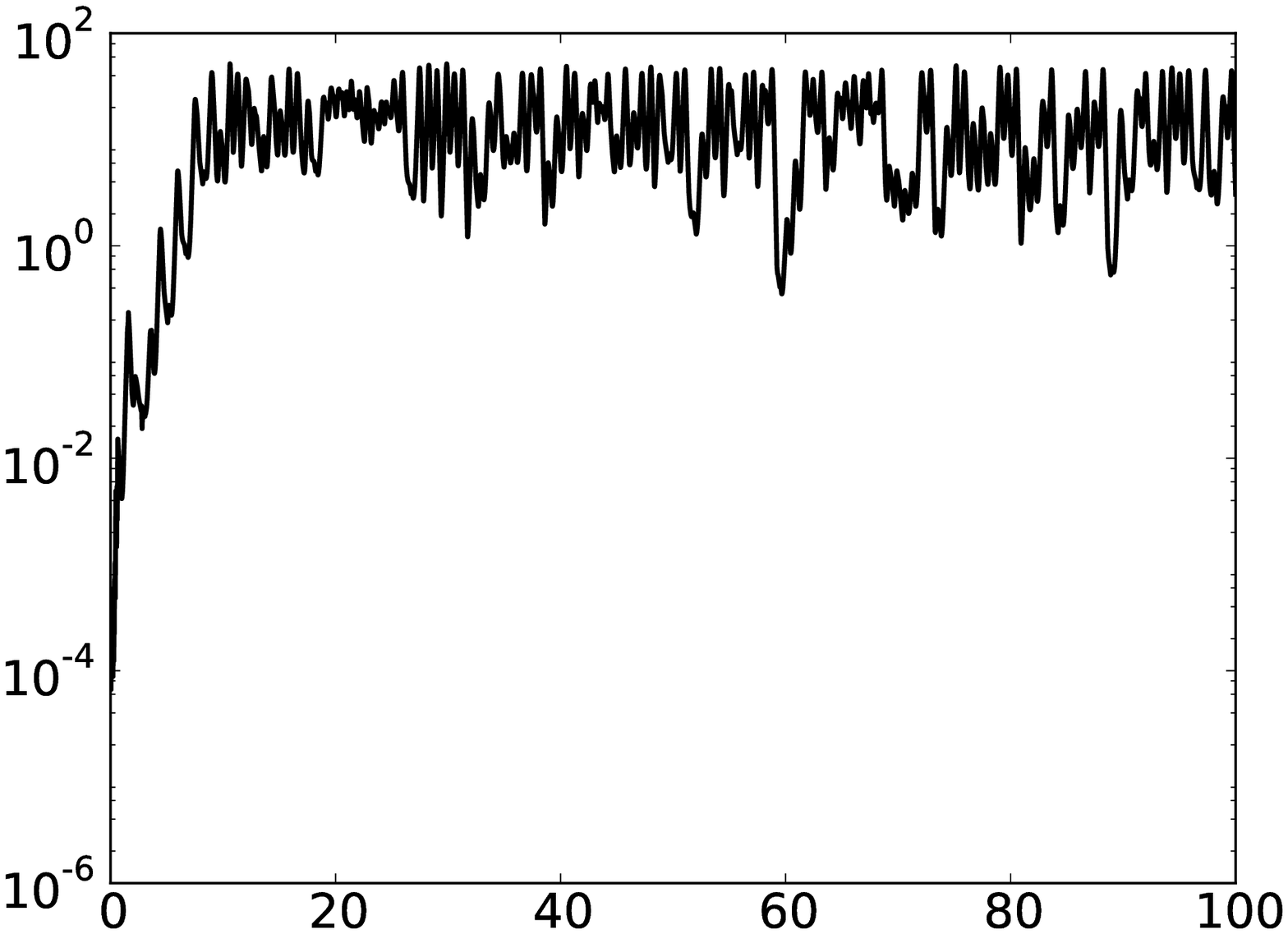}}
\caption{Dynamics of the Lorenz'63 model in the chaotic regime $(a,b,r)=(10,\frac{8}{3},28)$ }
\label{fig:ex6_2}
\end{figure}

\end{example}

\begin{example}
\label{ex:ex7}
The Lorenz\index{Lorenz model!'96} '96 model is a simple dynamical system\index{dynamical system}, of tunable dimension,
which was designed as a caricature of the dynamics of Rossby
waves in atmospheric\index{atmospheric sciences} dynamics.  
The equations have a periodic ``ring''
formulation and take the form\footnote{Again, here index denotes
components of the solution, not discrete time.}
\begin{subequations} \label{eq:ex7}
\begin{eqnarray}
\frac{dv_k}{dt}&=&v_{k-1}\bigl(v_{k+1}-v_{k-2}\bigr)-v_k
+F,\;\; k \in \{1,\cdots,K\},\\
v_0&=&v_K,\;\; v_{K+1}=v_{1}, \;\;v_{-1}=v_{K-1}. 
\end{eqnarray}
\end{subequations}
Equation \eqref{eq:ex7} satisfies the same dissipativity property 
\eqref{eq:diss} satisfied by the Lorenz '63 model, for appropriate
choice of $\alpha,\beta>0$, and hence also satisfies the absorbing ball 
property \eqref{eq:ab} thus having a global attractor\index{global attractor}
(see Definition \ref{def:ga}).

\begin{figure}
\centering
\subfigure[$v_{1}$ as a function of time]{\includegraphics[scale=0.364]{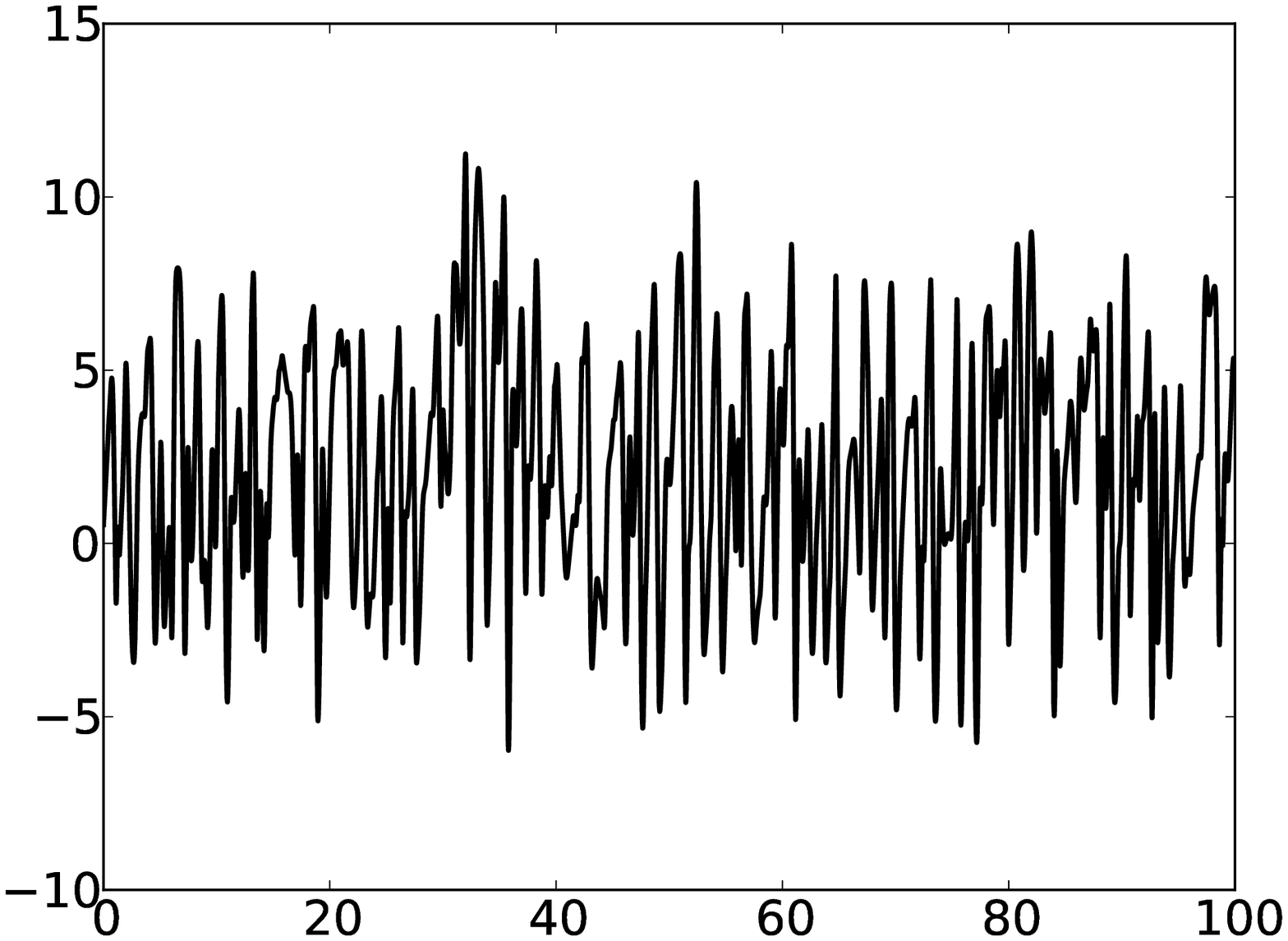}}
\subfigure[Evolution of error for a small  initial perturbation]{\includegraphics[scale=0.364]{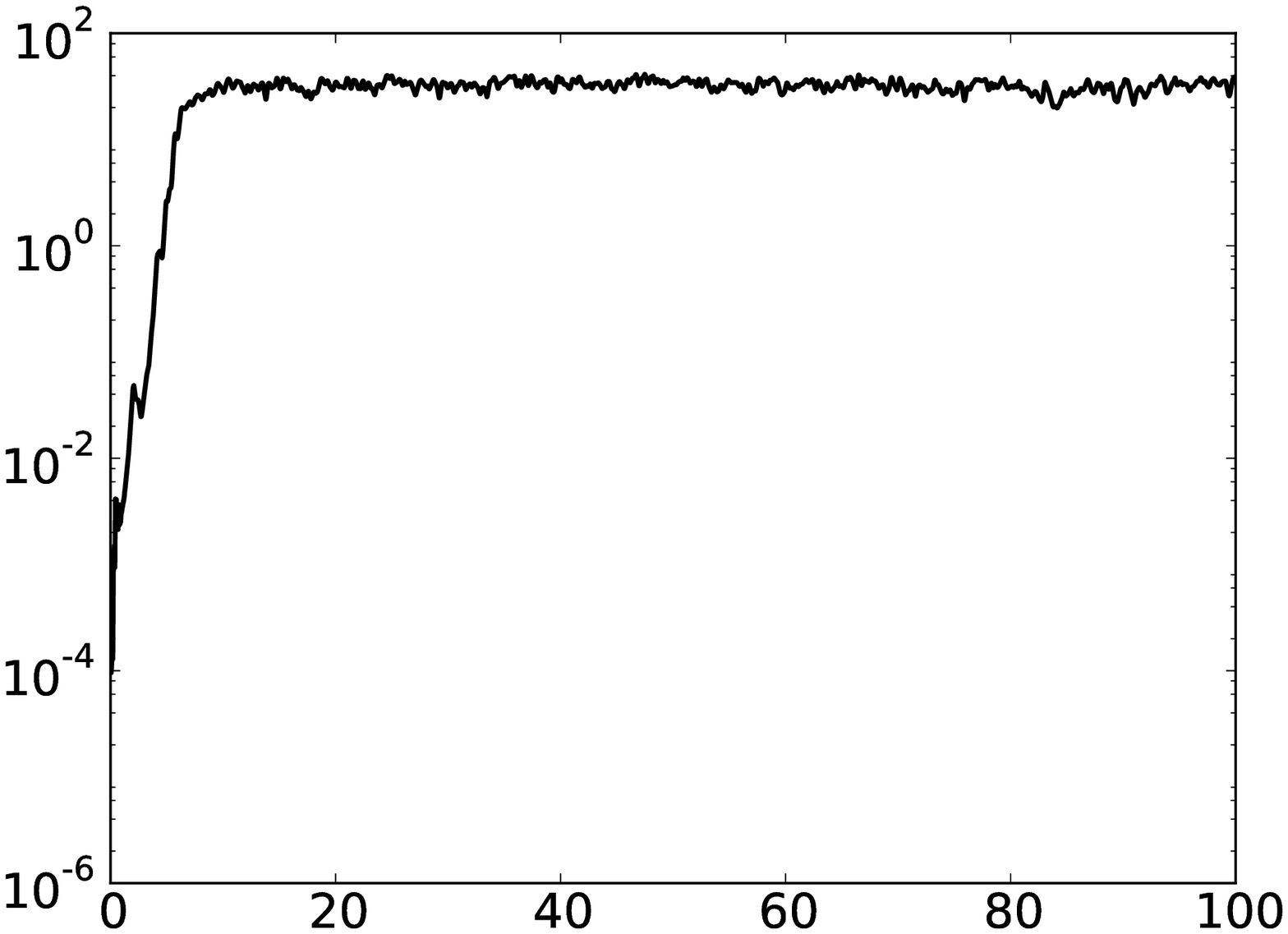}}
\caption{Dynamics of the Lorenz'96 model in the chaotic regime $(F,K)=(8,40)$ }
\label{fig:ex7_2}
\end{figure}
In Figure \ref{fig:ex7_2}a we plot a trajectory of $v_{1}$ versus time for $F=8$ and $K=40$. Furthermore, as we did in the case of the Lorenz '63 model,
we also show the evolution of the Euclidean norm of the error $|\cdot|$ for an initial perturbation of magnitude $10^{-4}$; this is displayed in 
Figure \ref{fig:ex7_2}b and clearly demonstrates
sensitive dependence on initial conditions.  We visualize the attractor, projected onto two different pairs of coordinates,
in Figure \ref{fig:ex7_1}.

\begin{figure}
\centering
\subfigure[$v_{1}$ vs $v_{K}$]{\includegraphics[scale=0.365]{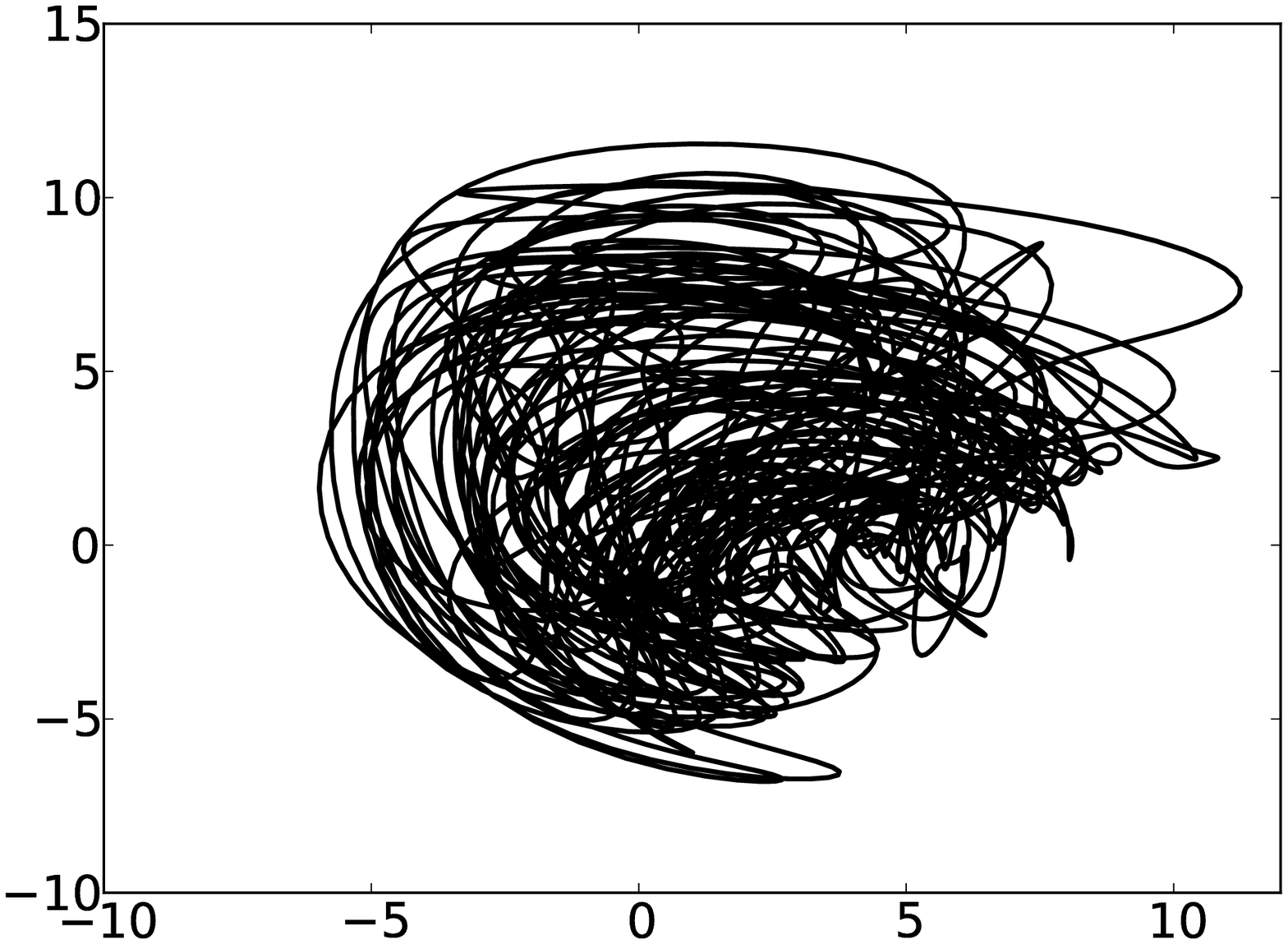}}
\subfigure[$v_{1}$ vs $v_{K-1}$]{\includegraphics[scale=0.365]{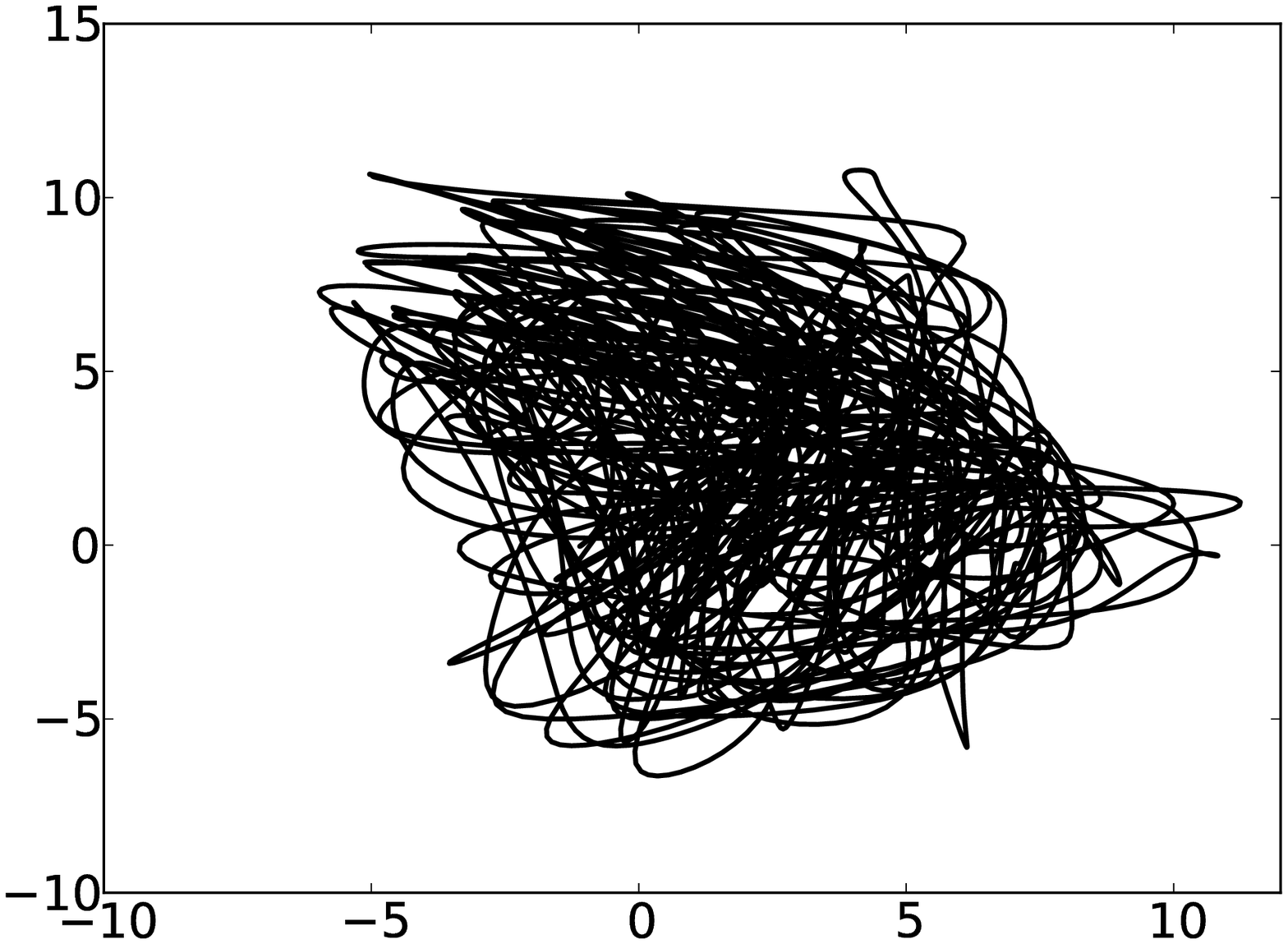}}
\caption{Projection of the Lorenz'96 attractor onto two different pairs of coordinates.}
\label{fig:ex7_1}
\end{figure}
\end{example}

\section{Smoothing Problem}\label{ssec:sp}

\subsection{Probabilistic Formulation of Data Assimilation}

Together \eqref{eq:dtf1} and \eqref{eq:dtf2} provide a probabilistic
model for the jointly varying random variable $(v,y)$. In the case
of deterministic dynamics, \eqref{eq:dtf11} and \eqref{eq:dtf2}
 provide a probabilistic model for the jointly varying random variable 
$(v_0,y)$. 
Thus in both cases we have a random variable $(u,y)$,
with $u=v$ (resp. $u=v_0$) in the stochastic (resp. deterministic) case.
Our aim is to find out information about the signal\index{signal} $v$, 
in the stochastic case, or $v_0$ in the deterministic case, 
from observation of a single instance of the data\index{data} $y$.
The natural probabilistic approach to this problem is to try and find
the probability measure\index{probability measure}
describing the random variable $u$ given $y$, denoted $u|y.$ This constitutes
the Bayesian \index{Bayesian} formulation of the problem of determining information 
about the signal arising in a noisy dynamical model, based on noisy
observations of that signal. We will refer to
the conditioned random variable $u|y$, in the case of either the stochastic
dynamics or deterministic dynamics, as the 
{\bf smoothing distribution}\index{smoothing!distribution}.
It is a random variable which contains all the probabilistic information
about the signal, given our observations. The key concept 
which drives this approach is Bayes' formula  \index{Bayes' formula}
from subsection \ref{ssec:BF} which we use repeatedly in what follows.

\subsection{Stochastic Dynamics}
\label{ssec:sd}

We wish to find the signal\index{signal} $v$ from \eqref{eq:dtf1} from
a single instance of data\index{data} $y$ given by \eqref{eq:dtf2}. To be
more precise we wish to condition the signal on
a discrete time interval $\J_0=\{0,...,J\}$, given data on
the discrete time interval $\J=\{1,...,J\}$; {we refer to
$\J_0$ as the data assimilation window\index{data assimilation!window}.} We
define $v=\{v_j\}_{j\in\J_0}, y=\{y_j\}_{j\in\J}, \xi=\{\xi_j\}_{j\in\J_0}$ and $\eta=\{\eta_j\}_{j\in\J}.$ 
The smoothing distribution\index{smoothing!distribution} here is 
the distribution of the conditioned random variable
$v|y$. Recall that we have assumed that
$v_0, \xi$ and $\eta$ are mutually independent random variables. With this
fact in hand we may apply Bayes' formula \index{Bayes' formula} to find the pdf $\bbP(v|y).$

{\bf Prior} The prior\index{prior} on $v$ is specified by (\ref{eq:dtf1}), together with the independence of $u$ and $\xi$ and the i.i.d. \index{i.i.d.} structure of $\xi.$
First note that, using \eqref{eq:cpeq} and the i.i.d. \index{i.i.d.}structure of $\xi$
in turn, we obtain
\begin{align*}
\bbP(v)&=\bbP(v_J,v_{J-1},\cdots,v_0)\\
&=\bbP(v_J|v_{J-1},\cdots,v_0)\bbP(v_{J-1},\cdots,v_0)\\
&=\bbP(v_J|v_{J-1})\bbP(v_{J-1},\cdots,v_0).
\end{align*}
Proceeding inductively gives
$$\bbP(v)=\prod_{j=0}^{J-1}\bbP(v_{j+1}|v_j)\bbP(v_0).$$
Now
$$\bbP(v_0) \propto \exp\Bigl(-\frac12\bigl|C_0^{-\frac12}(v_0-m_0)\bigr|^2\Bigr)$$
whilst
$$\bbP(v_{j+1}|v_j) \propto \exp\Bigl(-\frac12\Bigl|\Sigma^{-\frac12}\bigl(
v_{j+1}-\PPsi(v_{j})\bigr)\Bigr|^2\Bigr).$$ 
The probability distribution $\bbP(v)$ that we now write down
is {\em not} Gaussian, but the distribution on the initial
condition $\bbP(v_0)$, and the conditional distributions
$\bbP(v_{j+1}|v_j)$, are all Gaussian, making the explicit
calculations above straightforward.

Combining the preceding information we obtain
$$\bbP(v)\propto \exp(-\Jj(v))$$
where
\begin{subequations}
\label{eq:JJ}
\begin{eqnarray}
\Jj(v):=&
\frac12\bigl|C_0^{-\frac12}(v_0-m_0)\bigr|^2+\sum_{j=0}^{J-1}\frac12\bigl|\Sigma^{-\frac12}\bigl(v_{j+1}-\PPsi(v_j)\bigr)\bigr|^2\\
=&\frac12\bigl|v_0-m_0\bigr|_{C_0}^2+\sum_{j=0}^{J-1}\frac12\bigl|v_{j+1}-\PPsi(v_j)\bigr|_{\Sigma}^2.
\end{eqnarray}
\end{subequations}
The pdf $\rp(v)=\rho_0(v)$
proportional to $\exp(-\Jj(v))$
determines a prior measure $\mu_0$ on $\R^{|\J_0|\times n}$. 
The fact that the probability is not, in general, Gaussian
follows from the fact that $\PPsi$ is not, in general, linear.

{\bf Likelihood\index{likelihood}} 
The likelihood\index{likelihood} of the data $y|v$ is determined as follows.
It is a (Gaussian) probability distribution on $\R^{|\J|\times m}$, with pdf $\rp(y|v)$ proportional to $\exp(-\PPhi(v;y))$, where 
\begin{equation}\label{eq:dtf3}
\PPhi(v;y)=\sum_{j=0}^{J-1}\frac12\bigl|y_{j+1}-h(v_{j+1})\bigr|_{\Gamma}^2.
\end{equation}
To see this note that, because of the i.i.d. \index{i.i.d.} nature of the sequence $\eta$, 
it follows that
\begin{align*}
\bbP(y|v)&=\prod_{j=0}^{J-1} \bbP(y_{j+1}|v)\\
&=\prod_{j=0}^{J-1} \bbP(y_{j+1}|v_{j+1})\\
& \propto \prod_{j=0}^{J-1} \exp\Bigl(-\frac12 \bigl|\Gamma^{-\frac12}\bigl(y_{j+1}-h(v_{j+1})\bigr)\bigr|^2\Bigr)\\
&=\exp(-\PPhi(v;y)).
\end{align*}
In the applied literature 
$m_0$ and $C_0$ are often referred to as the {\bf background mean}\index{background!mean} and {\bf 
background covariance}\index{background!penalization} respectively;
we refer to $\PPhi$ as the {\bf model-data misfit}\index{model-data misfit} 
functional.

Using Bayes' formula \index{Bayes' formula} \eqref{eq:bayes} we can combine the prior and
the likelihood \index{likelihood} to determine the posterior distribution\index{posterior
distribution}, that is the
smoothing distribution\index{smoothing!distribution}, on $v|y.$ We denote the measure with this
distribution by $\mu.$

\begin{theorem} \label{th11}The posterior smoothing 
distribution\index{smoothing!distribution} on $v|y$ for
the stochastic dynamics \index{stochastic dynamics} model \eqref{eq:dtf1}, \eqref{eq:dtf2}
is a probability measure $\mu$ on 
$\R^{|\J_0|\times n}$ 
with pdf $\rp(v|y)=\rho(v)$ 
proportional to $\exp(-\Ii(v;y))$ where 
\begin{equation}\label{eq:dtf4}
\Ii(v;y)=\Jj(v)+\PPhi(v;y).
\end{equation}
\end{theorem}
\begin{proof}
Bayes' formula \index{Bayes' formula} \eqref{eq:bayes} gives us 
\[\rp(v|y)=\frac{\rp(y|v)\rp(v)}{\rp(y)}.\] Thus, ignoring constants of proportionality which depend only on $y$, 
\begin{align*}\rp(v|y)&\propto\rp(y|v)\rp(v_0)\\
&\propto \exp(-\PPhi(v;y)) \exp(-\Jj(v))\\
&=\exp(-\Ii(v;y)).
\end{align*}
\end{proof}

Note that, although the preceding calculations required only knowledge
of the pdfs of Gaussian distributions, the resulting posterior 
distribution\index{posterior distribution}
is non-Gaussian in general, unless $\PPsi$ and $h$ are linear. This is
because, unless $\PPsi$ and $h$ are linear, $\Ii(\cdot;y)$ is not quadratic. 
We refer to $\Ii$ as the negative log-posterior\index{posterior!log-posterior}.
It will be helpful later to note that
\begin{equation}\label{eq:dtsa33}
\frac{\rho(v)}{\rho_0(v)}\propto\exp\bigl(-\PPhi(v;y)\bigr).
\end{equation}

\subsection{Reformulation of Stochastic Dynamics}
\label{ssec:reform}

For the development of algorithms to probe the posterior distribution, 
the following reformulation of the stochastic dynamics \index{stochastic dynamics} problem can
be very useful. For this we define the vector 
$\xi=(v_0,\xi_0, \xi_1, \cdots, \xi_{J-1}) \in \bbR^{|\J_0|n}$. The
following lemma is key to what follows.

\begin{lemma}
\label{l:sicab}
Define the mapping $G:\bbR^{|\J_0| \times n} \mapsto \bbR^{|\J_0| \times n}$ by
$$G_j(v_0,\xi_0, \xi_1, \cdots, \xi_{J-1})=v_j, \quad j=0, \cdots, J,$$
where $v_j$ is determined by \eqref{eq:dtf1}. Then this mapping is invertible.
Furthermore, if  $\PPsi \equiv 0$, then $G$ is the identity mapping.
\end{lemma} 
\begin{proof}
In words the mapping $G$ takes the initial condition and noise into the
signal. Invertibility requires determination of the initial condition
and the noise from the signal. From the signal we may compute the noise as
follows noting that, of course, the initial condition is specified and
that then we have
$$\xi_{j}=v_{j+1}-\PPsi(v_j), \quad j=0, \cdots, J-1.$$
The fact that $G$ becomes the identity mapping when $\PPsi \equiv 0$
follows directly from \eqref{eq:dtf1} by inspection. 
\end{proof}

We may thus consider the smoothing problem as finding the probability
distribution of $\xi$, as defined prior to the lemma, given data\index{data}
$y$, with $y$ as defined in section \ref{ssec:sd}. Furthermore we have,
using the notion of pushforward\index{pushforward},
\begin{equation}
\label{eq:relate}
\bbP(v|y)=G \star \bbP(\xi|y), \quad \bbP(\xi|y)=G^{-1} \star \bbP(v|y).
\end{equation}
These formulae mean that it is easy to move between the two measures:
samples from one can be converted into samples from the other simply
by applying $G$ or $G^{-1}$. This means that algorithms can be applied
to, for example, generate samples from $\xi|y$, and then convert into
samples from $v|y.$ We will use this later on. In order to use this idea
it will be helpful to have an explicit expression for the pdf of $\xi|y$.
We now find such an expression.

To start we introduce the measure $\vartheta_0$ with density $\pi_0$ found from $\mu_0$ and $\rho_0$ in the case where $\PPsi\equiv0$. 
Thus 
\begin{subequations}
\label{eq:nuz}
\begin{eqnarray}
\pi_0(v)&\propto &\exp\left(-\frac12 \left|C_0^{-\frac12}(v_0-m_0)\right|^2
-\sum_{j=0}^{J-1} \frac12|\Sigma^{-\frac12}v_{j+1}|^2\right)\\
&\propto &\exp\left(-\frac12 \left|v_0-m_0\right|_{C_0}^2
-\sum_{j=0}^{J-1} \frac12|v_{j+1}|_{\Sigma}^2\right)
\end{eqnarray}
\end{subequations}
and hence $\vartheta_0$ is a Gaussian measure,
independent in each component $v_j$ for $j=0,\cdots, J.$
By Lemma \ref{l:sicab} we also deduce that measure $\vartheta_0$ 
with density $\pi_0$ is the prior on $\xi$ as defined above:
\begin{equation}
\label{eq:xipdf}
\pi_0(\xi)\propto \exp\left(-\frac12 \left|v_0-m_0\right|_{C_0}^2
-\sum_{j=0}^{J-1} \frac12|\xi_{j}|_{\Sigma}^2\right).
\end{equation}

We now compute the likelihood of $y|\xi$. For this we define
\begin{equation}
\label{eq:z1}
{\cal G}_j(\xi)=h\bigl(G_j(\xi)\bigr)
\end{equation}
and note that we may then concatenate the data and write
\begin{equation}
\label{eq:z2}
y={\cal G}(\xi)+\eta
\end{equation}
where $\eta=(\eta_1,\cdots,\eta_J)$ is a the Gaussian random
variable $N(0,\Gamma_{J})$ where $\Gamma_J$ is a block diagonal $nJ \times nJ$ matrix
with $n \times n$ diagonal blocks $\Gamma$.
It follows that the likelihood is determined by
$\bbP(y|\xi)=N\bigl({\cal G}(\xi),\Gamma_J\bigr)$.
Applying Bayes formula from \eqref{eq:bayes} to find the pdf for
$\xi|y$ we find the posterior $\vartheta$ on $\xi|y$, as summarized in the
following theorem. 

\begin{theorem} \label{t:101}The posterior smoothing 
distribution\index{smoothing!distribution} on $\xi|y$ for
the stochastic dynamics \index{stochastic dynamics} model \eqref{eq:dtf1}, \eqref{eq:dtf2}
is a probability measure $\vartheta$ on 
$\R^{|\J_0|\times n}$ 
with pdf $\rp(\xi|y)=\pi(\xi)$ 
proportional to $\exp(-\Ir(\xi;y))$ where 
\begin{equation}
\label{eq:z3}
\Ir(\xi;y)=\Jr(\xi)+\Phr(\xi;y),
\end{equation}
$$\Phr(\xi;y):=\frac12\left|\left(y-{\cal G}(\xi)\right)\right|_{\Gamma_J}^2$$
and
$$\Jr(\xi):=
\frac12 \left|v_0-m_0\right|_{C_0}^2
+\sum_{j=0}^{J-1} \frac12\left|\xi_{j}\right|_{\Sigma}^2.$$
\end{theorem}

We refer to $\Ir$ as the negative log-posterior\index{posterior!log-posterior}.

\subsection{Deterministic Dynamics}
\label{ssec:dd}
It is also of interest to study the posterior distribution\index{posterior
distribution}
on the initial condition in the case where the model dynamics 
contains no noise, and is given by \eqref{eq:dtf11}; this we
now do. Recall that
$\PPsi^{(j)}(\cdot)$ denotes the $j-$fold composition of
$\PPsi(\cdot)$ with itself. 
In the following we sometimes refer to $\Jjd$ as the {\bf background}
penalization\index{background!penalization}, 
and $m_0$ and $C_0$ as the background\index{background!mean} mean and 
covariance\index{background!covariance};
we refer to $\Phid$ as the {\bf model-data misfit}\index{model-data misfit} 
functional.

\begin{theorem} \label{th112}
The posterior smoothing distribution\index{smoothing!distribution} on $v_0|y$ 
for the the deterministic dynamics model \eqref{eq:dtf11}, \eqref{eq:dtf2}
is a probability measure $\nu$ on $\R^n$ with density $\rp(v_0|y)=\varrho(v_0)$ 
proportional to $\exp(-\Iid(v_0;y))$ where 
\begin{subequations}
\label{eq:Ip}
\begin{align}
\Iid(v_0;y)&=\Jjd(v_0)+\Phid(v_0;y),\\
\Jjd(v_0)&=\frac12\bigl|v_0-m_0\bigr|_{C_0}^2,\\
\Phid(v_0;y)&=\sum_{j=0}^{J-1}\frac12\bigl|y_{j+1}-h\bigl(\PPsi^{(j+1)}(v_0)\bigr)\bigr|_{\Gamma}^2.
\end{align}
\end{subequations}
\end{theorem}
\begin{proof}
We again use Bayes' rule which states that 
\[\rp(v_0|y)=\frac{\rp(y|v_0)\rp(v_0)}{\rp(y)}.\] 
Thus, ignoring constants of proportionality which depend only on $y$, 
\begin{align*}\rp(v_0|y)&\propto\rp(y|v_0)\rp(v_0)\\
&\propto \exp\bigl(-\Phid(v_0;y)\bigr)\exp\bigl(-\frac12|v_0-m_0|_{C_0}^2\bigr)\\
&=\exp(-\Iid(v_0;y)).
\end{align*}
Here we have used the fact that $\rp(y|v_0)$ is proportional to 
$\exp\bigl(-\Phid(v_0;y)\bigr)$; this follows from the fact that
$y_j|v_0$ form an i.i.d \index{i.i.d.} sequence of Gaussian random variables
$N\bigl(h(v_{j}),\Gamma)$ with $v_j=\PPsi^{(j)}(v_0).$
\end{proof}

We refer to $\Iid$ as the negative log-posterior\index{posterior!log-posterior}.

\section{Filtering\index{filtering} Problem}\label{ssec:fp}
The smoothing\index{smoothing} problem considered in the previous
section involves, potentially, conditioning $v_j$ on data
$y_k$ with $k>j.$ Such conditioning can only be performed
{\em off-line} and is of no use in {\em on-line} scenarios
where we want to determine information on the state of the
signal {\em now} hence using only data from the past up to the
present. To study this situation,
let $Y_j=\{y_l\}_{l=1}^{j}$ denote the accumulated 
data\index{data!accumulated} up to time $j$. {\bf Filtering} is concerned with 
determining $\rp(v_j|Y_j)$, the pdf 
associated with the probability measure\index{probability measure}
on the random variable $v_j|Y_j$; in particular filtering is concerned
with the sequential updating this pdf as the index $j$ is incremented. This update is defined by the following procedure
which provides a prescription for computing $\rp(v_{j+1}|Y_{j+1})$ from $\rp(v_j|Y_j)$ via two steps: {\bf prediction}\index{prediction} which computes the
mapping $\rp(v_j|Y_j) \mapsto \rp(v_{j+1}|Y_j)$ and {\bf analysis}\index{analysis}
which computes $\rp(v_{j+1}|Y_j) \mapsto \rp(v_{j+1}|Y_{j+1})$ by 
application of Bayes' formula.\index{Bayes' formula}
 
\noindent{\bf Prediction} Note that $\rp(v_{j+1}|Y_j,v_j)=\rp(v_{j+1}|v_j)$
because $Y_j$ contains noisy and indirect information about $v_j$ and cannot
improve upon perfect knowledge of the variable $v_j.$ Thus,
by \eqref{eq:cpeq}, we deduce that 
\begin{subequations}\label{eq:dtf5}
\begin{align}
\rp(v_{j+1}|Y_j)&=\int_{\R^n}\rp(v_{j+1}|Y_j,v_j)\rp(v_j|Y_j)dv_j\\
&=\int_{\R^n}\rp(v_{j+1}|v_j)\rp(v_j|Y_j)dv_j
\end{align}
\end{subequations}
Note that, since the forward model equation 
(\ref{eq:dtf1}) determines $\rp(v_{j+1}|v_j)$, this prediction step
provides the map from  $\rp(v_j|Y_j)$ to  $\rp(v_{j+1}|Y_j).$
This prediction \index{filtering!prediction step} step simplifies in the case of deterministic
dynamics \eqref{eq:dtf11}; in this case it simply corresponds
to computing the pushforward\index{pushforward} of $\rp(v_j|Y_j)$ under the map $\PPsi$.

\noindent{\bf Analysis}\index{analysis}
Note that $\rp(y_{j+1}|v_{j+1},Y_j)=\rp(y_{j+1}|v_{j+1})$ because 
 $Y_j$ contains noisy and indirect information about $v_{j}$ and cannot
improve upon perfect knowledge of the variable $v_{j+1}.$ 
Thus, using Bayes' 
formula \eqref{eq:bayes}, we deduce that 
\begin{align}\label{eq:dtf6}
\rp(v_{j+1}|Y_{j+1})&=\rp(v_{j+1}|Y_j, y_{j+1})\notag\\
&=\frac{\rp(y_{j+1}|v_{j+1},Y_j)\rp(v_{j+1}|Y_j)}{\rp(y_{j+1}|Y_j)}\notag\\
&=\frac{\rp(y_{j+1}|v_{j+1})\rp(v_{j+1}|Y_j)}{\rp(y_{j+1}|Y_j)}.
\end{align}
Since the observation equation (\ref{eq:dtf2}) determines $\rp(y_{j+1}|v_{j+1})$,
this analysis step  \index{filtering!analysis step} provides a map from $\rp(v_{j+1}|Y_j)$ to $\rp(v_{j+1}|Y_{j+1}).$
 
\noindent{\bf Filtering Update}
Together, then, the prediction \index{filtering!prediction step} and analysis step  \index{filtering!analysis step} provide a mapping from
$\rp(v_j|Y_j)$ to  $\rp(v_{j+1}|Y_{j+1}).$ 
Indeed if we let $\mu_j$ denote the probability measure 
on $\bbR^n$ corresponding
to the density $\rp(v_j|Y_j)$ and $\hmu_{j+1}$ be the probability 
measure on $\bbR^n$ corresponding to the density $\rp(v_{j+1}|Y_j)$
then the prediction \index{filtering!prediction step} step maps $\mu_j$ to $\hmu_{j+1}$ whilst
the analysis step  \index{filtering!analysis step} maps $\hmu_{j+1}$ to $\mu_{j+1}.$ 
However there is, in general, no easily usable closed form expression for 
the density of $\mu_j$, namely $\rp(v_j|Y_j)$. Nevertheless, formulae (\ref{eq:dtf5}), (\ref{eq:dtf6}) form the starting point for numerous algorithms to approximate $\rp(v_j|Y_j)$. In terms of analyzing the particle filter\index{filter!particle} it is
helpful conceptually to write the prediction  \index{filtering!prediction step} and analysis steps  \index{filtering!analysis step}
as
\begin{equation}
\label{eq:summary}
\hmu_{j+1}=P\mu_j\quad\quad\quad \mu_{j+1}=L_j \hmu_{j+1}.
\end{equation}
Note that  $P$ does not depend on $j$ as the same Markov process
governs the prediction step  \index{filtering!prediction step} at each $j$; however $L_j$ depends on
$j$ because the likelihood \index{likelihood} sees different data at each $j$.  Furthermore, the formula $\hmu_{j+1}=P\mu_j$ summarizes \eqref{eq:dtf5}
whilst $\mu_{j+1}=L_j \hmu_{j+1}$ summarizes \eqref{eq:dtf6}.
Note that $P$ is a linear mapping, whilst $L_j$ is nonlinear;
this issue is discussed in subsections \ref{ssec:mk} and \ref{ssec:btm},
at the level of pdfs.

\section{Filtering and Smoothing are Related}\label{ssec:fasar}

The filtering\index{filtering} and smoothing\index{smoothing} 
approaches to determining the signal from
the data are distinct, but related. They are related by the fact that
in both cases the solution computed at the {\em end} of any specified 
time-interval is conditioned on the same data, and must hence coincide;
this is made precise in the following.

\begin{theorem}\label{th12}
Let $\rp(v|y)$ denote the smoothing\index{smoothing} distribution on the
discrete time interval $j\in\J_0$, and $\rp(v_J|Y_J)$ the 
filtering distribution\index{filtering!distribution} at time $j=J$ for the stochastic dynamics \index{stochastic dynamics} model \eqref{eq:dtf1}.
Then the marginal of the smoothing\index{smoothing} distribution on $v_J$ 
is the same as the filtering distribution\index{filtering!distribution}
at time $J$:
\[\int\rp(v|y)dv_0dv_1...dv_{J-1}=\rp(v_J|Y_J).\]
\end{theorem}
\begin{proof}
Note that $y=Y_J$. Since $v=(v_0,...,v_{J-1},v_J)$ the result follows trivially.
\end{proof}

\begin{remark}\label{r13}
Note that the marginal of the smoothing\index{smoothing} distribution on say $v_j, \;j<J$ is 
{\em not} equal to the filter $\rp(v_j|Y_j)$. This is because the smoother induces a distribution on $v_j$ which is influenced by the entire data set
$Y_J=y=\{y_l\}_{l\in\J}$; in contrast the filter at $j$ involves only the data $Y_j=\{y_l\}_{l\in\{1,...,j\}}$.
\end{remark}

It is also interesting to mention the relationship
between filtering\index{filtering} and smoothing\index{smoothing} in the case of
noise-free dynamics. In this case the filtering 
distribution\index{filtering!distribution} 
$\rp(v_j|Y_j)$ is simply found as the pushforward\index{pushforward}
of the smoothing\index{smoothing} distribution on $\rp(v_0|Y_j)$
under $\PPsi^{(j)}$, that is under $j$ applications of $\PPsi.$

\begin{theorem}\label{th12a}
Let $\rp(v_0|y)$ denote the smoothing\index{smoothing} distribution on the
discrete time interval $j\in\J_0$, and $\rp(v_J|Y_J)$ the 
filtering distribution\index{filtering!distribution} at time $j=J$ for the deterministic dynamics model \eqref{eq:dtf11}.
Then the pushforward\index{pushforward} of the smoothing 
distribution\index{smoothing!distribution} on $v_0$ 
under $\PPsi^{(J)}$ is the same as the 
filtering distribution\index{filtering!distribution} at time $J$:
$$\PPsi^{(J)}\star\rp (v_0|Y_J)=\rp (v_J|Y_J).$$
\end{theorem}

\section{Well-Posedness}\label{ssec:wp}

Well-posedness of a mathematical problem refers, generally,
to the existence of a unique solution which depends continuously
on the parameters defining the problem. We have shown, for both filtering\index{filtering}
and smoothing\index{smoothing}, how to construct a uniquely 
defined probabilistic solution to the
problem of determining the signal\index{signal}
given the data\index{data}. In this setting it
is natural to consider well-posedness with respect to the data itself.
Thus we now
investigate the continuous dependence of the probabilistic solution
on the observed data; indeed we will show Lipschitz dependence.
To this end we need probability metrics, as
introduced in section \ref{ssec:pm}.

As we do throughout the notes, we perform all calculations using
the existence of everywhere positive Lebesgue densities for our
measures.  We let $\mu_0$ denote the prior measure on $v$ for the 
smoothing\index{smoothing} problem
arising in stochastic dynamics,\index{stochastic dynamics} as defined by \eqref{eq:dtf1}.
Then $\mu$ and  $\mu'$ denote the posterior measures 
resulting from two different instances of the data, $y$ and $y'$ respectively. 
Let $\rho_0, \rho$ and $\rho'$ denote the Lebesgue densities on $\mu_0, \mu$ and $\mu'$ 
respectively. Then, for $\Jj$ and $\PPhi$ as defined in \eqref{eq:JJ} and
\eqref{eq:dtf3}, 
\begin{subequations}
\label{eq:splus}
\begin{align}
\rho_0(v)&=\frac1{Z_0}\exp(-\Jj(v)),\\
\rho(v)&=\frac1{Z}\exp(-\Jj(v)-\PPhi(v;y)),\\
\rho'(v)&=\frac1{Z'}\exp(-\Jj(v)-\PPhi(v;y')),
\end{align}
\end{subequations}
where
\begin{subequations}
\label{eq:szero}
\begin{align}
Z_0&=\int\exp(-\Jj(v))dv,\\
Z&=\int\exp(-\Jj(v)-\PPhi(v;y))dv,\\
Z'&=\int\exp(-\Jj(v)-\PPhi(v;y'))dv.
\end{align}
\end{subequations}
Here, and in the proofs that follow in this section, all integrals are over $\bbR^{|\J_0| \times n}$
(or, in the case of the deterministic dynamics model at the end of the section, over $\bbR^n$).
Note that $|\J_0|$ is the cardinality of the set $\J_0$ and is hence 
equal to $J+1$.
To this end we note explicitly that (\ref{eq:splus}a) implies that
\begin{equation}
\label{eq:sstar}
\exp\bigl(-\Jj(v)\bigr)dv=Z_0\rho_0(v)dv=Z_0\mu_0(dv),
\end{equation}
indicating that integrals weighted by $\exp\bigl(-\Jj(v)\bigr)$ may be
rewritten as expectations with respect to $\mu_0.$
We use the identities \eqref{eq:splus}, \eqref{eq:szero} and \eqref{eq:sstar}
repeatedly in what follows to express all integrals as expectations
with respect to the measure $\mu_0$. In particular the assumptions
that we make for the subsequent theorems and corollaries in this
section are all expressed in terms of expectations under $\mu_0$ (or,
under $\nu_0$ for the deterministic dynamics problem considered at the end of the
section).
This is convenient because it relates to the unconditioned problem
of stochastic dynamics \index{stochastic dynamics} for $v$, in the absence of any data, and may thus be
checked once and for all, independently of the particular data set $y$ or
$y'$ which are used to condition $v$ and obtain $\mu$ and $\mu'$.

We assume throughout what follows that $y,y'$ are both contained
in a ball of radius $r$ in the Euclidean norm on $\bbR^{|\J| \times n}.$
Again $|\J|$ is the cardinality of the set $\J$ and is hence 
equal to $J$.
We also note that $Z_0$ is bounded from above independently of $r$,
because $\rho_0$ is the density associated with the probability measure 
$\mu_0$, which is therefore normalizable, and this measure is independent
of the data. It also follows that $Z \le Z_0$, $Z' \le Z_0$ by using
\eqref{eq:sstar} in (\ref{eq:szero}b), (\ref{eq:szero}c), together with
the fact that $\PPhi(\cdot;y)$ is a positive function. 
Furthermore, if we assume that
\begin{equation}
\label{eq:vsf}
\sfv:=\sum_{j\in \J}\bigl(1+|h(v_{j})|^2\bigr)
\end{equation}
satisfies $\bbE^{\mu_0}\sfv<\infty$, then both $Z$ and $Z'$ are positive
with common lower bound depending only on $r$,
as we now demonstrate.  It is sufficient to prove the result for $Z$,
which we now do. In the following, and in the proofs which follow,
$K$ denotes a generic constant, which may depend on $r$ and $J$ but
not on the solution sequence $v$, and 
which may change from instance to instance. Note first that, by
\eqref{eq:szero}, \eqref{eq:sstar},
$$\frac{Z}{Z_0} = \int \exp\bigl(-\PPhi(v;y)\bigr) \rho_0(v)dv \ge
\int \exp\bigl(-K\sfv \bigr) \rho_0(v)dv.$$
Since $\bbE^{\mu_0} \sfv<\infty$ we deduce from \eqref{eq:mie2}, that
for $R$ sufficiently large,
\begin{align*}
\frac{Z}{Z_0} \ge \exp\bigl(-KR \bigr)\int_{|\sfv|<R} \rho_0(v)dv&=
\exp\bigl(-KR \bigr) \bbP^{\mu_0}(|\sfv|<R)\\
&\ge \exp\bigl(-KR \bigr)\bigl(1-R^{-1}\bbE^{\mu_0}\sfv\bigr).
\end{align*}
Since $K$ depends on $y,y'$ only through $r$, we deduce that, by choice
of $R$ sufficiently large, we have found lower bounds on $Z,Z'$
which depend on $y,y'$ only through $r$.

Finally we note that, since all norms are equivalent on finite dimensional
spaces, there is constant $K$, such that
\begin{equation}
\label{eq:lunch}
\left(\sum_{j=0}^{J-1}|y_{j+1}-y_{j+1}'|_{\Gamma}^2\right)^\frac12 \le K|y-y'|.
\end{equation}
The following theorem then shows that the
posterior measure is in fact Lipschitz continuous, in the Hellinger metric,
with respect to the data.
 
\begin{theorem}\label{th14} 
Consider the smoothing\index{smoothing}
problem arising from the stochastic dynamics \index{stochastic dynamics} model \eqref{eq:dtf1}, resulting
in the posterior probability distributions $\mu$ and $\mu'$ associated with
two different data sets $y$ and $y'$. Assume that 
$\bbE^{\mu_0}\sfv <\infty$ where $\sfv$ is given by \eqref{eq:vsf}. 
Then there exists $c=c(r)$ such that,
for all $|y|, |y'|\leq r$,
 \[d_{{\rm{\tiny Hell}}}(\mu, \mu')\leq c|y-y'|.\]
\end{theorem}
\begin{proof}
We have, by \eqref{eq:splus}, \eqref{eq:sstar},
\begin{align*}
d_{{\rm{\tiny Hell}}}(\mu,\mu')^2&=\frac12\int|\sqrt{\rho(v)}-\sqrt{\rho'(v)}|^2dv\\
&=\frac12\int Z_0\Bigl|\frac{1}{\sqrt{Z}}e^{-\frac12\PPhi(v;y)}-\frac{1}{\sqrt{Z'}}e^{-\frac12\PPhi(v;y')}\Bigr|^2\rho_0(v)dv\\
&\leq I_1+I_2,
\end{align*}
where \[I_1=Z_0\int\frac1{Z}\Bigl|e^{-\frac12\PPhi(v;y)}-e^{-\frac12\PPhi(v;y')}\Bigr|^2\rho_0(v)dv\]
and, using \eqref{eq:splus} and \eqref{eq:szero}, 
\begin{align*}
I_2&=Z_0\Bigl|\frac{1}{\sqrt{Z}}-\frac1{\sqrt{Z'}}\Bigr|^2\int e^{-\PPhi(v;y')}\rho_0(v)dv\\
&=Z'\Bigl|\frac1{\sqrt{Z}}-\frac{1}{\sqrt{Z'}}\Bigr|^2.
\end{align*}

We estimate $I_2$ first.
Since, as shown before the theorem, $Z,Z'$ are bounded below
by a positive constant depending only on $r$, we have 
$$I_2 = \frac{1}{Z}|\sqrt{Z}-\sqrt{Z'}|^2 = \frac{1}{Z}
\frac{|Z-Z'|^2}{|\sqrt{Z}+\sqrt{Z'}|^2} \le K|Z-Z'|^2.$$
As $\PPhi(v;y)\ge 0$ and $\PPhi(v;y')\ge 0$ we have from \eqref{eq:splus}, \eqref{eq:szero},
using the fact that $e^{-x}$ is Lipschitz on $\bbR^+$,
\begin{align*}
|Z-Z'|&\leq Z_0\int|e^{-\PPhi(v;y)}-e^{-\PPhi(v;y')}|\rho_0(v)dv\\
&\leq Z_0\int|\PPhi(v;y)-\PPhi(v;y')|\rho_0(v)dv.
\end{align*}
By definition of $\PPhi$ and use of \eqref{eq:lunch} 
\begin{align*}
|\PPhi(v;y)-\PPhi(v;y')|&\leq\frac12\sum_{j=0}^{J-1}|y_{j+1}-y_{j+1}'|_{\Gamma}|y_{j+1}+y_{j+1}'-2h(v_{j+1})|_{\Gamma}\\
&\leq \frac12\left(\sum_{j=0}^{J-1}|y_{j+1}-y_{j+1}'|_{\Gamma}^2\right)^\frac12\left(\sum_{j=0}^{J-1}|y_{j+1}+y_{j+1}'-2h(v_{j+1})|_{\Gamma}^2\right)^\frac12\\
&\leq K|y-y'|\left(\sum_{j=0}^{J-1}\bigl(1+|h(v_{j+1})|^2\bigr)\right)^{\frac12}\\
&=K|y-y'|\sfv^{\frac12}.
\end{align*}
Since $\bbE^{\mu_0}\sfv<\infty$ implies that $\bbE^{\mu_0}\sfv^{\frac12}<\infty$
it follows that 
\[|Z-Z'|\leq K|y-y'|.\]
Hence $I_2 \le K|y-y'|^2$ 

Now, using that $Z_0$ is bounded above independently of $r$,
that $Z$ is bounded below, depending on data only through $r$,
and that $e^{-\frac12 x}$ is Lipschitz on $\bbR^+$, it
follows that $I_1$ satisfies
$$I_1 \le K\int |\PPhi(v;y)-\PPhi(v;y')|^2 \rho_0(v)dv.$$
Squaring the preceding bound on $|\PPhi(v;y)-\PPhi(v;y')|$ gives
$$|\PPhi(v;y)-\PPhi(v;y')|^2 \le K|y-y'|^2\sfv$$
and so
$I_1 \le K|y-y'|^2$ as required.
\end{proof}

\begin{corollary}\label{c15}
Consider the smoothing\index{smoothing}
problem arising from the stochastic dynamics \index{stochastic dynamics} model \eqref{eq:dtf1}, resulting
in the posterior probability distributions $\mu$ and $\mu'$ associated with
two different data sets $y$ and $y'$. Assume that 
$\bbE^{\mu_0}\sfv <\infty$ where $\sfv$ is given by \eqref{eq:vsf}. 
Let $f:\R^{|\J_0|\times n}\to\R^p$ be such that $\bbE^{\mu_0}|f(v)|^2<\infty$. 
Then there is $c=c(r)>0$ such that, for all $|y|, |y'|<r$, 
\[|\bbE^{\mu}f(v)-\bbE^{\mu'}f(v)|\leq c|y-y'|.\]
\end{corollary}
\begin{proof} 
First note that, since $\PPhi(v;y)\geq0$, 
$Z_0$ is bounded above independently of $r$, and since
$Z$ is bounded from below depending only on $r$, 
$\bbE^{\mu}|f(v)|^2\leq c\bbE^{\mu_0}|f(v)|^2$; 
and a similar bound holds under $\mu'$. 
The result follows from (\ref{eq:dtf8}) and Theorem \ref{th14}.
\end{proof}

Using the relationship between filtering \index{filtering} and smoothing \index{smoothing} as described in the previous
section, we may derive a corollary concerning the filtering distribution \index{filtering!distribution}.

\begin{corollary}\label{c16}
Consider the smoothing\index{smoothing}
problem arising from the stochastic dynamics \index{stochastic dynamics} model \eqref{eq:dtf1}, resulting
in the posterior probability distributions $\mu$ and $\mu'$ associated with
two different data sets $y$ and $y'$. Assume that 
$\bbE^{\mu_0}\sfv <\infty$ where $\sfv$ is given by \eqref{eq:vsf}. 
Let $g:\R^n\to\R^p$ be such that $\bbE^{\mu_0}|g(v_J)|^2<\infty$. 
Then there is $c=c(r)>0$ such that, for all $|y|, |y'|<r$, 
\[|\bbE^{\mu_J}g(u)-\bbE^{\mu_J'}g(u)|\leq c|Y_J-Y_J'|,\]
where $\mu_J$ and $\mu_J'$ denote the filtering 
distributions\index{filtering!distribution}
at time $J$ corresponding to data $Y_J, Y_J'$ respectively
(i.e. the marginals of $\mu$ and $\mu'$ on the coordinate
at time $J$). 
\end{corollary} 
\begin{proof} Since, by Theorem \ref{th12}, $\mu_J$ is the marginal of the smoother on the $v_J$ coordinate, the result follows from Corollary \ref{c15} by choosing $f(v)=g(v_J)$.
\end{proof}

A similar theorem, and corollaries, may be proved for the
case of deterministic dynamics \eqref{eq:dtf11}, 
and the posterior $\bbP(v_0|y).$ We state the theorem and leave
its proof to the reader.
We let $\nu_0$ denote the prior Gaussian measure $N(m_0,C_0)$ on $v_0$ 
for the smoothing\index{smoothing} problem arising in deterministic dynamics, and $\nu$ and  $\nu'$ the posterior measures on $v_0$
resulting from two different instances of the data, 
$y$ and $y'$ respectively. 
We also define
$$\sfv_0:=\sum_{j=0}^{J-1}\bigl(1+\bigl|h\bigl(\PPsi^{(j+1)}(v_0)\bigr)\bigr|^2\bigr).$$
 
\begin{theorem}\label{th14a} Consider the smoothing\index{smoothing} problem arising
from the deterministic dynamics model \eqref{eq:dtf11}.
Assume that $\bbE^{\nu_0}\sfv_0<\infty$. 
Then there is $c=c(r)>0$ such that, for all $|y|, |y'|\leq r$, 
\[d_{{\rm{\tiny Hell}}}(\nu, \nu')\leq c|y-y'|.\]
\end{theorem}

\section{Assessing The Quality of Data Assimilation Algorithms}
\label{ssec:qual}

It is helpful when studying algorithms for data assimilation \index{data assimilation}
to ask two questions: (i) how informative is the data we have?
(ii) how good is our algorithm at extracting this information?
These are two separate questions, answers to both of which are
required in the quest to understand how well we perform at extracting
a signal, using model and data.
We take the two questions in separately, in turn; however we caution
that many applied papers entangle them both and simply measure algorithm
quality by ability to reconstruct the signal.

Answering question (i) is independent of any particular
algorithm: it concerns the properties of the Bayesian \index{Bayesian} posterior pdf itself.
In some cases we will be interested in studying the properties
of the probability distribution on the signal, or the initial
condition, for a particular instance of the data generated
from a particular instance of the signal, which we call
the {\bf truth}.\index{truth} In this context we will use the
notation $\yd=\{\yd_j\}$ to denote the {\it realization} of the data generated
from a particular realization of the truth $\vd=\{\vd_j\}.$
We first discuss properties 
of the smoothing\index{smoothing} problem for stochastic dynamics \index{stochastic dynamics}.
{\bf Posterior consistency}\index{posterior consistency}
concerns the question of the limiting behaviour of
$\bbP(v|\yd)$ as either $J \to \infty$ (large data sets)
or $|\Gamma| \to 0$ (small noise).
A key question is whether $\bbP(v|\yd)$ converges to
the truth in either of these limits; this might happen,
for example, if $\bbP(v|\yd)$ becomes closer and closer to
a Dirac probability measure centred on $\vd.$ When this occurs 
we say that the problem exhibits Bayesian \index{Bayesian} 
posterior consistency\index{posterior consistency!Bayesian}; 
it is then of interest to study the rate at which the limit is attained.
Such questions concern the information content of
the data; they do not refer to any algorithm and therefore they
are not concerned with the quality of any particular algorithm.
When considering filtering\index{filtering}, 
rather than smoothing\index{smoothing},
a particular instance of this question concerns 
marginal distributions\index{marginal distribution}:
for example one may be concerned with posterior consistency 
of $\bbP(v_J|\yd_J)$ with respect to a Dirac on $\vd_J$ in the 
filtering\index{filtering} case, 
see Theorem \ref{th12}; for the case of deterministic dynamics the distribution $\bbP(v|\yd)$
is completely determined by $\bbP(v_0|\yd)$ (see Theorem \ref{th12a}) 
so one may discuss posterior consistency of  $\bbP(v_0|\yd)$ with respect to 
a Dirac on $\vd_0.$

Here it is appropriate to mention the important concept of {\bf model
error}\index{model error}. In many (in fact most) applications the physical
system which generates the data set $\{y_j\}$ can be
(sometimes significantly) different from the mathematical model used, at least
in certain aspects. This can be thought of conceptually
by imagining data generated by \eqref{eq:dtf2}, with 
$\vd=\{\vd_j\}$ governed by the deterministic dynamics
\begin{subequations}
\begin{eqnarray}\label{eq:dtf111}
\vd_{j+1}=&\pst(\vd_j), \;j\in\Z^+. \\
\vd_0=&u\sim\G(m_0,C_0).
\end{eqnarray}
\end{subequations}
Here the function $\pst$ governs
the dynamics of the truth which underlies the data.
We assume that the true solution operator \index{solution operator} is not
known to us exactly, and seek instead to combine the data
with the stochastic dynamics \index{stochastic dynamics} model \eqref{eq:dtf1};
the noise $\{\xi_j\}$ is used to allow for the discrepancy
between the true solution operator \index{solution operator} $\pst$ and
that used in our model, namely $\PPsi.$
It is possible to think of many variants on this situation. For
example, the dynamics of the truth may be stochastic; or
 the dynamics of the truth may take place in a higher-dimensional space
than that used in our models, and may need to be projected
into the model space.  
Statisticians sometimes refer to the situation where the data source
differs from the model used as {\bf model misspecification}.\index{model misspecification}

We now turn from the information content, or quality, of the data to
the quality of algorithms for data assimilation\index{data assimilation}. 
We discuss three 
approaches to assessing quality. The first fully Bayesian \index{Bayesian} approach 
can be defined independently of the quality of the data. The
second estimation approach entangles the properties of the
algorithm with the quality of the data. We discuss these two
approaches in the context of the smoothing\index{smoothing} problem for
stochastic dynamics \index{stochastic dynamics}. The reader will easily see how to generalize
to smoothing\index{smoothing} for deterministic dynamics, or to 
filtering\index{filtering}.
The third approach is widely used in operational numerical weather
prediction\index{weather forecasting}
and judges quality by the ability to predict.

{\bf Bayesian Quality Assessment}\index{Bayesian quality assessment}. 
Here we assume that the
algorithm under consideration provides an approximation
$\bbaP(v|y)$ to the true posterior distribution\index{posterior distribution}
 $\bbP(v|y).$
We ask the question: how close is $\bbaP(v|y)$ to $\bbP(v|y)$.
We might look for a distance measure between probability distributions,
or we might simply compare some important moments of the distributions,
such as the mean and covariance. Note that this version of quality
assessment does not refer to the concept of a true solution $\vd$.
We may apply it with $y=\yd$, but we may also apply it when there
is model error \index{model error} present and the data comes from outside the model used
to perform data assimilation \index{data assimilation}. However, if combined with Bayesian \index{Bayesian}
posterior consistency, when $y=\yd$, then the triangle inequality
relates the output of the algorithm to the truth $\vd$. 
Very few practitioners evaluate their algorithms by this measure.
This reflects the fact that knowing the true distribution $\bbP(v|y)$
is often difficult in practical high dimensional problems. However
it is arguably the case that practitioners should spend more time
querying their algorithms from the perspective of
Bayesian quality assessment since the algorithms are often used to
make probabilistic statements and forecasts.

{\bf Signal Estimation Quality Assessment}\index{signal estimation quality assessment}. 
Here we assume that the algorithm under consideration provides an approximation
to the signal $v$ underlying the data, which we denote by $v_{\rm approx};$ 
thus $v_{\rm approx}$ attempts to determine and then {\em track} the 
true\index{truth} signal\index{signal} from the data\index{data}.
If the algorithm actually provides a probability
distribution, then this estimate might be, for example, the
mean. We ask the question: if the algorithm is applied in the situation
where the data $\yd$ is generated from the the signal $\vd$, 
how close is $\va$ to $\vd$?
There are two important effects at play here: the first is the information
content of the data -- does the data actually contain enough information
to allow for accurate reconstruction of the signal in principle; and the 
second is the role of the specific algorithm used -- does the 
specific algorithm in question
have the ability to extract this information when it is present.
This approach thus measures the overall effect of these two in combination.

{\bf Forecast Skill}\index{forecast skill}. 
In many cases the goal of data assimilation is 
to provide better forecasts of the future, for example in numerical
weather prediction\index{weather forecasting}. 
In this context data assimilation algorithms can 
be benchmarked by their ability to make forecasts. 
This can be discussed in both the Bayesian quality and signal
estimation\index{signal estimation quality assessment}\index{Bayesian quality
assessment} senses. 
We first discuss Bayesian \index{Bayesian} Estimation forecast skill \index{forecast skill} in the context of
stochastic dynamics \index{stochastic dynamics}. 
The Bayesian \index{Bayesian} $k$-lag forecast skill \index{forecast skill} can be defined by studying 
the distance \index{distance} between the approximation $\bbaP(v|y)$ and $\bbP(v|y)$ 
when both are pushed forward from the end-point of the data assimilation
window by $k$ applications of the dynamical model \eqref{eq:dtf1};
this model defines a Markov transition kernel which is applied $k-$times
to produce a forecast. 
We now discuss signal estimation\index{signal estimation quality assessment} forecast skill \index{forecast skill} in the context of
deterministic dynamics. Using $v_{{\rm approx}}$ at the end point of
the assimilation window as an initial condition,
we run the model \eqref{eq:dtf11} forward by $k$ steps and compare
the output with $v^\dagger_{j+k}.$ 
In practical application, this forecast methodology inherently
confronts the effect of model error \index{model error}, 
since the data used to test forecasts is real
data which is not generated by the model used to assimilate, as well
as information content in the data and algorithm quality.

\section{Illustrations}
\label{ssec:ill}

In order to build intuition concerning the probabilistic viewpoint
on data assimilation
we describe some simple examples where the posterior distribution\index{posterior distribution} may
be visualized easily. For this reason we concentrate on the case
of one-dimensional deterministic dynamics; the posterior pdf 
$\bbP(v_0|y)$ for deterministic dynamics \index{deterministic dynamics} is given by 
Theorem \ref{th112}.  It is one-dimensional when the dynamics is 
one-dimensional and takes place in $\bbR$. 
In section \ref{sec:dtsa} we will introduce 
more sophisticated sampling methods to probe probability distributions 
in higher dimensions which arise from noisy dynamics and/or from
high dimensional models.

Figure \ref{fig:smooth1}  concerns the scalar linear
problem from Example \ref{ex:ex1} (recall that throughout this
section we consider only the case of deterministic
dynamics)  with $\lambda=0.5.$ 
We employ a prior $N(4,5)$,
we assume that $h(v)=v$, and we set $\Gamma=\gamma^{2}$ and consider
two different values of $\gamma$ and two different values of $J$,
the number of observations. 
The figure shows the posterior distribution\index{posterior distribution}
in these various parameter regimes. 
The true value of the initial condition which underlies the
data is $\vd_0=0.5.$ For both $\gamma=1.0$ and $0.1$ we see that,
as the number of observations $J$ increases, the posterior distribution 
{appears to converge to a limiting distribution. 
However for smaller $\gamma$ the limiting distribution has much smaller
variance, and is centred closer to the true initial condition at $0.5$.
Both of these observations can be explained, using the fact that
the problem is explicitly solvable: we show that for fixed $\gamma$
and $J \to \infty$ the posterior distribution has a limit, which is a
Gaussian with non-zero variance. And for fixed $J$ as $\gamma \to 0$ the
posterior distribution converges to a Dirac measure (Gaussian with zero
variance) centred at the truth $\vd_0.$ 

To see these facts we start by noting that from Theorem \ref{th112} the posterior
distribution on $v_0|y$ is proportional to the exponential of
$$\Iid(v_0;y)=\frac{1}{2\gamma^2}\sum_{j=0}^{J-1}|y_{j+1}-\lambda^{j+1}v_0|^2
+\frac{1}{2\sigma_0^2}|v_0-m_0|^2$$
where $\sigma_0^2$ denotes the prior variance $C_0.$ As a quadratic
form in $v_0$ this defines a Gaussian posterior distribution and
we may complete the square to find the posterior mean $m$ and variance
$\sigp^2$:
$$\frac{1}{\sigp^{2}}=\frac{1}{\gamma^2}\sum_{j=0}^{J-1}\lambda^{2(j+1)}+\frac{1}{\sigma_0^2}=\frac{1}{\gamma^2}\Bigl(\frac{\lambda^2-\lambda^{2J+2}}{1-\lambda^2}\Bigr)+\frac{1}{\sigma_0^2}$$
and
$$\frac{1}{\sigp^2} m= \frac{1}{\gamma^2}\sum_{j=0}^{J-1}\lambda^{(j+1)}y_{j+1}+\frac{1}{\sigma_0^2}m_0.$$
We note immediately that the posterior variance is independent of the
data. Furthermore, if we fix $\gamma$ and let $J \to \infty$ then
for any $|\lambda|<1$ we see that the large $J$ limit of the posterior
variance is determined by
$$\frac{1}{\sigp^{2}}=\frac{1}{\gamma^2}\Bigl(\frac{\lambda^2}{1-\lambda^2}\Bigr)+\frac{1}{\sigma_0^2}$$
and is non-zero; thus uncertainty remains in the posterior, even in the
limit of large data. 
On the other hand, if 
we fix $J$ and let $\gamma \to 0$
then $\sigp^2 \to 0$ so that uncertainty disappears in this limit. It
is then natural to ask what happens to the mean. To this end we assume
that the data is itself generated by the linear model of Example \ref{ex:ex1}
so that 
$$y_{j+1}=\lambda^{j+1}\vd_0+\gamma \zeta_{j+1}$$
where $\zeta_j$ is an i.i.d. \index{i.i.d.} Gaussian sequence with
$\zeta_1 \sim N(0,1).$
Then
$$\frac{1}{\sigp^2} m=\frac{1}{\gamma^2}\Bigl(\frac{\lambda^2-\lambda^{2J+2}}{1-\lambda^2}\Bigr)\vd_0+\frac{1}{\gamma} \sum_{j=0}^{J-1}\lambda^{(j+1)} \zeta_{j+1}+\frac{1}{\sigma_0^2}m_0.$$
Using the formula for $\sigp^2$ we obtain
$$\Bigl(\frac{\lambda^2-\lambda^{2J+2}}{1-\lambda^2}\Bigr)m+\frac{\gamma^2}{\sigma_0^2}m=\Bigl(\frac{\lambda^2-\lambda^{2J+2}}{1-\lambda^2}\Bigr)\vd_0+\gamma \sum_{j=0}^{J-1}\lambda^{(j+1)} \zeta_{j+1}+\frac{\gamma^2}{\sigma_0^2}m_0.$$
From this it follows that, for fixed $J$ and as $\gamma \to 0$,
$m \to \vd_0$, almost surely with respect to the noise realization
$\{\zeta_j\}_{j \in \J}.$ This is an example of posterior consistency.
\index{posterior consistency}}

\begin{figure}[h]
\centering
\subfigure[$\gamma=1$]{\includegraphics[scale=0.365]{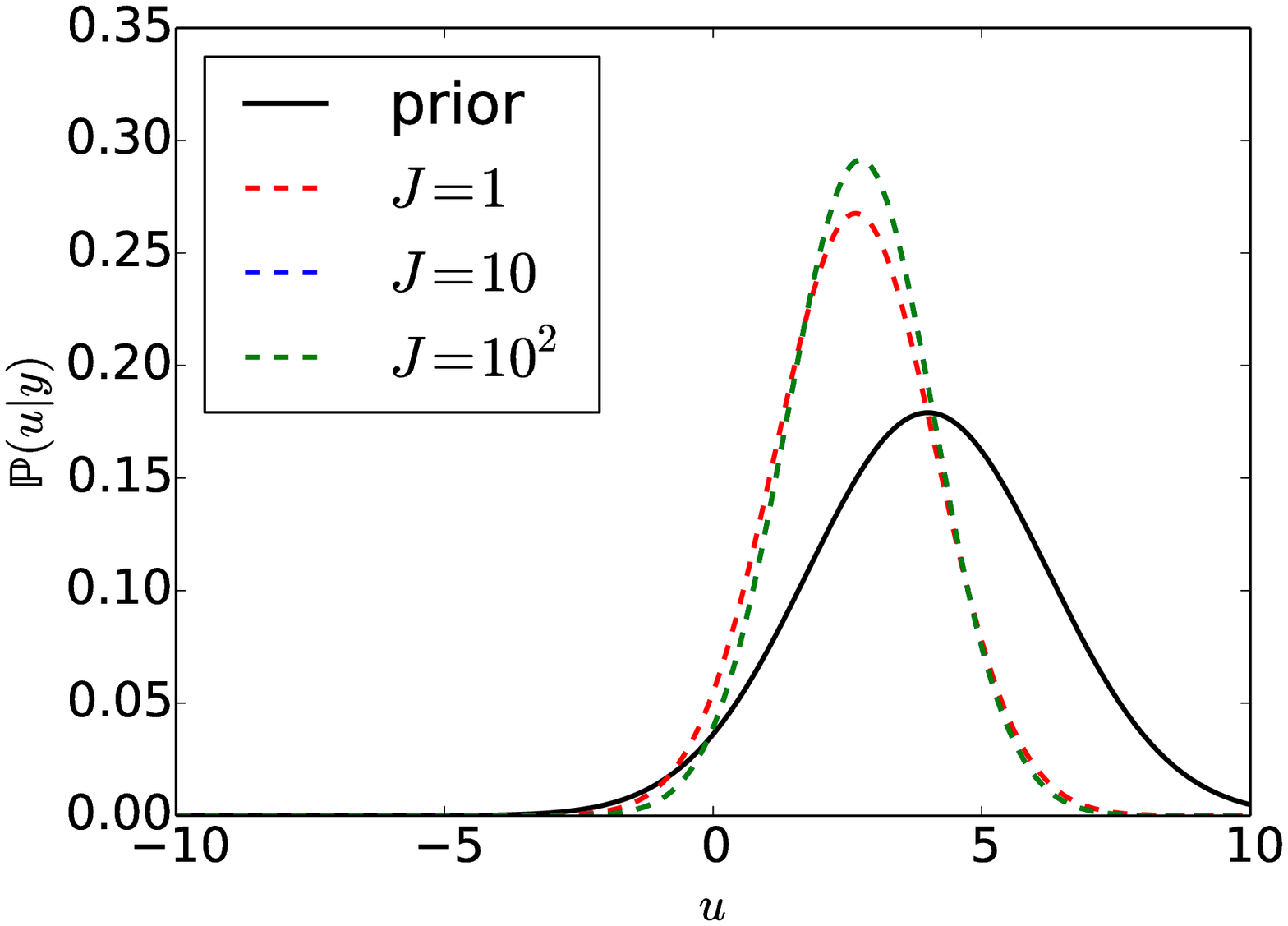}}
\subfigure[$\gamma=0.1$]{\includegraphics[scale=0.365]{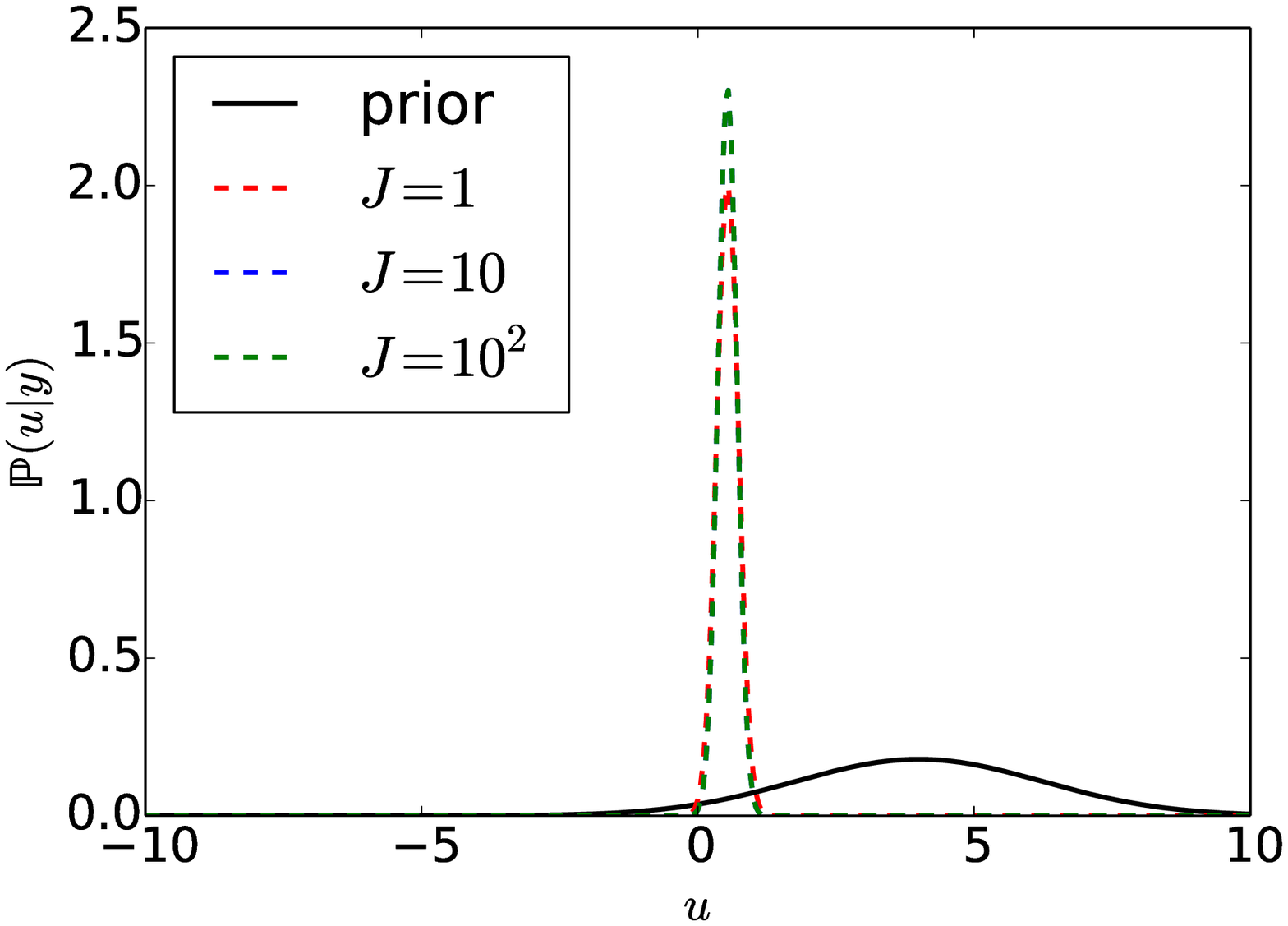}}
\caption{Posterior distribution for Examples \ref{ex:ex1} for different levels of observational noise. The true initial condition used in both cases is $v_{0}=0.5$, while we have assumed that $C_{0}=5$ and $m_{0}=4$ for the prior distribution.}
\label{fig:smooth1}
\end{figure}

We now study Example \ref{ex:ex4} in which the true 
dynamics are no longer linear. We start our investigation 
taking $r=2$ and investigate the effect of choosing 
different prior distributions. Before discussing the properties
of the posterior we draw attention to two facts.
Firstly, as Figure \ref{fig:ex4}a shows, the system converges
in a small number of steps
to the fixed point  \index{fixed point} at $1/2$ for this value of $r=2$.
And secondly the initial conditions $v_0$ and $1-v_0$ both result
in the same trajectory, if the initial condition is ignored.
The first point implies that, after a small number of steps,
the observed trajectory contains very little information about
the initial condition. The second point means that, since
we observe from the first step onwards, only the prior can distinguish
between $v_0$ and $1-v_0$ as the starting point. 

Figure \ref{fig:smooth2} concerns an experiment in which the
true initial condition underlying the data is $\vd_0=0.1$. Two different
priors are used, both with $C_0=0.01$, giving a standard deviation of
$0.1$, but with different means. The figure illustrates two facts:
firstly, even with $10^3$ observations, the posterior contains
considerable uncertainty, reflecting  the first point above. Secondly
the prior mean has an important role in the form of the posterior pdf:
shifting the prior mean to the right, from $m_0=0.4$ to $m_0=0.7$,
results in a posterior which favours the initial condition $1-\vd_0$
rather than the truth \index{truth} $\vd_0.$

\begin{figure}[h]
\centering
\subfigure[$r=2$, $m_{0}=0.4$]{\includegraphics[scale=0.365]{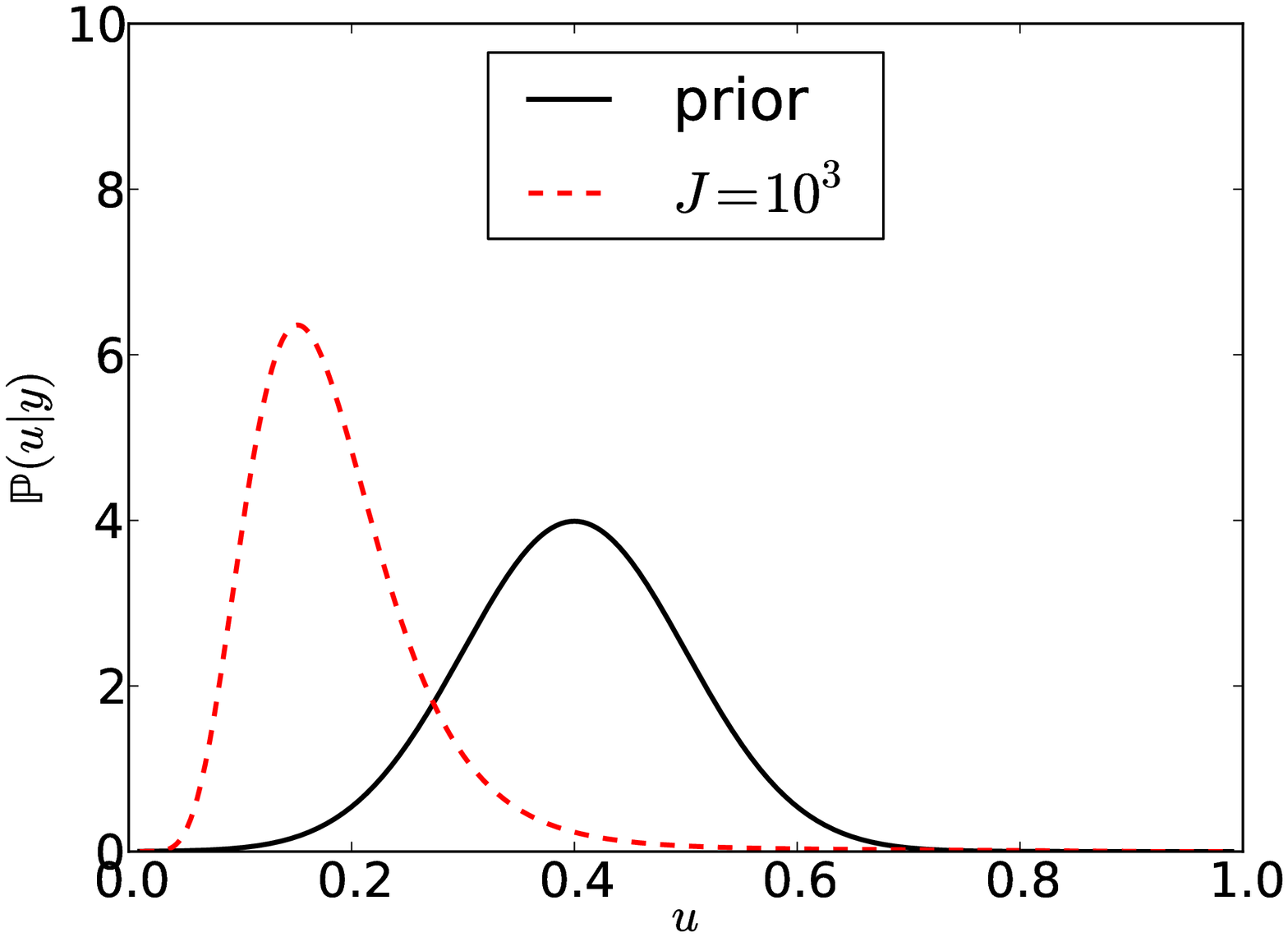}}
\subfigure[$r=2$ $m_{0}=0.7$]{\includegraphics[scale=0.365]{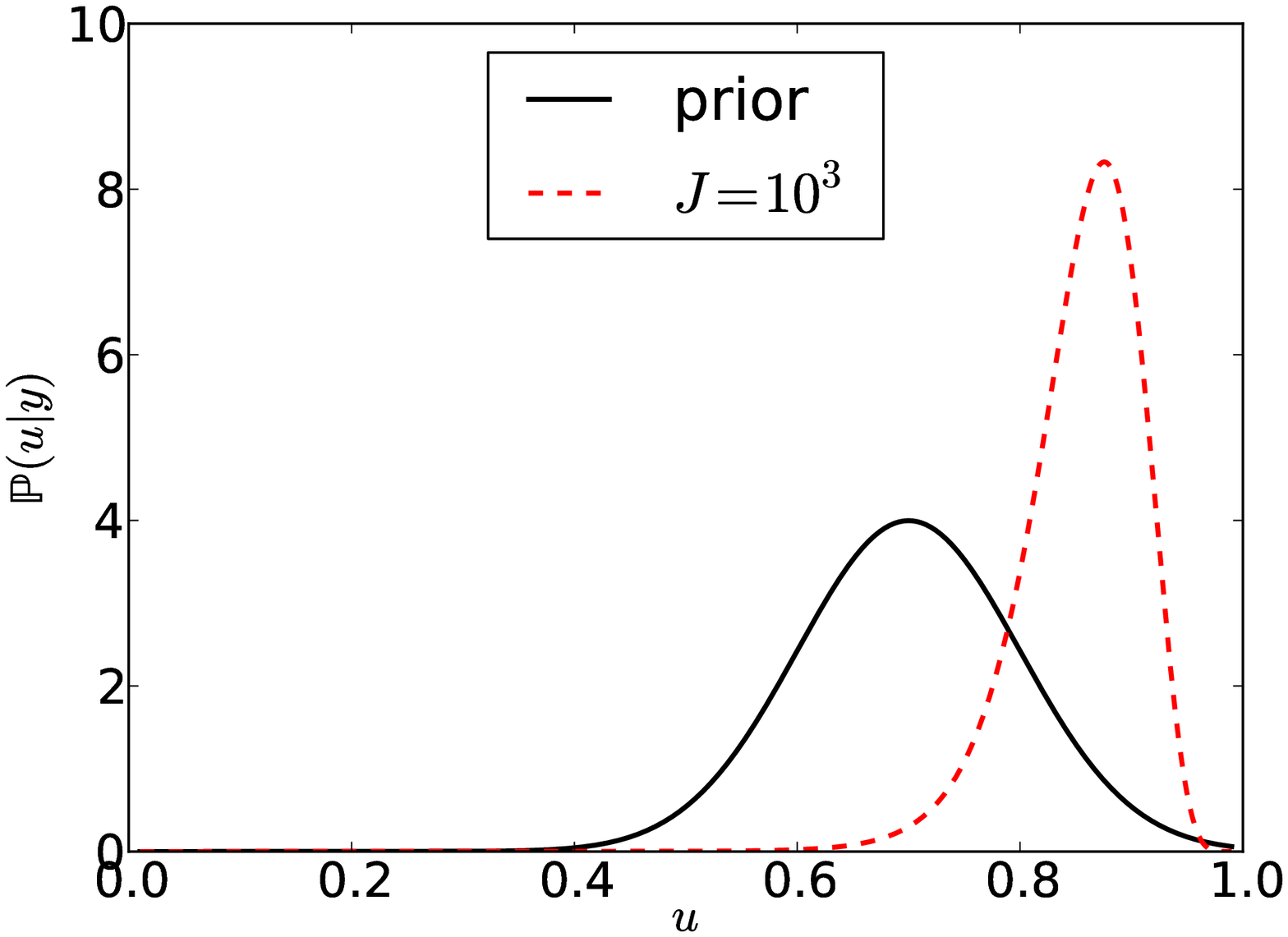}}
\caption{Posterior distribution for Example \ref{ex:ex4} for $r=2$ in the case of different means for the prior distribution. We have used $C_{0}=0.01$, $\gamma=0.1$ and true initial condition $v_{0}=0.1$, see also {\tt p2.m} in section \ref{ssec:p2}.}
\label{fig:smooth2}
\end{figure}

\begin{figure}[h]
\centering
\subfigure[$r=2$, $C_{0}=0.5$]{\includegraphics[scale=0.365]{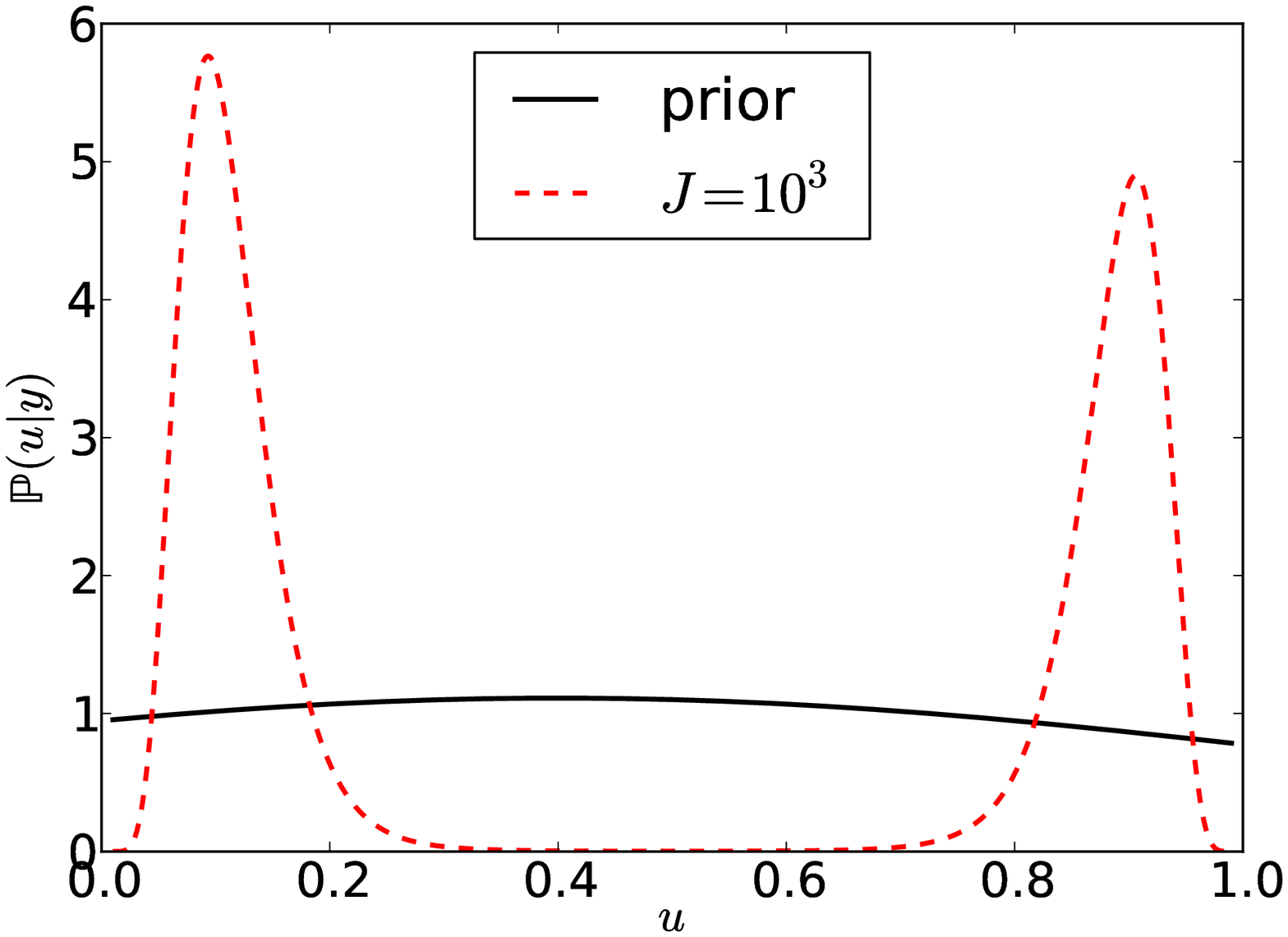}}
\subfigure[$r=2$ $C_{0}=5$]{\includegraphics[scale=0.365]{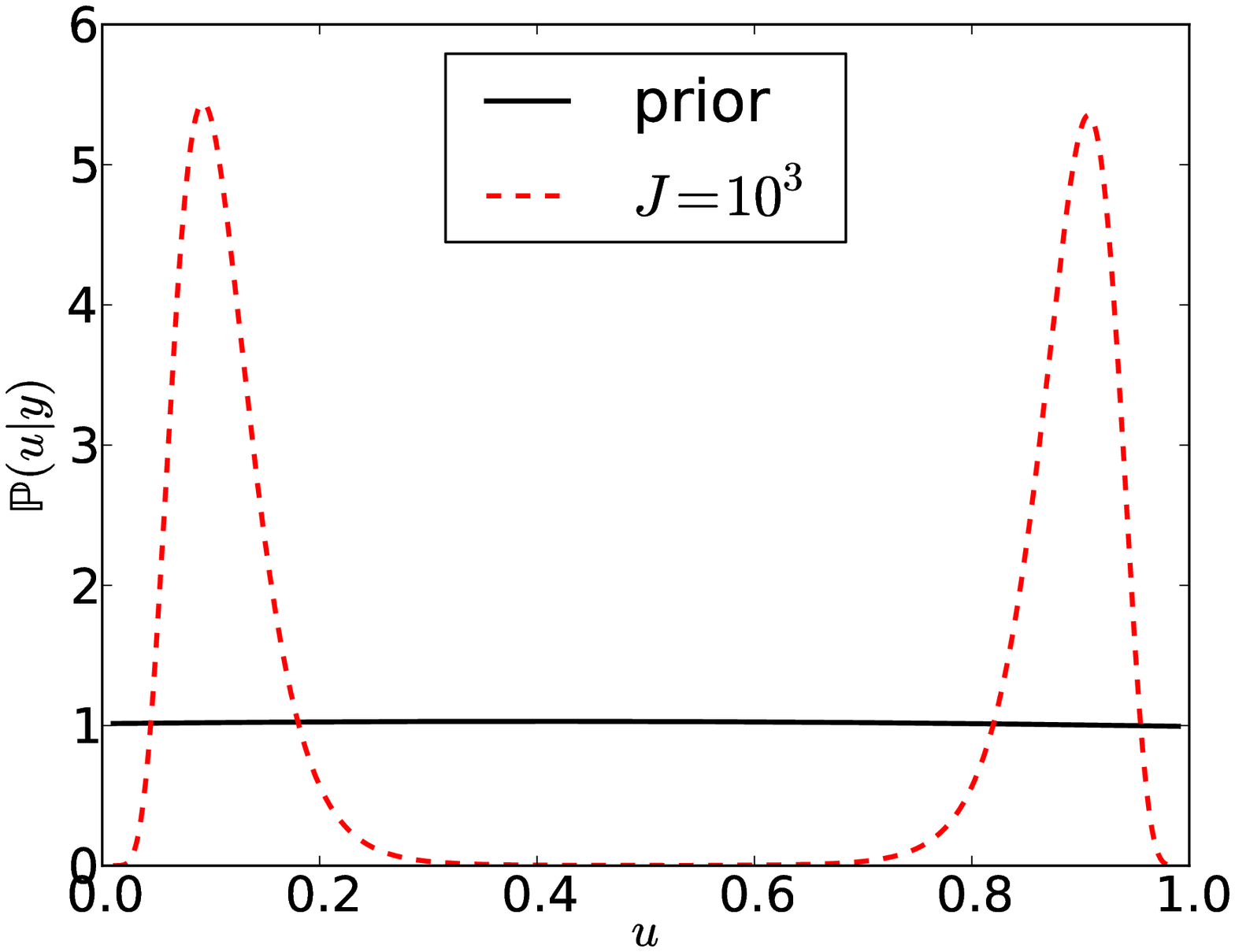}}
\caption{Posterior distribution for Example \ref{ex:ex4} for $r=2$ in the case of different covariance for the prior distribution. We have used $m_{0}=0.4$, $\gamma=0.1$ and true initial condition $v_{0}=0.1$.}
\label{fig:smooth2a}
\end{figure}

This behaviour of the posterior changes completely if we assume a flatter prior.
This is illustrated in Figure \ref{fig:smooth2a} where we consider the
prior $N(0.4,C_0)$ with $C_0=0.5$ and $5$ respectively. As we increase the 
prior covariance the mean plays a much weaker role than in the
preceding experiments: we now obtain a bimodal posterior centred 
around both the true initial condition 
$\vd_{0}$, and also around $1-\vd_{0}$.

In Figure  \ref{fig:smooth3} we consider the quadratic map \eqref{eq:ex4}
with $r=4,$ $J=5$ and prior $N(0.5,0.01)$, with
observational {standard deviation} $\gamma=0.2$. Here, 
after only five observations \index{observations} the posterior is very peaked,
although because of the $v \mapsto 1-v$ symmetry mentioned above, there
are two symmetrically related peaks; see Figure  \ref{fig:smooth3}a.
It is instructive to look at the negative of the logarithm of the posterior
pdf which, upto an additive constant, is given by $\Iid(v_0;y)$ in
Theorem \ref{th112}. The function $\Iid(\cdot;y)$
is shown in Figure  \ref{fig:smooth3}b. Its complexity indicates the 
considerable complications underlying solution of the smoothing\index{smoothing} problem.
We will return to this last point in detail later. 
Here we simply observe that normalizing the posterior distribution\index{posterior distribution} requires evaluation of the integral
\[
\int_{\bbR^n} e^{-\Ii(v_{0},y)}dv_{0}.
\] 
This integral may often be determined almost entirely 
by very small subsets of $\bbR^n$, meaning that this calculation
requires some care; indeed if $\Ii(\cdot)$ is very large over much of its 
domain then it may be impossible to compute the normalization
constant\index{normalization constant} numerically.  
We note, however, that the sampling 
methods that we will describe in the next chapter do not require evaluation
of this integral.
 
\begin{figure}[h]
\centering
\subfigure[$r=4$]{\includegraphics[scale=0.365]{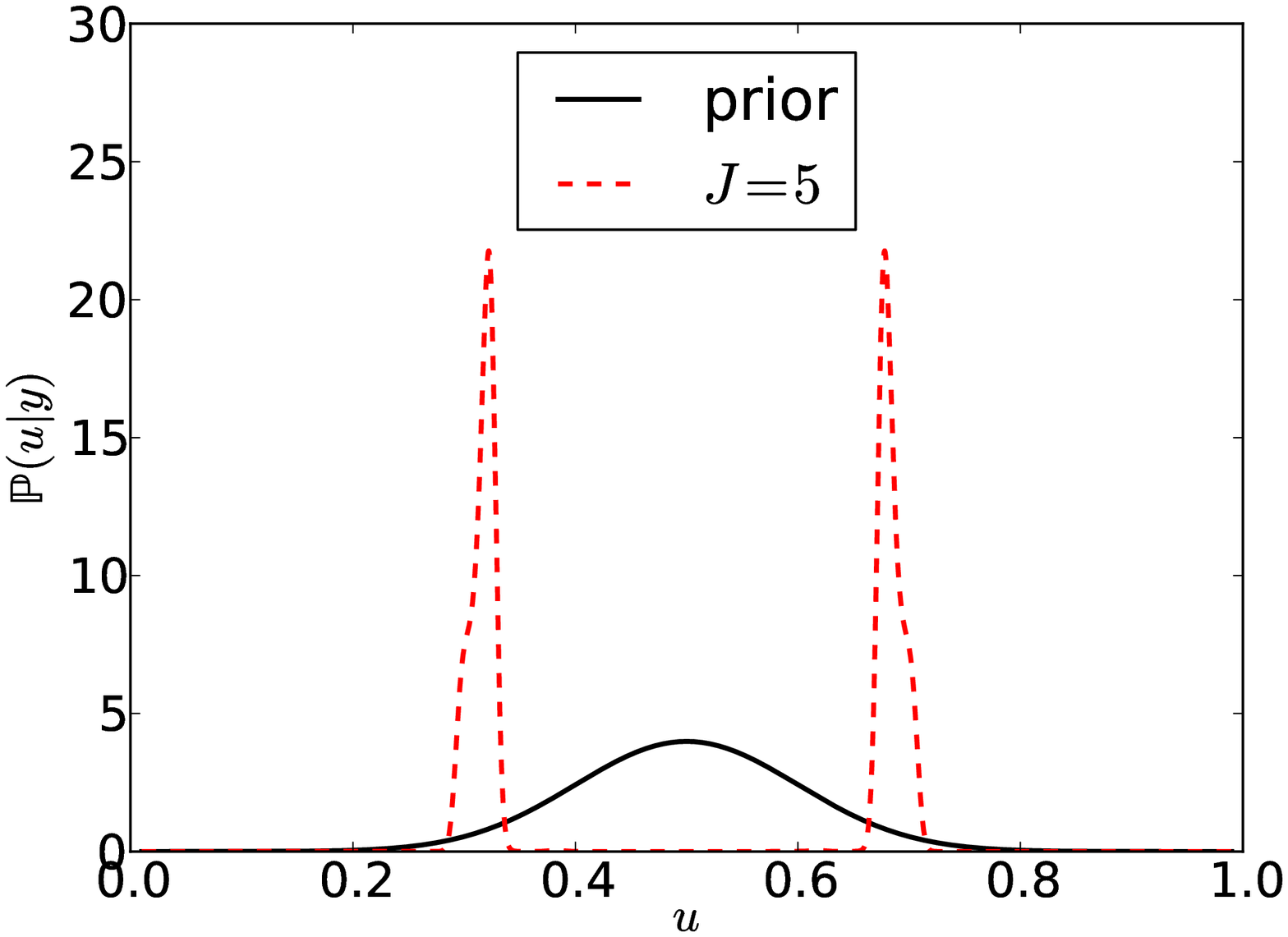}}
\subfigure[$\Ii(v_{0};y)$ $r=4, J=5$]{\includegraphics[scale=0.365]{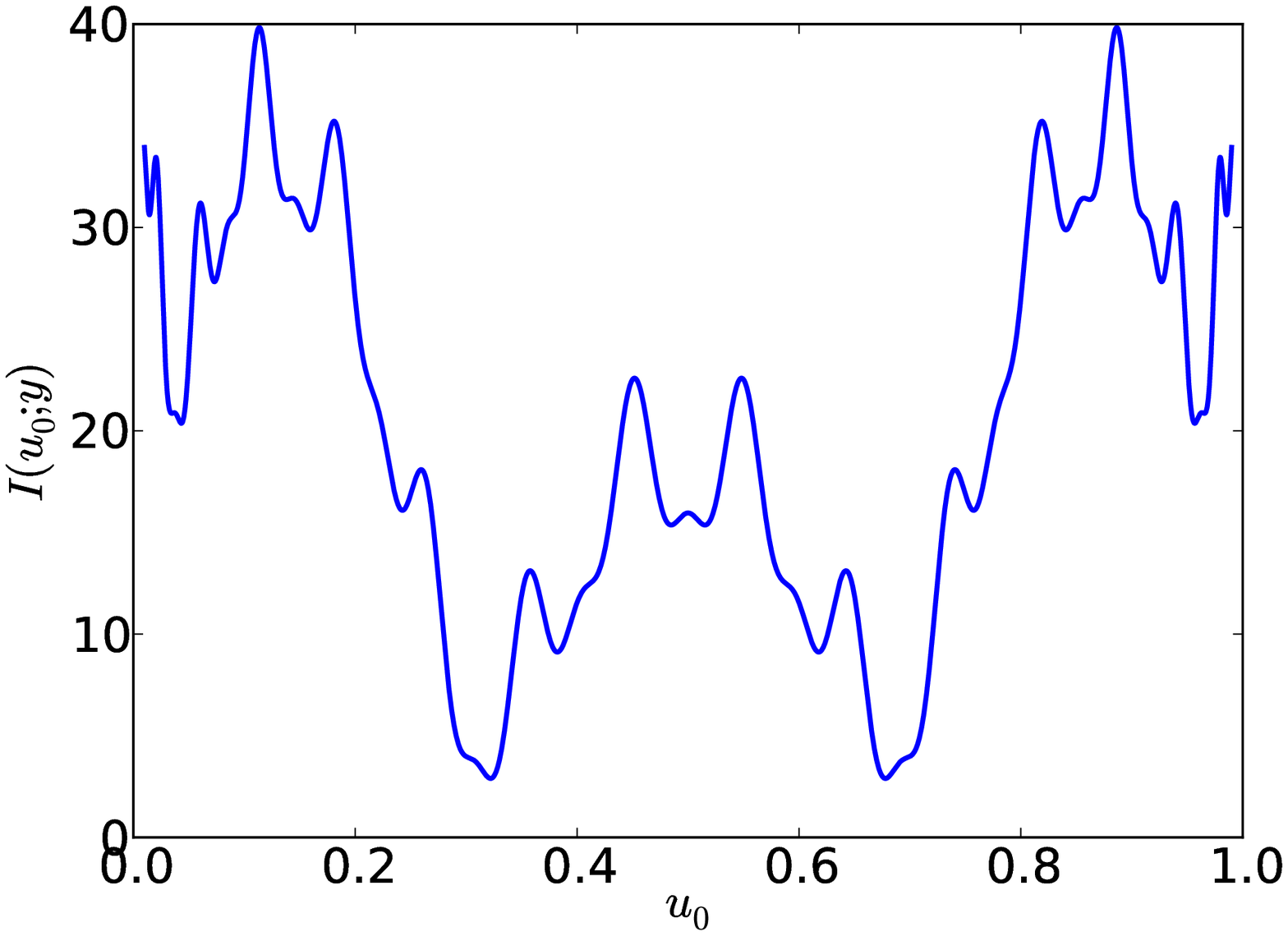}}
\caption{Posterior distribution and negative log posterior
for Example \ref{ex:ex4} for $r=4$ and $J=5$. We have used $C_{0}=0.01, m_{0}=0.5$, $\gamma=0.2$ and true initial condition $v_{0}=0.3$.}
\label{fig:smooth3}
\end{figure}

\section{Bibliographic Notes}
\label{ssec:bib}

\begin{itemize}

\item Section \ref{ssec:s} Data Assimilation has its
roots in the geophysical sciences\index{geophysical sciences}, 
and is driven by the desire to improve
inaccurate models of complex dynamically evolving phenomena by means
of incorporation of data. The book \cite{kal03} describes
data assimilation  \index{data assimilation}from the viewpoint of the atmospheric 
sciences\index{atmospheric sciences} and weather 
prediction\index{weather forecasting},
whilst the book \cite{ben02} describes the subject from the viewpoint
of oceanography\index{oceanography}. 
These two subjects were the initial drivers for
evolution of the field. However, other applications are increasingly
using the methodology of data assimilation, and the oil industry\index{oil recovery} 
in particular is heavily involved in the use, and development, of
algorithms in this area \cite{orl08}.
The recent book \cite{abarbanel13} provides a perspective on the
subject from the viewpoint of physics and nonlinear dynamical systems, and
includes motivational examples from neuroscience, as well as the geophysical
sciences. The article \cite{ICGL97} is a useful one to
read because it establishes a notation which is now widely used
in the applied communities and the articles \cite{Nic02,ApteJV08} provide
simple introductions to various aspects of
the subject from a mathematical perspective.
The special edition of the journal PhysicaD, devoted to
 Data Assimilation, \cite{JI07}, provides an overview of the state of
the art around a decade ago.

\item It is useful to comment on generalizations of the set-up
described in section \ref{ssec:s}. First we note that we have
assumed a {\em Gaussian} structure for the additive noise appearing
in both the signal model \eqref{eq:dtf1} and the data model 
\eqref{eq:dtf2}. This is easily relaxed in much of what
we describe here, provided that an explicit formula for the
probability density function \index{probability density function}  of the noise is known. However
the Kalman filter\index{Kalman filter}, described in the next chapter,
relies explicitly on the closed 
Gaussian form of the probability distributions resulting from
the assumption of Gaussian noise. There are also other
parts of the notes, such as the pCN \index{pCN} MCMC methods 
and the minimization principle underlying approximate Gaussian 
filters\index{filter!approximate Gaussian}, also both described
in the next chapter, which require the Gaussian structure.
Secondly we note that we have assumed {\em additive} noise. 
This, too, can be relaxed but has the complication that most
non-additive noise models do not yield explicit formulae for
the needed conditional probability density functions;\index{probability density function}
for example this situation arises if one looks at stochastic
differential equations\index{stochastic differential equations}
over discrete time-intervals -- see \cite{beskos2006exact} and the
discussion therein. However some of the methods
we describe rely only on drawing samples from the desired
distributions and do not require the explicit conditional
probability density function \index{probability density function}. Finally we note that much 
of what we describe here translates to {\em infinite dimensional}
spaces with respect to both the signal space, and the the
data space; however in the infinite dimensional data space 
case the additive Gaussian observational noise is currently the
only situation which is well-developed \cite{article:Stuart2010}. 

\item Section \ref{ssec:ge}. The subject of deterministic
discrete time dynamical systems of the form \eqref{eq:dtf11} 
is overviewed in numerous texts; see \cite{Wig90} and the
references therein, and Chapter 1 of \cite{SH96}, for example. The
subject of stochastic discrete time dynamical systems of the form  
\eqref{eq:dtf1}, and in particular the property of ergodicity\index{ergodic}
which underlies Figure \ref{fig:ex34}, is covered in some depth
in \cite{MT93}. The exact solutions of the quadratic map
(\ref{eq:ex4}) for $r=2$ and $r=4$ may be found in
\cite{raey} and \cite{TUS} respectively.
The Lorenz '63 model was introduced
in \cite{lorenz1963deterministic}. Not only does this paper demonstrate
the possibility of chaotic behaviour and  sensitivity with respect to
initial conditions, but it also makes a concrete connection between
the three dimensional continuous time dynamical system and
a one-dimensional chaotic map of the form \eqref{eq:dtf1}.
Furthermore, a subsequent computer assisted proof demonstrated
rigorously that the ODE does indeed exhibit chaos \cite{tucker1, tucker2}.
The book \cite{Sparrow82} discusses properties of the Lorenz '63 model
in some detail and the book \cite{falconer} discussed properties
such as fractal dimension.
The shift of origin that we have adopted for the Lorenz '63
model is explained in \cite{book:Temam1997}; it enables the
model to be written in an abstract form which  includes many 
geophysical\index{geophysical sciences} models of interest, 
such as the Lorenz '96 model
introduced in \cite{lorenz1996predictability}, and the Navier-Stokes
equation on a two-dimensional torus \cite{MW06,book:Temam1997}. 
We now briefly describe this common abstract form. 
The vector $u \in \bbR^J$ ($J=3$ for Lorenz '63, $J$ arbitrary for Lorenz' 96) 
solves the equation
\begin{equation}
\frac{du}{dt}+Au+B(u,u)=f,\quad u(0)=u_0,\label{eq:0}
\end{equation}
where there is $\lambda>0$ such that, for all $w \in \bbR^J,$
$$\langle Aw,w \rangle \ge \lambda|w|^2, \quad\langle B(w,w),w \rangle=0.$$
Taking the inner-product with $u$ shows that
$$\frac12\frac{d}{dt}|u|^2+\lambda|u|^2 \leq \langle f,u \rangle.$$
If $f$ is constant in time then this inequality may be used to show that
\eqref{eq:diss} holds:
$$\frac12\frac{d}{dt}|u|^2 \le \frac{1}{2\lambda}|f|^2-\frac{\lambda}{2}|u|^2.$$
Integrating this inequality gives the existence
of an absorbing set\index{absorbing set}
and hence leads to the existence of a 
global attractor;\index{global attractor} see Example \ref{ex:diss}, the book \cite{book:Temam1997} or Chapter 2 of \cite{SH96}, for example.

\item Section \ref{ssec:sp} contains the formulation of Data Assimilation
as a fully nonlinear and non-Gaussian problem in Bayesian \index{Bayesian} statistics.
This formulation is not yet the basis of practical algorithms in the 
geophysical\index{geophysical sciences}
systems such as weather forecasting\index{weather forecasting}.  This 
is because global weather\index{weather forecasting}
forecast models involve $n={\cal O}(10^9)$ unknowns, and incorporate
$m={\cal O}(10^6)$ data points daily; sampling the posterior on $\bbR^n$
given data in $\bbR^m$ in an online fashion, usable for forecasting, is
beyond current algorithmic and computational capability.
However the fully Bayesian \index{Bayesian} perspective provides a fundamental mathematical
underpinning of the subject, from which other more tractable approaches
can be systematically derived. See \cite{article:Stuart2010} for discussion of
the Bayesian \index{Bayesian} approach to inverse problems. Historically, data assimilation
has not evolved from this Bayesian \index{Bayesian} perspective, but has rather
evolved out of the control theory perspective. This perspective
is summarized well in the book \cite{jaz70}. However, the importance
of the Bayesian \index{Bayesian} perspective is increasingly being recognized in the
applied communities. In addition to providing a starting point
from which to derive approximate algorithms, it also provides
a gold standard against which other more {\em ad hoc} algorithms
can be benchmarked; this use of Bayesian \index{Bayesian} methodology was
suggested in \cite{LS12} in the context of meteorology (see discussion
that follows), and then 
employed in \cite{ILS13} in the context of 
subsurface\index{geophysical sciences!subsurface} inverse
problems arising in geophysics\index{geophysical sciences}.

\item Section \ref{ssec:fp} describes the filtering\index{filtering}, 
or sequential,
approach to data assimilation, within the fully Bayesian \index{Bayesian} framework.
For low dimensional systems the use of particle filters\index{filter!particle}, which
may be shown to rigorously approximate the required 
filtering distribution\index{filtering!distribution}
as it evolves in discrete time, has been enormously successful;
see \cite{DdFG01} for an overview. Unfortunately, these filters
can behave poorly in high dimensions \cite{BLB08,BBL08,SBBA08}.
Whilst there is ongoing work to overcome these problems with high-dimensional
particle filtering\index{filter!particle}, see \cite{bcj11,CMT10,vl10} for example, this work
has yet to impact practical data assimilation in, for example,
operational weather forecasting\index{weather forecasting}. 
For this reason the {\em ad hoc}
filters, such as 3DVAR\index{3DVAR}, extended Kalman Filter\index{Kalman filter!extended} and ensemble Kalman
Filter\index{Kalman filter!ensemble}, 
described in Chapter \ref{sec:dtfa}, are of
great practical importance. Their analysis is hence an important
challenge for applied mathematicians.

\item Section \ref{ssec:wp} Data assimilation may be viewed as
an inverse problem to determine the signal from the observations.
Inverse problems in differential equations are often ill-posed
when viewed from a classical non-probabilistic perspective. One
reason for this is that the data may not be informative about the
whole signal so that many solutions are possible.
However taking the Bayesian \index{Bayesian} viewpoint, in which the many solutions
are all given a probability, allows for well-posedness to be
established. This idea is used for data assimilation
problems arising in fluid mechanics in \cite{CDRS08}, for
inverse problems arising in subsurface geophysics\index{geophysical sciences!subsurface} in \cite{ClDS10,DHS12}
and described more generally in \cite{article:Stuart2010}. 
Well-posedness with respect to changes in the data is of importance
in its own right, but also more generally because it underpins other
stability results which can be used to control perturbations.
In particular the effect of numerical approximation on integration
of the forward model can be understood in terms of its effect on
the posterior distribution\index{posterior distribution}; see \cite{CDS09}.
A useful overview of probability metrics, including 
Hellinger \index{metric!Hellinger}
and total variation metrics,\index{metric!total variation}
is contained in \cite{GS02}.

\item Section \ref{ssec:qual}. The subject of posterior consistency
is central to the theory of statistics in general \cite{van2000asymptotic}, 
and within Bayesian \index{Bayesian} statistics in particular 
\cite{berger1985statistical,bs94,ghosal1999consistency}.
Assessing the quality of data assimilation\index{data assimilation}
algorithms is typically
performed in the ``signal\index{signal} estimation'' framework using
{\bf identical twin experiments}\index{identical twin experiments} in which the data is generated
from the same model used to estimate the signal; see \cite{JI07}
and the references therein.  
The idea of assessing ``Bayesian quality'' has only recently
been used within the data assimilation literature; see
\cite{LS12} where this approach is taken for the Navier-Stokes
inverse problem formulated in \cite{CDRS08}.
The evaluation of algorithms by means of forecast skill \index{forecast skill}
is enormously influential in the field of numerical weather
prediction\index{weather forecasting}
and drives a great deal of algorithmic selection.
The use of information theory to understand the effects of
model error \index{model error}, and to evaluate filter performance, is
introduced in \cite{MGH} and \cite{MB} respectively.
There are also a number of useful {\bf consistency checks} \index{consistency checks} which can
be applied to evaluate the computational model and its fit
to the data \cite{eve06, anderson1996method, anderson1999monte}.  
We discuss the idea of the variant known as {\em rank
histograms\index{histogram!rank}} at the end of Chapter \ref{sec:dtfa}.
If the empirical statistics of the innovations are 
inconsistent with the assumed model, then they can
be used to improve the model used in the future; this is known as 
{\em reanalysis.}\index{reanalysis} 

\end{itemize}

\section{Exercises}
\label{ex:first}

\begin{enumerate}

\item Consider the map given in Example \ref{ex:ex3}  and related program
{\tt p1.m}. By experimenting with the code determine approximately
the value of $\alpha$, denoted by $\alpha_1$, 
at which the noise-free dynamics
changes from convergence to an equilibrium point \index{equilibrium point} to convergence to a
period $2$ solution. Can you find a value of $\alpha=\alpha_2$ for which you
obtain a period $4$ solution? Can you find a value of $\alpha=\alpha_3$
for which you
obtain a non-repeating (chaotic) solution? For the values $\alpha=2, \alpha_2$ and $\alpha_3$ 
compare the trajectories of the dynamical system obtained with the initial
condition $1$ and with the initial condition $1.1$. Comment on what you find. 
Now fix the initial condition at $1$ and consider the same values
of $\alpha$, with and without noise ($\sigma \in \{0,1\}$). Comment on what
you find. Illustrate your answers with graphics. To get interesting
displays you will find it helpful to change the value of $J$ (number of
iterations) depending upon what you are illustrating.

\item  Consider the map given in Example \ref{ex:ex4} and verify the explicit
solutions given for $r=2$ and $r=4$ in formulae \eqref{eq:solr2}--\eqref{eq:solr4}. 

\item Consider the Lorenz '63 model given in Example \ref{ex:ex6}. 
Determine values of $\{\alpha,\beta\}$ for which \eqref{eq:diss} holds.

\item Consider the Lorenz '96 model given in Example \ref{ex:ex7}. Program
{\tt p19.m} plots solutions of the model, as well as studying
sensitivity to initial conditions. Study the behaviour of the equation
for $J=40, F=2$, for $J=40, F=4$ and report your results. Fix $F$ at $8$
and play with the value of the dimension of the system, $J$. 
Report your results. Again, illustrate your answers with graphics.

\item Consider the posterior smoothing distribution from Theorem \ref{th11}. 
Assume that the stochastic dynamics \index{stochastic dynamics} model \eqref{eq:dtf1}
is scalar and defined by $\PPsi(v)=av$ for
some $a \in \bbR$ and $\Sigma=\sigma^2$; and that the observation model 
\eqref{eq:dtf2} is defined by
$h(v)=v$ and $\Gamma=\gamma^2.$ Find explicit formulae for 
$\Jj(v)$ and $\PPhi(v;y)$, assuming that $v_0 \sim N(m_0,\sigma_0^2).$

\item Consider the posterior smoothing distribution from Theorem \ref{th112}. 
Assume that the dynamics model \eqref{eq:dtf11a}
is scalar and defined by $\PPsi(v)=av$ for
some $a \in \bbR$; and that the observation model \eqref{eq:dtf2} is defined by
$h(v)=v$ and $\Gamma=\gamma^2.$ Find explicit formulae for 
$\Jjd(v_0)$ and $\Phid(v_0;y)$, assuming that $v_0 \sim N(m_0,\sigma_0^2).$

\item Consider the definition of total variation\index{metric!total variation} 
distance given 
in Definition \ref{def:tvd}.
State and prove a theorem analogous to Theorem \ref{th14}, but employing the
total variation\index{metric!total variation} 
distance instead of the Hellinger distance.

\item Consider the filtering distribution \index{filtering!distribution} from section \ref{ssec:fp} 
in the case where the stochastic dynamics \index{stochastic dynamics} model \eqref{eq:dtf1}
is scalar and defined by $\PPsi(v)=av$ for
some $a \in \bbR$ and $\Sigma=\sigma^2$; and that the observation model
\eqref{eq:dtf2} is defined by
$h(v)=v$ and $\Gamma=\gamma^2$; and $v_0 \sim N(m_0,\sigma_0^2).$
Demonstrate that the prediction and analysis steps preserve Gaussianity
so that $\mu_j=N(m_j,\sigma_j^2).$
Find iterative formulae which update $(m_j,\sigma_j^2)$ to
give $(m_{j+1},\sigma_{j+1}^2)$.

\item Prove Theorem \ref{th14a}.

\end{enumerate}

\graphicspath{{./figs/chapter2/}}
\chapter{Discrete Time: Smoothing Algorithms}\label{sec:dtsa}\index{smoothing}

The formulation of the data assimilation\index{data assimilation}
problem described in the previous chapter is probabilistic, and its computational resolution requires the probing of 
a posterior probability distribution on signal given data. This probability
distribution is on the signal sequence $v=\{v_j\}_{j=0}^{J}$
when the underlying dynamics is stochastic
and given by \eqref{eq:dtf1}; the posterior is specified in Theorem
\ref{th11} and is proportional to $\exp\bigl(-\Ii(v;y)\bigr)$ 
given by \eqref{eq:dtf4}. On the other hand, if 
the underlying dynamics is deterministic
and given by \eqref{eq:dtf11}, then the probability distribution is
on the initial condition $v_0$ only; it is given in Theorem \ref{th112}
and is proportional to $\exp\bigl(-\Iid(v_0;y)\bigr)$, with $\Iid$ given 
by \eqref{eq:Ip}. Generically, in this chapter,
we refer to the unknown variable as
$u$, and then use $v$ in the specific case of stochastic dynamics,\index{stochastic dynamics} 
and $v_0$ in the specific case of deterministic dynamics \index{deterministic dynamics}. 
The aim of this chapter is to understand $\bbP(u|y).$ In this
regard we will do three things:
\begin{itemize}
\item find explicit expressions for the pdf $\bbP(u|y)$ in the linear,
Gaussian setting;
\item generate samples $\{u^{(n)}\}_{n=1}^N$ from $\bbP(u|y)$ by algorithms
applicable in the non-Gaussian setting;
\item find points where $\bbP(u|y)$ is maximized with respect to $u$,
for given data $y$.
\end{itemize}

In general the probability distributions of interest cannot be
described by a finite set of parameters, except in a few simple
situations such as the Gaussian scenario where the mean and covariance
determine the distribution in its entirety -- the {\bf Kalman Smoother}\index{Kalman smoother}.
When the probability distributions cannot be described by
a finite set of parameters, an expedient computational approach
to approximately representing the measure
is through the idea of {\bf Monte Carlo sampling}\index{Monte Carlo!sampling}.
The basic\index{Monte Carlo} idea is to approximate a measure $\nu$ by a set of $N$ samples
$\{u^{(n)}\}_{n \in \Z^+}$ drawn, or approximately drawn, 
from $\nu$ to obtain the measure $\nu^N \approx \nu$ given by:
\begin{equation}
\label{eq:mci}
\nu^N=\frac{1}{N}\sum_{n=1}^N \delta_{u^{(n)}}.
\end{equation}
We may view this as defining a (random) map $S^N$ on measures which takes
$\nu$ into $\nu^N$.
If the $u^{(n)}$ are exact draws from $\nu$ then the resulting approximation
$\nu^N$ converges to the true measure $\nu$ as $N \to \infty$.
\footnote{Indeed we prove such a result in Lemma \ref{l:s}
in the context of the particle filter\index{filter!particle}.}
For example if $v=\{v_j\}_{j=0}^{J}$ is governed by the 
probability distribution $\mu_0$ defined by the unconditioned
dynamics \eqref{eq:dtf1}, and with pdf determined by \eqref{eq:JJ}, 
then exact independent samples $v^{(n)}=\{v_j^{(n)}\}_{j=0}^{J}$
are easy to generate, simply by running the dynamics model forward in 
discrete time.  However for the complex probability measures of interest 
here, where the signal\index{signal} is conditioned on data\index{data},
exact samples are typically not possible and so instead
we use the idea of 
{\bf Monte Carlo Markov Chain (MCMC)}\index{MCMC}\index{Monte Carlo!Markov Chain} methods
which provide a methodology for generating approximate samples.
These methods do not require knowledge of the normalization
constant\index{normalization constant} for the measure $\bbP(u|y)$;
as we have discussed, Bayes' formula \eqref{eq:bayes} readily delivers
$\bbP(u|y)$ upto normalization, but the normalization itself
can be difficult to compute.
It is also of interest to simply maximize the posterior probability
distribution, to find a single point estimate of the solution, leading to
{\bf variational methods}\index{variational method}, which we also consider.

Section \ref{ssec:kalkal} gives explicit formulae for the solution
of the smoothing problem in the setting where the stochastic dynamics \index{stochastic dynamics}
model is linear, and subject to Gaussian noise, for which the observation
operator\index{observation operator} is linear, and for which
the distributions of the initial condition and the observational noise
are Gaussian; this is the Kalman smoother\index{Kalman smoother}. 
These explicit formulae help to build intuition about
the nature of the smoothing distribution.
In section \ref{ssec:mcmcm} we provide some background concerning 
MCMC \index{MCMC} methods, and in particular,
the {\bf Metropolis-Hastings \index{Metropolis-Hastings}} variant of 
MCMC\index{MCMC}, and show how they can be
used to explore the posterior distribution. 
It can be very difficult to sample the probability distributions
of interest with high accuracy, 
because of the two problems of high dimension and
sensitive dependence on initial conditions. Whilst we do not
claim to introduce the optimal algorithms to deal with these
issues, we do discuss such issues in relation to the samplers we introduce,
and we provide references to the active research ongoing in this area.
Furthermore, although sampling of the posterior 
distribution\index{posterior distribution} may be computationally 
infeasible in many situations, where possible, it provides an important 
benchmark solution, enabling other algorithms to be compared against 
a ``gold standard'' Bayesian \index{Bayesian} solution.

However, because sampling the posterior distribution can
be prohibitively expensive, a widely used computational 
methodology is simply to find the
point which maximizes the probability, using techniques
from optimization\index{optimization}. 
These are the variational methods\index{variational method},
also known as {\bf 4DVAR}\index{4DVAR} and {\bf weak constraint 4DVAR} \index{4DVAR!weak constraint}. 
We introduce this approach to the problem in section \ref{ssec:vm}. 
In section \ref{ssec:mci} we provide numerical illustrations
which showcase the MCMC \index{MCMC} and variational methods\index{variational method}.
The chapter concludes in sections \ref{ssec:bs} and \ref{ssec:ex3}
with bibliographic notes and exercises.

\section{Linear Gaussian Problems: The Kalman Smoother}\label{ssec:kalkal}

The Kalman smoother \index{Kalman smoother} plays an important role because it is one of the few examples for which the smoothing distribution can be explicitly determined. This explicit characterization occurs because the signal dynamics and observation operator are assumed to be linear. When combined with the Gaussian assumptions on the initial condition for the signal, and on the signal and observational noise, this gives rise to a posterior smoothing distribution which is also Gaussian. 

To find formulae for this Gaussian Kalman smoothing distribution we set 
\begin{equation}\label{eq:linearmodel}
\PPsi(v) = Mv, \quad h(v)=Hv
\end{equation}
for matrices $M\in \bbR^{n\times n},$ $H\in \bbR^{m\times n}$ and consider the signal/observation model \eqref{eq:dtf1}, \eqref{eq:dtf2}. Given 
data\index{data} $y=\{y_j\}_{j\in\J}$ and signal\index{signal}
$v=\{v_j\}_{j\in\J_0}$ we are interested in the probability distribution of $v|y,$ as characterized in subsection \ref{ssec:sd}. By specifying the linear model \eqref{eq:linearmodel}, and applying Theorem \ref{th11} we find the following:

\begin{theorem}
\label{t:ks1}
The posterior smoothing distribution on $v|y$ for the linear stochastic dynamics \index{stochastic dynamics} model \eqref{eq:dtf1}, \eqref{eq:dtf2}, \eqref{eq:linearmodel} with $C_0,\,\Sigma$ and $\Gamma$ symmetric positive-definite is a Gaussian probability measure $\mu=N(m,C)$ on $\bbR^{|\J_0|\times n}.$ The covariance $C$ is the inverse
of a symmetric positive-definite block tridiagonal 
precision\index{precision} matrix 
\begin{eqnarray*}
L=\left(
\begin{array}{ccccc}
L_{11}  & L_{12}  &              &        &   \\
L_{21}  & L_{22}  & L_{23}   &               &  \\
           &  \ddots &\ddots  &\ddots     &          \\
           &             &  \ddots          &  \ddots             &  L_{J J+1} \\
           &             &                 & L_{J+1 J} & L_{J+1 J+1}
\end{array}
\right)
\end{eqnarray*}  
with  $L_{ij} \in \bbR^{n\times n}$ given by  $L_{11}=C_0^{-1} + M^T \Sigma^{-1} M,$ $L_{jj} = H^T \Gamma^{-1} H + M^T \Sigma^{-1} M+\Sigma^{-1}$ for  $j=2,\ldots,J,$ $L_{J+1,J+1}=H^T \Gamma^{-1} H+\Sigma^{-1}$, $L_{j j+1}= -M^T \Sigma^{-1}$ and $ L_{j+1 j}=-\Sigma^{-1}M$ for $j=1,\ldots,J.$  
Furthermore the mean $m$ solves the equation 
\begin{equation*}
Lm  = r,
\end{equation*}
where $$r_1 = C_0^{-1}m_0, \quad \quad r_j = H^T \Gamma^{-1} y_{j-1},\quad j=2,\cdots,J+1.$$
This mean is also the unique minimizer of the functional
\begin{align}
\label{eq:Ilin}
\Ii (v;y)&= \frac12 \bigl|C_0^{-1/2} (v_0 -m_0)\bigr|^2 + \sum_{j=0}^{J-1} \frac12 \bigl|\Sigma^{-1/2}(v_{j+1} -M v_j)\bigr|^2 + \sum_{j=0}^{J-1} \frac12 \bigl|\Gamma^{-1/2}(y_{j+1}-Hv_{j+1})\bigr|^2\notag\\
&= \frac12 \bigl|v_0 -m_0\bigr|_{C_0}^2 + \sum_{j=0}^{J-1} \frac12 \bigl|v_{j+1} -M v_j|_{\Sigma}^2 + \sum_{j=0}^{J-1} \frac12 \bigl|y_{j+1}-Hv_{j+1}|_{\Gamma}^2
\end{align}
with respect to $v$ and, as such, is a maximizer, with respect to $v$, for the posterior pdf $\bbP(v|y)$.
\end{theorem}
\begin{proof}
The proof is based around Lemma \ref{l:qf}, and identification of
the mean and covariance by study of an appropriate quadratic form.
From Theorem \ref{th11} we know that the desired distribution has pdf proportional to $\exp\bigl(-\Ii(v;y)\bigr),$ where $\Ii(v;y)$ is given in \eqref{eq:Ilin}.
This is a quadratic form in $v$ and we deduce that the inverse covariance $L$ is given by $\partial_v^2 \Ii(v;y),$ the Hessian of $\Ii$ with respect to $v.$ To
determine $L$ we note the following identities:
\begin{align*}
D_{v_0}^2 I(v;y)&=C_0^{-1}+M^T\Sigma^{-1}M,\\
D_{v_j}^2 I(v;y)&=\Sigma^{-1}+M^T\Sigma^{-1}M+H^T\Gamma^{-1}H,\quad j=1,\dots, J-1\\
D_{v_J}^2 I(v;y)&=\Sigma^{-1}+H^T\Gamma^{-1}H,\\
D^2_{v_j,v_{j+1}}I(v;y)&=-M^T\Sigma^{-1},\\
D^2_{v_{j+1},v_{j}}I(v;y)&=-\Sigma^{-1}M.
\end{align*}

We may then complete the square and write 
$$\Ii (v;y)=\frac12 \langle (v-m),L(v-m)\rangle + q,$$
where $q$ is independent of $v.$ From this it follows that the mean
does indeed minimize $\Ii(v;y)$ with respect to $v$, and hence
maximizes $\bbP(v|y) \propto \exp\bigl(-\Ii(v;y)\bigr)$ with respect
to $v$. By differentiating with respect to $v$ we obtain 
\begin{equation}
Lm =r,\quad r=-\nabla_v \Ii(v;y)\Bigl|_{v=0},
\end{equation}
where $\nabla_v$ is the gradient of $\Ii$ with respect to $v.$
This characterization of $r$ gives the desired equation for the
mean. Finally we show that $L$, and hence $C$, is positive
definite symmetric. Clearly $L$ is symmetric and hence so is $C.$ It remains to check that $L$ is strictly positive definite. To see this note that if we set $m_0=0$ then
\begin{equation}
\frac12 \langle v,L v \rangle = \Ii (v;0)\ge 0.
\end{equation}
Moreover, $\Ii(v;0)=0$ with $m_0=0$ implies, since $C_0>0$ and $\Sigma>0$, 
\begin{subequations}
\begin{align*}
v_0 &= 0, \\
v_{j+1}&=Mv_j,  \quad j=0,\ldots,J-1, 
\end{align*}
\end{subequations}
i.e. $v=0.$ Hence we have shown that $\langle v,L v \rangle =0$ implies $v=0$ and the proof is complete.
\end{proof}

We now consider the Kalman smoother \index{Kalman smoother} in the case of deterministic dynamics. Application of Theorem \ref{th112} gives the following:

\begin{theorem}
\label{t:ks2}
The posterior smoothing distribution on $v_0|y$ for the deterministic
linear dynamics model \eqref{eq:dtf11}, \eqref{eq:dtf2}, \eqref{eq:linearmodel} with $C_0$ and $\Gamma$ symmetric positive definite is a Gaussian probability measure $\nu=N(\mdet,\Cdet)$ on $\bbR^n.$ The covariance $\Cdet$ is the inverse
of the positive-definite symmetric matrix $\Ldet$ given by the expression 
\begin{equation*}
\Ldet= C_0^{-1} + \sum_{j=0}^{J-1} (M^T)^{j+1} H^T \Gamma^{-1} H M^{j+1}.
\end{equation*}
The mean $\mdet$ solves
\begin{equation*}
\Ldet\mdet= C_0^{-1} m_0+ \sum_{j=0}^{J-1} (M^T)^{j+1} H^T \Gamma^{-1} y_{j+1}.
\end{equation*}
This mean is a minimizer of the functional
\begin{equation}
\label{eq:Iid2}
\Iid (v_0;y)= \frac12 \bigl|v_0 -m_0|_{C_0}^2 +  \sum_{j=0}^{J-1} \frac12 \bigl|y_{j+1}-HM^{j+1}v_0\bigr|_{\Gamma}^2
\end{equation}
with respect to $v_0$ and, as such, is a maximizer, with respect to
$v_0$, of the posterior pdf $\bbP(v_0|y).$
\end{theorem}
\begin{proof}
By Theorem \ref{th112} we know that the desired distribution has pdf proportional to $\exp\bigl(-\Iid(v_0;y)\bigr)$ given by \eqref{eq:Iid2}. 
The inverse covariance $\Ldet$ can be found as the Hessian of $\Iid,$ $\Ldet= \partial_v^2 \Iid(v_0;y),$ and the mean $\mdet$ solves
\begin{equation}
\Ldet \mdet = -\nabla_v \Iid (v_0;y)\Bigl|_{v_0=0}.
\end{equation}
As in the proof of the preceding theorem, we have that
$$\Iid(v_0;y)=\frac12 \langle \Ldet(v_0-\mdet),(v_0-\mdet)\rangle+q$$
where $q$ is independent of $v_0;$ this shows that $\mdet$ minimizes
$\Iid(\cdot\,;y)$ and maximizes $\bbP(\cdot|y).$

We have thus characterized $\Ldet$ and $\mdet$ and using this
characterization gives the desired expressions. It remains to check that $\Ldet$ is positive definite, since it is clearly symmetric by definition.
Positive-definiteness follows from the assumed 
positive-definiteness of $C_0$ and $\Gamma$ since, 
for any nonzero $v_0\in \bbR^n,$ 
\begin{equation}
\langle v_0, \Ldet v_0 \rangle \ge \langle v_0 C_0^{-1} v_0 \rangle > 0.
\end{equation}
\end{proof}

\section{Markov Chain-Monte Carlo Methods}\index{MCMC}\index{Monte Carlo!Markov Chain}\label{ssec:mcmcm}
In the case of stochastic dynamics,\index{stochastic dynamics} equation \eqref{eq:dtf1},
the posterior distribution\index{posterior distribution} of interest is the measure $\mu$ 
on $\bbR^{|\J_0|\times n}$, with density $\bbP(v|y)$ 
given in Theorem \ref{th11};
in the case of deterministic dynamics, equation \eqref{eq:dtf11}, it is
the measure $\nu$ on $\bbR^n$ with density $\bbP(v_0|y)$ given
in Theorem \ref{th112}. In this section we describe
the idea of Markov Chain-Monte Carlo (MCMC \index{MCMC}) methods for exploring
such probability distributions.

We will start by describing the general
MCMC methodology, after which we discuss the
specific Metropolis-Hastings\index{Metropolis-Hastings} 
instance of this methodology. This material makes no reference to
the specific structure of our sampling problem; it works in
the general setting of creating a Markov chain which is invariant for a
an arbitrary measure $\mu$ on $\bbR^\ell$ with pdf $\rho$. 
We then describe applications of the Metropolis-Hastings\index{Metropolis-Hastings} family of
MCMC methods to the smoothing problems
of noise-free dynamics and noisy dynamics respectively. 
When describing the generic Metropolis-Hastings\index{Metropolis-Hastings} methodology we will 
use $u$ (with indices) to denote the state of the Markov chain \index{Markov chain} and 
$w$ (with indices) the proposed moves. 
Thus the current state $u$ and proposed state $w$ live in the space
where signal sequences $v$ lie, in the case of stochastic 
dynamics, and in the space where initial conditions $v_0$
lie, in the case of deterministic dynamics.

\subsection{The MCMC \index{MCMC} Methodology}
\label{ssec:mcmcidea}

Recall the concept of a Markov 
chain\index{Markov chain}$\{u^{(n)}\}_{n \in \Z^+}$ 
introduced in subsection \ref{ssec:mk}. The idea of MCMC methods
is to construct a Markov chain \index{Markov chain} which is invariant with respect to
a given measure $\mu$ on $\bbR^{\ell}$ and, of particular interest to
us, a measure $\mu$ with positive Lebesgue density $\rho$ on $\bbR^{\ell}$. 
We now use a superscript $n$ to denote the index of the resulting 
Markov chain,\index{Markov chain}
instead of subscript $j$, to provide a clear distinction between the
Markov chain \index{Markov chain}defined by the stochastic (respectively
deterministic) dynamics model \eqref{eq:dtf1} (respectively \eqref{eq:dtf11})
and the Markov chains \index{Markov chain} that we will use to sample the posterior
distribution on the signal\index{signal} $v$ 
(respectively initial condition $v_0$)
given data\index{data} $y$. 

We have already seen that  Markov chains \index{Markov chain} allow the computation
of averages with respect to the invariant measure \index{invariant measure} by computing the
running time-average of the chain -- see \eqref{eq:erg}. More
precisely we have the following theorem (for which it is useful
to recall the notation for the iterated kernel $p^n$ from the
very end of subsection \ref{ssec:mk}):

\begin{theorem}\label{th21} Assume that, if $u^{(0)}\sim\mu$ with Lebesgue density $\rho$, then $u^{(n)}\sim\mu$ for all $n\in\Z^+$ so that $\mu$ is
invariant for the Markov chain\index{Markov chain}. If, in addition,  
the Markov chain\index{Markov chain} is ergodic\index{ergodic},
then for any bounded continuous $\varphi: \bbR^\ell \to \bbR$,
\[\frac1N\sum_{n=1}^{N}\varphi(u^{(n)})\stackrel{\text{a.s.}}{\longrightarrow}\bbE^{\mu}\varphi(u)\]
for $\mu$ a.e. initial condition $u^{(0)}$.
In particular, if there is probability measure ${\mathsf p}$ on $\bbR^\ell$
and $\varepsilon>0$ such that, for all $u \in \bbR^{\ell}$ and all Borel sets
$A \subseteq \calB(\bbR^\ell)$,
$p(u,A) \ge \varepsilon {\mathsf p}(A)$ 
then, for all $u \in \bbR^{\ell}$,
\begin{equation}
\label{eq:tv}
d_{{\tiny{\rm TV}}}\bigl(p^n(u,\cdot),\mu\bigr) \le 2(1-\varepsilon)^n.
\end{equation}
Furthermore, there is then  $K=K(\varphi)>0$ such that
\begin{equation}
\label{eq:clt}
\frac1N\sum_{n=1}^{N}\varphi(u^{(n)})=\bbE^{\mu}\varphi(u)+K\xi_N N^{-\frac12}
\end{equation}
where $\xi_N$ converges weakly\index{weak convergence} to $N(0,1)$ as $N \to \infty.$
\end{theorem}

\begin{remark}
\label{rem:mcmc}
This theorem is the backbone behind MCMC. As we will see, there is
a large class of methods which ensure invariance of a given measure
$\mu$ and, furthermore, these methods are often provably ergodic\index{ergodic} 
so that the preceding theorem applies. As with all algorithms in computational
science, the optimal algorithm is the one the delivers smallest error
for given unit computational cost. In this regard there are two
observations to make about the preceding theorem.

\begin{itemize}

\item The constant $K$ measures the size of the variance of the estimator
of $\bbE^{\mu}\varphi(x),$ multiplied by $N$. It is thus a surrogate
for the error incurred in running MCMC over a finite number of steps.
The constant $K$ depends on 
$\varphi$ itself, but will also reflect general properties of the
Markov chain. For a given MCMC method there will often be tunable
parameters whose choice will affect the size of $K$, without affecting the
cost per step of the Markov chain. The objective of choosing these
parameters is to minimize the constant $K$, within a given class of
methods all of which have the same cost per step.
In thinking about how to do this it is important to appreciate that
$K$ measures the amount of correlation in the Markov chain; lower
correlation leads to decreased constant $K$. More precisely, $K$ is computed
by integrating the autocorrelation\index{autocorrelation} of the Markov chain.

\item A further tension in designing MCMC methods is in the choice of
the class of methods themselves. Some 
Markov chains\index{Markov chain} 
are expensive to implement, but the convergence
in \eqref{eq:clt} is rapid (the constant $K$ can be made small by appropriate
choice of parameters), whilst other
Markov chains\index{Markov chain} are cheaper
to implement, but the convergence in \eqref{eq:clt} 
is slower (the constant $K$ is much larger). 
Some compromise between ease of implementation and rate of convergence
needs to be made.

\end{itemize}

\end{remark}

\subsection{Metropolis-Hastings\index{Metropolis-Hastings} Methods}
\label{ssec:mh}

The idea of Metropolis-Hastings\index{Metropolis-Hastings} methods is to build an MCMC
method for measure $\mu$, by adding an accept/reject
test on top of a Markov chain\index{Markov chain} which is easy to implement, but
which is not invariant with respect to $\mu$; the accept/reject step
is designed to enforce invariance with respect to $\mu.$
This is done by enforcing {\em detailed balance}\index{detailed balance}:
\begin{equation}
\label{eq:db1}
\rho(u)p(u,w)=\rho(w)p(w,u)\quad\forall u,w \in \bbR^{\ell}\times\bbR^{\ell}.
\end{equation}
Note that integrating with respect to $u$ and using the fact that
$$\int_{\bbR^{\ell}} p(w,u)du=1$$
we obtain
$$\int_{\bbR^{\ell}} \rho(u)p(u,w)du=\rho(w)$$
so that \eqref{eq:comeon2} is satisfied and density $\rho$ is indeed
invariant. We now exhibit an algorithm designed to satisfy detailed
balance, by correcting a given Markov chain,\index{Markov chain} which is not invariant with respect to
$\mu$, by the addition of an accept/reject mechanism.

We are given a probability density function \index{probability density function}  $\rho$ hence satisfying
$\rho:\R^\ell\to\R^+$, with $\int \rho(u)du=1$. Now consider a Markov transition kernel $q:\R^\ell\times \R^\ell\to\R^+$ with the property that $\int q(u,w)dw=1$ for every $u\in\R^\ell$. Recall 
the notation, introduced in subsection \ref{ssec:mk},
that we use function $q(u,w)$ to denote a pdf and, simultaneously,
a probability measure $q(u,dw).$ We create a Markov chain\index{Markov chain} $\{u^{(n)}\}_{n\in\N}$ which is invariant for $\rho$ as follows. 
Define\footnote{Recall that we use the 
$\wedge$ operator to denote the minimum between
the two real numbers.}
\begin{equation}
\label{eq:dtma1}a(u,w)=1\wedge \frac{\rho(w)q(w,u)}{\rho(u)q(u,w)}.
\end{equation}
The algorithm is:
\begin{enumerate}
\item[1.] Set $n=0$ and choose $u^{(0)}\in\R^{\ell}$.
\item[2.] $n\to n+1$.
\item[3.] Draw {$w^{(n)}\sim q(u^{(n-1)},\cdot)$.}
\item[4.] Set $u^{(n)}=w^{(n)}$ with probability $a(u^{(n-1)},w^{(n)})$, $u^{(n)}=u^{(n-1)}$ otherwise.
\item[5.] Go to step 2.
\end{enumerate}

At each step in the algorithm there are two sources
or randomness: that required for drawing $w^{(n)}$ in step 3; and
that required for accepting or rejecting  $w^{(n)}$ as the next $u^{(n)}$
in step 4. These two sources of randomness are chosen to 
be independent of one another.
Furthermore, all the randomness at discrete algorithmic time $n$ is 
independent of randomness at preceding discrete algorithmic times,
conditional on $u^{(n-1)}$.  Thus the whole procedure gives a
Markov chain.\index{Markov chain} If $\us=\{\us^{(j)}\}_{j \in \N}$ 
is an i.i.d. \index{i.i.d.} sequence of $U[0,1]$ random variables
then we may write the algorithm as follows:
\begin{align*}
w^{(n)}&\sim q(u^{(n-1)},\cdot)\\
u^{(n)}&=w^{(n)}\bbI\bigl(\us^{(n)} \le a(u^{(n-1)},w^{(n)})\bigr)+
u^{(n-1)}\bbI\bigl(\us^{(n)} > a(u^{(n-1)},w^{(n)})\bigr).
\end{align*}
Here $\bbI$ denotes the indicator function of a set.
We let $p:\R^{\ell} \times \R^{\ell} \to \bbR^+$ denote the
transition kernel of the resulting Markov chain,\index{Markov chain}, 
and we let $p^n$ denote the transition kernel over $n$ steps; recall that 
hence $p^n(u,A)=\rp(u^{(n)} \in A|u^{(0)}=u).$
Similarly as above, for fixed $u$, $p^n(u,dw)$ denotes a probability measure on $\bbR^\ell$
with density $p^n(u,w).$
The resulting algorithm is known as a Metropolis-Hastings\index{Metropolis-Hastings}
MCMC \index{MCMC} algorithm, and satisfies detailed balance\index{detailed balance}
with respect to $\mu.$

\begin{remark}\label{r22} 

The following two observations are central to
Metropolis-Hastings \index{Metropolis-Hastings} MCMC \index{MCMC} methods.

\begin{itemize}

\item The construction of Metropolis-Hastings \index{Metropolis-Hastings} MCMC \index{MCMC} methods is 
designed to ensure the {\em detailed balance}\index{detailed balance} condition \eqref{eq:db1}.
We will use the condition expressed in this form in
what follows later. It is also sometimes written in integrated form
as the statement
\begin{equation}
\label{eq:db}
\int_{\R^{\ell}\times \R^{\ell}} f(u,w) \rho(u)p(u,w) dudw = \int_{\R^{\ell}\times \R^{\ell}}  f(u,w) \rho(w)p(w,u) dudw
\end{equation}
for all $f:\R^{\ell}\times \R^{\ell} \rightarrow \R.$
Once this condition is obtained it follows trivially that the measure
$\mu$ with density $\rho$ is invariant since, for $f=f(w)$, 
we obtain
\begin{align*}
\int_{\R^{\ell}} f(w) \left ( \int_{\R^{\ell}} \rho(u)p(u,w)du \right) dw & = \int_{\R^{\ell}} f(w) \rho(w) dw \int_{\R^{\ell}} p(w,u)du\\
&=\int_{\R^{\ell}} f(w) \rho(w) dw.
\end{align*}
Note that $\int_{\R^{\ell}} \rho(u)p(u,w)du$ is the density of the
distribution of the Markov chain\index{Markov chain} after one step, given that it is
initially distributed according to density $\rho.$ Thus the preceding
identity shows that the expectation of $f$ is unchanged by the
Markov chain,\index{Markov chain} if it is initially distributed with density $\rho.$
This means that if the Markov chain\index{Markov chain} is distributed according to measure
with density $\rho$ initially then it will be distributed according
to the same measure for all algorithmic time.

\item Note that, in order to implement Metropolis-Hastings \index{Metropolis-Hastings} MCMC \index{MCMC} methods,
it is not necessary to know the normalisation constant for $\rho(\cdot)$ 
since only its ratio appears in the definition of the 
acceptance probability $a$.  
\end{itemize}

\end{remark}

The Metropolis-Hastings \index{Metropolis-Hastings} algorithm defined above satisfies the following,
which requires definition of TV distance given in section \ref{ssec:pm}:

\begin{corollary}\label{c21} For the Metropolis-Hastings MCMC methods
we have that the detailed balance \index{detailed balance} condition \eqref{eq:db1} is satisfied
and that hence $\mu$ is invariant: if $u^{(0)}\sim\mu$ with Lebesgue density $\rho$, then $u^{(n)}\sim\mu$ for all $n\in\Z^+$. 
Thus, if the Markov chain \index{Markov chain} is ergodic\index{ergodic}, 
then the conclusions of Theorem \ref{th21} hold.
\end{corollary}

We now describe some exemplars of Metropolis-Hastings \index{Metropolis-Hastings} methods
adapted to the data assimilation problem. These are not to be
taken as optimal MCMC \index{MCMC} methods for data assimilation, but rather
as examples of how to construct proposal\index{proposal}
distributions $q(u,\cdot)$ for
Metropolis-Hastings \index{Metropolis-Hastings} methods in the context of data assimilation.
In any given application the proposal\index{proposal} distribution plays a central
role in the efficiency of the MCMC \index{MCMC} method and tailoring it to the
specifics of the problem can have significant impact on efficiency
of the MCMC \index{MCMC} method. 
Because of the level of generality at which we are presenting the
material herein (arbitrary $f$ and $h$), 
we cannot discuss such tailoring in any detail.

\subsection{Deterministic Dynamics}
\label{ssec:dd2}

In the case of deterministic dynamics \eqref{eq:dtf11}, the measure of interest is a measure on
the initial condition $v_0$ in $\mathbb{R}^{n}$. 
Perhaps the simplest Metropolis-Hastings \index{Metropolis-Hastings} algorithm is the {\bf Random Walk
Metropolis}\index{Random Walk
Metropolis} (RWM)\index{RWM}\index{Random Walk
Metropolis!RWM} sampler which employs a Gaussian
proposal,\index{proposal!Gaussian}\index{proposal!RWM}
centred at the current state; we now illustrate this
for the case of deterministic dynamics.
Recall that the measure of interest is $\nu$ with pdf $\varrho.$
Furthermore $\varrho \propto \exp\bigl(-\Iid(v_0;y)\bigr)$
as given in Theorem \ref{th112}.
 
The RWM method proceeds as follows:
given that we are at $u^{(n-1)} \in \bbR^n$,
a current approximate sample from the posterior distribution\index{posterior distribution}
on the initial condition, we propose
\begin{equation}
\label{eq:prop1}
w^{(n)}=u^{(n-1)}+\beta\iota^{(n-1)}
\end{equation}
where $\iota^{(n-1)} \sim N(0,C_{{\rm prop}})$
for some symmetric positive-definite proposal covariance\index{proposal!covariance}
$C_{{\rm prop}}$ and proposal variance\index{proposal!variance} 
scale parameter $\beta>0$; 
natural choices for this proposal
covariance\index{proposal!covariance}
include the identity $I$ or the prior covariance
$C_0.$ Because of the
symmetry of such a random walk proposal it follows that
$q(w,u)=q(u,w)$ and hence that
\begin{align*}
a(u,w)&=1\wedge \frac{\varrho(w)}{\varrho(u)}\\
&=1\wedge \exp\bigl(\Iid(u;y)-\Iid(w;y)\bigr).
\end{align*}

\begin{remark} \label{r:connect}
The expression for the acceptance probability shows that
the proposed move to $w$ is accepted with probability
one if the value of $\Iid(\cdot\,;y)$, the
log-posterior\index{posterior!log-posterior},
is decreased by moving to $w$ from the current
state $u$. On the other hand, if $\Iid(\cdot\,;y)$ increases then the proposed
state is accepted only with some probability less than one.
Recall that $\Iid(\cdot\,;y)$ is the sum of the prior penalization (background)
and the model-data misfit functional. 
The algorithm thus has a very natural interpretation in terms of the
data assimilation problem: it biases samples towards decreasing $\Iid(\cdot\,;y)$
and hence to improving the fit to both the model and the data in
combination. 

The algorithm has two key tuning parameters:
the proposal covariance\index{proposal!covariance}
$C_{{\rm prop}}$ and the scale parameter $\beta.$
See Remark \ref{rem:mcmc}, first bullet, for discussion of the role of
such parameters.
The covariance can encode any knowledge, or guesses, about the relative
strength of correlations in the model; given this, the parameter $\beta$ should
be tuned to give an acceptance probability that is neither close to $0$ nor to
$1$. This is because if the acceptance probability is small then successive 
states of the Markov chain are highly correlated, leading to a large constant
$K$ in \eqref{eq:clt}. On the other hand if the acceptance probability
is close to one then this is typically because $\beta$ is small, also leading 
to highly correlated steps and hence to a large constant
$K$ in \eqref{eq:clt}. 
\end{remark}

Numerical results illustrating the method are given in section \ref{ssec:mci}.

\subsection{Stochastic Dynamics}
\label{ssec:sd2}

We now apply the Metropolis-Hastings \index{Metropolis-Hastings} methodology to the data assimilation smoothing problem in the case of the stochastic dynamics \index{stochastic dynamics} model \eqref{eq:dtf1}. 
Thus the probability measure is on an entire signal\index{signal}
sequence $\{v_j\}_{j=0}^J$ and not just on $v_0$; hence it  
lives on $\bbR^{|\J_0|\times n}$. 
{It is possible to apply the random walk method to this
situation, too, but we take the opportunity to introduce several 
different Metropolis-Hastings \index{Metropolis-Hastings} methods, in order to highlight the
flexibility of the methodology. Furthermore, it is also possible
to take the ideas behind the proposals introduced in this
section and apply them in the case of deterministic dynamics.

In what follows recall the measures $\mu_0$ and $\mu$ defined in
section \ref{ssec:sp}, with densities $\rho_0$ and $\rho$, representing
(respectively) the measure on sequences $v$ generated by \eqref{eq:dtf1}
and the resulting measure when the signal\index{signal}
is conditioned on the data\index{data} $y$ from \eqref{eq:dtf2}.}
We now construct,  via the Metropolis-Hastings\index{Metropolis-Hastings} 
methodology, two Markov chains \index{Markov chain} $\{u^{(n)}\}_{n\in\N}$ which are invariant with respect to $\mu$. Hence 
we need only specify the transition  kernel $q(u,w)$, and identify the resulting acceptance probability $a(u,w)$. The sequence $\{w^{(n)}\}_{n\in\Z^+}$
will denote the proposals.\index{proposal}

{\bf Independence Dynamics Sampler}\index{Independence Dynamics Sampler} 
Here we choose the proposal\index{proposal!independence dynamics sampler}
$w^{(n)}$, independently of the current state $u^{(n-1)}$, from the prior $\mu_0$ with
density $\rho_0.$ Thus we are simply proposing independent draws from
the dynamical model \eqref{eq:dtf1}, with no information from the data used
in the proposal. 
Important in what follows is the observation that 
\begin{equation}\label{eq:dtsa2}
\frac{\rho(v)}{\rho_0(v)}\propto\exp(-\PPhi(v;y)).
\end{equation}
With the given definition of proposal we have that
$q(u,w)=\rho_0(w)$ and hence that 
\begin{align*}a(u,w)&=1\wedge\frac{\rho(w)q(w,u)}{\rho(u)q(u,w)}\\
&=1\wedge\frac{\rho(w)/\rho_0(w)}{\rho(u)/\rho_0(u)}\\
&=1\wedge\exp(\PPhi(u;y)-\PPhi(w;y)).
\end{align*}

\begin{remark} \label{r:connect2}
The expression for the acceptance probability shows that
the proposed move to $w$ is accepted with probability
one if the value of $\PPhi(\cdot\,;y)$ is decreased by moving to $w$ from the current
state $u$. On the other hand, if $\PPhi(\cdot\,;y)$ increases then the proposed
state is accepted only with some probability less than one, with the
probability decreasing exponentially fast with respect to the size of 
the increase. Recall that $\PPhi(\cdot\,;y)$ measures the fit of the signal 
to the data. Because the proposal\index{proposal} builds in the underlying signal model the
acceptance probability does not depend on $\Ii(\cdot\,;y)$, the
negative log-posterior,\index{posterior!log-posterior} 
but only the part reflecting the data, namely the negative 
log-likelihood.\index{likelihood!log-likelihood}
In contrast the RWM\index{Random Walk Metropolis}\index{proposal!RWM} 
method, explained in the
context of deterministic dynamics, does not build the model into
its proposal and hence the accept-reject mechanism depends on the
entire log-posterior\index{posterior!log-posterior}; 
see Remark \ref{r:connect}. 
\end{remark}

The Independence Dynamics Sampler\index{Independence Dynamics Sampler} 
does not have any tuning parameters
and hence can be very inefficient as there are no parameters to
modify in order to obtain a reasonable acceptance probability;
as we will see in the illustrations section \ref{ssec:mci} below, the
method can hence be quite inefficient because of the resulting
frequent rejections.  We now discuss this point and an approach to resolve it. 
The rejections are caused by attempts to move far from the
current state,  and in particular to proposed states which
are based on the underlying stochastic dynamics,\index{stochastic dynamics} 
but not on the observed data. 
This typically leads to increases in the model-data misfit 
functional $\PPhi(.\,;y)$ once the Markov chain has found a state
which fits the data reasonably well. Even if data is not explicitly used in
constructing the proposal\index{proposal}, this effect can be ameliorated
by making local proposals, which do not move far from the current state. 
These are exemplified in the following MCMC \index{MCMC} algorithm.

{\bf The pCN \index{pCN} Method.}
It is helpful in what follows to recall the measure $\vartheta_0$ with 
density $\pi_0$ found from $\mu_0$ and $\rho_0$ 
in the case where $\PPsi\equiv0$ and given by equation \eqref{eq:nuz}.
We denote the mean by $m$ and covariance 
by $C$, noting that $m=(m_0^T,0^T,\cdots,0^T)^T$ and that
$C$ is block diagonal with first block $C_0$ and the remainder
all being $\Sigma.$  Thus $\vartheta_0=N(m,C).$
The basic idea of this method is to make proposals\index{proposal!pCN} 
with the
property that, if $\PPsi \equiv 0$ so that the dynamics is Gaussian
and with no time correlation, and if $h \equiv 0$ so that the data
is totally uninformative, then the proposal\index{proposal} would be accepted
with probability one. Making small incremental proposals of this type
then leads to a Markov chain which
incorporates the effects of $\PPsi \ne 0$ and $h \ne 0$ through
the accept-reject mechanism. We describe the details of how this works.

Recall the prior on the stochastic dynamics \index{stochastic dynamics} model with density
$\rho_0(v)\propto \exp\bigl(-\Jj(v)\bigr)$ given by \eqref{eq:JJ}.
It will be useful to rewrite $\pi_0$ as follows:
\[\pi_0(v)\propto\exp(-\Jj(v)+F(v)),\]where
\begin{equation} \label{eq:Gu}
F(v)=\sum_{j=0}^{J-1}\left(\frac12\left|\Sigma^{-\frac12}\PPsi(v_{j})\right|^2-\left\langle 
\Sigma^{-\frac12}v_{j+1}, \Sigma^{-\frac12}\PPsi(v_{j})\right\rangle\right).
\end{equation}
We note that
\begin{equation*}
\frac{\rho_0(v)}{\pi_0(v)}\propto\exp(-F(v))
\end{equation*}
and hence that, using \eqref{eq:dtsa33},
\begin{equation}\label{eq:dtsa3}
\frac{\rho(v)}{\pi_0(v)}\propto\exp(-\PPhi(v;y)-F(v)).
\end{equation}

Recall the Gaussian measure $\vartheta_0=N(m,C)$
defined via its pdf in \eqref{eq:nuz}.
The pCN \index{pCN} method is a variant of random walk type methods, based on the following Gaussian proposal\index{proposal!Gaussian}\index{proposal!pCN} 
\begin{equation}
w^{(n)}=m+\left(1-\beta^2\right)^\frac12\left(u^{(n-1)}-m\right)+\beta\iota^{(n-1)},
\label{eq:pcnprop}
\end{equation}
\[\beta \in (0,1], \quad \iota^{(n-1)}\sim N(0,C).\]
Here $\iota^{(n-1)}$ is assumed to be independent of $u^{(n-1)}$. 


\begin{lemma} Consider the Markov chain \index{Markov chain}
\begin{equation}
u^{(n)}=m+\left(1-\beta^2\right)^\frac12\left(u^{(n-1)}-m\right)+\beta\iota^{(n-1)},
\label{eq:pcnprop2}
\end{equation}
\[\beta \in (0,1], \quad \iota^{(n-1)}\sim N(0,C)\]
with $\iota^{(n-1)}$ independent of $u^{(n-1)}.$
The Markov kernel \index{Markov kernel} for this chain $q(u,w)$ satisfies detailed
balance \eqref{eq:db1}
with respect to the measure $\vartheta_0$ with density $\pi_0$:
\begin{equation}
\label{eq:deen}
\frac{\pi_0(w)q(w,u)}{\pi_0(u)q(u,w)}=1.
\end{equation}
\end{lemma}
\begin{proof}
We show that $\pi_0(u)q(u,w)$ is symmetric in $(u,w).$ To demonstrate
this it suffices to consider the quadratic form found by taking the
negative of the logarithm of this expression. This is given by
$$\frac12|u-m|_{C}^2+\frac{1}{2\beta^2}|w-m-(1-\beta^2)^{\frac12}(u-m)|^2_{C}.$$
This is the same as
$$\frac{1}{2\beta^2}|u-m|_{C}^2+\frac{1}{2\beta^2}|w-m|_{C}^2-
\frac{(1-\beta^2)^{\frac12}}{\beta^2}\langle w-m,u-m\rangle_{C}$$
which is clearly symmetric in $(u,w).$ The result follows.
\end{proof}

By use of \eqref{eq:deen} and (\ref{eq:dtsa3})
we deduce that the acceptance probability for the MCMC method 
with proposal\index{proposal} \eqref{eq:pcnprop} is 
\begin{align*}
a(u,w)&=1\wedge\frac{\rho(w)q(w,u)}{\rho(u)q(u,w)}\\
&=1\wedge\frac{\rho(w)/\pi_0(w)}{\rho(u)/\pi_0(u)}\\
&=1\wedge\exp(\PPhi(u;y)-\PPhi(w;y)+F(u)-F(w)).
\end{align*}
{Recall that the proposal preserves the underlying Gaussian structure
of the stochastic dynamics \index{stochastic dynamics} model; the accept-reject mechanism then
introduces non-Gaussianity into the stochastic dynamics \index{stochastic dynamics} model,
via $F$, and 
introduces the effect of the data, via $\PPhi.$
By choosing $\beta$ small, so that $w^{(n)}$ is close to $u^{(n-1)}$, we can make $a(v^{(n-1)}, w^{(n)})$ reasonably large and obtain a usable algorithm.
This is illustrated in section \ref{ssec:mci}.

Recall from subsection \ref{ssec:reform} that, if $\PPsi \equiv 0$ 
(as assumed to define the measure $\vartheta_0$), then the noise sequence $\{\xi_{j-1}\}_{j=1}^{\infty}$
is identical with the signal sequence $\{v_j\}_{j=1}^{\infty}$.
More generally, even if $\PPsi \ne 0$, the noise sequence $\{\xi_j\}_{j=1}^{\infty}$, together with $v_0$, a vector which we denote in 
subsection \ref{ssec:reform} by $\xi$, uniquely determines the  
signal sequence $\{v_j\}_{j=0}^{\infty}$: see Lemma \ref{l:sicab}. 
This motivates a different formulation of the smoothing
problem for stochastic dynamics \index{stochastic dynamics} where one views the noise sequence
and initial condition as the unknown, rather than the signal sequence
itself. Here we study the implication of this perspective for 
MCMC \index{MCMC} methodology, in the context of the pCN \index{pCN} Method,
leading to our third sampler within this subsection: the
pCN \index{pCN} Dynamics Sampler. We now describe this algorithm.

{\bf The pCN \index{pCN} Dynamics Sampler} is so-named because the
proposal\index{proposal!pCN dynamics sampler} 
(implicitly, via the mapping $G$ defined in Lemma \ref{l:sicab})
samples from the dynamics as in the Independence Sampler\index{Independence Dynamics Sampler}, 
while the proposal also includes a parameter $\beta$ allowing small
steps to be taken and chosen to ensure good acceptance probability,
as in the pCN \index{pCN} Method.
The posterior measure we wish to sample is given in Theorem \ref{t:101}. 
Note that this theorem implicitly contains the fact that 
$$\vartheta(d\xi) \propto \exp\bigl(-\Phr(\xi;y)\bigr)\vartheta_0(d\xi).$$
Furthermore $\vartheta_0=N(m,C)$ 
where the mean $m$ and covariance $C$ are as
described above for the standard pCN\index{pCN} method.
We  use the pCN \index{pCN} proposal\index{proposal!pCN} \eqref{eq:pcnprop}:
\begin{equation*}
\zeta^{(n)}=m+\left(1-\beta^2\right)^\frac12\left(\xi^{(n-1)}-m\right)+\beta\iota^{(n-1)},
\end{equation*}
and the acceptance probability is given by
$$
a(\xi,\zeta)=1\wedge\exp\left(\Phr(\xi;y)-\Phr(\zeta;y)\right). 
$$
When interpreting this formula it is instructive to note that
$$\Phr(\xi;y)=\frac12\left|y-{\cal G}(\xi)\right|_{\Gamma_J}^2=\frac12\left|\Gamma_J^{-\frac12}\left(y-{\cal G}(\xi)\right)\right|^2 =
\PPhi\bigl({\cal G}(\xi);y\bigr),$$
and that $\xi$ comprises both $v_0$ and the noise sequence
$\{\xi\}_{j=0}^{J-1}.$
Thus the method has the same acceptance probability as the Independence
Dynamics Sampler\index{Independence Dynamics Sampler}, 
albeit expressed in terms of initial condition
and model noise rather than signal, and also possesses a tunable parameter
$\beta$; it thus has the nice conceptual interpretation of
the acceptance probability that is present in the Independence
Dynamics Sampler\index{Independence Dynamics Sampler}, 
as well as the advantage of the pCN\index{pCN} method 
that the proposal\index{proposal!variance} variance $\beta$ 
may be chosen to ensure a reasonable acceptance probability.

\section{Variational Methods}\label{ssec:vm}
Sampling the posterior using MCMC \index{MCMC} methods can be prohibitively 
expensive. This is because, in general, sampling involves generating
many different points in the state space of the Markov chain. \index{Markov chain} It can
be of interest to generate a single point, or small number of points,
which represent the salient features of the probability distribution, when this is possible. If the probability is peaked at one, or a small number of places, 
then simply locating these peaks may be sufficient in some applied contexts. 
This is the basis for variational methods\index{variational method} 
which seek to maximize the posterior 
probability, thereby locating such peaks. 
In practice this boils down to minimizing the negative log-posterior.
\index{posterior!log-posterior}

We start by illustrating the idea in the context of the Gaussian distributions
highlighted in section \ref{ssec:kalkal} concerning the Kalman smoother.\index{Kalman smoother}
In the case of stochastic dynamics,\index{stochastic dynamics} Theorem \ref{t:ks1} shows that
$\bbP(v|y)$, the pdf of the posterior distribution, has the form
$$P(v|y) \propto \exp\Bigl(-\frac12|v-m|_{L}^2\Bigr).$$
Now consider the problem
$$v^{\star}={\rm argmax}_{v \in \bbR^{|\J_0| \times n}} \bbP(v|y).$$ 
From the structure of $\bbP(v|y)$ we see that
$$v^{\star}={\rm argmin}_{v \in \bbR^{|\J_0| \times n}} \Ii(v;y)$$
where
$$\Ii(v;y)=\frac12|v-m|_{L}^2=\frac12|L^{-\frac12}(v-m)|^2.$$
Thus $v^{\star}=m$, the mean of the posterior.
Similarly, using Theorem \ref{t:ks2}, we can show that
in the case of deterministic dynamics \index{deterministic dynamics},
$$v_0^{\star}={\rm argmax}_{v_0 \in \bbR^{|\J_0| \times n}} \bbP(v_0|y),$$
in the case of deterministic dynamics,
is given by $v_0^{\star}=\mdet.$

In this section we show how to characterize peaks in the posterior
probability, in the general non-Gaussian case,
leading to problems in the calculus of variations\index{calculus of variations}.
The methods are termed {\bf variational methods}\index{variational method}. 
In the atmospheric sciences\index{atmospheric sciences}
these variational methods
are referred to as {\bf 4DVAR}\index{4DVAR}; this nomenclature
reflects the fact that they are variational methods\index{variational method} which incorporate data over three spatial dimensions and one temporal dimension (thus four dimensions in
total), in order to estimate the state. In Bayesian \index{Bayesian} statistics the methods are called {\bf MAP estimators}\index{MAP estimator}: 
maximum {\em a posteriori} estimators\index{maximum {\em a posteriori} estimators}. 
It is helpful to realize that the MAP estimator \index{MAP estimator} is not,
in general, equal to the mean of the posterior distribution.
However, in the case of Gaussian posteriors, it is equal to the mean.
Computation of the mean of a posterior distribution, in general,
requires integrating against the posterior distribution. This can
be achieved, via sampling for example, but is typically quite
expensive, if sampling is expensive. MAP estimators \index{MAP estimator},
in contrast, only require solution of an optimization\index{optimization} 
problem.
Unlike the previous section on MCMC methods we do not attempt to
overview the vast literature on relevant algorithms (here
optimization\index{optimization} algorithms); instead 
references are given in the bibliographic notes of section \ref{ssec:bs}.

First we consider the case of stochastic dynamics \index{stochastic dynamics}.
 
\begin{theorem}\label{th22}
Consider the data assimilation problem for
stochastic dynamics \index{stochastic dynamics}: \eqref{eq:dtf1}, \eqref{eq:dtf2}, with $\PPsi \in C^1(\bbR^n,\bbR^n)$ and $h \in C^1(\bbR^n,\bbR^m).$
Then:

\begin{itemize} 

\item (i)  the infimum of $\Ii(\cdot\,;y)$ given in (\ref{eq:dtf4})
is attained at at least one point $v^\star$ in $\bbR^{|\J_0|\times n}$.
It follows that the density $\rho(v)=\bbP(v|y)$ on $\bbR^{|\J_0|\times n}$
associated with the posterior probability $\mu$ given by Theorem \ref{th11}
is maximized at $v^\star;$ 

\item (ii) furthermore, 
let $B(u,\delta)$ denote a ball in $\R^{|\J_0|\times n}$ of radius $\delta$ and centred at $u$. Then 
\begin{equation}
\label{eq:missed}
\lim_{\delta\to0}\frac{\rp^\mu\bigl(B(u_1,\delta)\bigr)}{\rp^\mu\bigl(B(u_2,\delta)\bigr)}=\exp\bigl(\Ii(u_2;y)-\Ii(u_1;y)\bigr) 
\quad {\rm for ~all~} u_1, u_2 \in \bbR^{|\J_0|\times n}.
\end{equation}

\end{itemize}
\end{theorem}
\begin{proof}
Note that $\Ii(\cdot\,;y)$ is non-negative and continuous so that the infimum
$\Ibi$ is finite and non-negative.
To show that the infimum of $\Ii(\cdot\,;y)$ is attained in 
$\bbR^{|\J_0|\times n}$ we let $v^{(n)}$ denote a minimizing 
sequence. Without loss of generality we may assume that, for
all $n \in \bbN$,
$$\Ii(v^{(n)};y) \le \Ibi+1.$$
From the structure of $\Ii(\cdot\,;y)$ it follows that
\begin{align*}
v_0&=m_0+C_0^{\frac12}r_0,\\
v_{j+1}&=\PPsi(v_j)+\Sigma^{\frac12}r_{j+1}, \quad j \in \Z^+
\end{align*}
where $\frac12|r_j|^2 \le \Ibi+1$ for all $j \in \Z^+.$ By iterating
and using the inequalities on the $|r_j|$ we deduce the existence
of $K>0$ such that $|v^{(n)}| \le K$ for all $n \in \bbN$. From this
bounded sequence we may extract a convergent subsequence, relabelled
to $v^{(n)}$ for simplicity, with limit $v^\star$. By construction
we have that $v^{(n)} \to v^\star$ and, for any $\epsilon>0$, there is
$N=N(\epsilon)$ such that
$$\Ibi \le \Ii(v^{(n)};y) \le \Ibi+\epsilon,\quad \forall n \ge N.$$
Hence, by continuity of $\Ii(\cdot\,;y)$, it follows that
$$\Ibi \le \Ii(v^\star;y) \le \Ibi+\epsilon.$$ 
Since $\epsilon>0$ is arbitrary it follows that $\Ii(v^\star;y)=\Ibi.$
Because
\begin{align*} \mu(dv)&=\frac1Z\exp(-\Ii(v;y))dv\\
&=\rho(v)dv\end{align*}
it follows that $v^\star$ also maximizes the posterior pdf $\rho$. 

For the final result we first note that,
because $\PPsi$ and $h$ are continuously differentiable, the function
$I(\cdot\,;y)$ is continuously differentiable. Thus we have 
\begin{align*}
\rp^\mu\bigl(B(u,\delta)\bigr)&=\frac1Z\int_{|v-u|<\delta}\exp\bigr(-\Ii(v;y)\bigl)dv\\
&=\frac1Z\int_{|v-u|<\delta}\Bigl(\exp(-\Ii(u;y))+e(u;v-u)\Bigr)dv
\end{align*}
where 
$$e(u;v-u)=\Bigl\langle -\int_0^1 D_v I\bigl(u+s(v-u);y\bigr)ds,v-u \Bigr\rangle.$$ 
As a consequence we have, for $K^{\pm}>0$,
$$-K^-|\delta| \le e(u;v-u) \le K^+|\delta|$$
for $u=u_1,u_2$ and $|v-u|<\delta.$
Using the preceding we find that, for
$E:=\exp\bigl(\Ii(u_2;y)-\Ii(u_1;y)\bigr)$
$$\frac{\rp^\mu\bigl(B(u_1,\delta)\bigr)}{\rp^\mu\bigl(B(u_2,\delta)\bigr)}
\le E\frac{\int_{|v-u_1|<\delta} \exp\bigl(K^+|\delta|\bigr)dv}
{\int_{|v-u_2|<\delta} \exp\bigl(-K^-|\delta|\bigr)dv}=
E\frac{\exp\bigl(K^+|\delta|\bigr)}{\exp\bigl(-K^-|\delta|\bigr)}.$$
Similarly we have that 
$$\frac{\rp^\mu\bigl(B(u_1,\delta)\bigr)}{\rp^\mu\bigl(B(u_2,\delta)\bigr)}
\ge E\frac{\int_{|v-u_1|<\delta} \exp\bigl(-K^-|\delta|\bigr)dv}
{\int_{|v-u_2|<\delta} \exp\bigl(K^+|\delta|\bigr)dv}=
E\frac{\exp\bigl(-K^-|\delta|\bigr)}{\exp\bigl(K^+|\delta|\bigr)}.$$
Taking the limit $\delta \to 0$ gives the desired result.
\end{proof}

\begin{remark}\label{r23} The second statement in Theorem \ref{th22} may appear a little abstract. It is, however, essentially a complicated way of restating 
the first statement. To see this fix $u_2$ and note that the right hand
side of \eqref{eq:missed}
 is maximized at point $u_1$ which minimizes $\Ii(\cdot\,;y).$ Thus,
independently of the choice of  
{\em any} fixed $u_2$, the identity \eqref{eq:missed} shows that the 
probability of a small ball of radius $\delta$ centred at $u_1$ is,
approximately, maximized by choosing centres at minimizers of $\Ii(\cdot\,;y).$
Why, then, do we bother with the second statement? We do so because
it makes no reference to Lebesgue density. As such it can be generalized to 
infinite dimensions, 
as is required in continuous time for example. 
We include the second statement for precisely this reason.
We also remark that our assumption on continuous differentiability of
$\PPsi$ and $h$ is stronger than what is needed, but makes for
the rather explicit bounds used in the preceding proof and is hence
pedagogically desirable.
\end{remark}

The preceding theorem leads to a natural algorithm: compute 
$$v={\rm argmin}_{u \in \bbR^{|\J_0|\times n}}\, \Ii(u;y).$$
In applications to meteorology this algorithm
is known as {\bf weak constraint 4DVAR}\index{4DVAR!weak constraint}, and we
denote this as {\bf w4DVAR} in what follows. 
The word ``weak'' in this context is used to indicate that the deterministic
dynamics model (\ref{eq:dtf11}a) is not imposed as a strong
constraint. Instead the objective functional $\Ii(\cdot\,;y)$ is minimized;
this penalizes deviations from exact satisfaction of the
deterministic dynamics model, as well as deviations from the data.

The w4DVAR  method generalizes the standard {\bf 4DVAR} 
method which may be derived
from w4DVAR in the limit $\Sigma \to 0$ so that the prior on the
model dynamics \eqref{eq:dtf1} is deterministic, 
but with a random initial condition, as in \eqref{eq:dtf11}.
In this case the appropriate minimization is of $\Iid(v_0;y)$
given by \eqref{eq:Ip}. This has the advantage of being a lower
dimensional minimization problem than w4DVAR;
however it is often a harder minimization problem, especially
when the dynamics is chaotic. The basic 4DVAR algorithm is sometimes
called {\bf strong constraint 4DVAR}\index{4DVAR!strong constraint}
to denote the fact that the dynamics model (\ref{eq:dtf11}a) is imposed
as a strong constraint on the minimization of the model-data
misfit\index{model-data misfit} with respect to the
initial condition; we simply refer to the method as 4DVAR.
The following theorem may be
proved similarly to Theorem \ref{th22}.

\begin{theorem}\label{th22a}
Consider the data assimilation problem for
deterministic dynamics: \eqref{eq:dtf11},\eqref{eq:dtf2} with $\PPsi \in C^1(\bbR^n,\bbR^n)$
and $h \in C^1(\bbR^n,\bbR^m).$
Then:

\begin{itemize}

\item (i) the infimum of $\Iid(\cdot\,;y)$ given in (\ref{eq:Ip})
is attained at at least one point $v^\star_0$ in $\bbR^{n}$.
It follows that the density $\varrho(v_0)=\bbP(v_0|y)$ on $\bbR^{n}$
associated with the posterior probability $\nu$ given by Theorem \ref{th112}
is maximized at $v^\star_0;$ 

\item (ii) furthermore, if $B(z,\delta)$ denotes 
a ball in $\R^n$ of radius $\delta$, centred at $z$, then 
\[\lim_{\delta\to0}\frac{\rp^\nu\bigl(B(z_1,\delta)\bigr)}{\rp^\nu\bigl(B(z_2,\delta)\bigr)}=\exp(\Iid(z_2;y)-\Iid(z_1;y)).\]

\end{itemize}

\end{theorem}

As in the case of stochastic dynamics\index{stochastic dynamics} 
we do not discuss optimization\index{optimization} methods 
to perform minimization associated with 
variational\index{calculus of variations} problems; this is
because  optimization is a well-established 
and mature research area which is hard to do justice to within the
confines of this book. However we conclude this section 
with an example which illustrates certain advantages of the Bayesian \index{Bayesian}
perspective over the optimization\index{optimization} or variational perspective.
Recall from Theorem \ref{th14} that the Bayesian \index{Bayesian} posterior distribution
is continuous with respect to small changes in the data. In contrast,
computation of the global maximizer of the probability may be
discontinuous as a function of data. To illustrate this consider
the probability measure $\mue$ on $\bbR$ with Lebesgue density
proportional to $\exp\bigl(-\Ve(u)\bigr)$ where
\begin{equation}
\label{eq:V}
\Ve(u)=\frac14(1-u^2)^2+\epsilon u.
\end{equation} 
It is a straightforward application of the methodology behind the
proof of Theorem \ref{th14} to show that $\mue$ is Lipschitz
continuous in $\epsilon$, with respect to the Hellinger metric.
Furthermore the methodology behind Theorems \ref{th22} and \ref{th22a} 
shows that the probability with respect to this measure is maximized
where $\Ve$ is minimized.
The global minimum, however, changes discontinuously, even though
the posterior distribution changes smoothly. This is illustrated 
in Figure \ref{fig:varsky}, where the left hand panel shows the
continuous evolution of the probability density function,\index{probability density function} whilst
the right hand-panel shows the discontinuity in the global maximizer
of the probability (minimizer of $\Ve$) as $\epsilon$ passes through zero.
The explanation for this difference
between the fully Bayesian \index{Bayesian} approach and MAP\index{MAP estimator}
estimation is as follows. 
The measure $\mue$ has two peaks, for small $\epsilon$,
close to $\pm 1.$ The Bayesian \index{Bayesian} approach accounts for both of these
peaks simultaneously and weights their contribution to expectations.
In contrast the MAP\index{MAP estimator}
estimation approach leads to a global minimum located
near $u=-1$ for $\epsilon>0$ and near $u=+1$ for $\epsilon<0$, resulting
in a discontinuity.

\begin{figure}[h]
\centering
\subfigure[$V^\epsilon$ for $\epsilon>0$, $\epsilon<0$, and $\epsilon=0$.]{\includegraphics[scale=0.365]{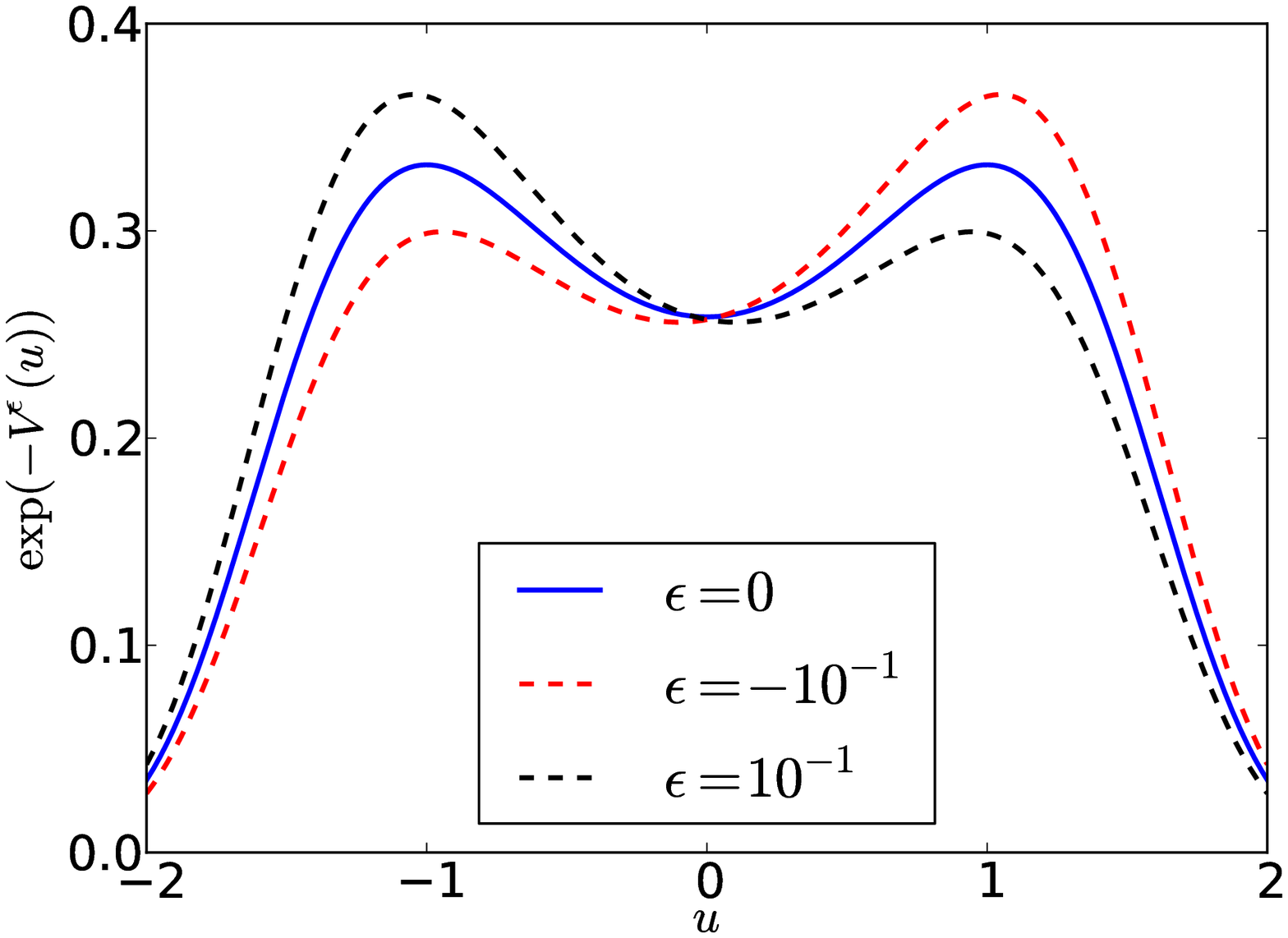}}
\subfigure[Global minima of $V^\epsilon$ as a function of $\epsilon$]{\includegraphics[scale=0.365]{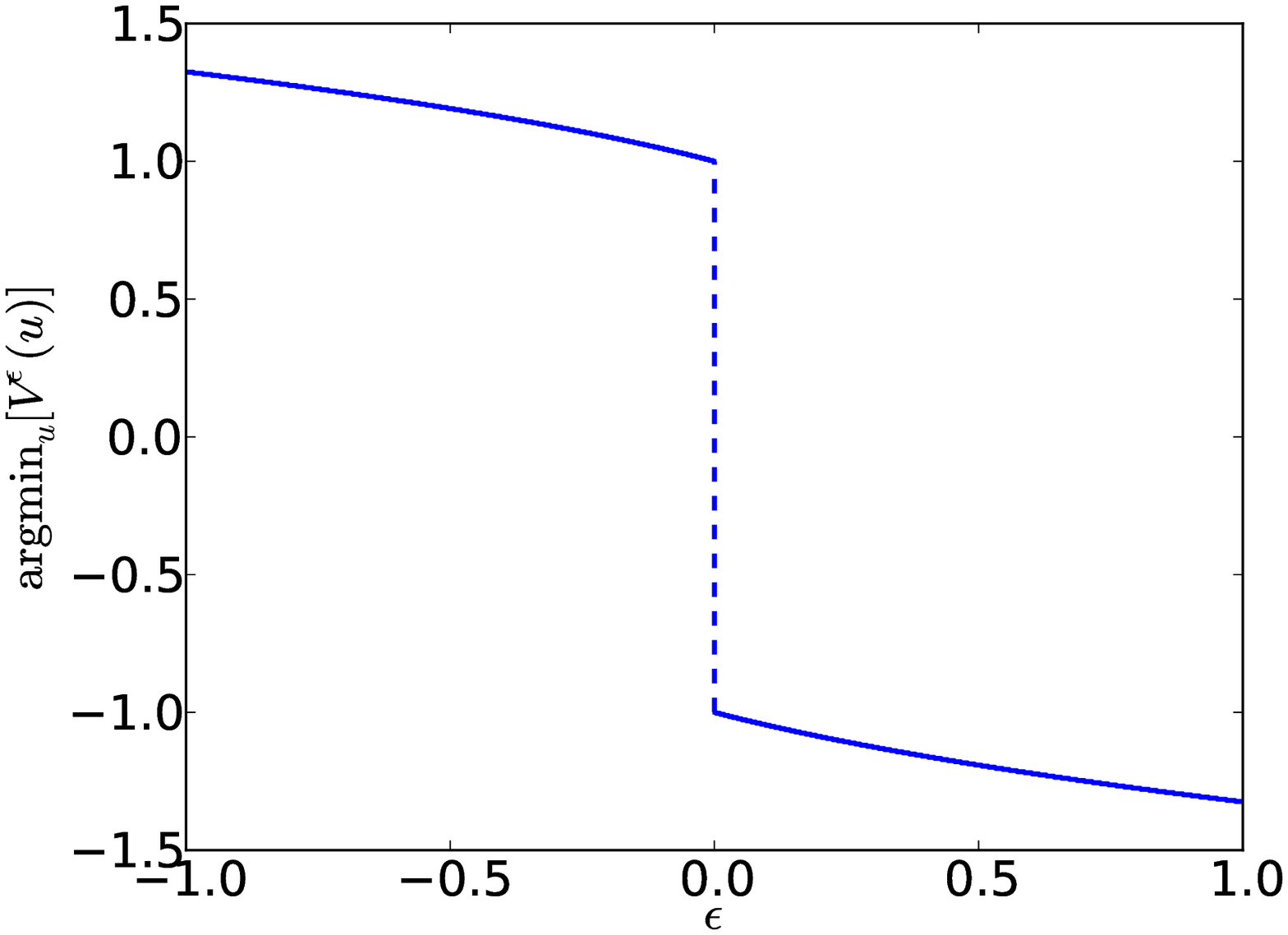}}
\caption{Plot of \eqref{eq:V} shows discontinuity of the global maximum as a function of $\epsilon$.}
\label{fig:varsky}
\end{figure}

\section{Illustrations}
\label{ssec:mci}

We describe  a range of numerical experiments which illustrate the
application of MCMC \index{MCMC} methods and variational methods\index{variational method} to the smoothing
problems which arise in both deterministic and stochastic dynamics \index{stochastic dynamics}.

The first illustration concerns use of
the RWM\index{Random Walk Metropolis} algorithm to 
study the smoothing distribution 
for Example \ref{ex:ex4} in the case of deterministic dynamics where
our aim is to find $\bbP(v_0|y).$  
Recall Figure \ref{fig:smooth3}a which shows the true posterior
pdf, found by plotting the formula given in Theorem
\ref{th11}. We now approximate the true posterior pdf by the MCMC \index{MCMC}
method, using the same parameters, namely $m_0=0.5, C_0=0.01, \gamma=0.2$
and $\vd_0=0.3$. 
In Figure \ref{fig:MCMC1} we compare the posterior pdf calculated by 
the RWM\index{Random Walk Metropolis} method (denoted by $\rho^N$, the histogram \index{histogram}
of the output of the Markov chain) with the true posterior pdf $\rho$.
\begin{figure}[h]
\centering
\subfigure[$\rho^N$ vs $\rho$]{\includegraphics[scale=0.365]{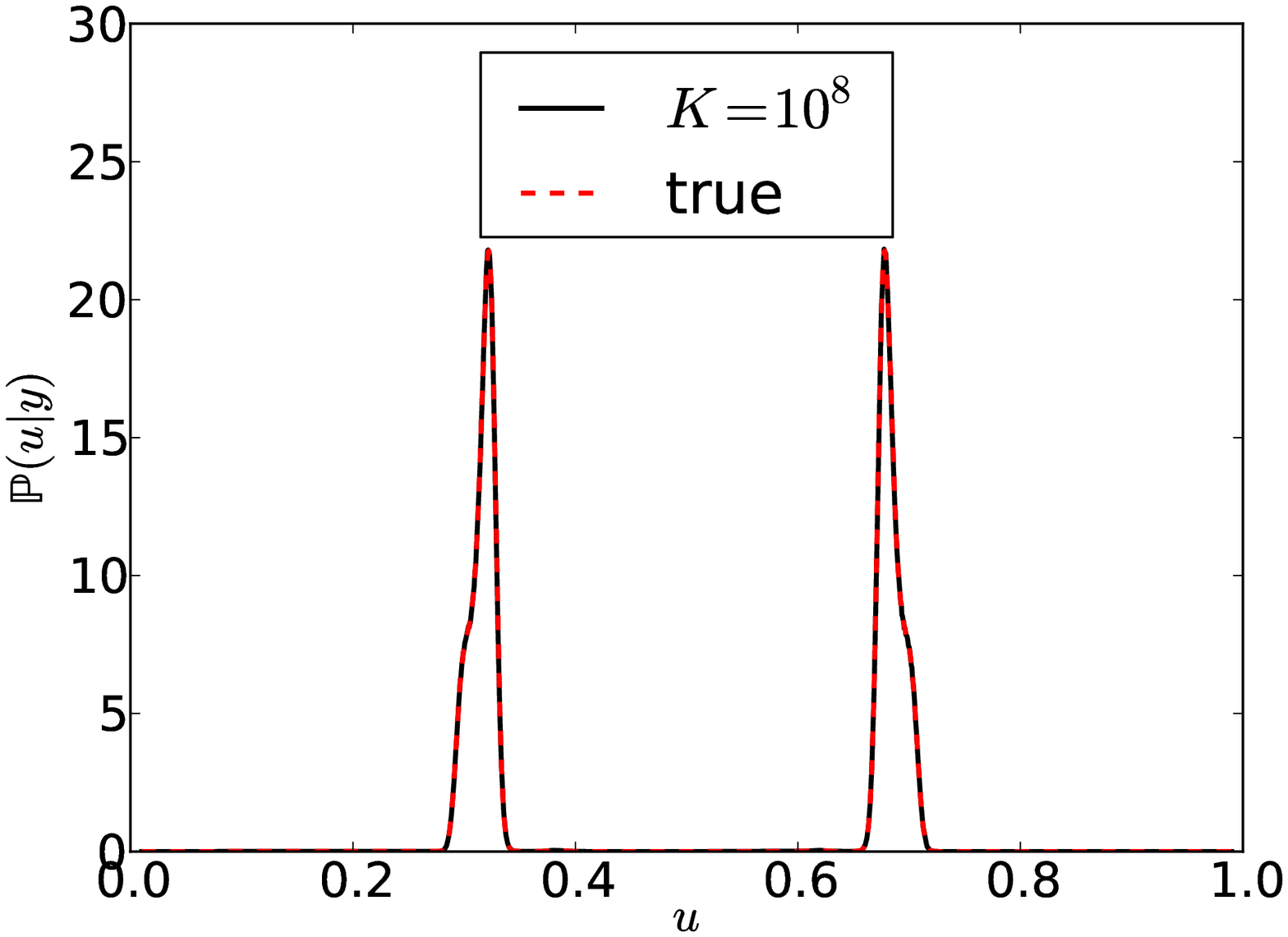}}
\subfigure[$\rho^N-\rho$]{\includegraphics[scale=0.365]{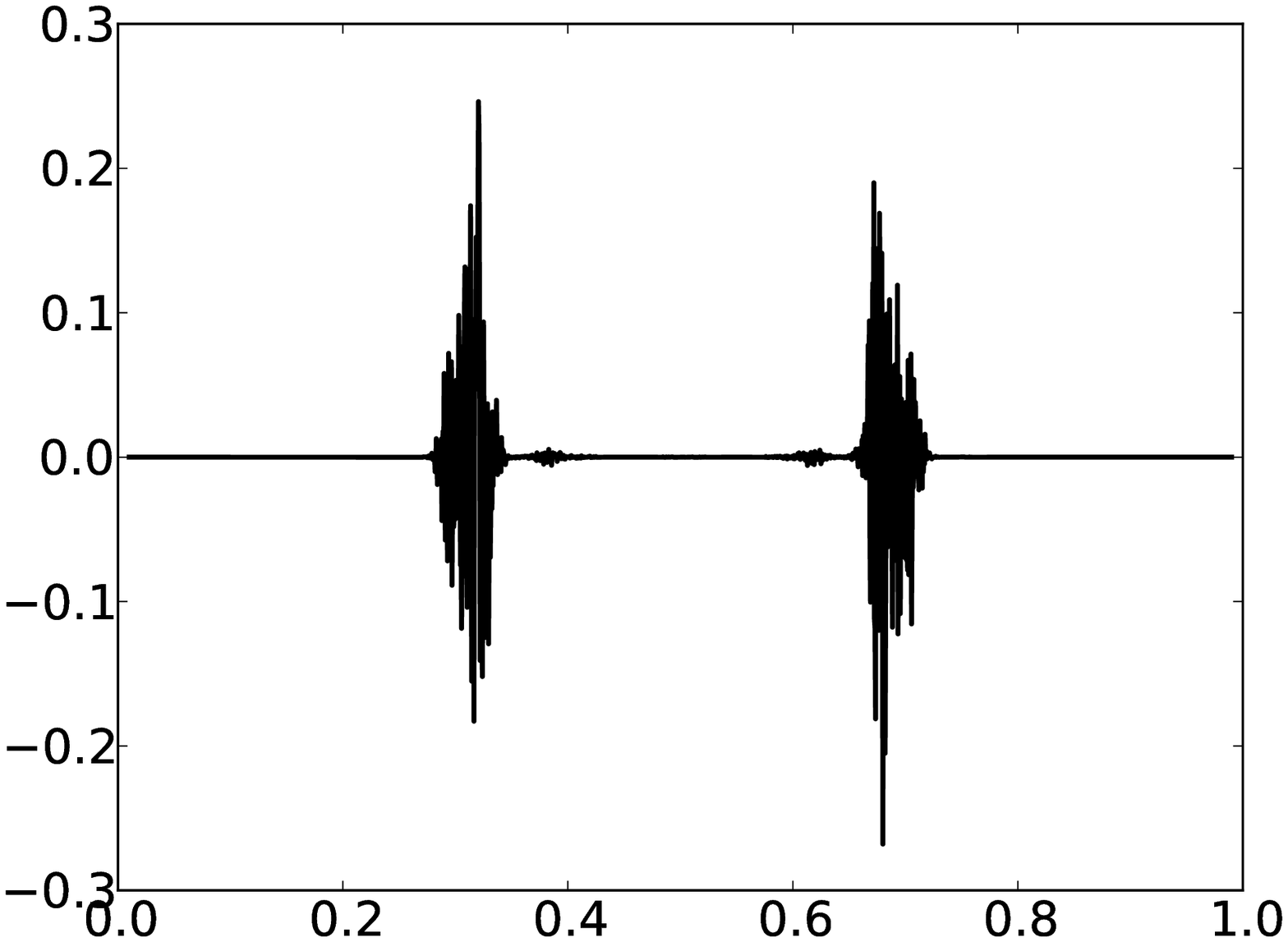}}
\caption{Comparison of the posterior for Example \ref{ex:ex4} for $r=4$ using random walk metropolis and equation \eqref{eq:Ip} directly as in the  \MAT program {\tt p2.m}. We have used  $J=5$ $C_{0}=0.01, m_{0}=0.5$, $\gamma=0.2$ and true initial condition $v_{0}=0.3$, see also {\tt p3.m} in section \ref{ssec:p3}. We have used  $N=10^8$ samples from the MCMC \index{MCMC} algorithm.}
\label{fig:MCMC1}
\end{figure}
The two distributions are almost indistinguishable when plotted together
in Figure \ref{fig:MCMC1}a; in Figure \ref{fig:MCMC1}b we plot their difference, which as we can see is small, relative to the true value.
We deduce that the number of samples used, $N=10^8$, results here in
accurate sampling of the posterior.

We now turn to the use of MCMC \index{MCMC} methods to sample the smoothing
pdf $\bbP(v|y)$ in the case of stochastic dynamics \index{stochastic dynamics} \eqref{eq:dtf1},
using  the Independence Dynamics Sampler\index{Independence Dynamics Sampler} and both pCN \index{pCN} methods.
Before describing application of numerical methods we study the
ergodicity\index{ergodic} of the Independence Dynamics Sampler\index{Independence Dynamics Sampler} in a simple, but illustrative, setting.
For simplicity assume that the observation operator $h$ is bounded so that,
for all $u \in \R^N$, $|h(u)| \le \hx.$ Then, recalling the notation
$Y_j=\{y_{\ell}\}_{\ell=1}^{j}$ from the filtering problem, we have
\begin{align*}
\PPhi(u;y) & \le \sum_{j=0}^{J-1}\bigl(|\Gamma^{-\frac12}y_{j+1}|^2+|\Gamma^{-\frac12}
h(u_{j+1})|^{2}\bigr)\\
& \le |\Gamma^{-\frac12}|^2\Bigl(\sum_{j=0}^{J-1}|y_{j+1}|^2+J\hx^2\Bigr)\\
& \le |\Gamma^{-\frac12}|^2\Bigl(|Y_J|^2+J\hx^2\Bigr)\\
& =: \ppx.
\end{align*} 
Since $\PPhi \ge 0$ this shows that every proposed step is accepted
with probability exceeding $e^{-\ppx}$ and hence that, since proposals
are made with the prior measure $\mu_0$ describing the unobserved
stochastic dynamics \index{stochastic dynamics},
$$p(u,A) \ge e^{-\ppx}\mu_0(A).$$
Thus Theorem \ref{th21} applies and, in particular, \eqref{eq:tv}
and \eqref{eq:clt} hold, with $\varepsilon=e^{-\ppx}$, under these
assumptions. This positive result about the ergodicity\index{ergodic}
of the MCMC method, also indicates the potential difficulties
with the Independence Dynamics Sampler\index{Independence Dynamics Sampler}.
The Independence Sampler\index{Independence Dynamics Sampler} 
relies on draws from the prior matching the data well. Where the data set is large ($J\gg1$) or the noise covariance small 
($|\Gamma|\ll 1$) this will happen infrequently, because $\ppx$ will be
large, and the MCMC \index{MCMC} method will reject frequently and be inefficient. 
To illustrate this we consider application of the method to the Example \ref{ex:ex3}, using the same parameters as in Figure \ref{fig:p1};
specifically we take $\alpha=2.5$ and $\Sigma=\sigma^2=1$.
We now sample the posterior distribution and then plot the resulting
accept-reject ratio $a$ for the Independence Dynamics Sampler\index{Independence Dynamics Sampler}, employing
different values of noise $\Gamma$ and different sizes of the data set $J$.
This is illustrated in Figure \ref{fig:smooth_MCMC1}.
\begin{figure}[h]
\centering
\subfigure[$J=5$, different values of $\gamma$]{\includegraphics[scale=0.365]{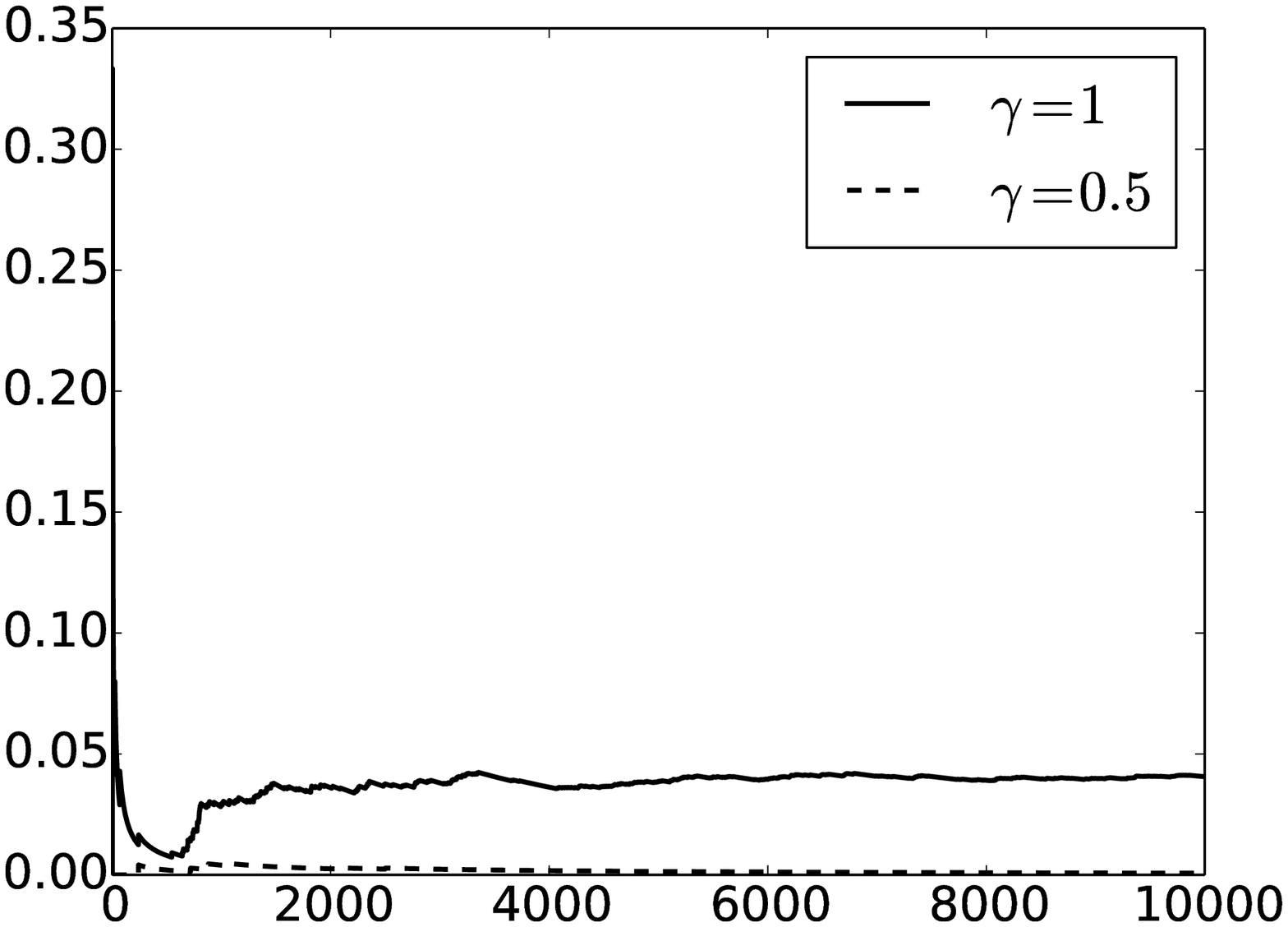}}
\subfigure[$\gamma=1$, different values of $J$]{\includegraphics[scale=0.365]{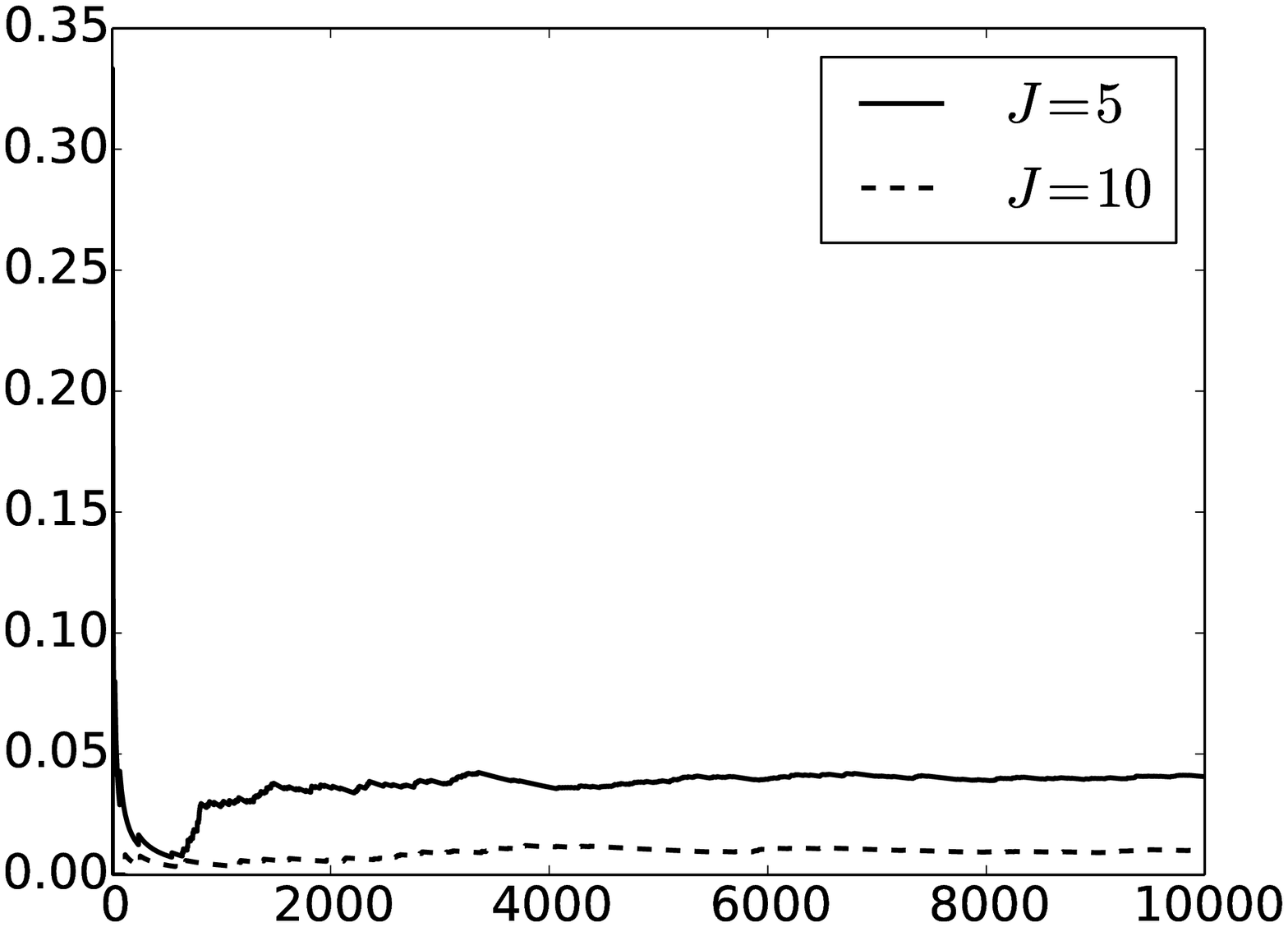}}
\caption{Accept-reject probability of the Independence Sampler\index{Independence Dynamics Sampler} for Example \ref{ex:ex3} for $\alpha=2.5$, $\Sigma=\sigma^{2}=1$ and $\Gamma=\gamma^{2}$ for different values of $\gamma$ and $J$.}
\label{fig:smooth_MCMC1}
\end{figure}

\begin{figure}[h]
\centering
\subfigure[First element of $v^{(k)}$]{\includegraphics[scale=0.365]{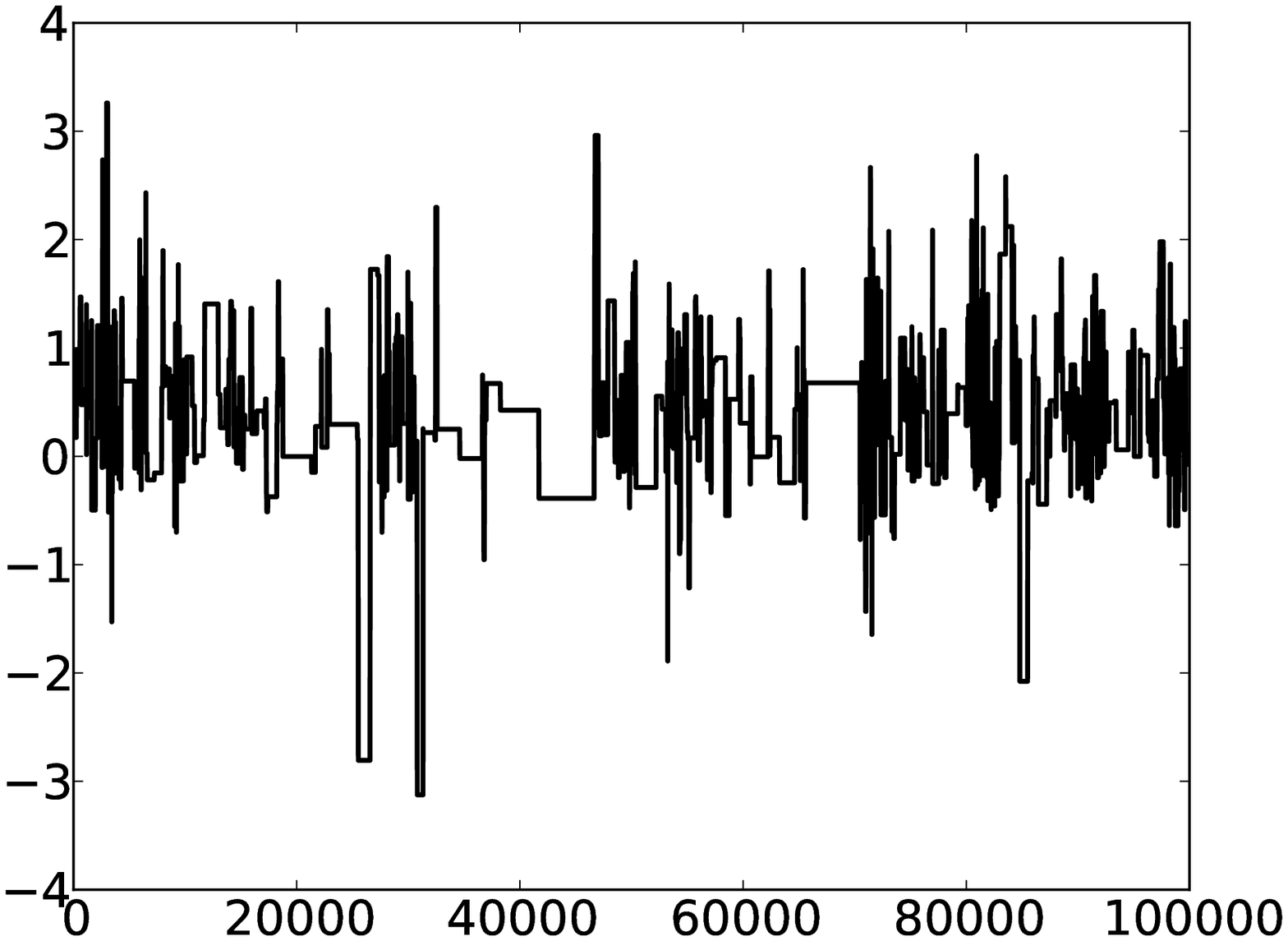}}
\subfigure[Running average of first element of $v^{(k)}$]{\includegraphics[scale=0.365]{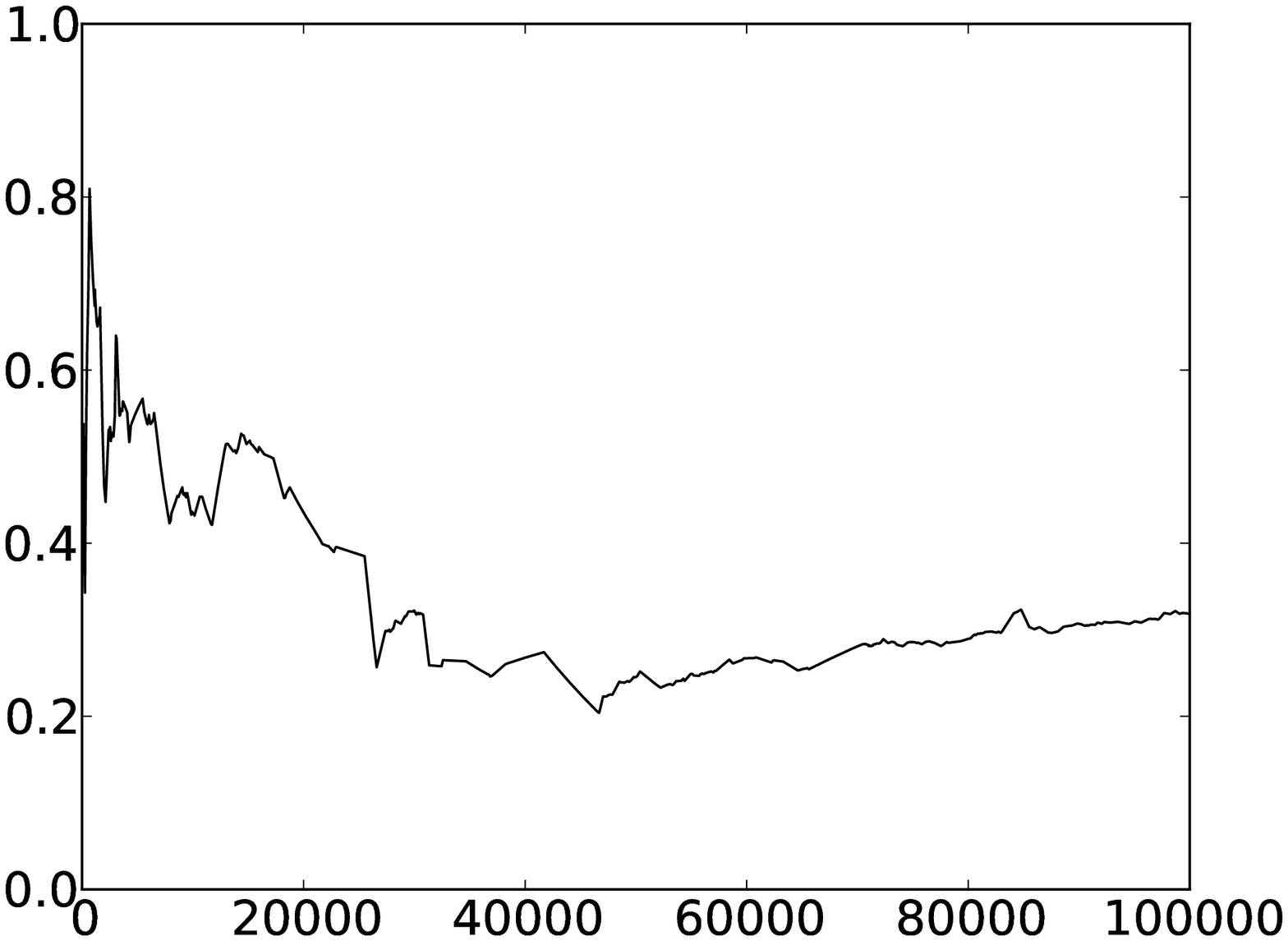}}
\caption{Output and running average of the  Independence  Dynamics Sampler\index{Independence Dynamics Sampler} after $K=10^5$ steps,  for Example \ref{ex:ex3} for $\alpha=2.5$, $\Sigma=\sigma^{2}=1$ and $\Gamma=\gamma^{2}=1$, with 
$J=10$, see also {\tt p4.m} in section \ref{ssec:p4}.}
\label{fig:smooth_MCMC2}
\end{figure}
In addition, in Figure \ref{fig:smooth_MCMC2}, we plot the output,
and the running average of the output, projected into the first element of 
the vector $v^{(k)}$, the initial condition  -- remember that we are defining a Markov chain \index{Markov chain}
on $\mathbb{R}^{J+1}$ -- for $N=10^5$ steps. Figure \ref{fig:smooth_MCMC2}a 
clearly exhibits the fact that there are many rejections caused
by the low average acceptance probability. Figure \ref{fig:smooth_MCMC2}b
shows that the running average has not converged after $10^5$ steps, indicating
that the chains needs to be run for longer.
If we run the Markov chain \index{Markov chain}over $N=10^8$ steps then we do get convergence.
This is illustrated in Figure \ref{fig:smooth_MCMC3}. In Figure \ref{fig:smooth_MCMC3}a we see that the running average has converged to its limiting value 
when this many steps are used. In Figure \ref{fig:smooth_MCMC3}b where we plot 
the marginal probability distribution for the first element of $v^{(k)}$,
calculated from this converged Markov chain\index{Markov chain}.

\begin{figure}[h]
\centering
\subfigure[Running average of first element of $v^{(k)}$]{\includegraphics[scale=0.36]{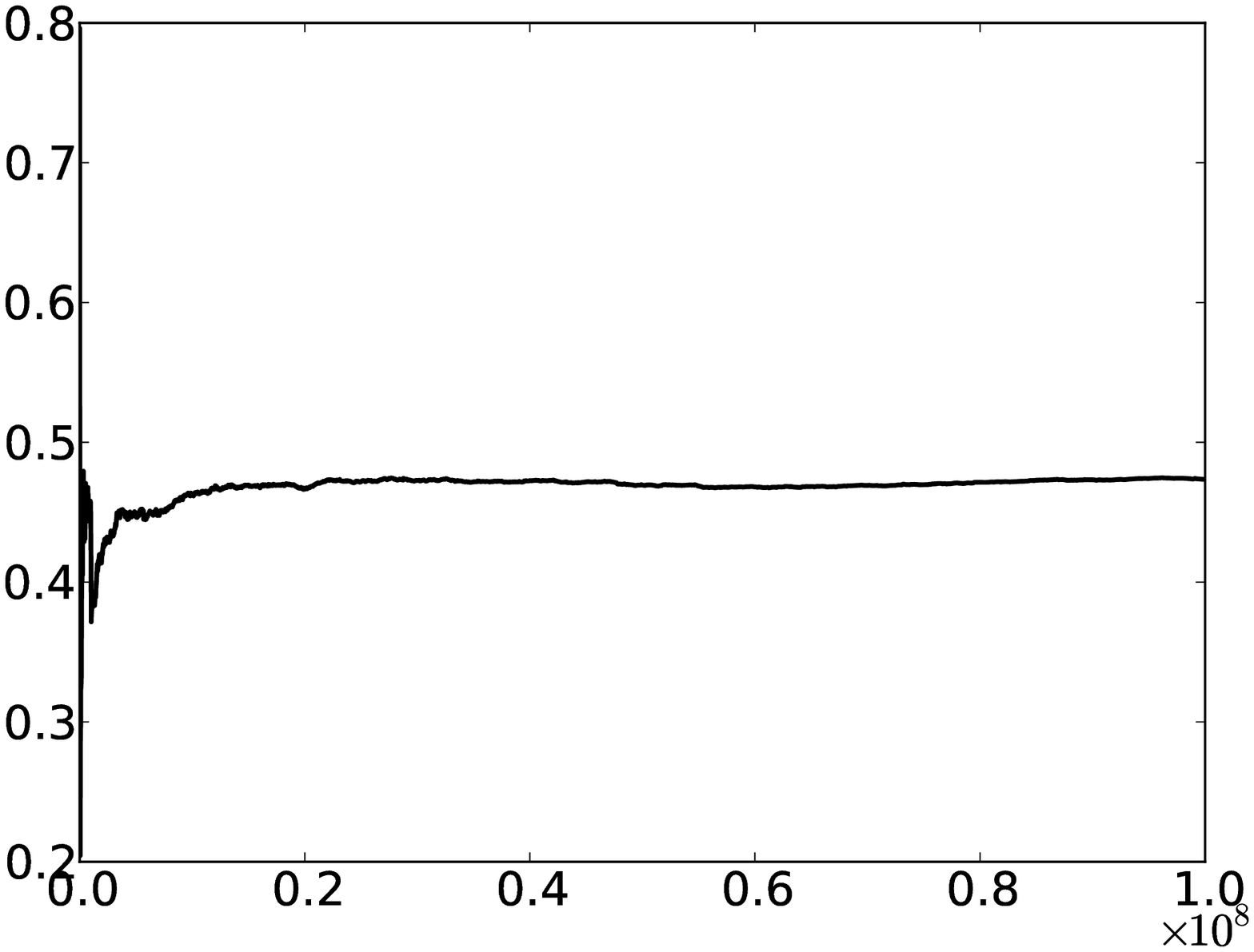}}
\subfigure[Probability density of first element of $v^{(k)}$]{\includegraphics[scale=0.36]{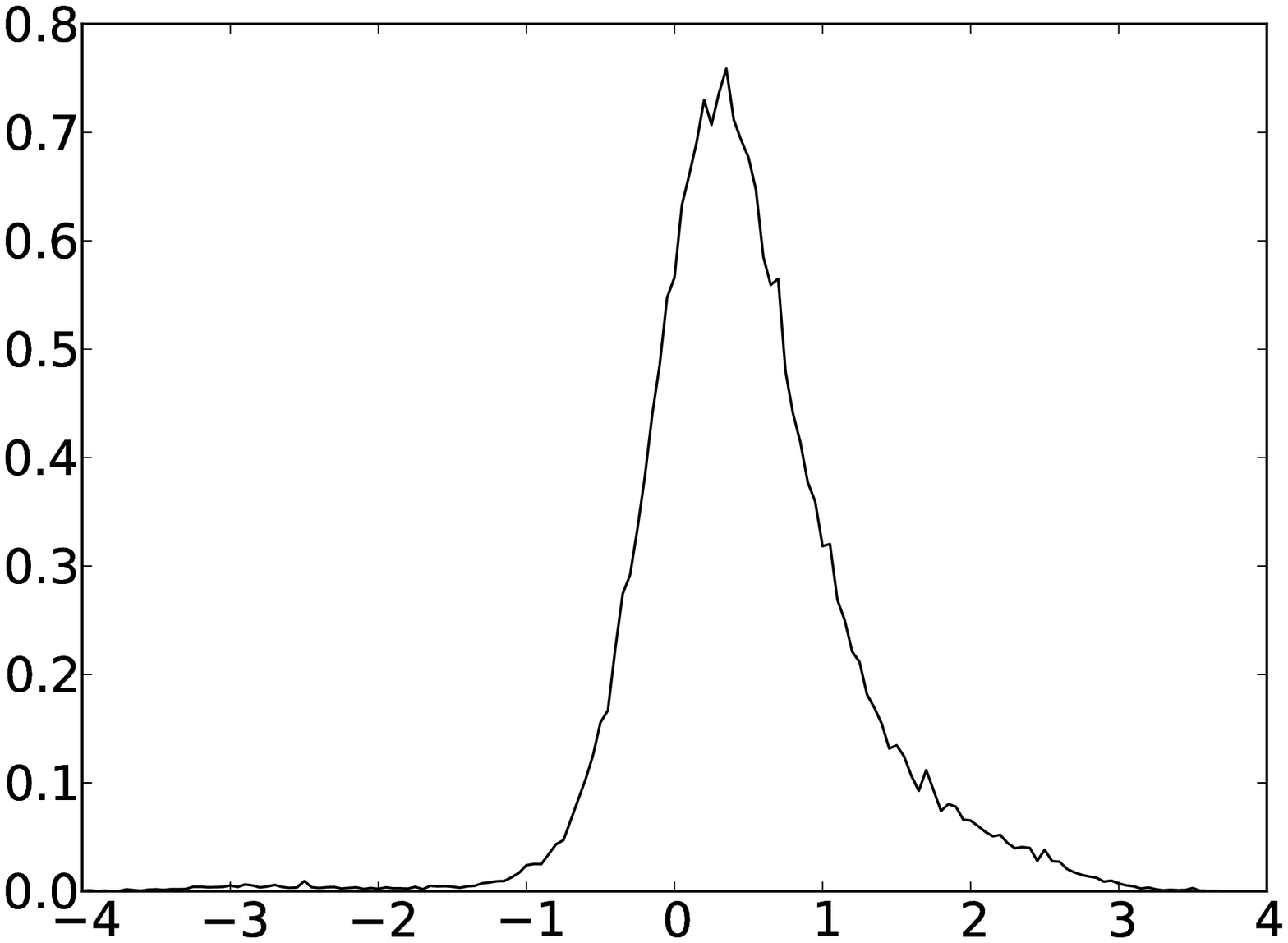}}
\caption{Running average and probability density of the first element of $v^{(k)}$  for the  Independence  Dynamics Sampler\index{Independence Dynamics Sampler} after $K=10^8$ steps,  for  Example \ref{ex:ex3} for $\alpha=2.5$, $\Sigma=\sigma^{2}=1$ and $\Gamma=\gamma^{2}$, with $\gamma=1$ and $J=10$, see also {\tt p4.m} in section \ref{ssec:p4}.}
\label{fig:smooth_MCMC3}
\end{figure}

In order to get faster convergence when sampling the 
posterior distribution\index{posterior distribution}
we turn to application of the pCN \index{pCN} method. Unlike the Independence Dynamics Sampler\index{Independence Dynamics Sampler}, this contains a tunable parameter which can vary the size of
the proposals. In particular, the possibility of making small
moves, with resultant higher acceptance probability, makes this a more
flexible method than the Independence Dynamics Sampler\index{Independence Dynamics Sampler}.
In Figure \ref{fig:pCN} we show application of the pCN \index{pCN} sampler, again
considering 
 Example \ref{ex:ex3} for $\alpha=2.5$, $\Sigma=\sigma^{2}=1$ and $\Gamma=\gamma^{2}=1$, with 
 $J=10$, the same parameters used in Figure \ref{fig:smooth_MCMC2}.

In the case that the dynamics are significantly influencing the trajectory, 
i.e. the regime of large $\PPsi$ or small $\sigma$, 
it may be the case that the standard pCN \index{pCN} method
is not effective, due to large effects of the $G$ term, and the improbability of 
Gaussian samples being close to samples of the prior on the dynamics.
The pCN \index{pCN} Dynamics sampler, recall, acts on the space comprising the 
the initial condition and forcing, both of which are Gaussian
under the prior, and so may sometimes have an advantage given that
pCN \index{pCN}-type methods are based on Gaussian proposals\index{proposal!Gaussian}. 
The use of this method is explored in Figure \ref{fig:pCND}
for Example \ref{ex:ex3} for $\alpha=2.5$, $\Sigma=\sigma^{2}=1$ and $\Gamma=\gamma^{2}=1$, with 
 $J=10$.  

\begin{figure}[h]
\centering
\subfigure[First element of $v^{(k)}$]{\includegraphics[scale=0.365]{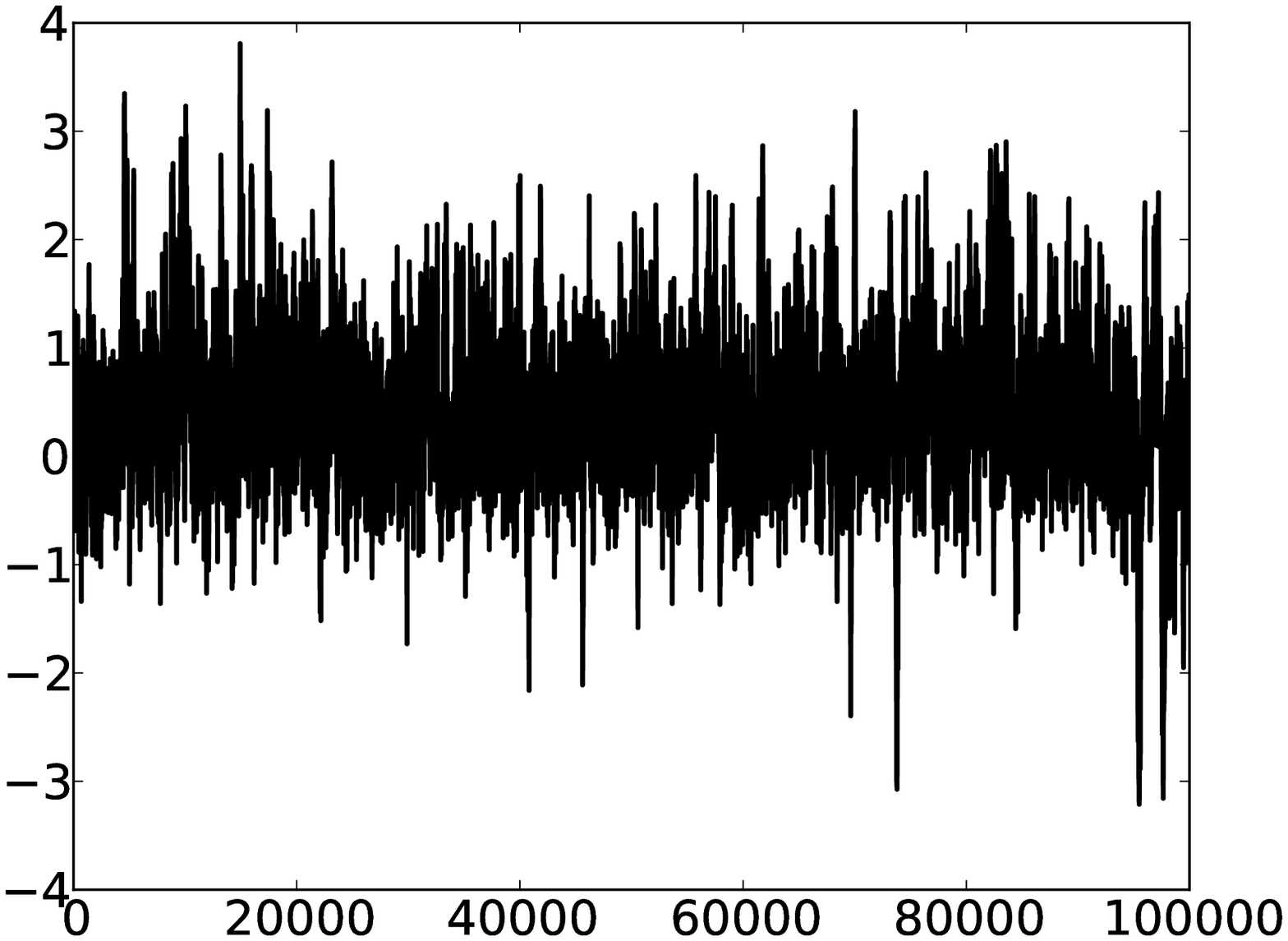}}
\subfigure[Running average of first element of $v^{(k)}$]{\includegraphics[scale=0.365]{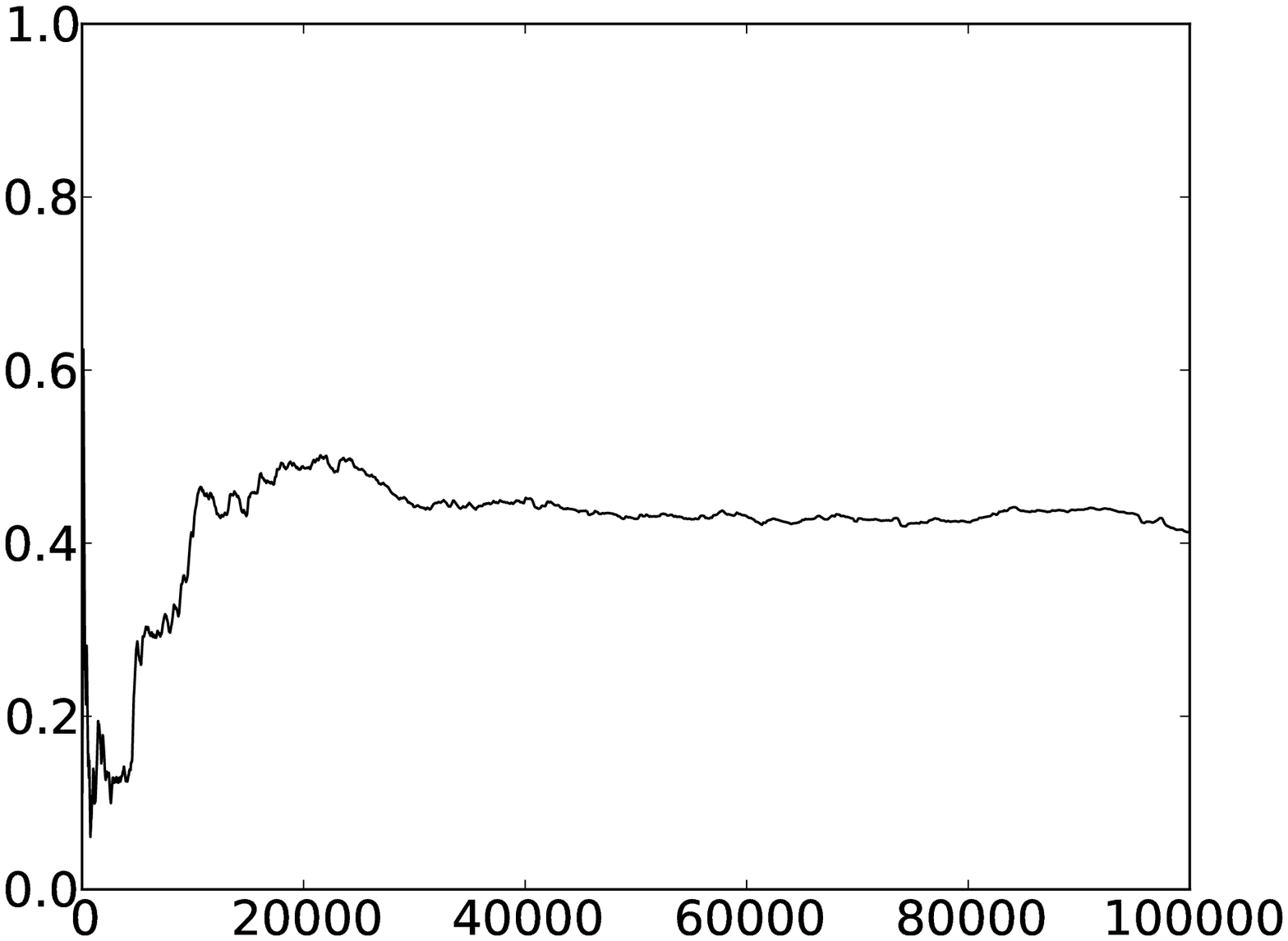}}
\caption{Trace-plot and running average of the first element of $v^{(k)}$  for the  pCN \index{pCN} sampler after $K=10^5$ steps,  
for  Example \ref{ex:ex3} with $\alpha=2.5$, $\Sigma=\sigma^{2}=1$ and $\Gamma=\gamma^{2}=1$, with 
$J=10$, see also {\tt p5.m} in section \ref{ssec:p5}.}
\label{fig:pCN}
\end{figure}

\begin{figure}[h]
\centering
\subfigure[First element of $v^{(k)}$]{\includegraphics[scale=0.365]{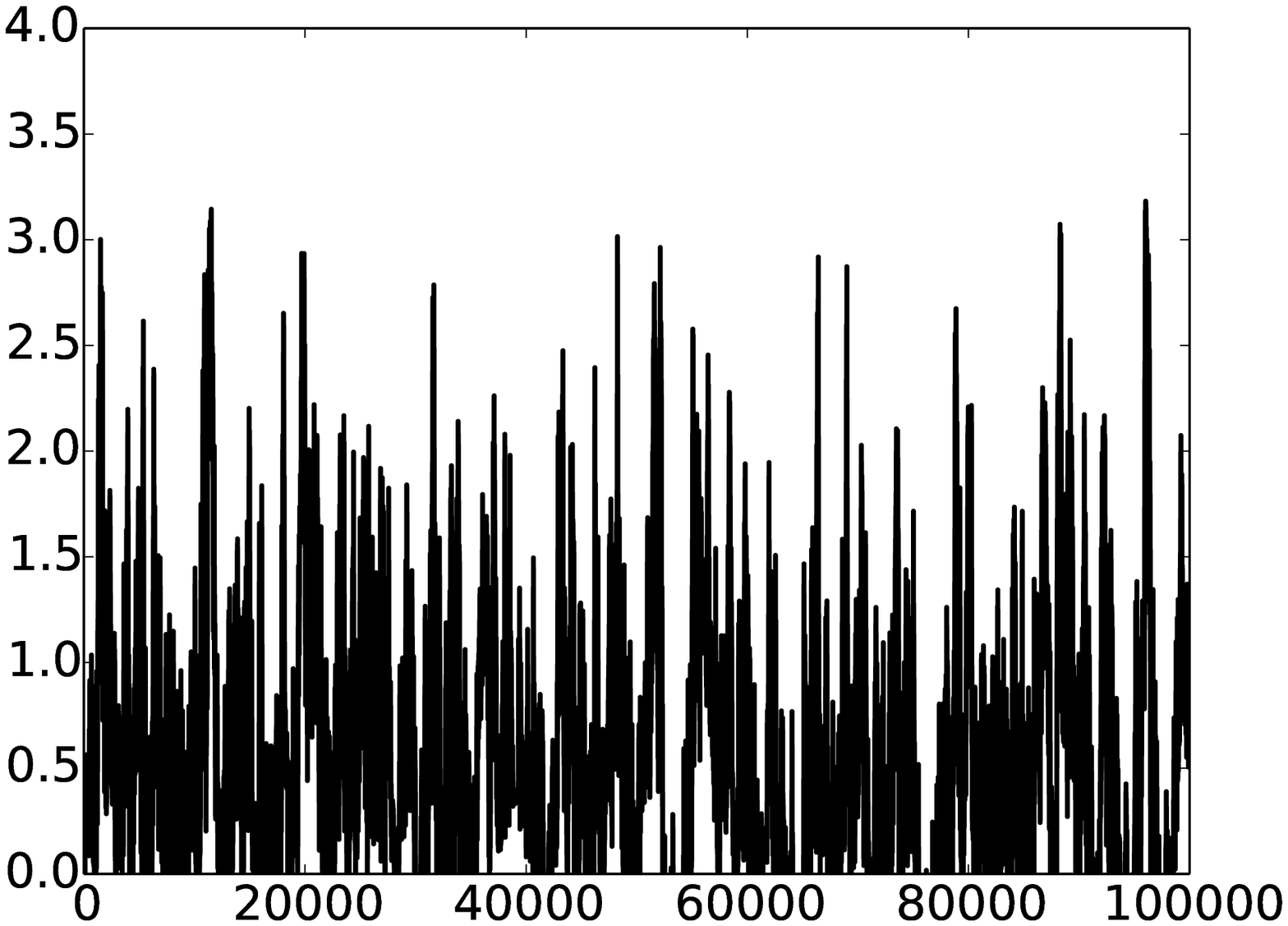}}
\subfigure[Running average of first element of $v^{(k)}$]{\includegraphics[scale=0.365]{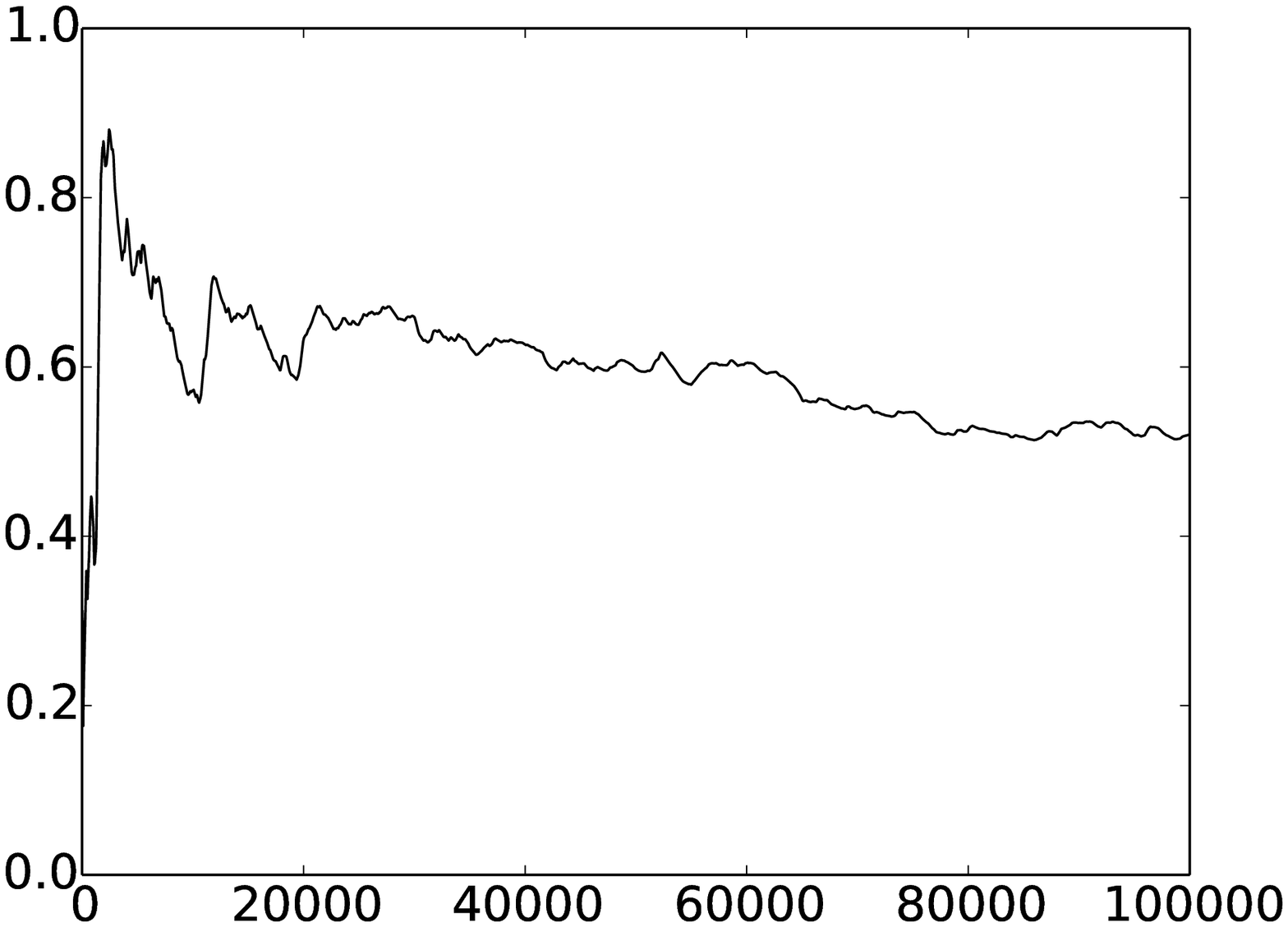}}
\caption{Trace-plot and running average of the first element of $v^{(k)}$  for the  pCN \index{pCN} dynamics sampler after $K=10^5$ steps,  
for  Example \ref{ex:ex3} with $\alpha=2.5$, $\Sigma=\sigma^{2}=1$ and $\Gamma=\gamma^{2}=1$, with 
$J=10$, see also {\tt p6.m} in section \ref{ssec:p5}.}
\label{fig:pCND}
\end{figure}


We now turn to variational methods\index{variational method}; 
recall Theorems \ref{th22} and
\ref{th22a} in the stochastic and deterministic cases respectively.
In Figure \ref{fig:4DVAR}a we plot the MAP\index{MAP estimator} 
(4DVAR)\index{4DVAR} estimator for 
our Example \ref{ex:ex1}, choosing exactly the same parameters and data as for Figure \ref{fig:smooth1}a, in the case where $J=10^{2}$. 
In this case the function $\Iid(\cdot\,;y)$ is
quadratic and has a unique global minimum. A straightforward
minimization routine will easily find this: we employed standard
\MAT optimization\index{optimization} software initialized at three different
points. From all three starting points
chosen the algorithm finds the correct global minimizer. 

In Figure \ref{fig:4DVAR}b we plot the MAP\index{MAP estimator} (4DVAR) estimator for 
our Example \ref{ex:ex4} for  the case $r=4$ choosing exactly the same parameters and data as for Figure \ref{fig:smooth3}. We again employ
a \MAT optimization\index{optimization} routine, and we again initialize it 
at three different points. 
\begin{figure}[h]
\centering
\subfigure[$I(v_{0};y)$ and its minima for Example \ref{ex:ex1}]{\includegraphics[scale=0.365]{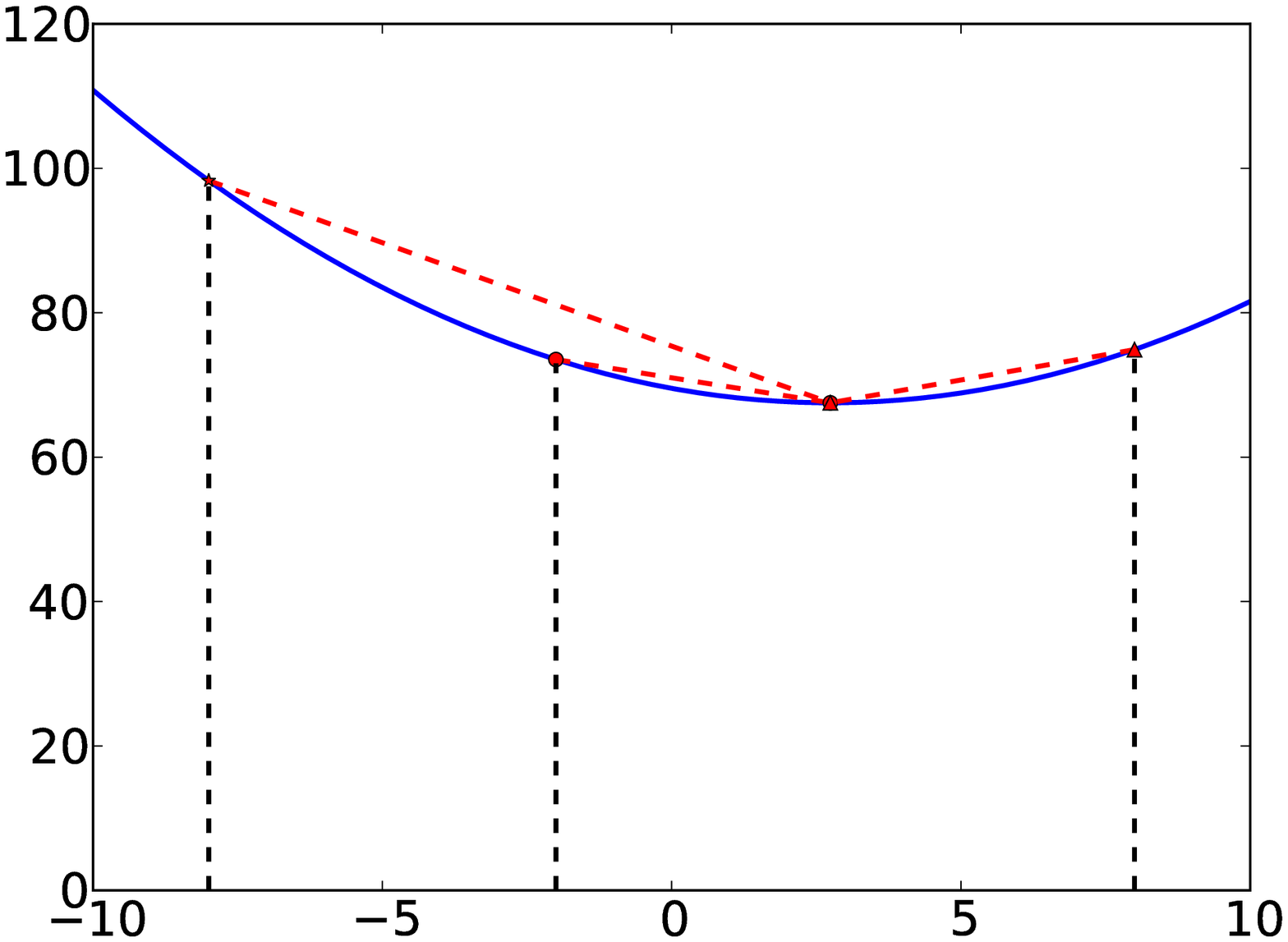}}
\subfigure[$I(v_{0};y)$ and its minima for Example \ref{ex:ex4}]{\includegraphics[scale=0.365]{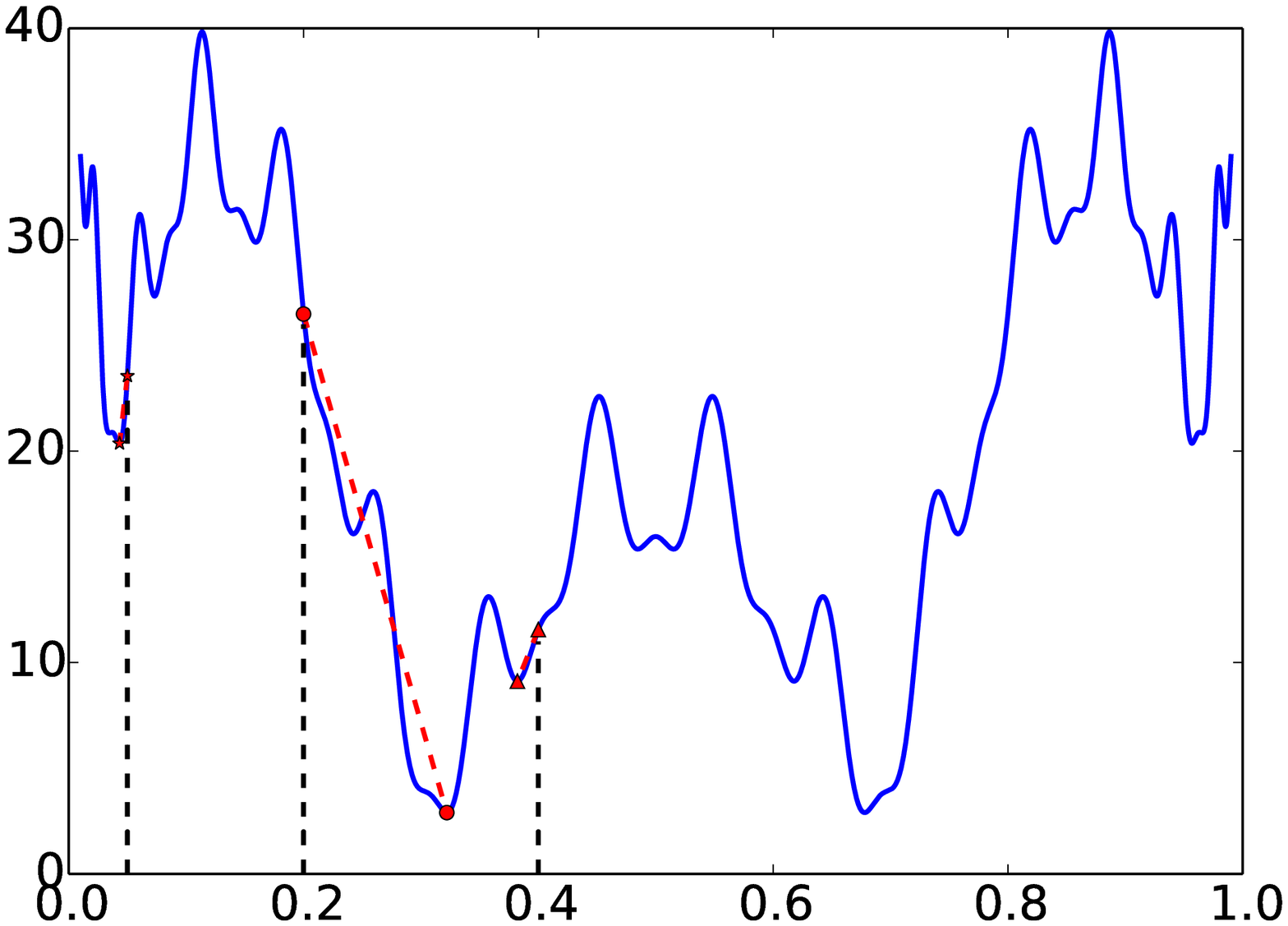}}
\caption{Finding local minima of $I(v_{0};y)$ for Examples \ref{ex:ex1} and \ref{ex:ex4}. The values and the data used are the same as for Figures \ref{fig:smooth1}a and \ref{fig:smooth3}b. $(\circ,\star,\Box)$ denote three different initial conditions for the starting the minimization process. $(-8,-2,8)$ for Example \ref{ex:ex1} and  $(0.05,0.2,0.4)$ for Example \ref{ex:ex4}. }
\label{fig:4DVAR}
\end{figure}
The value obtained for our MAP\index{MAP estimator} estimator depends crucially on the choice of initial condition in our minimization procedure. In particular, on the choices of starting point presented: for the three initializations 
shown, it is only when we start from $0.2$ are we able to find the global minimum of $\Iid(v_{0};y)$. 
By Theorem \ref{th22a} this global minimum corresponds to the maximum
of the posterior distribution\index{posterior distribution}, and we see that finding the MAP\index{MAP estimator} estimator
is a difficult task for this problem. Starting with the other two initial 
conditions displayed we converge to one of the many local minima of 
$\Iid(v_{0};y)$; these local minima are in fact regions of very low
probability, as we can see in Figure \ref{fig:smooth3}a. 
This illustrates the care required when computing 4DVAR\index{4DVAR} solutions
in cases where the forward problem exhibits
sensitivity to initial conditions.

Figure \ref{fig:w4DVAR} shows application of the w4DVAR\index{4DVAR!weak constraint}
method, or MAP estimator\index{MAP estimator} given by Theorem \ref{th22}, in the case of
the Example \ref{ex:ex3} with parameters set at 
$J=5, \gamma=\sigma=0.1$.  
In contrast to the previous example, this
is no longer a one-dimensional minimization problem: we are minimizing
$\Ii(v;y)$ given by \eqref{eq:dtf4} over $v \in \bbR^6$, given the
data $y \in \bbR^5.$  
The figure shows that there are at least $2$ local
minimizers for this problem, with $v^{(1)}$ closer to the truth than
$v^{(2)}$, and with $I(v^{(1)};y)$ considerably smaller that $I(v^{(2)};y).$ 
However $v^{(2)}$ has a larger basin of attraction for the optimization\index{optimization} software
used: many initial conditions lead to $v^{(2)}$, while fewer lead to $v^{(1)}$.  
Furthermore, whilst we believe that 
$v^{(1)}$ is the global minimizer, it is difficult to state this with 
certainty, even for this relatively low-dimensional model.
To get greater certainty an exhaustive and expensive search of
the six dimensional parameter space would be needed.

\begin{figure}[h]
\centering
{\includegraphics[scale=0.45]{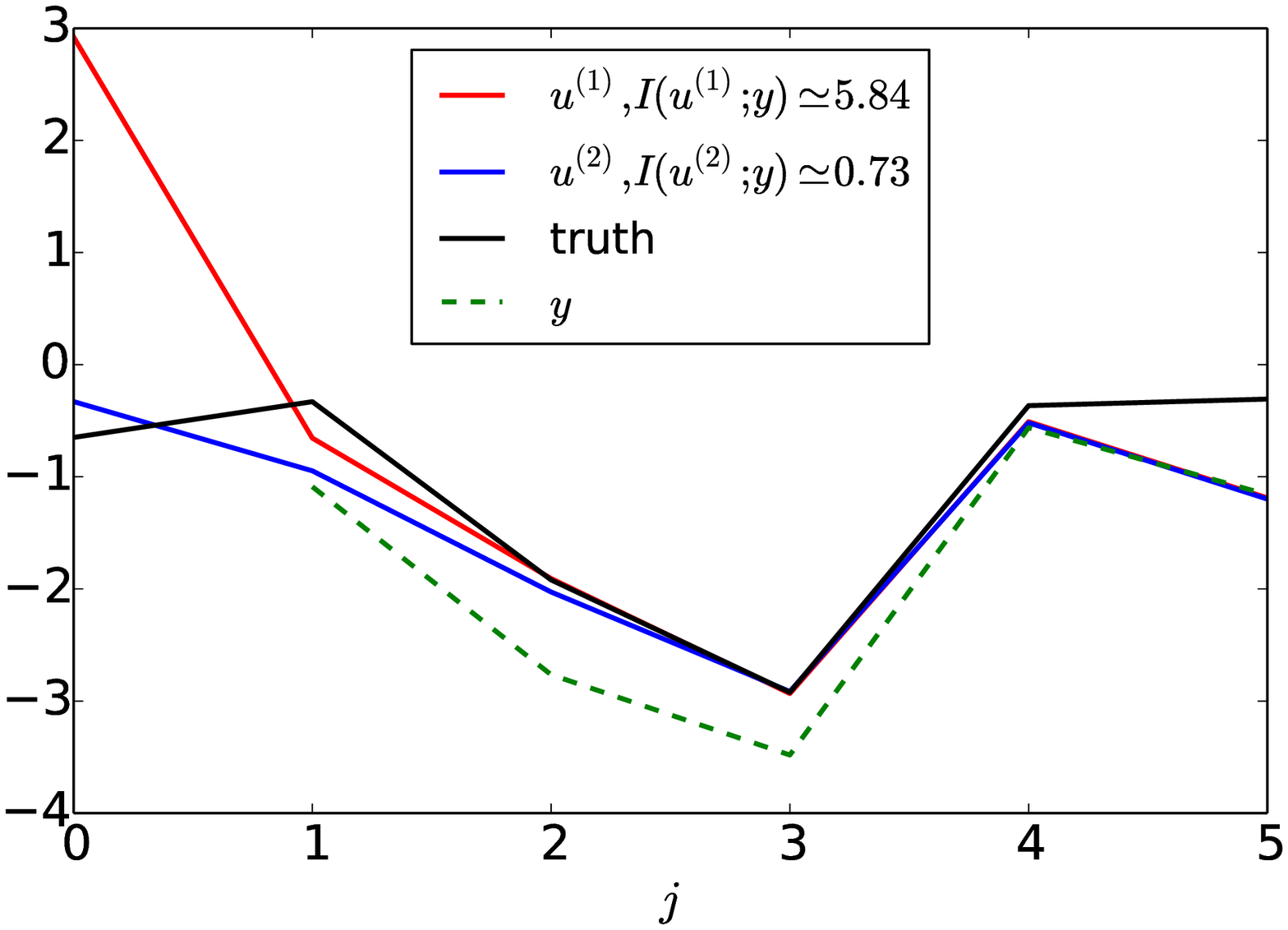}}
\caption{Weak constraint 4DVAR for $J=5, \gamma=\sigma=0.1$, illustrating
two local minimizers $v^{(1)}$ and $v^{(2)}$, see also {\tt p7.m} in section \ref{ssec:p6}.}
\label{fig:w4DVAR}
\end{figure}

\section{Bibliographic Notes}
\label{ssec:bs}

\begin{itemize}

\item The Kalman Smoother from subsection
\ref{ssec:kalkal} leads to a system of linear equations,
characterized in Theorem \ref{t:ks1}. These equations are of
block tridiagonal form, and may be solved by 
LU factorization\index{LU factorization}. The Kalman filter
corresponds to the LU sweep in this factorization, a fact that
was highlighted in \cite{BF63}.

\item Section \ref{ssec:mcmcm}.
Monte Carlo Markov Chain\index{MCMC}\index{Monte Carlo!Markov Chain}
methods have a long history,
initiated in the 1953 paper \cite{MRRT53} and then generalized
to an abstract formulation in the 1970 paper \cite{Has}. The
subject is overviewed from an algorithmic point of
view in \cite{Liu}. Theorem \ref{th21} is contained
in \cite{MT93}, and that reference also contains many 
other convergence theorems for Markov chains\index{Markov chain}; in particular we
note that it is often possible to increase substantially the class
of functions $\varphi$ to which the theorem applies by means
of Lyapunov function techniques, which control the tails of the
probability distribution.  The specific form of the
pCN-MCMC\index{pCN}\index{MCMC} method which we introduce here has been chosen to be
particularly effective in high dimensions; 
see \cite{CRSW12} for an overview, \cite{BRSV08} for
the introduction of pCN \index{pCN} and other methods for sampling probability
measures  in infinite dimensions, in the context of
conditioned diffusions, and \cite{CDS11} for the application to
a data assimilation problem.

The key point about pCN \index{pCN} methods is that the proposal is 
reversible with respect to an underlying Gaussian measure.
Even in the absence of data, if $\PPsi \ne 0$ then this Gaussian measure is far from the
measure governing the actual dynamics. In contrast, still in the absence of
data, this Gaussian
measure is {\em precisely} the measure governing the
noise and initial condition, giving the pCN \index{pCN} Dynamics Sampler
a natural advantage over the standard pCN\index{pCN} method. 
In particular, notice that the acceptance probability is now determined
only by the model-data misfit for the pCN \index{pCN} Dynamics Sampler, 
and does {\it not} have to account
for incorporation of the dynamics as it does in the original pCN \index{pCN} method; 
this typically improves the acceptance rate of the pCN \index{pCN} Dynamics Sampler over
the standard pCN \index{pCN} method. Therefore, this method may be 
preferable, particularly in the case of unstable dynamics.
The pCN \index{pCN} Dynamics Sampler was introduced in \cite{CDS11}
and further trialled in \cite{HLS13}; it shows
considerable promise.

The subject of MCMC methods is an enormous one to which we cannot
do justice in this brief presentation. There are two relevant
time-scales for the Markov chain: the burn-in time\index{burn-in time}
which determines the time to reach part of state-space where
most of the probability mass is concentrated, and the 
mixing time\index{mixing time} which determines the time taken to
fully explore the probability distribution. Our
brief overview would not be complete without a cursory discussion of 
convergence diagnostics \cite{gelman2013bayesian} which attempt to ensure
that the Markov chain is run long enough to have both burnt-in and mixed.
Whilst none of the diagnostics
are foolproof, there are many simple tests that can and
should be undertaken. The first is to simply study (as we have done in this
section) trace plots of quantities of interest (components of the
solution, acceptance probabilities) and the running average of 
these quantities of interest. More sophisticated diagnostics 
are also available. For example, 
comparison of the within-chain and between-chain 
variances of multiple chains beginning from over-dispersed initial conditions
is advocated in the works \cite{gelman1992inference, brooks1998general}.  
The authors of those works advise to apply a range of tests
based on comparing inferences from individual chains and 
a mixture of chains.  These and other more sophisticated
diagnostics are not considered further here, 
and the reader is referred to the cited works for further details and discussion.

\item Section \ref{ssec:vm}. Variational Methods, known as 
4DVAR\index{4DVAR}
in the meteorology community and widely used in practice, 
have the distinction, when compared
with the {\em ad hoc} non-Gaussian filters described in the next chapter
which are also widely used in practice in their EnKF and 3DVAR formulations,
of being well-founded statistically: they correspond
to the maximum {\em a posteriori} estimator (MAP estimator)
\index{MAP estimator} for the fully
Bayesian \index{Bayesian} posterior distribution\index{posterior distribution} on model state given data \cite{KS05}.
See \cite{Z97} and the references therein for a discussion of
the applied context; see \cite{dlsv13} for a more theoretical
presentation, including connections to the Onsager-Machlup
functional\index{Onsager-Machlup functional}
arising in the theory of diffusion processes \index{diffusion process}.
The European Centre for Medium-Range Weather Forecasts
(ECMWF) runs a weather prediction code based on spectral approximation
of continuum versions of Newton's balance laws, together with
various sub-grid scale models. 
Initialization of this prediction code 
is based on the use of 4DVAR like methods. The conjunction of this
computational forward model, together with the use of 4DVAR to
incorporate data, results in what is the best weather predictor, 
worldwide, according to a widely adopted
metric by which the prediction skill of forecasts is measured.
The subject of algorithms for optimization\index{optimization},
which of course underpins variational methods,
is vast and we have
not attempted to cover it here; we mention briefly that
many methods use first derivative information (for example
steepest descent methods) and second derivative information
(Newton methods); the reader is directed to 
\cite{nocedal1999numerical} for details.
Derivatives can also be useful in making MCMC proposals, leading the
Langevin the hybrid Monte Carlo\index{Monte Carlo!hybrid}
methods, for example; see \cite{RC99} and
the references therein.

\end{itemize}

\section{Exercises}
\label{ssec:ex3}

\begin{enumerate}

\item Consider the posterior distribution on the initial condition,
given by Theorem \ref{th112}, in the case of deterministic dynamics.
In the case of Example \ref{ex:ex4}, program {\tt p2.m} plots the 
prior and posterior distributions for this problem for data generated
with true initial condition $v_0=0.1$ Why is the posterior distribution
concentrating much closer to $0.9$ than to the true initial condition at $0.1$? 
Change the mean of the prior from $0.7$ to $0.3$; what do you observe regarding
the effect on the posterior. Explain what you observe. Illustrate
your findings with graphics.

\item Consider the posterior distribution on the initial condition, 
given by Theorem \ref{th112}, in the case of deterministic dynamics.
In the case of Example \ref{ex:ex4}, program {\tt p3.m} approximates the 
posterior distribution for this problem for data generated
with true initial condition $v_0=0.3$. Why is the posterior distribution
in this case approximately symmetric about $0.5$? What happens if the mean of the
prior is changed from $0.5$ to $0.1$? Explain what you observe. Illustrate
your findings with graphics. 

\item Consider the posterior distribution on the initial condition,
given by Theorem \ref{th112}, in the case of deterministic dynamics.
In the case of Example \ref{ex:ex4}, program {\tt p3.m} approximates the 
posterior distribution for this problem. Modify the program so that the prior
and data are the same as for the first exercise in this section.
Compare the approximation to the posterior obtained by use of
program {\tt p3.m} with the true posterior as computed by program
{\tt p2.m}. Carry out similar comparisons for different choices of
prior, ensuring that programs {\tt p2.m} and {\tt p3.m} share the same
prior and the same data.  In all cases experiment with the choice of the 
parameter $\beta$ in the proposal distribution within {\tt p3.m}, 
and determine its effect on the displayed approximation  of the
true posterior computed from {\tt p2.m}. 
Illustrate your findings with graphics.

\item Consider the posterior distribution on the initial condition,
given by Theorem \ref{th112}, in the case of deterministic dynamics.
In the case of Example \ref{ex:ex4}, program {\tt p3.m} approximates the
posterior distribution for this problem. Modify the program so that it
applies to Example \ref{ex:ex3}. Experiment with the choice of the 
parameter $J$, which determines the length of the Markov chain\index{Markov chain}
simulation, within {\tt p3.m}. Illustrate your findings with graphics.

\item Consider the posterior distribution on the signal,
given by Theorem \ref{th11}, in the case of stochastic dynamics \index{stochastic dynamics}.
In the case of Example \ref{ex:ex3}, program {\tt p4.m} approximates the
posterior distribution for this problem, using the Independence Dynamics
Sampler\index{Independence Dynamics Sampler}. 
Run this program for a range of values of $\gamma$. Report 
and explain the effect of $\gamma$ on the acceptance probability curves.

\item Consider the posterior distribution on the signal,
given by Theorem \ref{th11}, in the case of stochastic dynamics \index{stochastic dynamics}.
In the case of Example \ref{ex:ex3}, program {\tt p5.m} approximates the
posterior distribution for this problem, using the pCN \index{pCN} 
sampler. Run this program for a range of values of $\gamma$. Report 
and explain the effect of $\beta$ on the acceptance probability curves.

\item Consider the posterior distribution on the signal,
given by Theorem \ref{th11}, in the case of stochastic dynamics \index{stochastic dynamics}.
In the case of Example \ref{ex:ex3}, program {\tt p6.m} approximates the
posterior distribution for this problem, using the pCN \index{pCN} dynamics 
sampler. Run this program for a range of values of $\gamma$. Report 
and explain the effect of $\sigma$ and of $J$ on the acceptance probability curves.

\item Consider the MAP estimator for the posterior distribution on the signal,
given by Theorem \ref{th22} in the case of stochastic dynamics \index{stochastic dynamics}. 
Program {\tt p7.m} finds the MAP estimator for Example \ref{ex:ex3}.
Increase $J$ to $50$ and display your results graphically. Now
repeat your experiments for the values $\gamma=0.01,0.1$ and $10$ and
display and discuss your findings. Repeat the experiments using the
``truth'' as the initial condition for the minimization. What effect
does this have? Explain this effect. 

\item Prove Theorem \ref{th22a}. 

\item Consider application of the RWM \index{Random Walk Metropolis} proposal \eqref{eq:prop1},
applied in the case of stochastic dynamics \index{stochastic dynamics}. Find the form of the
Metropolis-Hastings\index{Metropolis-Hastings}
acceptance probability in this case. 

\item Consider the family of probability measures $\mu^{\epsilon}$ on $\bbR$
with Lebesgue density proportional to $\exp\bigl(-V^{\epsilon}(u)\bigr)$
with $V^{\epsilon}(u)$ given by \eqref{eq:V}. Prove that the family of
measure $\mu^{\epsilon}$ is locally Lipschitz in the Hellinger\index{metric!Hellinger} metric
and in the total variation metric\index{metric!total variation}.

\end{enumerate}

\graphicspath{{./figs/chapter3/}}
\chapter{Discrete Time: Filtering Algorithms}\label{sec:dtfa}\index{filtering}

In this chapter we describe various algorithms for the filtering \index{filtering}
problem. Recall from section \ref{ssec:fp}
that filtering \index{filtering} refers to the sequential update of
the probability distribution on the state given the data, as 
data is acquired, and that $Y_j=\{y_{\ell}\}_{\ell=1}^j$ denotes
the data\index{data!accumulated} accumulated up to time $j$. The filtering
update \index{filtering!update} from time $j$ to time $j+1$ may be broken into two steps: {\em
prediction}\index{prediction}
which is based on the equation for the state evolution,
using the Markov kernel \index{Markov kernel} for the stochastic 
or deterministic dynamical system \index{dynamical system!stochastic}
which maps $\bbP(v_{j}|Y_j)$ into $\bbP(v_{j+1}|Y_j)$;
and {\em analysis}\index{analysis} which incorporates data via Bayes' formula\index{Bayes' formula} and
maps $\bbP(v_{j+1}|Y_j)$ into $\bbP(v_{j+1}|Y_{j+1})$.
All but one of the algorithms we study (the optimal proposal\index{proposal!optimal} version of the particle filter) will also reflect these two steps.

We start in section \ref{ssec:tkf}
with the Kalman filter\index{Kalman filter} which provides
an exact algorithm to determine the filtering distribution \index{filtering!distribution}
for linear problems with additive Gaussian noise. 
Since the filtering distribution  \index{filtering!distribution}
is Gaussian in this case, the algorithm comprises an iteration
which maps the mean and covariance from time $j$ to time $j+1$.
In section \ref{ssec:ngf} we show how the idea of Kalman
filtering\index{Kalman filter} may be used to combine dynamical model with data
for nonlinear problems; in this case the posterior distribution\index{posterior distribution} is
not Gaussian, but the algorithms proceed by invoking a Gaussian
ansatz in the analysis step  \index{filtering!analysis step} of the filter. This results in
algorithms which do not provably approximate the true filtering
distribution\index{filtering!distribution} in general; 
in various forms they are, however, robust to
use in high dimension. In section \ref{ssec:pf} we introduce
the particle filter\index{particle filter} methodology which leads to provably accurate
estimates of the true filtering distribution  \index{filtering!distribution} but which is, in its current forms,
poorly behaved in high dimensions. 
The algorithms in sections \ref{ssec:tkf}--\ref{ssec:pf} are
concerned primarily with stochastic dynamics,\index{stochastic dynamics} but setting $\Sigma=0$
yields the corresponding algorithms for deterministic dynamics.\index{deterministic dynamics}
In section \ref{ssec:stab} we study the long time behaviour of some of the filtering algorithms  \index{filtering!algorithm} introduced in the previous sections. Finally, in section \ref{ssec:filti}  we present some numerical illustrations and conclude with bibliographic notes and exercises 
in sections \ref{ssec:dtnb3} and \ref{ex:filt}.

For clarity of exposition we again recall the form of the data assimilation
problem. The signal is governed by the model of equations \eqref{eq:dtf1}:
\begin{subequations}
\begin{eqnarray*}
&v_{j+1}=\PPsi(v_j)+\xi_j, \;j\in\Z^+, \\
&v_0\sim \G(m_0,C_0),
\end{eqnarray*}
\end{subequations}
where $\xi=\{\xi_j\}_{j\in\N}$ is an i.i.d. \index{i.i.d.} sequence, independent of $v_0$,
with $\xi_0\sim\G(0,\Sigma)$. The data is given by equation \eqref{eq:dtf2}:
\begin{equation*}
y_{j+1}=h(v_{j+1})+\eta_{j+1}, \;j\in\Z^+,
\end{equation*}
where $h:\R^n\to\R^m$ and $\eta=\{\eta_j\}_{j\in\Z^+}$ is an i.i.d. \index{i.i.d.} sequence,
independent of $(v_0,\xi)$, with $\eta_1\sim\G(0,\Gamma)$.

\section{Linear Gaussian Problems: The Kalman\index{Kalman filter} Filter}\label{ssec:tkf}
This algorithm provides a sequential method for updating the filtering 
distribution  \index{filtering!distribution} $\rp(v_j|Y_j)$ from time $j$ to time $j+1$, 
when $\PPsi$ and $h$ are linear maps. 
In this case the filtering distribution  \index{filtering!distribution} is Gaussian and it can be characterized entirely through its mean and covariance. 
To see this we note that the prediction   \index{filtering!prediction step} step preserves Gaussianity 
by Lemma \ref{lem:to}; the analysis step  \index{filtering!analysis step} preserves Gaussianity because it
is an application of Bayes' formula\index{Bayes' formula} \eqref{eq:bayes} and then Lemma \ref{l:qf}
establishes the required Gaussian property since the log pdf 
is quadratic in the unknown.

To be concrete we let 
\begin{equation}\label{eq:dtfa1}
\PPsi(v)=Mv, \;h(v)=Hv
\end{equation}
for matrices $M\in\R^{n\times n}, H\in\R^{m\times n}$. We assume that $m\leq n$ and $\mbox{Rank}(H)=m$. We let $(m_j, C_j)$ denote the mean and covariance of 
$v_j|Y_j$, noting that this entirely characterizes the random variable since
it is Gaussian. We let 
$(\widehat{m}_{j+1}, \widehat{C}_{j+1})$ denote the mean and covariance of 
$v_{j+1}|Y_j$, noting that this too completely characterizes the random
variable, since it is also Gaussian. 
We now derive the map $(m_j, C_j)\mapsto(m_{j+1}, C_{j+1})$, 
using the intermediate variables  $(\widehat{m}_{j+1}, \widehat{C}_{j+1})$ so that 
we may compute the prediction  \index{filtering!prediction step} and analysis   \index{filtering!analysis step} steps separately.
This gives the Kalman filter in a form where the update is expressed in
terms of precision\index{precision} rather than covariance\index{covariance}.

\begin{theorem}\label{th31}
Assume that $C_0, \Gamma, \Sigma>0$. Then $C_j>0$ for all $j\in\Z^+$ and 
\begin{subequations} \label{eq:kalman_update}
\begin{eqnarray}
C_{j+1}^{-1}&=&(MC_jM^T+\Sigma)^{-1}+H^T\Gamma^{-1}H,\\
C_{j+1}^{-1}m_{j+1}&=&(MC_jM^T+\Sigma)^{-1}Mm_j+H^T\Gamma^{-1}y_{j+1}.
\end{eqnarray}
\end{subequations}
\end{theorem}
\begin{proof}
We assume for the purposes of induction that $C_j>0$ noting that this
is true for $j=0$ by assumption.
The prediction  \index{filtering!prediction step} step is determined by (\ref{eq:dtf1}) in the case $\PPsi(\cdot)=M\cdot$:
\[v_{j+1}=Mv_j+\xi_j, \;\xi_j\sim\G(0,\Sigma).\]
From this it is clear that \[\bbE(v_{j+1}|Y_j)=\bbE(Mv_j|Y_j)+\bbE(\xi_j|Y_j).\]
Since $\xi_j$ is independent of $Y_j$ we have
\begin{equation}\label{eq:dtfa2}
\widehat{m}_{j+1}=Mm_j.
\end{equation}
Similarly 
\begin{align*}
\bbE\bigl((v_{j+1}-\widehat{m}_{j+1})\otimes(v_{j+1}-\widehat{m}_{j+1})|Y_j\bigr)&=\bbE
\bigl(M(v_j-m_j)\otimes M(v_j-m_j)|Y_j\bigr)+\bbE\bigl(\xi_j\otimes \xi_j|Y_j\bigr)\\&+\bbE\bigl(M(v_j-m_j)\otimes \xi_j|Y_j\bigr)+\bbE\bigl(\xi_j\otimes M(v_j-m_j)|Y_j\bigr).
\end{align*}
Again, since $\xi_j$ is independent of $Y_j$ and of $v_j$, we have 
\begin{align}\label{eq:dtfa3}
\widehat{C}_{j+1}&=M\bbE((v_j-m_j)\otimes (v_j-m_j)| Y_j)M^T +\Sigma\notag\\
&=MC_jM^T+\Sigma.
\end{align}
Note that $\widehat{C}_{j+1}>0$ because $C_j>0$ by the inductive hypothesis
and $\Sigma>0$ by assumption.

Now we consider the analysis step  \index{filtering!analysis step}. By (\ref{eq:dtf6}), which is just Bayes' formula\index{Bayes' formula}, and using Gaussianity, we have 
\begin{subequations}
\begin{eqnarray}
\exp\Bigl(-\frac12\bigl|v-m_{j+1}\bigr|_{C_{j+1}}^2\Bigr)
&\propto&\exp\Bigl(-\frac12\bigl|\Gamma^{-\frac12}(y_{j+1}-Hv)\bigr|^2-
\frac12\bigl|\hc_{j+1}^{-\frac12}(v-\widehat{m}_{j+1})\bigr|^2\Bigr)\\
&=&\exp\Bigl(-\frac12\bigl|y_{j+1}-Hv\bigr|_{\Gamma}^2-\frac12\bigl|v-\widehat{m}_{j+1}\bigr|_{\hc_{j+1}}^2\Bigr).
\end{eqnarray}
\label{eq:this}
\end{subequations}
Equating quadratic terms in $v$ gives, since $\Gamma>0$ by assumption, 
\begin{equation}\label{eq:dtfa4} C_{j+1}^{-1}=\widehat{C}_{j+1}^{-1}+H^T\Gamma^{-1}H\end{equation}
and equating linear terms in $v$ gives 
\footnote{We do not need to match the constant terms (with respect to $v$) since
the normalization constant \index{normalization constant} in Bayes theorem deals with matching these.}
\begin{equation}\label{eq:dtfa5}C_{j+1}^{-1}m_{j+1}=\widehat{C}_{j+1}^{-1}\widehat{m}_{j+1}+H^T\Gamma^{-1}y_{j+1}.\end{equation}
Substituting the expressions (\ref{eq:dtfa2}) and (\ref{eq:dtfa3}) for $(\widehat{m}_{j+1}, \widehat{C}_{j+1})$ gives the desired result. It remains to verify that $C_{j+1}>0$.  From \eqref{eq:dtfa4} it follows, since $\Gamma^{-1}>0$
by assumption and $\widehat{C}_{j+1}>0$ (proved above), that $C_{j+1}^{-1}>0$.
Hence $C_{j+1}>0$ and the induction is complete. 
\end{proof}

We may now reformulate the Kalman filter using covariances directly, rather
than using precisions\index{precision}.

\begin{corollary}\label{c32}
Under the assumptions of Theorem \ref{th31}, the 
formulae for the Kalman filter\index{Kalman filter} given there may be rewritten as follows:
\begin{align*}
d_{j+1}&=y_{j+1}-H\widehat{m}_{j+1},\\
S_{j+1}&=H\widehat{C}_{j+1}H^T+\Gamma,\\
K_{j+1}&=\widehat{C}_{j+1}H^TS_{j+1}^{-1},\\
m_{j+1}&=\widehat{m}_{j+1}+K_{j+1}d_{j+1},\\
C_{j+1}&=(I-K_{j+1}H)\widehat{C}_{j+1},
\end{align*}
with $(\widehat{m}_{j+1}, \widehat{C}_{j+1})$ given in (\ref{eq:dtfa2}), (\ref{eq:dtfa3}).
\end{corollary}
\begin{proof}
By (\ref{eq:dtfa4}) we have \[C_{j+1}^{-1}=\widehat{C}_{j+1}^{-1}+H^T\Gamma^{-1}H\]
and application of Lemma \ref{l34} below gives 
\begin{align*}C_{j+1}&=
\hc_{j+1}-\widehat{C}_{j+1}H^T(\Gamma+H\hc_{j+1}H^T)^{-1}H \hc_{j+1}\\
&=\Bigl(I-\widehat{C}_{j+1}H^T(\Gamma+H\hc_{j+1}H^T)^{-1}H\Bigr)\hc_{j+1}\\
&=(I-\hc_{j+1}H^TS_{j+1}^{-1}H)\hc_{j+1}\\
&=(I-K_{j+1}H)\hc_{j+1}
\end{align*}
as required. Then the identity (\ref{eq:dtfa5}) gives
\begin{align}\label{eq:dtfa6}
m_{j+1}&=C_{j+1}\hc_{j+1}^{-1}\hm_{j+1}+C_{j+1}H^T\Gamma^{-1}y_{j+1}\notag\\
&=(I-K_{j+1}H)\hm_{j+1}+C_{j+1}H^T\Gamma^{-1}y_{j+1}.
\end{align}
Now note that, again by (\ref{eq:dtfa4}),\[C_{j+1}(\hc_{j+1}^{-1}+H^T\Gamma^{-1}H)=I\]so that
\begin{align*}
C_{j+1}H^T\Gamma^{-1}H&=I-C_{j+1}\hc_{j+1}^{-1}\\
&=I-(I-K_{j+1}H)\\
&=K_{j+1}H.
\end{align*}
Since $H$ has rank $m$ we deduce that \[C_{j+1}H^T\Gamma^{-1}=K_{j+1}.\]
Hence (\ref{eq:dtfa6}) gives 
\begin{align*}
m_{j+1}&=(I-K_{j+1}H)\hm_{j+1}+K_{j+1}y_{j+1}
=\hm_{j+1}+K_{j+1}d_{j+1}
\end{align*}
as required.
\end{proof}
\begin{remark}\label{r33}
The key difference between the Kalman\index{Kalman filter} update formulae in Theorem \ref{th31} and in Corollary \ref{c32} is that, in the former matrix inversion takes place in the state space, with dimension $n$, whilst in the latter matrix inversion takes place in the data space, with dimension $m$. In many applications
$m\ll n$, as the observed subspace dimension is much less than
the state space dimension, and thus the formulation in 
Corollary \ref{c32} is frequently employed in practice. 
The quantity $d_{j+1}$ is referred to as the {\em innovation}\index{innovation}
at time-step $j+1$ and measures the mismatch of the predicted state from 
the data. The matrix $K_{j+1}$ is known as the 
{\em Kalman gain}\index{Kalman gain}.
\end{remark}

The following matrix identity was used to derive the formulation of the 
Kalman Filter\index{Kalman filter} in which inversion takes place in the data space.

\begin{lemma} {\bf Woodbury Matrix Identity}\index{Woodbury matrix identity} 
\label{l34} Let $A \in \bbR^{p \times p},
U \in \bbR^{p \times q}, C \in \bbR^{q \times q}$ and $V \in
\bbR^{q \times p}.$ If $A$ and $C$ are {positive} 
then $A+UCV$ is
invertible and
$$(A+UCV)^{-1}=A^{-1}-A^{-1}U\Bigl(C^{-1}+VA^{-1}U\Bigr)^{-1}VA^{-1}.$$
\end{lemma}

\section{Approximate Gaussian Filters}\label{ssec:ngf}

Here we introduce a family of methods, based on invoking
a minimization principle which underlies the Kalman filter\index{Kalman filter},
and which has a natural generalization to non-Gaussian problems.
The update equation for the Kalman filter\index{Kalman filter} mean, (\ref{eq:dtfa5}), can be  written as
$$m_{j+1}=\mbox{arg}\min_v\Jtd(v)$$
where
\begin{equation}\label{eq:dtfa7}
\Jtd(v):=\frac12|y_{j+1}-Hv|_{\Gamma}^2+\frac12|v-\hm_{j+1}|_{\hc_{j+1}}^2;
\end{equation}
here $\hm_{j+1}$ is calculated from \eqref{eq:dtfa2}, and $\hc_{j+1}$
is given by \eqref{eq:dtfa3}.
The fact that this minimization principle holds follows from  (\ref{eq:this}).
(We note that $\Jtd(\cdot)$ in fact depends on $j$, but we suppress explicit
reference to this dependence for notational simplicity.)

Whilst the Kalman filter\index{Kalman filter} itself is restricted to linear, Gaussian problems, the formulation via minimization generalizes to nonlinear problems. A natural generalization of (\ref{eq:dtfa7}) to the nonlinear case is to define 
\begin{equation}
\label{eq:J3}
\Jtd(v):=\frac12|y_{j+1}-h(v)|_{\Gamma}^2+\frac12|v-\hm_{j+1}|_{\hc_{j+1}}^2,
\end{equation}
where
$$\hm_{j+1}=\PPsi(m_j)+\xi_j,$$
and then to set \[m_{j+1}=\mbox{arg}\min_{v}\Jtd(v).\]
This provides a family of algorithms for updating the mean, differing
depending upon how $\hc_{j+1}$ is specified. In this section we
will consider several choices for this specification, and hence
several different algorithms. Notice that the minimization principle 
is very natural: it enforces a compromise between fitting the model
prediction  $\hm_{j+1}$ and the data $y_{j+1}$.

For simplicity we consider the case where observations \index{observations} are linear and $h(v)=Hv$ leading to the update algorithm 
$m_j\mapsto m_{j+1}$ defined by 
\begin{subequations}
\label{eq:dtfa8}
\begin{align}
\hm_{j+1}&=\PPsi(m_j)+\xi_j,\\
\Jtd(v)&=\frac12|y_{j+1}-Hv|_{\Gamma}^2+\frac12|v-\hm_{j+1})|_{\hc_{j+1}}^2,\\
m_{j+1}&=\mbox{arg}\min_v \Jtd(v).
\end{align}
\end{subequations}
This quadratic minimization problem is explicitly solvable and,
by the arguments used in deriving Corollary \ref{c32}, 
we deduce the following update formulae:
\begin{subequations}\label{eq:dtfa9}
\begin{align}
m_{j+1}&=(I-K_{j+1}H)\hm_{j+1}+K_{j+1}y_{j+1},\\
K_{j+1}&=\hc_{j+1}H^TS_{j+1}^{-1},\\
S_{j+1}&=H\hc_{j+1}H^T+\Gamma.
\end{align}
\end{subequations}

The next three subsections each correspond to algorithms derived in this way,
namely by minimizing $\Jtd(v)$, but corresponding to different choices 
of the model covariance $\hc_{j+1}$. We also note that in the  
first two of these subsections we choose $\xi_j \equiv 0$ in
equation (\ref{eq:dtfa8}a) so that
the prediction is made by the noise-free dynamical model; however is
not a necessary choice and whilst it is natural for the extended Kalman
filter\index{Kalman filter!extended} for 3DVAR\index{3DVAR} including
random effects in the  (\ref{eq:dtfa8}a) is also reasonable in some
settings. Likewise the ensemble Kalman filter\index{Kalman filter!ensemble}
can also be implemented with noise-free prediction models.

We refer to these three algorithms collectively
as {\bf approximate Gaussian filters}\index{filter!approximate Gaussian}. 
This is because they invoke
a Gaussian approximation when updating the estimate of the signal
via (\ref{eq:dtfa8}b). Specifically this update is the correct update
for the mean if the assumption that 
$\bbP(v_{j+1}|Y_j)=N\bigl(\hm_{j+1}),\hc_{j+1})$
is invoked for the prediction step  \index{filtering!prediction step}.
In general the approximation implied by this assumption will not be
a good one and this can invalidate the statistical accuracy of the
resulting algorithms. However the resulting algorithms may still have
desirable properties in terms of signal estimation; in subsection
\ref{sssec:3s} we will demonstrate that this is indeed so. 

\subsection{3DVAR}\index{3DVAR}\label{ssec:3}
This algorithm is derived from (\ref{eq:dtfa9}) by simply fixing the model covariance $\hc_{j+1}\equiv\hc$ for all $j$.
Thus we obtain 
\begin{subequations}\label{eq:dtfa10}
\begin{align}
\hm_{j+1}&=\PPsi(m_j),\\
m_{j+1}&=(I-KH)\hm_{j+1}+Ky_{j+1},\\
K&=\hc H^T S^{-1}, \quad S=H\hc H^T+\Gamma.
\end{align}
\end{subequations}
The nomenclature 3DVAR refers to the fact that the method is
variational (it is based on the minimization principle underlying
all of the approximate Gaussian methods), and it works sequentially
at each fixed time $j$; as such the minimization, when applied to
practical physical problems, is over three spatial dimensions.
This should be contrasted with 4DVAR\index{4DVAR} 
which involves a minimization over all spatial dimensions, as well
as time -- four dimensions in all.

We now describe two methodologies which generalize 3DVAR\index{3DVAR} by
employing model covariances which evolve from step $j$ to step $j+1$:
the extended and ensemble Kalman filters\index{Kalman filter!ensemble}\index{Kalman filter!extended}. We present both methods in
basic form but conclude the section with some discussion of
methods widely used in practice to improve their practical performance.

\subsection{Extended Kalman Filter}\label{ssec:exkf}
The idea of the extended Kalman filter\index{Kalman filter!extended} 
(ExKF\index{Kalman filter!extended}\index{Kalman filter!ExKF}) is to propagate covariances according to the linearization of (\ref{eq:dtf1}), and propagate the mean, using (\ref{eq:dtf11}). Thus we obtain, from modification of Corollary \ref{c32} and (\ref{eq:dtfa2}), (\ref{eq:dtfa3})
\[\mbox{Prediction}\;\left\{\begin{array}{ll} \hm_{j+1}&=\PPsi(m_j), \vspace{4pt}\\ \hc_{j+1}&=D\PPsi(m_j)C_jD\PPsi(m_j)^T+\Sigma.
                                    \end{array}\right.
                                     \] 
\[\quad\;\;\;\;\;\mbox{Analysis}\;\;\;\;\left\{\begin{array}{llll} S_{j+1}&=H\hc_{j+1}H^T+\Gamma, \vspace{4pt}\\K_{j+1}&=\hc_{j+1}H^TS_{j+1}^{-1},
\vspace{4pt}\\m_{j+1}&=(I-K_{j+1}H)\hm_{j+1}+K_{j+1}y_{j+1},\vspace{4pt}
\\C_{j+1}&=(I-K_{j+1}H)\hc_{j+1}.
                                    \end{array}\right.
                                     \]

\subsection{Ensemble Kalman Filter}\label{ssec:enkf}

The ensemble Kalman filter\index{Kalman filter!ensemble} (EnKF\index{Kalman filter!ensemble}\index{Kalman filter!EnKF}) generalizes the
idea of approximate Gaussian filters\index{filter!approximate Gaussian} 
in a significant
way: rather than using the minimization procedure
(\ref{eq:dtfa8}) to update a single estimate of the {\em mean},
it is used to generate an {\em ensemble} of particles
which all satisfy the model/data compromise inherent
in the minimization; the mean and covariance used in the minimization
are then estimated using this ensemble, thereby adding further
coupling to the particles, in addition to that introduced
by the data.

The EnKF\index{Kalman filter!ensemble} is executed in a variety of ways and 
we start by describing one of these, the perturbed observation\index{perturbed observations} EnKF\index{Kalman filter!ensemble}:

\[\;\;\;\quad\quad\mbox{Prediction}\;\left\{\begin{array}{lll}  
\widehat{v}_{j+1}^{(n)}&=\PPsi(v_{j}^{(n)}) + \xi_j^{(n)}, \;n=1,...,N,\vspace{4pt}\\
\hm_{j+1}&=\frac1N\sum_{n=1}^N\widehat{v}_{j+1}^{(n)},\vspace{4pt}\\
\hc_{j+1}&=\frac1{N-1}\sum_{n=1}^N(\widehat{v}_{j+1}^{(n)}-\hm_{j+1})(\widehat{v}_{j+1}^{(n)}-\hm_{j+1})^T.

 \end{array}\right.
                                   \] 
\[\mbox{Analysis}\;\;\;\;\left\{\begin{array}{llll} S_{j+1}&=H\hc_{j+1}H^T+\Gamma, \vspace{4pt}\\ K_{j+1}&=\hc_{j+1}H^TS_{j+1}^{-1},
\vspace{4pt}\\v_{j+1}^{(n)}&=(I-K_{j+1}H)\widehat{v}_{j+1}^{(n)}+K_{j+1}y_{j+1}^{(n)}, \;n=1,...,N, \vspace{4pt}\\y_{j+1}^{(n)}&=y_{j+1}+\eta_{j+1}^{(n)}, \;n=1,...,N.                                    \end{array}\right.
                                     \] 
Here $\eta_j^{(n)}$ are i.i.d. \index{i.i.d.} draws from $\G(0,\Gamma)$ and $\xi_j^{(n)}$ are i.i.d. draws from $\G(0,\Sigma)$.  Perturbed observation refers to the fact that each particle sees an observation perturbed by an independent draw from $\G(0,\Gamma)$.  This procedure gives the Kalman Filter\index{Kalman filter} in the linear case in the limit of infinite ensemble.  {Even though the algorithm is motivated through
our general approximate Gaussian filters\index{filter!approximate Gaussian} 
framework, notice that
the ensemble is not prescribed to be Gaussian. Indeed it evolves under
the full nonlinear dynamics in the prediction step. This fact, together
with the fact that covariance matrices are not propagated explicitly,
other than through the empirical properties of the ensemble,
has made the algorithm very appealing to practitioners. }

Another way to motivate the preceding algorithm 
is to introduce the family of cost functions
\begin{equation}
\label{eq:minn}
\Jtdn(v):=\frac12|y^{(n)}_{j+1}-Hv|_{\Gamma}^2+\frac12|v-\hv^{(n)}_{j+1}|_{\hc_{j+1}}^2.
\end{equation}
The analysis step\index{filtering!analysis step} proceeds
to determine the ensemble $\{v_{j+1}^{(n)}\}_{n=1}^N$ 
by minimizing $\Jtdn$ with $n=1,\cdots, N.$ 
The set $\{\hv_{j+1}^{(n)}\}_{n=1}^N$ is found from running the 
prediction step\index{filtering!prediction step}
using the fully nonlinear dynamics. These minimization
problems are coupled through $\hc_{j+1}$ which depends on the
entire set of $\{\hv_{j}^{(n)}\}_{n=1}^N.$ 
The algorithm thus provides update rules of the form
\begin{equation}
\label{eq:etadpu2}
\{ v_j^{(n)}\}_{n=1}^N \mapsto \{ \hv_{j+1}^{(n)}\}_{n=1}^{N},\quad\quad\quad \{ \hv_{j+1}^{(n)}\}_{n=1}^N \mapsto \{ v_{j+1}^{(n)}\}_{n=1}^{N},
\end{equation}
defining approximations of the prediction  \index{filtering!prediction step} and analysis steps  \index{filtering!analysis step} respectively.

It is then natural to think of the algorithm making the approximations
\begin{equation}
\label{eq:diraca}
\mu_j \approx \mu_j^N=\frac{1}{N}\sum_{n=1}^{N} \delta_{v_{j}^{(n)}},\quad\quad\quad
\hmu_{j+1} \approx \mu_j^N=\frac{1}{N}\sum_{n=1}^{N} \delta_{\hv_{j+1}^{(n)}}.
\end{equation}
Thus we have a form of Monte Carlo\index{Monte Carlo!approximation}
approximation of the distribution of interest.
However, except for linear problems, the approximations
given do not, in general, converge to the true distributions
$\mu_j$ and $\hmu_j$ as $N \to \infty.$

\subsection{Square Root Ensemble Kalman Filters}\index{Kalman filter!ensemble square root}
\label{ssec:sqrt}

We now describe another popular variant of the EnKF\index{Kalman filter!ensemble}.
The idea of this variant is to define the analysis
step in such a way that an ensemble of particles
is produced whose empirical covariance {\em exactly}
satisfies the Kalman identity 
\begin{equation}
\label{eq:K}
C_{j+1}=(I-K_{j+1}H)\hc_{j+1}
\end{equation}
which relates the covariances in the analysis step
to those in the prediction step.\index{filtering!prediction step}
This is done by mapping the mean of the predicted
ensemble according to the standard Kalman update,
and introducing a linear deterministic transformation of the differences
between the particle positions and their mean to enforce
(\ref{eq:K}). 
Doing so eliminates a sampling error 
inherent in the perturbed observation\index{perturbed observations} approach.
The resulting algorithm has the following form:

\[\;\;\;\quad\quad\mbox{Prediction}\;\left\{\begin{array}{lll}  
\widehat{v}_{j+1}^{(n)}&=\PPsi(v_{j}^{(n)}) + \xi_j^{(n)}, \;n=1,...,N,\vspace{4pt}\\
\hm_{j+1}&=\frac1N\sum_{n=1}^N\widehat{v}_{j+1}^{(n)},\vspace{4pt}\\
\hc_{j+1}&=\frac1{N-1}\sum_{n=1}^N(\widehat{v}_{j+1}^{(n)}-\hm_{j+1})(\widehat{v}_{j+1}^{(n)}-\hm_{j+1})^T,

 \end{array}\right.
                                   \] 
\[\mbox{Analysis}\;\;\;\;\left\{\begin{array}{llll} S_{j+1}&=H\hc_{j+1}H^T+\Gamma, \vspace{4pt}\\ K_{j+1}&=\hc_{j+1}H^TS_{j+1}^{-1},
\vspace{4pt}\\m_{j+1}&=(I-K_{j+1}H)\widehat{m}_{j+1}+K_{j+1}y_{j+1},\vspace{4pt}\\v_{j+1}^{(n)}&=m_{j+1}+\zeta_{j+1}^{(n)}.                                    \end{array}\right.
                                     \] 
Here the $\{\zeta_{j+1}^{(n)}\}_{n=1}^N$ 
are designed{ to 
have sample covariance $C_{j+1}=(I-K_{j+1}H)\hc_{j+1}.$}
There are several ways to do this and we now describe one of
them, referred to as the ensemble transform Kalman filter\index{Kalman filter!ensemble transform} (ETKF)\index{Kalman filter!ETKF}.

If we define 
$$
\widehat{X}_{j+1}=\frac{1}{\sqrt{N-1}}\left [ \widehat{v}_{j+1}^{(1)}-\widehat{m}_{j+1}, \dots, \widehat{v}_{j+1}^{(N)}-\widehat{m}_{j+1} \right ]
$$   
then $\widehat{C}_{j+1}=\widehat{X}_{j+1}\widehat{X}_{j+1}^T$.  
We now seek a transformation $T_{j+1}$ so that, if 
$X_{j+1}=\widehat{X}_{j+1} T_{j+1}^{\frac12}$, then
\begin{equation}
\label{trans}
C_{j+1}:={X}_{j+1}X_{j+1}^T=(I-K_{j+1}H)\hc_{j+1}.
\end{equation}
Note that the $X_{j+1}$ (resp. the $\widehat{X}_{j+1}$)
correspond to Cholesky
factors of the matrices $C_{j+1}$ (resp. $\hc_{j+1}$)
respectively. We may now define the 
$\{\zeta_{j+1}^{(n)}\}_{n=1}^N$ by 
$$
{X}_{j+1}=\frac{1}{\sqrt{N-1}}\left [ \zeta_{j+1}^{(1)}, \dots, \zeta_{j+1}^{(N)} \right ].
$$   
We now demonstrate how to find an appropriate
transformation $T_{j+1}.$ 
We assume that $T_{j+1}$ is symmetric and positive-definite and
the standard matrix square-root is employed. Choosing
$$
{T}_{j+1} = \left [  I + (H\widehat{X}_{j+1})^T \Gamma^{-1} (H \widehat{X}_{j+1}) \right ]^{-1}
$$
we see that
\[
\begin{array}{ll}
{X}_{j+1}X_{j+1}^T &= \widehat{X}_{j+1} T_{j+1} \widehat{X}_{j+1}^T \\
 &=   \widehat{X}_{j+1} \left [  I + (H\widehat{X}_{j+1})^T \Gamma^{-1} (H \widehat{X}_{j+1}) \right ]^{-1} \widehat{X}_{j+1}^T \\
&=  \widehat{X}_{j+1} \left \{ I  - (H\widehat{X}_{j+1})^T \left [ (H\widehat{X}_{j+1})(H\widehat{X}_{j+1})^T+ \Gamma \right ]^{-1} (H\widehat{X}_{j+1}) \right \} \widehat{X}_{j+1}^T \\
&= (I-K_{j+1}H)\hc_{j+1}
\end{array}
\]
as required, where the transformation between the second and third lines is justified by Lemma \ref{l34}.
It is important to ensure that
${\bf 1}$, the vector of all ones, 
is an eigenvector of the transformation $T_{j+1}$, 
and hence of $T_{j+1}^{\frac12}$, so that the mean 
of the ensemble is preserved. This is guaranteed 
by $T_{j+1}$ as defined.

\section{The Particle Filter}
\label{ssec:pf}

In this section we introduce an important class of filtering methods
known as {\em particle filters\index{particle filter}}.  In contrast to the filters introduced 
in the preceding section, the particle filter\index{particle filter} can be {\em proved} to 
reproduce the true posterior filtering distribution  \index{filtering!distribution} in the large particle 
limit and, as such, has a privileged places amongst all the filters 
introduced in this book. We will describe the method in its basic form -- 
the {\em bootstrap filter}  \index{filter!bootstrap}-- and then give a proof of convergence.  
It is important to appreciate that the form of particle filter\index{particle filter} introduced
here is far from state-of-the-art and that far more sophisticated versions
are used in practical applications. Nonetheless, despite this sophistication,
particle filter\index{particle filter}s do not perform well in applications such as those arising
in geophysical applications of data assimilation,\index{data assimilation} 
because the data in those
applications places very strong constraints on particle locations, making
efficient algorithms very hard to design. It is
for this reason that we have introduced particle filter\index{particle filter}s after the approximate 
Gaussian filters\index{filter!approximate Gaussian} 
introduced in the preceding section. The filters in the 
preceding section tend to be more robust to data specifications. However they 
do all rely on the invocation of ad hoc Gaussian assumptions in their 
derivation and hence do not provably produce the correct posterior filtering 
distribution, \index{filtering!distribution} notwithstanding their ability, in partially observed
small noise scenarios, to correctly identify the signal itself, as in
Theorem \ref{t35}. Because it can provably reproduce the correct filtering
distribution, \index{filtering!distribution} 
the particle filter\index{particle filter} thus plays an important role, conceptually,
even though it is not, in current form, a practical algorithm in geophysical 
applications. With further improvements it may, in time, form the basis
for practical algorithms in geophysical applications.

\subsection{The Basic Approximation Scheme}
\label{ssec:dirac}

All probability measures which possess density with respect to
Lebesgue measure can be approximated by a finite convex combination
of Dirac probability measures; an example of this is the 
{\bf Monte Carlo sampling}\index{Monte Carlo!sampling} 
idea that we described at the start of Chapter \ref{sec:dtsa}, and also underlies
the ensemble Kalman filter\index{Kalman filter!ensemble} of 
subsection \ref{ssec:enkf}. In practice 
the idea of approximation by a convex combination of probability
measures requires the determination of the locations and weights 
associated with these Dirac measures.
Particle filters are sequential algorithms which 
use this idea to approximate the true filtering distribution  \index{filtering!distribution} $\rp(v_j|Y_j).$

Basic Monte Carlo\index{Monte Carlo}, as in \eqref{eq:mci}, and the ensemble Kalman
filter\index{Kalman filter!ensemble}, as in \eqref{eq:diraca}, correspond to approximation by
equal weights.
Recall $\mu_j$, the probability measure on $\bbR^n$ corresponding
to the density $\rp(v_j|Y_j)$, and $\hmu_{j+1}$, the probability 
measure on $\bbR^n$ corresponding to the density $\rp(v_{j+1}|Y_j).$
The basic form of the particle filter\index{particle filter}
proceeds by allowing the weights to vary
and by finding $N$-particle Dirac measure approximations of the form 
\begin{equation}
\mu_j \approx 
\mu_j^{N}:= 
\sum_{n=1}^{N} w_j^{(n)} \delta_{v_j^{(n)}}, \quad\quad\quad
\hmu_{j+1} \approx 
\hmu_{j+1}^{N}:= 
\sum_{n=1}^{N} \hw_{j+1}^{(n)} \delta_{\hv_{j+1}^{(n)}}.
\label{eq:emp_filt}
\end{equation}
The weights must sum to one.
The approximate distribution $\mu_j^N$ is completely
defined by particle positions $v_j^{(n)}$ and weights $w_j^{(n)}$, and the approximate distribution $\hmu_{j+1}^N$ is completely
defined by particle positions $\hv_{j+1}^{(n)}$ and weights $\hw_{j+1}^{(n)}.$ 
Thus the objective of the method is to
find update rules 
\begin{equation}
\label{eq:etadpu}
\{ v_j^{(n)}, w_j^{(n)}\}_{n=1}^N \mapsto \{ \hv_{j+1}^{(n)}, \hw_{j+1}^{(n)}\}_{n=1}^{N},\quad\quad\quad \{ \hv_{j+1}^{(n)}, \hw_{j+1}^{(n)}\}_{n=1}^N \mapsto \{ v_{j+1}^{(n)}, w_{j+1}^{(n)}\}_{n=1}^{N}
\end{equation}
defining the prediction and analysis approximations respectively;
compare this with \eqref{eq:etadpu2} for the EnKF where the particle
weights are uniform and only the positions are updated.
Defining the updates for the particle filter
may be achieved by an application of sampling,
for the prediction step  \index{filtering!prediction step}, and of Bayesian \index{Bayesian} probability, for the analysis
step.

Recall the prediction and analysis formulae from \eqref{eq:dtf5}
and \eqref{eq:dtf6} which can be summarized as
\begin{subequations}
\label{eq:dt56}
\begin{align}
\label{eq:dtf56a}
\rp(v_{j+1}|Y_j)&=\int_{\R^n}\rp(v_{j+1}|v_j)\rp(v_j|Y_j)dv_j,\\
\rp(v_{j+1}|Y_{j+1})&=\frac{\rp(y_{j+1}|v_{j+1})\rp(v_{j+1}|Y_j)}{\rp(y_{j+1}|Y_j)}.
\label{eq:dtf56b}\end{align}
\end{subequations}
We may rewrite \eqref{eq:dt56} as 
{
\begin{subequations}
\label{eq:dtf66}
\begin{align}
\label{eq:dtf66a}
\hmu_{j+1}(\cdot) &= (P\mu_j)(\cdot) := \int_{\R^n}\rp(\cdot|v_j) \mu_j(dv_j)\\
\frac{d\mu_{j+1}}{d\hmu_{j+1}}(v_{j+1})&=\frac{\rp(y_{j+1}|v_{j+1})}{\rp(y_{j+1}|Y_j)}.
\label{eq:dtf66b}\end{align}
\end{subequations}
}
Writing the update formulae this way is important for us because 
they then make sense in the absence of Lebesgue densities; in particular 
we can use them in situations where Dirac masses \index{Dirac mass}  appear, as they
do in our approximate probability measures.  The formula \eqref{eq:dtf66b}
for the {\em density} or 
{\em Radon-Nikodym derivative}\index{Radon-Nikodym derivative} 
of $\mu_{j+1}$ with respect to that of $\hmu_{j+1}$ 
has a straightforward interpretation: the
righthand-side quantifies how to reweight expectations under
$\hmu_{j+1}$ so that they become expectations under $\mu_{j+1}$. 
To be concrete we may write
$$\bbE^{\mu_{j+1}} \varphi(v_{j+1})=\bbE^{\hmu_{j+1}}
\Bigl( \frac{d\mu_{j+1}}{d\hmu_{j+1}}(v_{j+1})\varphi(v_{j+1})\Bigr).$$

\subsection{Sequential Importance Resampling}
\label{ssec:bsf}

The simplest particle filter\index{particle filter}, which is
based on {\em sequential importance resampling}\index{sequential importance 
resampling} is now described. We start by assuming that we have an
approximation $\mu_j^N$ given by \eqref{eq:emp_filt} and
explain how to evolve the weights $\{ v_j^{(n)}, w_j^{(n)}\}_{n=1}^N$
into $\{ v_j^{(n+1)}, w_j^{(n+1)}\}_{n=1}^N$,
via $\{ \hv_{j+1}^{(n)}, \hw_{j+1}^{(n)}\}_{n=1}^N$ as in \eqref{eq:etadpu}. 

\vspace{0.1in}

\noindent{\bf Prediction} In this step we approximate the prediction
phase of the Markov chain\index{Markov chain}. To do this we simply draw
$\hv_{j+1}^{(n)}$ from the kernel $p$ of the Markov chain\index{Markov chain}
(\ref{eq:dtf1}a) started from $v_{j}^{(n)}$. Thus the
relevant kernel is $p(v_j,v_{v+1})=\rp(v_{j+1}|v_j).$
We then have $\hv_{j+1}^{(n)} \sim p(v_j^{(n)},\cdot).$
We leave the weights of the approximation unchanged so that
$\hw_{j+1}^{(n)}=w_j^{(n)}$. From these new particles and (in fact
unchanged) weights we have the particle approximation
\begin{equation}
\hmu_{j+1}^N=\sum_{n=1}^{N} w_j^{(n)} \delta_{\hv_{j+1}^{(n)}}.
\label{eq:emp_filt_ratz}
\end{equation}

\vspace{0.1in}

\noindent{\bf Analysis} In this step we approximate the incorporation
of data via Bayes' formula\index{Bayes' formula}. Define $g_j(v)$ by 
\begin{equation}
\label{eq:BRG}
g_j(v_{j+1}) \propto \rp(y_{j+1}|v_{j+1}),
\end{equation}
where the constant of proportionality is, for example, the normalization
for the Gaussian, and is hence independent of both $y_{j+1}$ and $v_{j+1}.$
We now apply Bayes' formula \index{Bayes' formula} in
the form \eqref{eq:dtf66b}. Thus we obtain 
{\begin{equation}
\mu_{j+1}^N=\sum_{n=1}^{N} w_{j+1}^{(n)} \delta_{\hv_{j+1}^{(n)}}
\label{eq:emp_filt2}
\end{equation}
where 
\begin{equation}
\label{eq:waits}
w_{j+1}^{(n)}=\tw_{j+1}^{(n)}\Bigl{/}\left(\sum_{n=1}^{N} \tw_{j+1}^{(n)}\right ),
\quad \quad \tw_{j+1}^{(n)} = g_j\bigl(\hv_{j+1}^{(n)}\bigr)w_{j}^{(n)}.
\end{equation}
The first equation in the preceding is required for normalization.
Thus in this step we do not change the particle positions, but we reweight
them.

\vspace{0.1in}

\noindent{\bf Resampling} 
The algorithm as described is deficient in two regards, both of which
can be dealt with by introducing a re-sampling step into the algorithm.
Firstly, the initial measure $\mu_0$ for the true filtering distribution  \index{filtering!distribution}
will not typically be made up of a combination of Dirac measures.
Secondly, the method can perform poorly if one of the particle weights
approaches $1$ (and then all others approach $0$). The effect of the
first can be dealt with by sampling the initial measure and approximating
it by an equally weighted (by $N^{-1}$) sum of Dirac measures at the samples.
The second can be ameliorated by drawing a set of $N$ particles
from the measure \eqref{eq:emp_filt2} and assigning weight $N^{-1}$ to
each; this has the effect of multiplying particles with high weights
and killing particles with low weights. 

\vspace{0.1in}

Putting together the three preceding steps leads to 
the following algorithm; for notational convenience we use 
$Y_0$ to denote the empty vector (no observations at the start): 

\begin{enumerate}
\item[1.] Set $j=0$ and $\mu_0^N(dv_0)=\mu_0(dv_0)$.
\item[2.] Draw $v_{j}^{(n)} \sim \mu_j^N$, $n=1,\dots,N$. 
\item[3.] Set $w_{j}^{(n)}=1/N$, $n=1,\dots,N$; redefine $\mu_j^{N}:=
\sum_{n=1}^{N} w_j^{(n)} \delta_{v_j^{(n)}}$.
\item[4.] Draw $\hv_{j+1}^{(n)} \sim p(v_{j}^{(n)}|\cdot).$
\item[5.] Define $w_{j+1}^{(n)}$ by \eqref{eq:waits}
and $\mu_{j+1}^{N}:= \sum_{n=1}^{N} w_{j+1}^{(n)} \delta_{\hv_{j+1}^{(n)}}$. 
\item[6.] $j+1\to j$.
\item[7.] Go to step 2.
\end{enumerate}

This algorithm is conceptually intuitive, proposing that each 
particle moves according to the dynamics of the underlying model 
itself, and is then re-weighted according to the likelihood \index{likelihood} of the
proposed  particle, i.e. according to the data.
This sequential importance resampling filter is also
sometimes termed the {\bf bootstrap filter}\index{filter!bootstrap}.
We will comment on important improvements to this basic algorithm
in the the following section and in the bibliographic notes. 
Here we prove convergence of this basic
method, as the number of particles goes to infinity, 
thereby demonstrating the potential power of the bootstrap
filter and more sophisticated variants on it.

Recall that, by \eqref{eq:summary}, the true filtering distribution  \index{filtering!distribution} 
simply satisfies the iteration
\begin{equation}
\label{eq:pfn}
\mu_{j+1}=L_j P \mu_j, \quad \mu_0=N(m_0,C_0),
\end{equation}
where $P$ corresponds to moving a point currently at $v$ according to the Markov
kernel $p(\cdot|v)$ describing the dynamics given by (\ref{eq:dtf1}a)
and $L_j$ denotes the application of Bayes' formula \index{Bayes' formula} 
with likelihood \index{likelihood} proportional to
$g_j(\cdot)$ given by \eqref{eq:BRG}. Recall also 
the sampling operator $S^N$ defined by \eqref{eq:mci}.
It is then instructive to write the particle filtering algorithm 
which approximates \eqref{eq:pfn} in the following form:
\begin{equation}
\label{eq:pfna}
\mu_{j+1}^N=L_j S^N P \mu_j^N, \quad \mu_0^N=\mu_0.
\end{equation}
There is a slight trickery here in writing application of the sampling
$S^N$ {\em after} application of $P$, but some reflection shows that
this is well-justified: applying $P$ followed by $S^N$ can be shown,
by first conditioning on the initial point and sampling with respect
to $P$, and then sampling over the distribution of the initial point,
to be the algorithm as defined. 

Comparison of \eqref{eq:pfn} and \eqref{eq:pfna} shows that
analyzing the particle filter\index{particle filter} requires estimation 
of the error induced by application of $S^N$ (the {\em resampling error})
together with estimation of the rate of accumulation of this error
in time under the application of $L_j$ and $P$. We now build the tools
to allow us to do this.  The operators $L_j, P$ and $S^N$ map the space 
$\cP(\bbR^n)$ of probability measures on $\bbR^n$ into itself according 
to the following: 
\begin{subequations}
\begin{align}
(L_j \mu) (dv) &= \frac{g_j(v) \mu(dv)}{\int_{\bbR^n} g_j(v) \mu(dv)}, \\
(P \mu) (dv) &= \int_{\bbR^n} p(v',dv) \mu(dv'), \\
(S^N \mu)(dv) &= \frac{1}{N} \sum_{n=1}^N \delta_{v^{(n)}} (dv) , \quad v^{(n)} \sim \mu ~~~ {\rm i.i.d.}\,. 
\end{align}
\end{subequations}
Notice that both $L_j$ and $P$ are deterministic maps, whilst $S^N$
is random. 
Let $\mu=\mu_\omega$ denote, for each $\omega$, an element of $\cP(\bbR^n)$.
If we then assume that $\omega$ is a random variable describing the
randomness required to define the sampling operator $S^N$, and let $\bbE^{\omega}$
denote expectation over $\omega$, then we may define a 
{``root mean square''}
distance $d(\cdot,\cdot)$
between two random probability measures $\mu_\omega, \nu_\omega,$ as follows:
$$
d(\mu,\nu) = {\rm sup}_{|f|_{\infty} \leq 1} \sqrt{ \bbE^{\omega} |\mu(f)-\nu(f)|^2 }.
$$
Here we have used the convention that $\mu(f) = \int_{\bbR^n} f(v)\mu(dv)$ 
for measurable $f:\bbR^n \to \bbR$, and similar for $\nu.$ Furthermore
$$|f|_{\infty}=\sup_{u} |f(u)|.$$ 
This distance does indeed generate a metric and, in particular,
satisfies the triangle inequality.
Note also that, in the absence of randomness within the measures,
the metric satisfies $d(\mu,\nu)=2\dtv(\mu,\nu)$, by
\eqref{eq:vtd}; that is, it
reduces to the total variation metric \index{metric!total variation}.
In our context the randomness within the probability measures comes
from the sampling operator $S^N$ used to define the numerical approximation.

\begin{theorem}\label{th39}
We assume in the following that there exists $\kappa \in (0,1]$ 
such that
for all $v\in \bbR^n$ and $j \in \bbN$
$$
\kappa \le g_j(v) \le  \kappa^{-1}.
$$ 
Then
$$d(\mu^N_J,\mu_J) \le \sum_{j=1}^J (2\kappa^{-2})^j \frac{1}{\sqrt N}.$$
\end{theorem}
\begin{proof}
The desired result is proved below 
in a straightforward way from the following three 
facts, whose proof we postpone to three lemmas at the end of the section:
\begin{subequations}
\begin{align}
 \sup_{\mu \in \cP(\bbR^n)}d(S^N \mu,\mu) &\leq \frac{1}{\sqrt{N}}, \\
 d(P \nu , P \mu) &\leq   d(\nu,\mu),\\
 d(L_j \nu , L_j \mu) &\leq 2 \kappa^{-2}  d( \nu , \mu ). 
\end{align}
\end{subequations}
By the triangle inequality we have, for $\nu_j^N=P\mu_j^N$,
\begin{align*}
d(\mu^N_{j+1},\mu_{j+1}) &= d(L_j S^N P \mu^N_j,L_j P \mu_j)\\
& \le d(L_j P \mu^N_j,L_j P \mu_j)+d(L_j S^N P \mu^N_j,L_j P \mu^N_j)\\
& \le 2\kappa^{-2}\Bigl(d(\mu^N_j,\mu_j)+d(S^N \nu_j^N,\nu^N_j)\Bigr)\\
& \le 2\kappa^{-2}\Bigl(d(\mu^N_j,\mu_j)+\frac{1}{\sqrt N}\Bigr).
\end{align*}
Iterating, after noting that $\mu_0^N=\mu_0$, gives the desired result.
\end{proof}

\begin{remark}
\label{rem:fil}
This important theorem shows that the particle filter\index{particle filter} 
reproduces the true filtering distribution,\index{filtering!distribution} in the large particle
limit. We make some comments about this.

\begin{itemize}

\item
This theorem shows that, at any fixed discrete time $j$, the filtering distribution  \index{filtering!distribution}
$\mu_j$ is well-approximated by the bootstrap filtering distribution
$\mu_j^N$ in the sense that, as the number of particles
$N \to \infty$, the approximating measure converges to the
true measure. However, since $\kappa<1$, the number of
particles required to decrease the upper bound on the error
beneath a specified tolerance grows with $J$.

\item If the likelihoods have a small lower bound then the constant
in the convergence proof may be prohibitively expensive, requiring
an enormous number of particles to obtain a small error. This
is similar to the discussion concerning the Independence 
Dynamics Sampler\index{Independence Dynamics Sampler}
in section \ref{ssec:mci} where we showed that large values in
the potential $\PPhi$ lead to slow convergence of the Markov chain\index{Markov chain},
and the resultant need for a large number of samples. 

\item In fact in many applications the likelihoods\index{likelihood} $g_j$ may 
not be bounded from above or below, uniformly in $j$, and more refined 
analysis is required.  However, if the Markov kernel\index{Markov kernel} $P$ 
is ergodic\index{ergodic} then it is possible to obtain 
bounds in which the error constant arising
in the analysis has  milder growth with respect to $J$.

\item Considering the case of deterministic dynamics shows just how
difficult it may be to make the theorem applicable in practice:
if the dynamics is deterministic then the original set of samples
from $\mu_0$, $\{v_0^{(n)}\}_{n=1}^N$ give rise to 
a set of particles $v_j^{(n)}=\PPsi^{(j)}(v_0^{(n)})$; in other
words the particle positions are {\em unaffected by the data}.
This is clearly a highly undesirable situation, in general,
since there is no reason at all why the pushforward\index{pushforward}
under the dynamics of the initial measure $\mu_0$ should have
substantial overlap with the filtering distribution\index{filtering!distribution} for a given fixed data set. Indeed for chaotic\index{chaos} 
dynamical systems
one would expect that it does not as the pushforward measure will
be spread over the global attractor, whilst the data will, at fixed time,
correspond to a single point on the attractor\index{global attractor}.  This
example motivates the improved proposals\index{proposal} of the next section.

\end{itemize}

\end{remark}

Before describing improvements to the basic particle filter\index{particle filter},
we prove the three lemmas underlying the convergence proof.

\begin{lemma}
\label{l:s}
The sampling operator satisfies
$$\sup_{\mu \in \cP(\bbR^n)}d(S^N \mu,\mu) \leq \frac{1}{\sqrt{N}}.$$
\end{lemma}
\begin{proof} 
Let $\nu$ be an element of $\cP(\bbR^n)$ and $\{v^{(n)}\}_{n=1}^N$
i.i.d. with $v^{(1)} \sim \nu$. In this proof the randomness
in the measure $S^N$ arises from these samples $\{v^{(n)}\}_{n=1}^N$
and expectation over this randomness is denoted $\bbE$. Then
$$S^N \nu(f)=\frac{1}{N}\sum_{n=1}^N f \left (v^{(n)} \right)$$
and, defining $\bF=f-\nu(f)$, we deduce that
$$S^N\nu(f)-\nu(f)=\frac{1}{N}\sum_{n=1}^N \bF\left (v^{(n)} \right).$$
It is straightforward to see that
$$\bbE \bF\left (v^{(n)} \right)\bF\left (v^{(l)} \right)=\delta_{nl}\bbE \left | \bF\left (v^{(n)} \right) \right |^2.$$
Furthermore, for $|f|_{\infty} \le 1$,
$$\bbE \left |\bF\left (v^{(1)} \right) \right |^2=\bbE \left |f\left (v^{(1)} \right) \right |^2-\left |\bbE f\left (v^{(1)} \right)\right|^2 \le 1.$$
It follows that, for $|f|_{\infty} \le 1$,
$$\bbE\left |\nu(f)- S^N\nu(f)\right|^2=\frac{1}{N^2}\sum_{n=1}^N \bbE \left |\bF\left (v^{(n)} \right)\right|^2 \le \frac{1}{N}.$$
Since the result is independent of $\nu$ we may take the supremum over all
probability measures and obtain the desired result.
\end{proof}

\begin{lemma}
Since $P$ is a Markov kernel \index{Markov kernel} we have
$$d(P \nu , P \nu') \leq   d(\nu,\nu').$$
\end{lemma}
\begin{proof}
Define
$$q(v')=\int_{\bbR^n} p(v',v)f(v)dv=\bbE(v_1|v_0=v'),$$
that is the expected value of $f$ under one-step of the Markov chain\index{Markov chain}
given by (\ref{eq:dtf1}a), started from $v'.$
Clearly, for $|f|_{\infty} \le 1$,
$$|q(v')| \le \int_{\bbR^n} p(v',dv)|f(v) \le \int_{\bbR^n} p(v',dv)=1.$$
Thus
$$|q|_{\infty} \le \sup_{v} |q(v)| \le 1.$$
Note that 
\begin{align*}
\nu(q)&=\bbE(v_1|v_0 \sim \nu)=\int_{\bbR^n}\int_{\bbR^n} p(v',v)f(v)
\nu(dv')\,dv\\
&=\int_{\bbR^n}\Bigl(\int_{\bbR^n} p(v',v)\nu(dv')\Bigr)f(v)dv=P\nu(f).
\end{align*}
Thus $P\nu(f)=\nu(q)$ and it follows that
$$|P\nu(f)-P\nu'(f)|=|\nu(q)-\nu'(q)|.$$
Thus
\begin{align*}
d(P\nu,P\nu')&=\sup_{|f|_{\infty}\le 1}\Bigl(\bbE^{\omega}|P\nu(f)-P\nu'(f)|^2\Bigr)^{\frac12}\\
& \le \sup_{|q|_{\infty}\le 1}\Bigl(\bbE^{\omega}|\nu(q)-\nu'(q)|^2\Bigr)^{\frac12}\\
& = d(\nu,\nu')
\end{align*}
as required.
\end{proof}

\begin{lemma}
Under the Assumptions of Theorem \ref{th39} we have
$$d(L_j \nu , L_j \mu) \leq 2 \kappa^{-2}  d( \nu , \mu ).$$ 
\end{lemma}
\begin{proof}
Notice that for $|f|_\infty <\infty$
we can rewrite 
\begin{subequations}
\begin{align}
(L_j \nu)(f) - (L_j \mu)(f)  = & \frac{\nu(f g_j)}{\nu(g_j)} - \frac{\mu(f g_j)}{\mu(g_j)} \\
 = & \frac{\nu(f g_j)}{\nu(g_j)}  - \frac{\mu(f g_j)}{\nu(g_j)} + \frac{\mu(f g_j)}{\nu(g_j)} - \frac{\mu(f g_j)}{\mu(g_j)} \\
 = & \frac{\kappa^{-1}}{\nu(g_j)}[\nu(\kappa f g_j) -\mu(\kappa f g_j)] + 
 \frac{\mu(f g_j)}{\mu(g_j)} \frac{\kappa^{-1}}{\nu(g_j)} [\mu(\kappa g_j) - \nu(\kappa g_j)].
\end{align}
\end{subequations}
Now notice that 
$\nu(g_j)^{-1} \leq \kappa^{-1}$ and that 
$\mu(fg_j)/\mu(g_j) \le 1$ since the expression corresponds to
an expectation with respect to measure found from $\mu$ by reweighting
with likelihood \index{likelihood} proportional to $g_j$. Thus 
$$|(L_j \nu)(f) - (L_j \mu)(f)| \le \kappa^{-2}|\nu(\kappa f g_j) -\mu(\kappa f g_j)|+\kappa^{-2}|\nu(\kappa g_j) - \mu(\kappa g_j)|.$$ 
Since $|\kappa g_j|_{\infty} \le 1$ it follows that $|\kappa f g_j|_{\infty}$
and hence that 
$$\bbE^{\omega}|(L_j \nu)(f) - (L_j \mu)(f)|^2 \le 4\kappa^{-4}
\sup_{|h|_{\infty} \le 1} \bbE^{\omega}|\nu(h) - \mu(h)|^2.$$
The desired result follows.
\end{proof}

\subsection{Improved Proposals}
\label{ssec:opt}

In the particle filter\index{particle filter} described in the previous section
we propose according to the underlying unobserved dynamics,
and then apply Bayes' formula \index{Bayes' formula} to incorporate the data.
The final point in Remarks \ref{rem:fil} demonstrates
that this may result in a very poor set of particles
with which to approximate the 
filtering distribution\index{filtering!distribution}.
Cleverer proposals\index{proposal}, which use the data, can lead
to improved performance and we outline this methodology here.

Instead of moving the particles $\{v_{j}^{(n)}\}_{n=1}^N$
according to the Markov kernel\index{Markov kernel} $P$, we use a Markov kernel
$Q_j$ with density $\rqq(v_{j+1}|v_{j},Y_{j+1})$.
The weights $w_{j+1}^{(n)}$ are found, as before,
by applying Bayes' formula \index{Bayes' formula} for each particle, and 
then weighting appropriately as in \eqref{eq:waits}: 
\begin{subequations}
\label{eq:waits2}
\begin{eqnarray}
\tw_{j+1}^{(n)} &=& w_{j}^{(n)}\frac{\rp\left(y_{j+1}|\hv_{j+1}^{(n)}\right)\rp\left(\hv_{j+1}^{(n)}|v_{j}^{(n)}\right)}
{\rqq\left(\hv_{j+1}^{(n)}|v_{j}^{(n)},Y_{j+1}\right)},\\
w_{j+1}^{(n)}&=&\tw_{j+1}^{(n)}\Bigl{/}\left(\sum_{n=1}^{N} \tw_{j+1}^{(n)}\right ).
\end{eqnarray}
\end{subequations}
The choice 
$$\rqq\left(v_{j+1}|v_{j}^{(n)},Y_{j+1}\right)=
\rp\left(v_{j+1}|v_{j}^{(n)}\right)$$
results in the bootstrap filter\index{filter!bootstrap}
from the preceding subsection.
In the more general case the approach
results in the following algorithm:

\begin{enumerate}
\item[1.] Set $j=0$ and $\mu_0^N(v_0)dv_0=\bbP(v_0)dv_0$.
\item[2.] Draw $v_{j}^{(n)} \sim \mu_j^N$, $n=1,\dots,N$. 
\item[3.] Set $w_{j}^{(n)}=1/N$, $n=1,\dots,N$.
\item[4.] Draw $\hv_{j+1}^{(n)} \sim \rqq(\cdot|v_{j+1}^{(n)},Y_{j+1}).$
\item[5.] Define $w_{j+1}^{(n)}$ by \eqref{eq:waits2} 
and $\mu_{j+1}^N=\rp^{N}(v_{j+1}|Y_{j+1})$ by \eqref{eq:emp_filt2}. 
\item[6.] $j+1\to j$.
\item[7.] Go to step 2.
\end{enumerate}

We note that the normalization constants in (\ref{eq:waits2}a), 
here assumed known in the definition
of the reweighting, or not of course needed.
The so-called \textbf{optimal proposal} \index{proposal!optimal} is found by
choosing
$$
\rqq\left(v_{j+1}|v_{j}^{(n)},Y_{j+1}\right) \equiv \rp\left(v_{j+1}|v_{j}^{(n)},y_{j+1}\right)$$
which results in
\begin{equation} 
\tw_{j+1}^{(n)} = w_{j}^{(n)} \rp\left(y_{j+1}|v_{j}^{(n)}\right).
\label{eq:opwts}
\end{equation}
{The above can be seen by observing that the definition of conditional probability gives
\begin{equation}
\begin{array}{lll}
\rp\left(y_{j+1}|\hv_{j+1}^{(n)}\right)\rp\left(\hv_{j+1}^{(n)}|v_{j}^{(n)}\right) &= & \rp\left(y_{j+1},\hv_{j+1}^{(n)}|v_{j}^{(n)}\right)\\
 &= & \rp\left(\hv_{j+1}^{(n)} | v_{j}^{(n)},y_{j+1}\right)\rp\left(y_{j+1}|v_{j}^{(n)}\right).
\end{array}
\end{equation}
Substituting the optimal proposal \index{proposal!optimal} into \eqref{eq:waits2} then immediately gives \eqref{eq:opwts}.

}

This small difference from the bootstrap filter\index{filter!bootstrap}
may seem trivial at a glance, and at the potentially
large cost of sampling from $\bbQ$. 
However, in the case of nonlinear Gaussian Markov models as we study here,
the distribution and the weights are given in closed form.  If the dynamics is highly nonlinear or the 
model noise is larger than the observational noise then 
the variance of the weights for the optimal proposal \index{proposal!optimal} may be much smaller than for the standard 
proposal.\index{proposal!standard} The corresponding particle filter\index{particle filter} will be referred to with the acronym
SIRS(OP) to indicate the optimal proposal \index{proposal!optimal}.
For deterministic dynamics the optimal proposal \index{proposal!optimal} reduces to the standard proposal\index{proposal!standard}.

\section{Large-Time Behaviour of Filters}
\label{ssec:stab}

With the exception of the Kalman filter\index{Kalman filter} for linear problems, and the
particle filter in the general case, the filtering methods presented
in this chapter do not, in general, give accurate approximations of
the true posterior distribution; in particular the approximate Gaussian 
filters\index{filter!approximate Gaussian} 
do not perform well as measured by the Bayesian \index{Bayesian} quality
assessment\index{Bayesian quality assessment} test of section \ref{ssec:qual}. 
However they may perform well as measured by the signal estimation quality assessment\index{signal estimation quality assessment} test and the purpose
of this section is to demonstrate this fact.

More generally, an important question concerning filters is their behaviour
when iterated over long times and, in particular, their ability to
recover the true signal underlying the data if iterated for long
enough, even when initialized far from the truth.
In this section we present some basic
large time asymptotic results for filters 
to illustrate the key issue which affects
the ability of filters to accurately recover the signal when iterated
for long enough. The main idea is that the data must be sufficiently
rich to stabilize any inherent instabilities within the underlying
dynamical model \eqref{eq:dtf1}; in rough terms it is necessary to
observe only the unstable directions as the dynamics of the model
itself will enable recovery of the true signal within the
space spanned by the stable directions. 
We illustrate this idea first, in subsection
\ref{sssec:kfo}, for the
explicitly solvable case of the Kalman filter\index{Kalman filter}
in one dimension, and then, in subsection \ref{sssec:3s}, 
for the 3DVAR\index{3DVAR} method. 

\subsection{The Kalman Filter\index{Kalman filter} in One Dimension}
\label{sssec:kfo}

We consider the case of one dimensional dynamics with 
\[
\PPsi(v)=\lambda v, \quad h(v)=v,
\]
while we will also assume that
\[
\Sigma=\sigma^{2}, \quad \Gamma=\gamma^{2}.
\]
With these definitions equations (\ref{eq:kalman_update}a,b) become
\begin{subequations} \label{eq:kalman_update1d_a}
\begin{eqnarray}  
\frac{1}{c_{j+1}} &=&\frac{1}{\sigma^{2}+\lambda^{2}c_{j}}+\frac{1}{\gamma^{2}}, \\
\frac{m_{j+1}}{c_{j+1}} &=& \frac{\lambda m_{j}}{\sigma^{2}+\lambda^{2}c_j}+\frac{1}{\gamma^{2}}y_{j+1},
\end{eqnarray}
\end{subequations}
which, after some algebraic manipulations, give 
\begin{subequations} \label{eq:kalman_update1d}
\begin{eqnarray}
c_{j+1}&=&g(c_{j}), \\
m_{j+1} &=&\left(1-\frac{c_{j+1}}{\gamma^{2}} \right)\lambda m_{j}+\frac{c_{j+1}}{\gamma^{2}}y_{j+1},
\end{eqnarray}
\end{subequations}
where we have defined 
\begin{equation} \label{eq:map1}
g(c):=\frac{\gamma^{2}(\lambda^{2} c+\sigma^{2})}{\gamma^{2}+\lambda^{2} c+\sigma^{2}}.
\end{equation}

We wish to study the behaviour of the Kalman filter\index{Kalman filter} 
as $j \rightarrow \infty$, \emph{i.e.} 
when more and more data points are assimilated into the model. 
Note that the covariance evolves independently of the data $\{y_j\}_{j \in \Z^+}$ and satisfies an autonomous nonlinear dynamical system. 
However it is of interest to note that, if $\sigma^2=0$, then the dynamical
system for $c_j^{-1}$ is linear.

We now study the
asymptotic properties of this map.
The fixed points  \index{fixed point} $c^{\star}$ of (\ref{eq:kalman_update1d}a) satisfy 
\begin{equation} \label{eq:kalman_eq1d}
c^{\star}=\frac{\gamma^{2}(\lambda^{2} c^{\star}+\sigma^{2})}{\gamma^{2}+\lambda^{2} c^{\star}+\sigma^{2}},
\end{equation}
and thus solve the quadratic equation
$$\lambda^2 (c^{\star})^2+\bigl(\gamma^2(1-\lambda^2)+\sigma^2\bigr)c^{\star}-
\gamma^2\sigma^2=0.$$
We see that, provided $\lambda\gamma\sigma \ne 0$, one root is positive and
one negative. The roots are given by 
\begin{equation} \label{eq:equilibria}
c^{\star}_{\pm}=\frac{-(\gamma^{2}+\sigma^{2}-\gamma^{2}\lambda^{2})
\pm \sqrt{(\gamma^{2}+\sigma^{2}-\gamma^{2}\lambda^{2})^{2}+4\lambda^{2}\gamma^{2}\sigma^{2}}}{2\lambda^{2}}.
\end{equation}
We observe that the update formula for the covariance ensures that, provided
$c_0 \ge 0$ then $c_j \ge 0$ for all $j \in \N.$ It also demonstrates that $c_j
\le \gamma^2$ for all $j \in \Z^+$ so that the variance of the
filter is no larger
than the variance in the data. We may hence fix our attention
on non-negative covariances, knowing that they are also uniformly
bounded by $\gamma^2.$
We will now study the stability of the non-negative fixed points  \index{fixed point}. 

We first start with the case $\sigma=0$, which corresponds to 
deterministic dynamics, and for which the dynamics of $c_j^{-1}$ is
linear. In this case we obtain
 \[
c^{\star}_{+}=0, \quad c^{\star}_{-}=\frac{\gamma^{2}(\lambda^{2}-1)}{\lambda^{2}}, 
\]
and 
\[
g'(c^{\star}_{+})= \lambda^{2}, \quad g'(c^{\star}_{-})=\lambda^{-2},
\]
which implies that when $\lambda^{2} <1$, $c^{\star}_{+}$ is an asymptotically 
stable fixed point, while when $\lambda^{2}>1$, $c^{\star}_{-}$ is an asymptotically stable fixed point  \index{fixed point}. When $|\lambda|=1$ the two roots are coincident at
the origin and neutrally stable. Using the aforementioned linearity,
for the case $\sigma=0$ it is possible to solve (\ref{eq:kalman_update1d_a}a) to obtain for $\lambda^{2} \neq 1$
\begin{equation} \label{eq:decay}
\frac{1}{c_{j}}=\left(\frac{1}{\lambda^{2}}\right)^{j}\frac{1}{c_{0}}+\frac{1}{\gamma^{2}}\left[\frac{\left(\frac{1}{\lambda^{2}}\right)^{j}-1}{\frac{1}{\lambda^{2}}-1}\right]. 
\end{equation}
This explicit formula shows that the fixed point  \index{fixed point} $c^{\star}_{+}$ (resp. $c^{\star}_{-}$) is globally asymptotically stable, and exponentially attracting on $\mathbb{R}^{+}$, when $\lambda^{2}<1$
(resp. $\lambda^{2}>1$). 
Notice also that $c^{\star}_{-}=\mathcal{O}(\gamma^{2})$ so that when $\lambda^{2}>1$, the asymptotic variance of the filter scales as the observational noise variance.
Furthermore, when $\lambda^{2}=1$ we may solve (\ref{eq:kalman_update1d_a}a) to obtain 
\[
\frac{1}{c_{j}}=\frac{1}{c_{0}}+\frac{j}{\gamma^{2}},
\]
showing that $c^{\star}_{-}=c^{\star}_{+}=0$ is globally asymptotically stable on $\mathbb{R}^{+}$, but is only algebraically attracting.

We now study the  stability of the fixed points  \index{fixed point} $c^{\star}_{+}$ and $c^{\star}_{-}$ in the case of $\sigma^{2} > 0$ 
corresponding to the case where the dynamics are stochastic. 
To this end we prove some bounds on 
$g'(c^{\star})$ that will also be useful when we study the behaviour of the error between the true signal and the estimated mean; here, and in what
follows in the remainder of this example, prime denotes differentiation with
respect to $c$. We start by noting that 
\begin{equation} \label{eq:map}
g(c)=\gamma^{2}-\frac{\gamma^{4}}{\gamma^{2}+\lambda^{2}c+\sigma^{2}},
\end{equation}
and so
\[
g'(c)=\frac{\lambda^{2}\gamma^{4}}{(\gamma^{2}+\lambda^{2}c+\sigma^{2})^{2}}.
\]
Using the fact that $c^{\star}$ satisfies \eqref{eq:kalman_eq1d} 
together with equation \eqref{eq:map}, we obtain 
\[
g'(c^{\star})=\frac{1}{\lambda^{2}}\frac{(c^{\star})^{2}}{\left(c^{\star}+\frac{\sigma^{2}}{\lambda^{2}} \right)^{2}} \quad  \text{and} \quad  
g'(c^{\star})=\lambda^{2}\left(1-\frac{c^{\star}}{\gamma^{2}} \right)^{2}.
\]
We can now see that from the first equation we obtain the following
two bounds, since $\sigma^{2}>0$: 
\[
g'(c^{\star})<\lambda^{-2}, \quad \text{for} \quad \lambda \in \bbR,\quad \text{and} \quad g'(c^{\star})<1, \quad \text{for} \quad \lambda^{2}=1,
\]
while from the second equality and the fact that, since $c^{\star}$ satisfies 
\eqref{eq:kalman_eq1d}, $c^{\star}<\gamma^{2}$
we obtain 
\[
g'(c^{\star})<\lambda^{2}.
\]
when $c^{\star}>0$. 
We thus conclude that when $\sigma^{2} > 0$ the fixed point  \index{fixed point} $c_{+}^{\star}$ of (\ref{eq:kalman_update1d}a) is always stable independently of the value of the parameter $\lambda$.

\begin{table} [h!]
\begin{center}
\scalebox{1.1}{
\begin{tabular}{ |c || c |  c|}
    \hline
             & Limiting covariance for $\sigma^{2}=0$ &  Limiting covariance for $\sigma^{2}>0$\\     \hline
   $ |\lambda|<1$ & $c_{j}\rightarrow0$  (exponentially) & $c_{j} \rightarrow c^{\star}_{+}=\mathcal{O}(\gamma^{2})$ (exponentially)  \\  \hline
 $ |\lambda|=1$ &  $c_{j}\rightarrow0$ (algebraically)  & $c_{j} \rightarrow c^{\star}_{+}=\mathcal{O}(\gamma^{2})$  (exponentially) \\
    \hline
    $ |\lambda|>1$  & $c_{j} \rightarrow c^{\star}_{-}=\mathcal{O}(\gamma^{2})$ (exponentially)  &$c_{j} \rightarrow c^{\star}_{+}=\mathcal{O}(\gamma^{2})$  (exponentially) \\
    \hline
  \end{tabular}}
\caption{Summary of the limiting behaviour of covariance $c_{j}$ for Kalman filter\index{Kalman filter} applied to one dimensional dynamics}
\label{tb:summary}
\end{center}
\end{table}

\begin{figure}[h]
\centering
\subfigure[$\sigma=0, \lambda=0.8, \gamma=1$]{\includegraphics[scale=0.365]{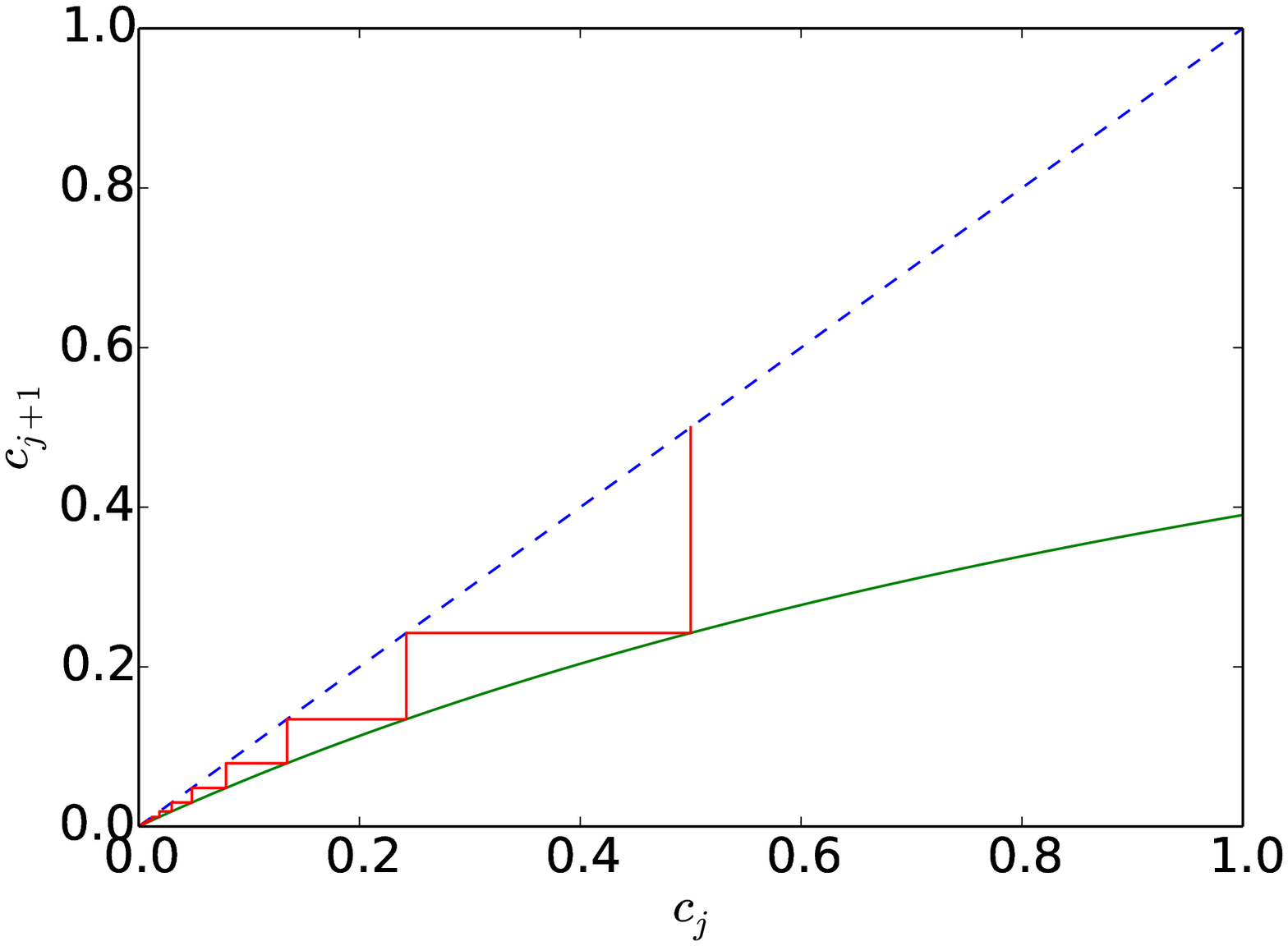}}
\subfigure[$\sigma=0, \lambda=1, \gamma=1$]{\includegraphics[scale=0.365]{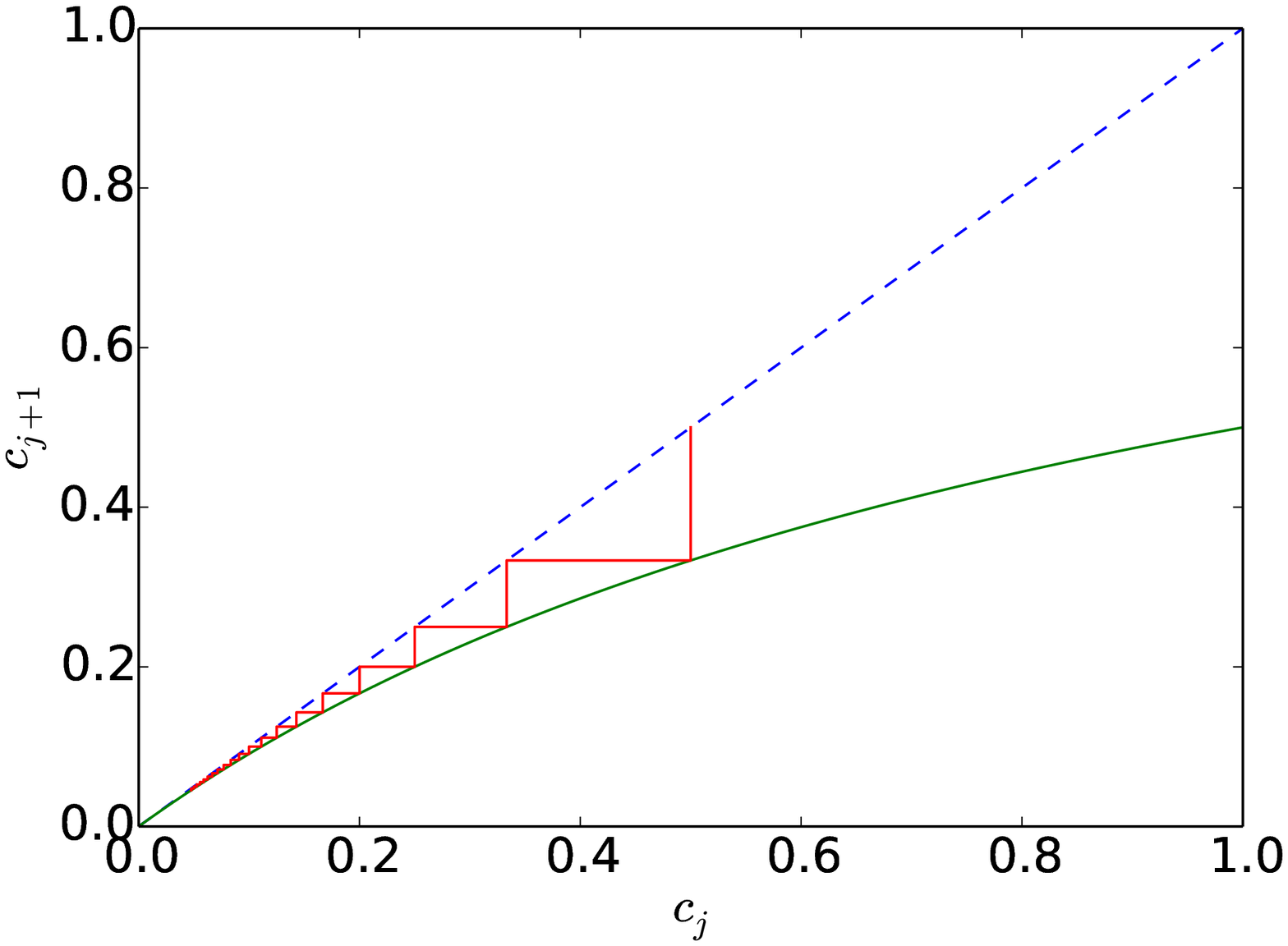}}
\caption{Cobweb diagram for equation (\ref{eq:kalman_update1d}a) }
\label{fig:cobwebb1}
\end{figure}

\begin{figure}[h]
\centering
\subfigure[$\sigma=0,\lambda=1.2,\gamma=1$]{\includegraphics[scale=0.365]{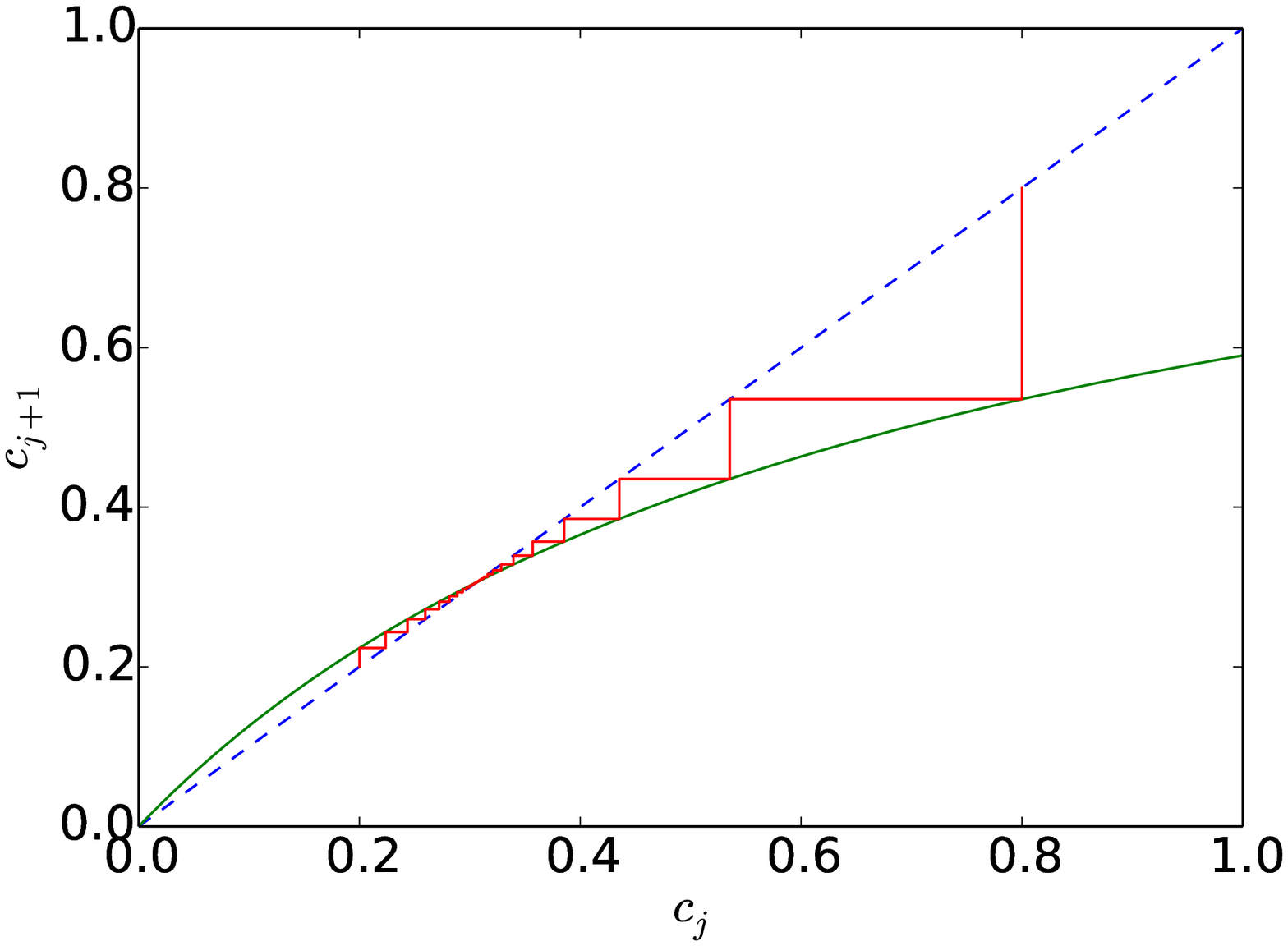}}
\subfigure[$\sigma=0.5,\lambda=1.2,\gamma=1$]{\includegraphics[scale=0.365]{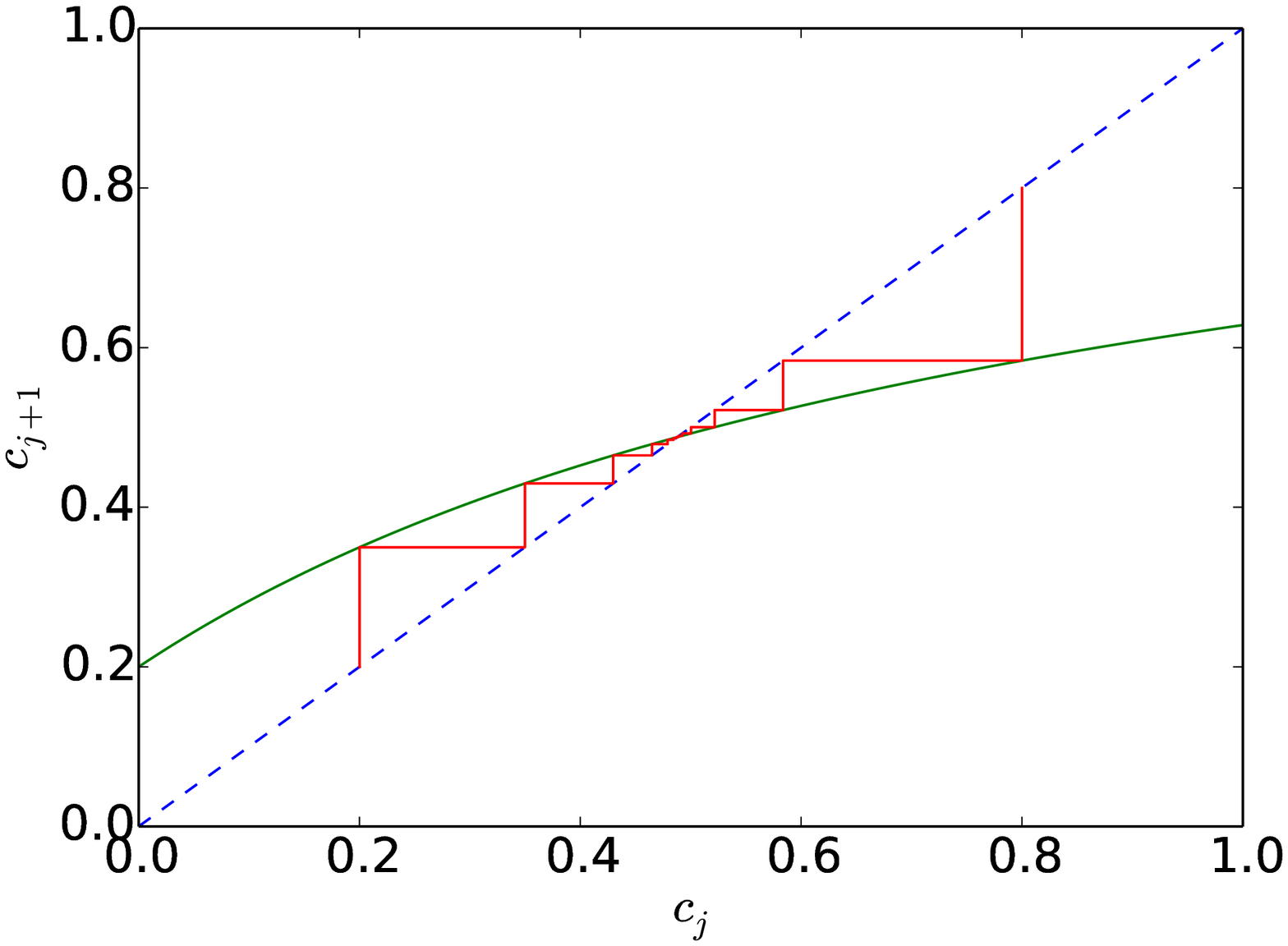}}
\caption{Cobweb diagram for equation (\ref{eq:kalman_update1d}a)}
\label{fig:cobwebb2}
\end{figure}

Table \ref{tb:summary} summarises the behaviour of the variance
of the Kalman filter\index{Kalman filter} in the case of one-dimensional dynamics.  This is illustrated further in Figures \ref{fig:cobwebb1} and \ref{fig:cobwebb2} where we plot the cobweb diagram for the map \eqref{eq:map}. In particular, in Figure \ref{fig:cobwebb1} we observe the difference between the algebraic and geometric convergence to 0, for different values of $\lambda$ in the case $\sigma=0$, while in Figure \ref{fig:cobwebb2} we observe the exponential convergence to $c^{\star}_{+}$ for the case of $|\lambda|>1$. 
The analysis of the error between the mean and the truth \index{truth} underlying
the data is left as an exercise at the end of the chapter. This shows
that the error in the mean is, asymptotically, of order $\gamma^2$
in the case where $\sigma=0.$

\subsection{The 3DVAR\index{3DVAR} Filter}
\label{sssec:3s}

In the previous subsection we showed that the Kalman filter\index{Kalman filter} accurately
recovers any one-dimensional signal, provided the observational noise
is small. The result allows for initialization far from the true
signal and is, in this sense, quite strong. On the other hand being
only one-dimensional it gives a somewhat limited picture. In this
subsection we study the 3DVAR\index{3DVAR} filter given by \eqref{eq:dtfa10}.
We study conditions under which the 3DVAR\index{3DVAR} filter will 
recover the true signal, to within a small observational noise level
of accuracy, in dimensions bigger than one, and when only part of
the system is observed.

To this end we assume that 
\begin{equation}\label{eq:dtfa11}
y_{j+1}=Hv_{j+1}^\dagger+\epsilon_j
\end{equation}
where the true signal $\{v_j^\dagger\}_{j\in\N}$ satisfies
\begin{subequations}\label{eq:dtfa12}
\begin{align}
v_{j+1}^\dagger&=\PPsi(v_j^\dagger), \;j\in\N\\
v_0^\dagger&=u
\end{align}
\end{subequations}
and, for simplicity, we assume that the observational noise satisfies 
\begin{equation}
\label{eq:error}
\sup_{j\in\N}|\epsilon_j|=\epsilon.
\end{equation}
We have the following result. 
\begin{theorem}\label{t35}
Assume that the data is given by \eqref{eq:dtfa11}, where the
signal follows equation \eqref{eq:dtfa12} and the error in the
data satisfies \eqref{eq:error}. Assume furthermore that $\hc$ 
is chosen so that
$(I-KH)\PPsi:\R^n\to\R^n$ is globally Lipschitz with constant $a<1$ 
in some norm $\norm{\cdot}$. Then there is constant $c>0$ such that
\[\limsup_{j\to\infty}\norm{m_j-v_j^\dagger}\leq \frac{c}{1-a}\epsilon.\]
\end{theorem}
\begin{proof}We may write (\ref{eq:dtfa10}), (\ref{eq:dtfa12}), using (\ref{eq:dtfa11}), as 
\begin{align*}
m_{j+1}&=(I-KH)\PPsi(m_j)+KH\PPsi(v_{j}^\dagger)+K\epsilon_j\\
v_{j+1}^\dagger&=(I-KH)\PPsi(v_{j}^\dagger)+KH\PPsi(v_{j}^\dagger).
\end{align*}
Subtracting, and letting $e_j=m_j-v_j^\dagger$ gives, for some
finite constant $c$ independent of $j$, 
\begin{align*}
\norm{e_{j+1}}&\leq\norm{(I-KH)\PPsi(m_j)-(I-KH)\PPsi(v_j^\dagger)}+\|K\epsilon_j\|\\
&\leq a\norm{e_j}+c\epsilon.
\end{align*}
Applying the Gronwall\index{Gronwall} Lemma \ref{lem:GL}
gives the desired result.
\end{proof}

We refer to a map with Lipschitz constant less than $1$ as a {\em contraction}
in what follows.

\begin{remark}
The preceding simple theorem shows that it is possible to construct filters
which can recover from being initialized far
from the truth and lock-on to a small neighbourhood of the true signal
underlying the data, when run for long enough. 
Furthermore, this can happen even when the system is
only partially observed, provided that the observational noise
is small and enough of the system is observed. This concept of
observing ``enough'' illustrates a key idea in 
filtering\index{filtering}: the question of 
whether  the fixed model covariance in 3DVAR, $\hc$, 
can be chosen to make $(I-KH)\PPsi$ into a contraction
involves a subtle interplay between the underlying dynamics,
encapsulated in $\PPsi$, and the observation operator $H$. In rough terms
the question of making $(I-KH)\PPsi$ into a contraction is the question
of whether the unstable parts of the dynamics are observed; if they
are then it is typically the case that $\hc$ can be designed to obtain the
desired contraction.
\end{remark}

\begin{example}
\label{ex:101}
Assume that $H=I$, so that the whole system is observed, that $\Gamma=\gamma^2I$ and $\hc=\sigma^2I$. Then, for $\eta^2=\frac{\gamma^2}{\sigma^2}$
\[S=(\sigma^2+\gamma^2)I, \quad K=\frac{\sigma^2}{(\s^2+\gamma^2)}I\]
and \[(I-KH)=\frac{\gamma^2}{(\s^2+\gamma^2)}I=\frac{\eta^2}{(1+\eta^2)}I.\]
Thus, if $\PPsi:\R^n\to\R^n$ is globally Lipschitz with constant $\lambda>0$ in the Euclidean norm, $|\cdot|$, then $(I-KH)\PPsi$ is globally Lipschitz with constant $a<1$, if $\eta$ is chosen so that $\frac{\eta^2\lambda}{1+\eta^2}<1$. Thus, by choosing $\eta$ sufficiently small the filter can be made to contract. This corresponds to trusting the data sufficiently in comparison to the model.
It is a form of {\bf variance inflation}\index{variance inflation} in
that, for given level of observational noise, $\eta$ can be made
sufficiently small by choosing the model variance scale $\sigma^2$
sufficiently large -- ``inflating'' the model variance.
\end{example}

\begin{example}
\label{ex:202}
Assume that there is a partition of the state space in which
$H=(I,0)^T$, so that only part of the system is observed. 
Set $\Gamma=\gamma^2I$ and $\hc=\sigma^2I$. Then, with $\eta$ as in the
previous example, 
\begin{eqnarray*}
I-KH=\left(
\begin{array}{cc}
\frac{\eta^2}{1+\eta^2}I & 0\\
0 & I
\end{array}
\right).
\end{eqnarray*}
Whilst the previous example shows that
variance inflation\index{variance inflation}
may help to stabilize the filter,
this example shows that, in general, more is required:
in this case it is clear that making $(I-KH)\PPsi(\cdot)$ into
a contraction will require a relationship between the subspace in which we
observe and the space in which the dynamics of the map is expanding
and contracting. For example, if $\PPsi(u)=Lu$ and
\begin{eqnarray*}
L=\left(
\begin{array}{cc}
2I & 0\\
0 & aI
\end{array}
\right)
\end{eqnarray*} 
then
\begin{eqnarray*}
(I-KH)L=\left(
\begin{array}{cc}
\frac{2\eta^2}{1+\eta^2}I & 0\\
0 & aI
\end{array}
\right)
\end{eqnarray*} 
When $|a|<1$ this can be made into a contraction by choosing $\eta$
sufficiently small; but for $|a| \ge 1$ this is no longer possible. 
The example thus illustrates the intuitive idea that the observations \index{observations} should
be sufficiently rich to ensure that the unstable directions within the
dynamics can be tamed by observing them.

\end{example}

\subsection{The Synchronization Filter}\label{ssec:4}
A fundamental idea underlying successful filtering of partially
observed dynamical systems, is synchronization\index{synchronization}. 
To illustrate this we introduce and study the idealized 
{\em synchronization filter}\index{synchronization}\index{filter!synchronization}.  
To this end consider a partition of the identity $P+Q=I$.
We write $v=(p,q)$ where $p=Pv, q=Qv$
and then, with a slight abuse of notation, write $\PPsi(v)=\PPsi(p,q)$. 
Consider a true signal
governed by the deterministic dynamics\index{deterministic dynamics} 
model \eqref{eq:dtfa12}
and write $\vd_k=(\pd_k,\qd_k)$, with $\pd_k=P\vd_k$ and $\qd_k=Q\vd_k$. Then
\begin{align*}
\pd_{k+1}&=P\PPsi(\pd_k,\qd_k),\\
\qd_{k+1}&=Q\PPsi(\pd_k,\qd_k).
\end{align*} 
Now imagine that we observe $y_k=\pd_k$ exactly, without noise.
Then the synchronization filter\index{synchronization}\index{filter!synchronization} simply
fixes the image under $P$ to $\pd_k$ and plugs this into the
image of the dynamical model under $Q$; if the filter is $m_k=(p_k,q_k)$
with $p_k=Pm_k$ and $q_k=Qm_k$ then
\begin{align*}
p_{k+1}&=\pd_{k+1},\\
q_{k+1}&=Q\PPsi(\pd_k,q_k).
\end{align*} 
We note that, expressed in terms of the data, this filter has the form
\begin{equation}
\label{eq:sync2}
m_{k+1}=Q\PPsi(m_k)+Py_{k+1}.
\end{equation}
A key question now is whether or not the filter synchronizes in the following
sense:
\begin{equation*}
|q_k-\qd_k| \to 0\,\,{\rm {as}}\,\, k \to \infty.
\end{equation*}
This of course is equivalent to 
\begin{equation}
\label{eq:sync}
|m_k-\vd_k| \to 0\,\,{\rm {as}}\,\, k \to \infty.
\end{equation}
Whether or not this happens involves, as for 3DVAR described above,
a subtle interplay between the underlying dynamics and the
observation operator, here $P$. The bibliography section \ref{ssec:dtnb3}
contains pointers to the literature studying this question.

In fact the following example shows how the synchronization filter can be
viewed as a distinguished parameter limit, corresponding to
infinite variance inflation\index{variance inflation}, 
for a particular family of 3DVAR filters.

\begin{example}
\label{ex:203}

Let $H=P$ and $\Gamma=\gamma^2 I$. If we choose 
 ${\widehat C}$ as in Example \ref{ex:202} 
then the 3DVAR filter can be written as
\begin{subequations}
\label{eq:sub}
\begin{align}
m_{k+1}&=S\PPsi(m_k)+(I-S)y_{k+1},\\
S&=\frac{\eta^2}{1+\eta^2}P+Q.
\end{align}
\end{subequations}
The limit $\eta \to 0$ is the the extreme limit of 
variance inflation\index{variance inflation} 
referred to in Example \ref{ex:101}. 
In this limit  the 3DVAR filter becomes the synchronization
filter \eqref{eq:sync2}. 
\end{example}

\section{Illustrations}
\label{ssec:filti}

\begin{figure}[h]
\centering
\subfigure[Solution.]{\includegraphics[scale=0.365]{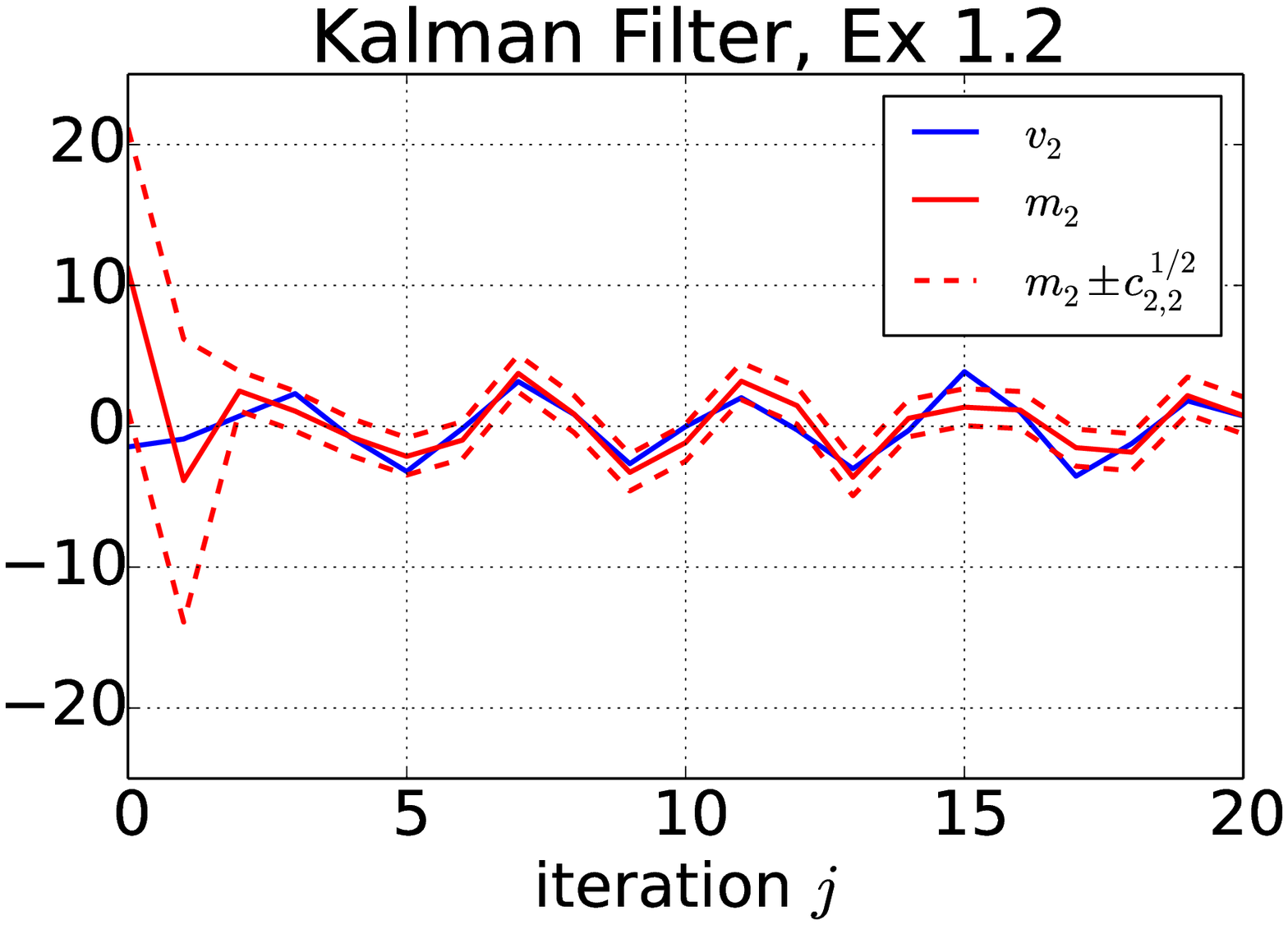}}
\subfigure[Covariance.]{\includegraphics[scale=0.365]{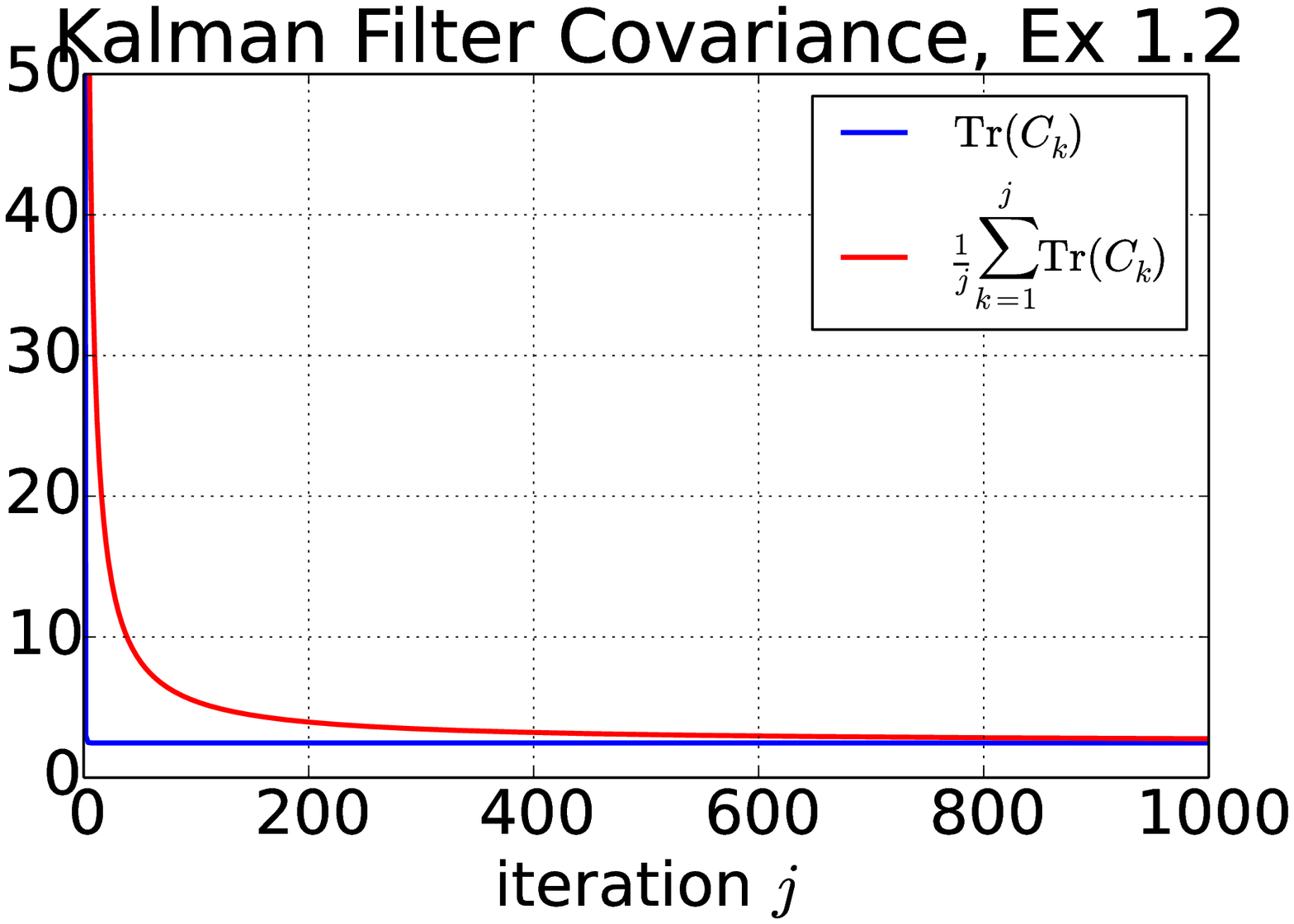}}
\subfigure[Error.]{\includegraphics[scale=0.365]{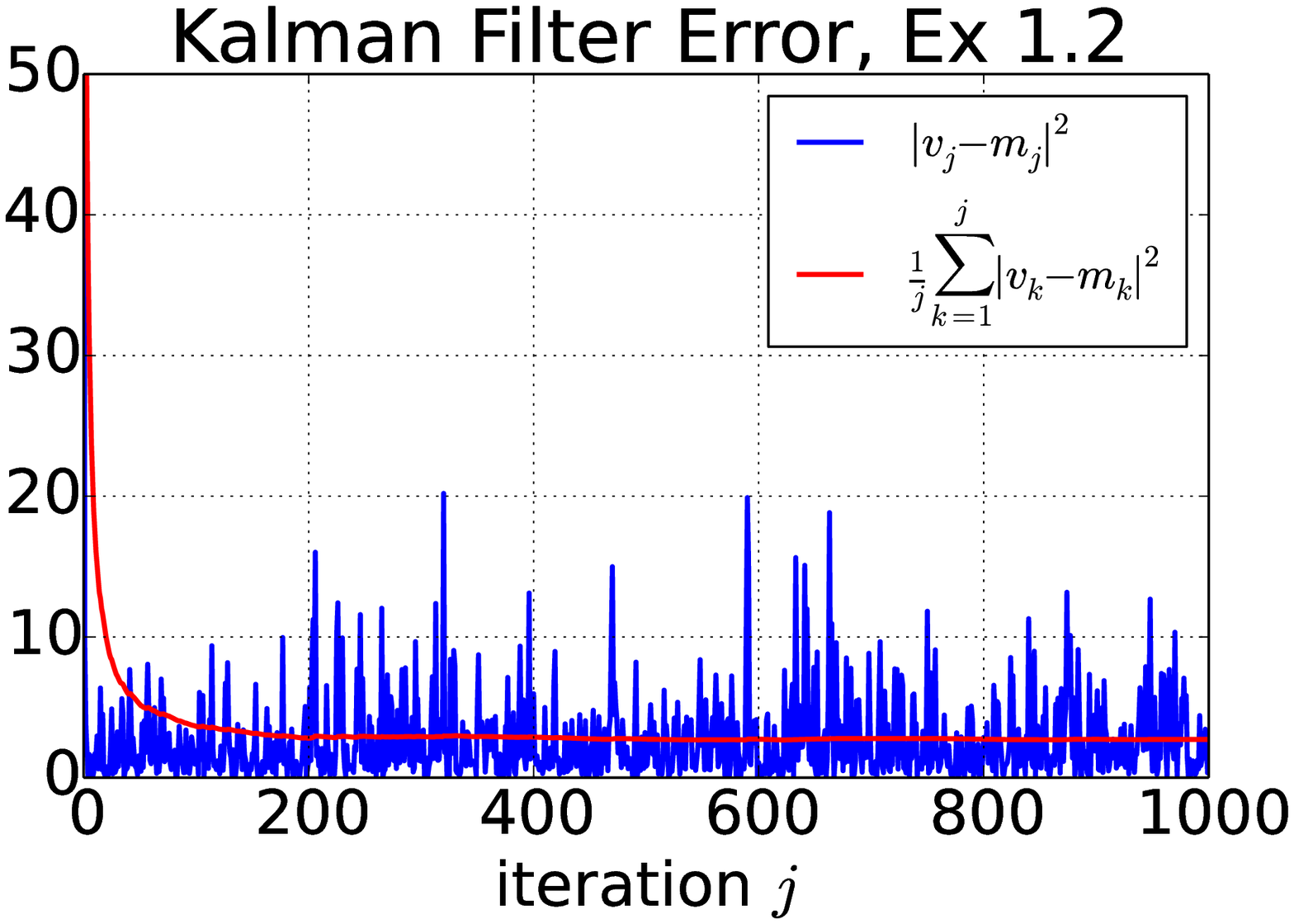}}
\caption{Kalman filter\index{Kalman filter} applied to the
linear system of Example \ref{ex:ex2} with $A=A_3$,
$H=(1,0)$, $\Sigma=I$, and $\Gamma=1$, see also {\tt p8.m} in Section \ref{ssec:p6}.
The problem is initialized with mean $0$ and covariance $10\,I$.}
\label{fig:KF}
\end{figure}

\begin{figure}[h]
\centering
\subfigure[Solution.]{\includegraphics[scale=0.365]{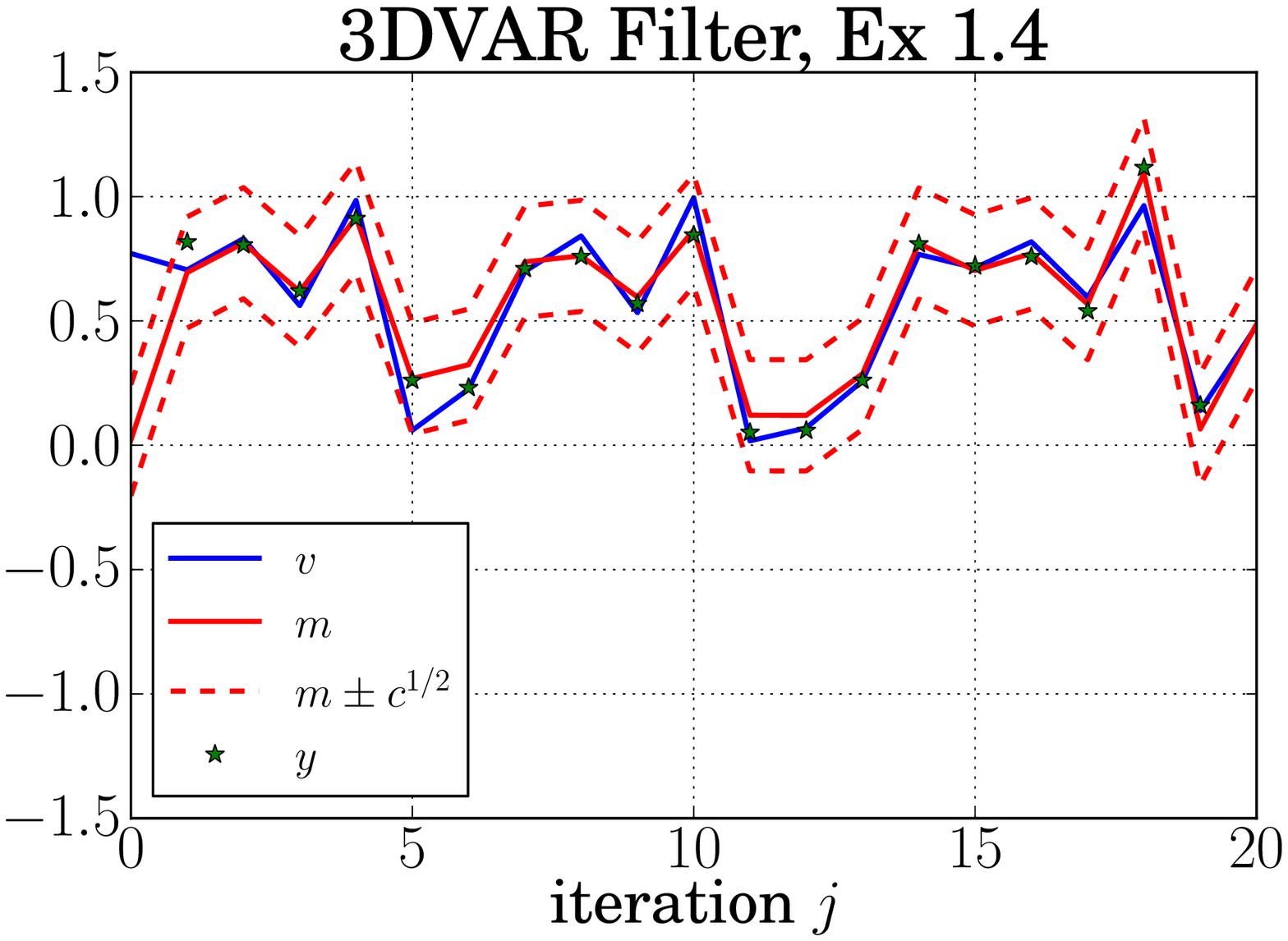}}
\subfigure[Covariance.]{\includegraphics[scale=0.365]{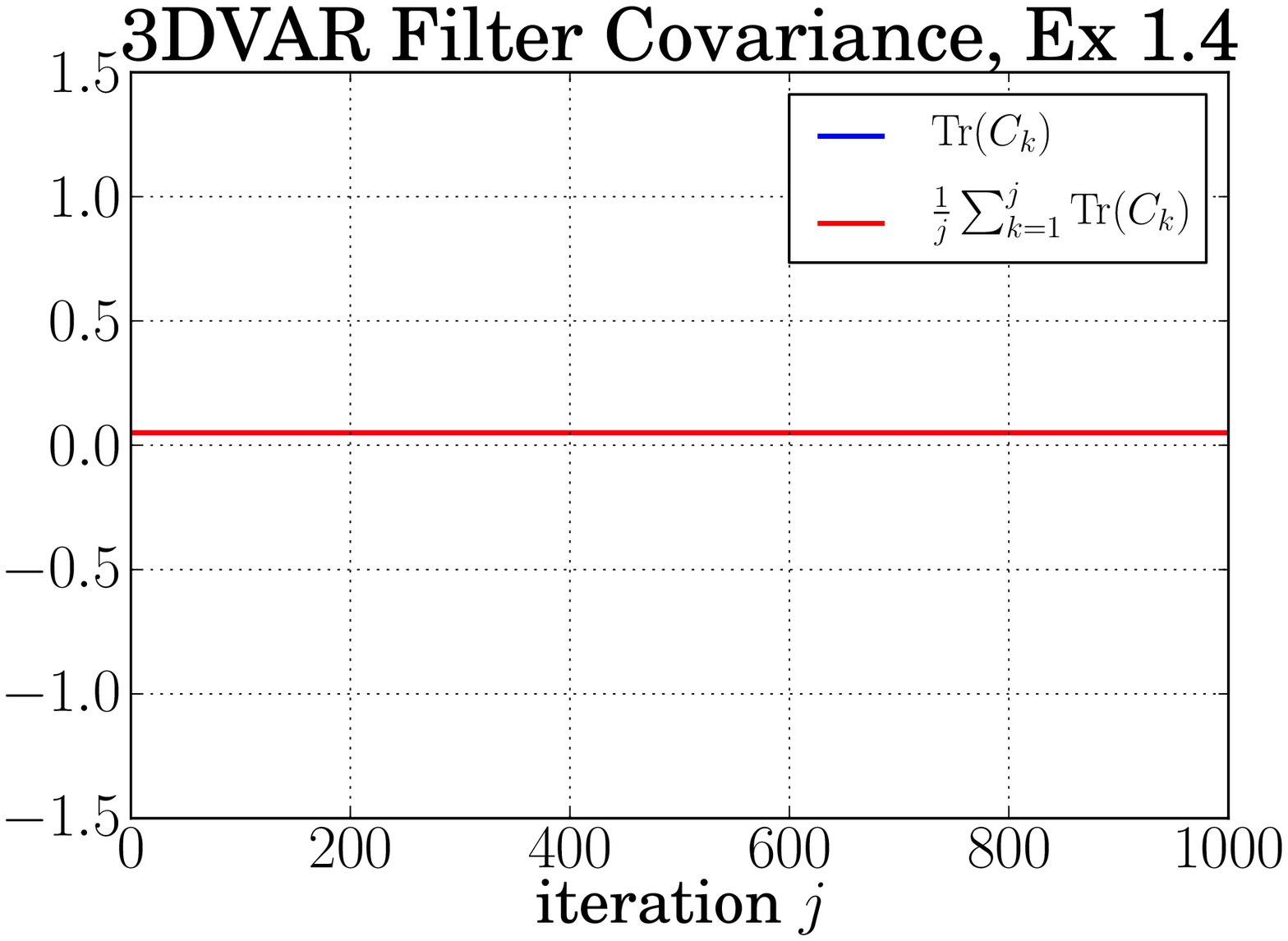}}
\subfigure[Error.]{\includegraphics[scale=0.365]{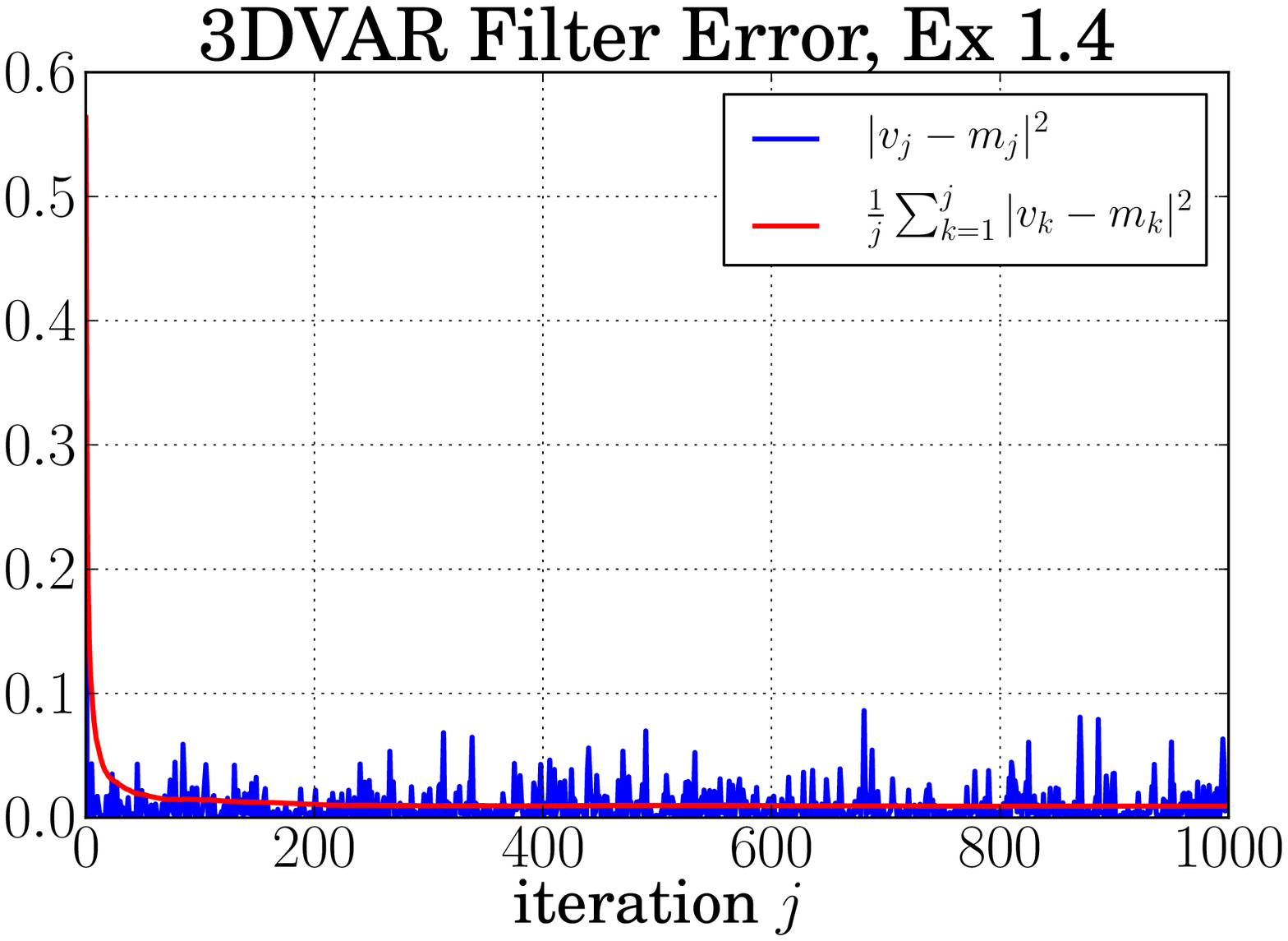}}
\caption{3DVAR\index{3DVAR} methodology applied to the logistic map
Example \ref{ex:ex4} with $r=4$, $\gamma^2=10^{-2}$,
and $c=\gamma^2/\eta$ with $\eta=0.2$, see also {\tt p9.m} in section \ref{ssec:p8}. }
\label{fig:3DVAR4}
\end{figure}

\begin{figure}[h]
\centering
\subfigure[Solution.]{\includegraphics[scale=0.365]{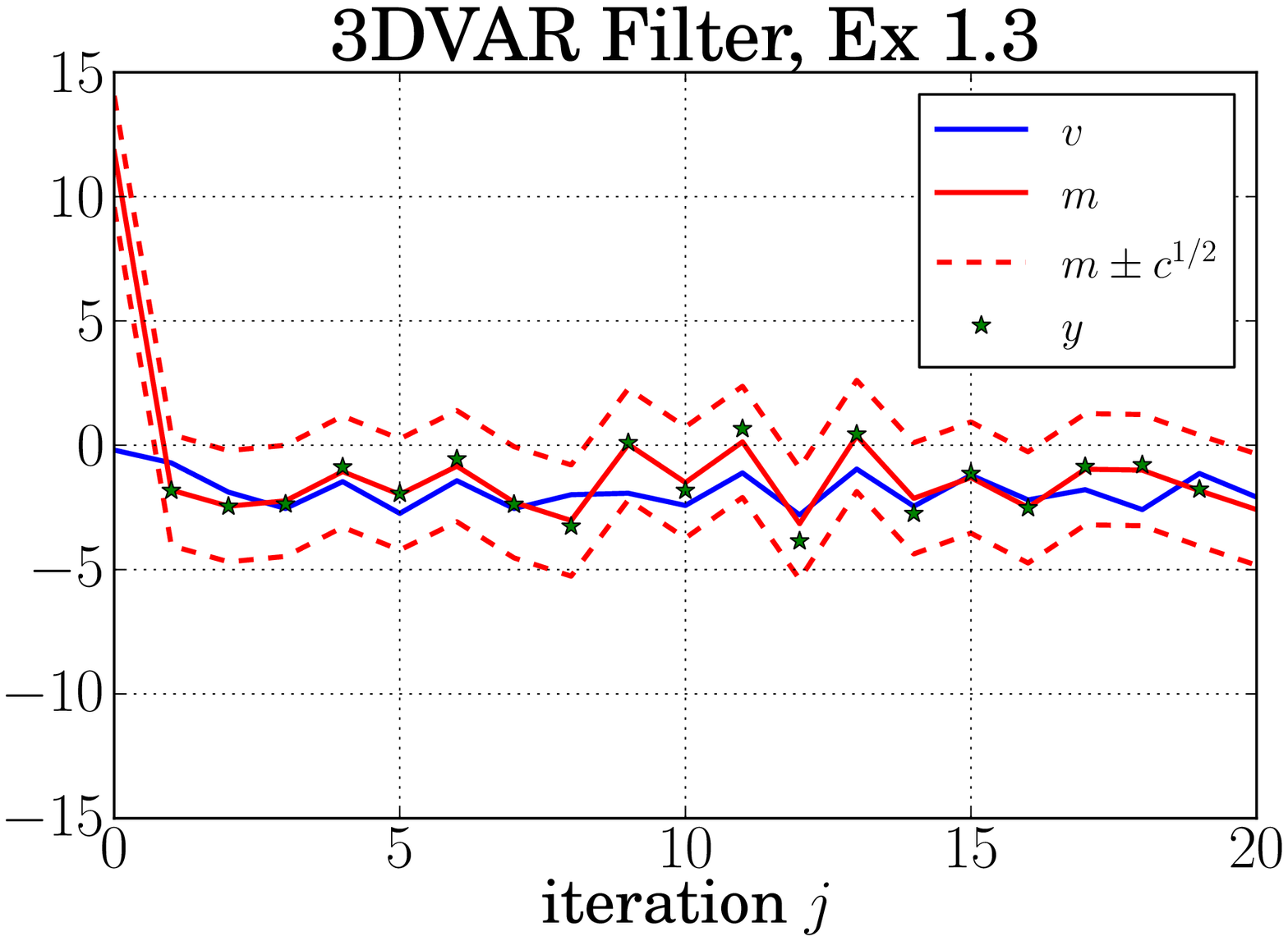}}
\subfigure[Covariance.]{\includegraphics[scale=0.365]{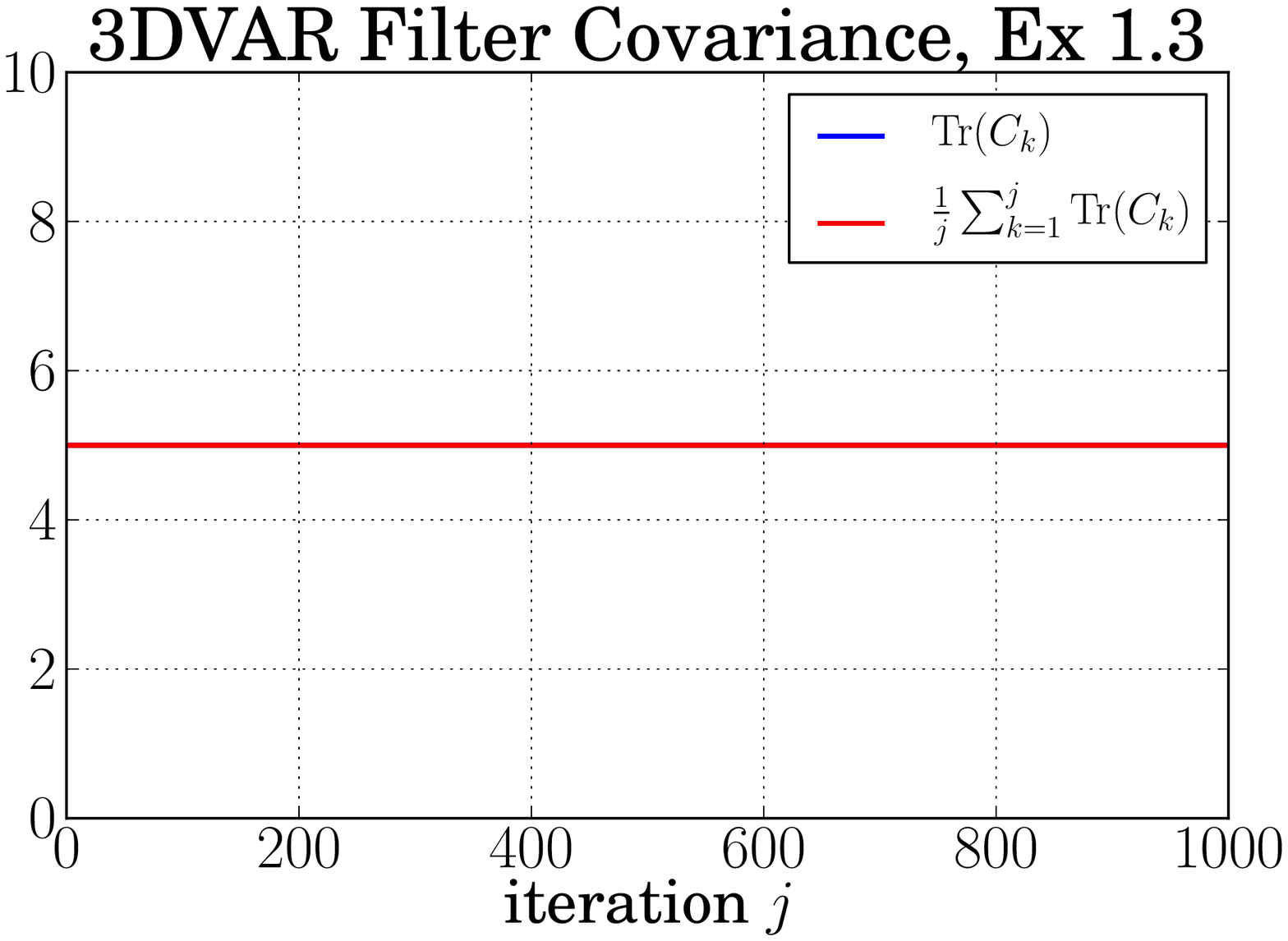}}
\subfigure[Error.]{\includegraphics[scale=0.365]{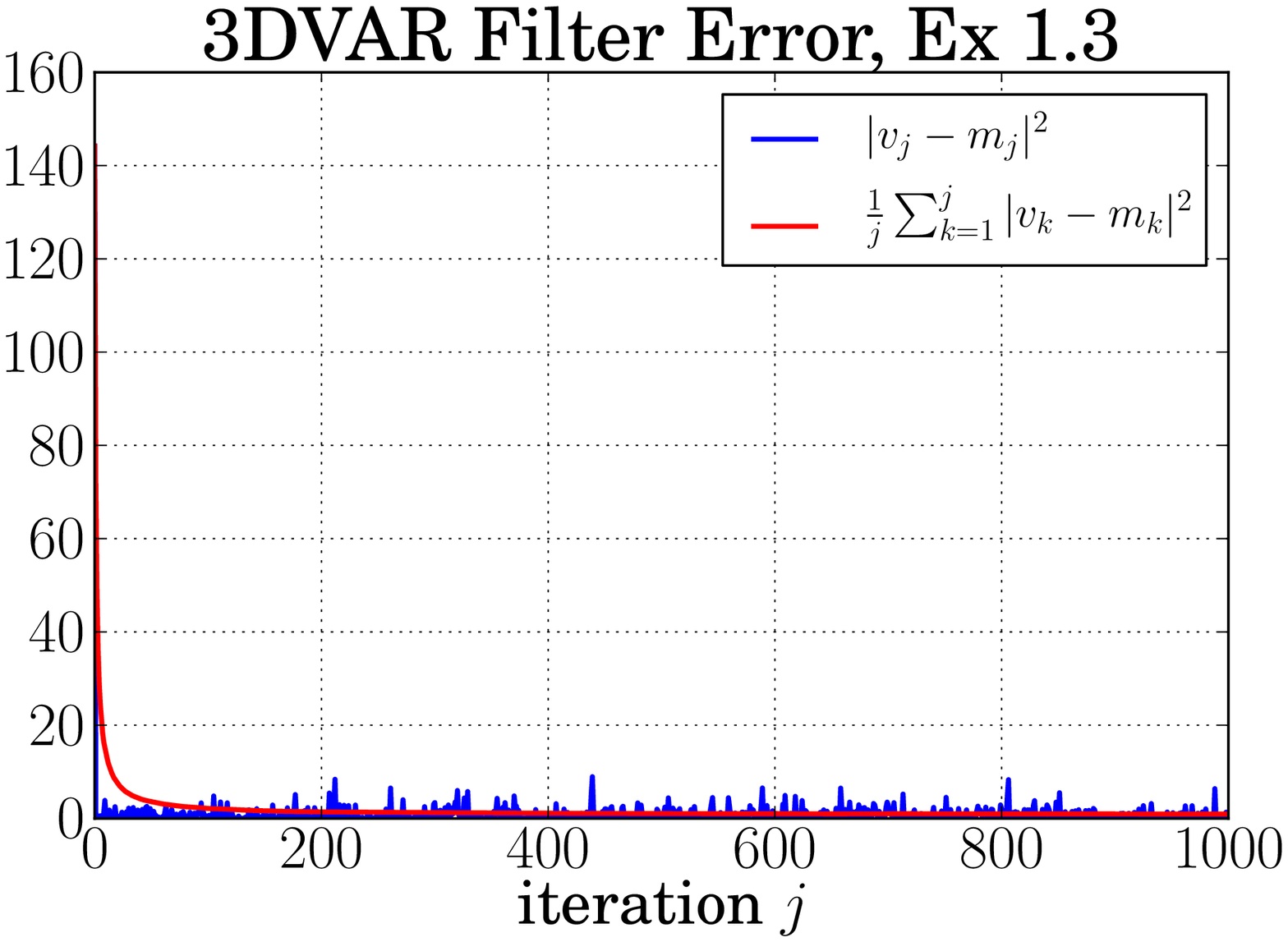}}
\caption{3DVAR\index{3DVAR} for the sin map Example
\ref{ex:ex3} with $\alpha=2.5$, $\sigma=0.3, \gamma=1$, 
and $\eta=0.2$, see also {\tt p10.m} in Section \ref{ssec:p9}.}
\label{fig:3DVAR3}
\end{figure}

\begin{figure}[h]
\centering
\subfigure[Solution.]{\includegraphics[scale=0.365]{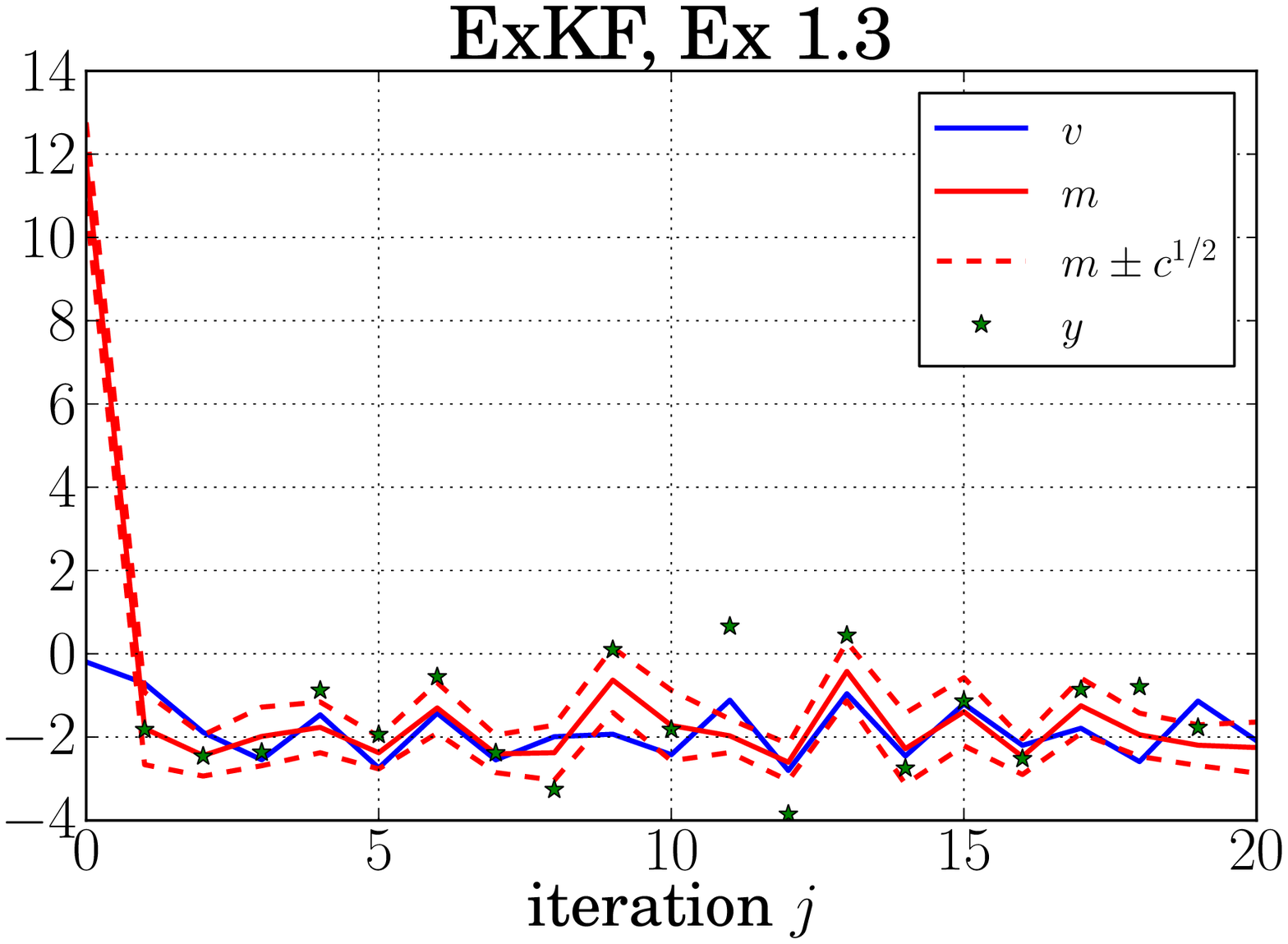}}
\subfigure[Covariance.]{\includegraphics[scale=0.365]{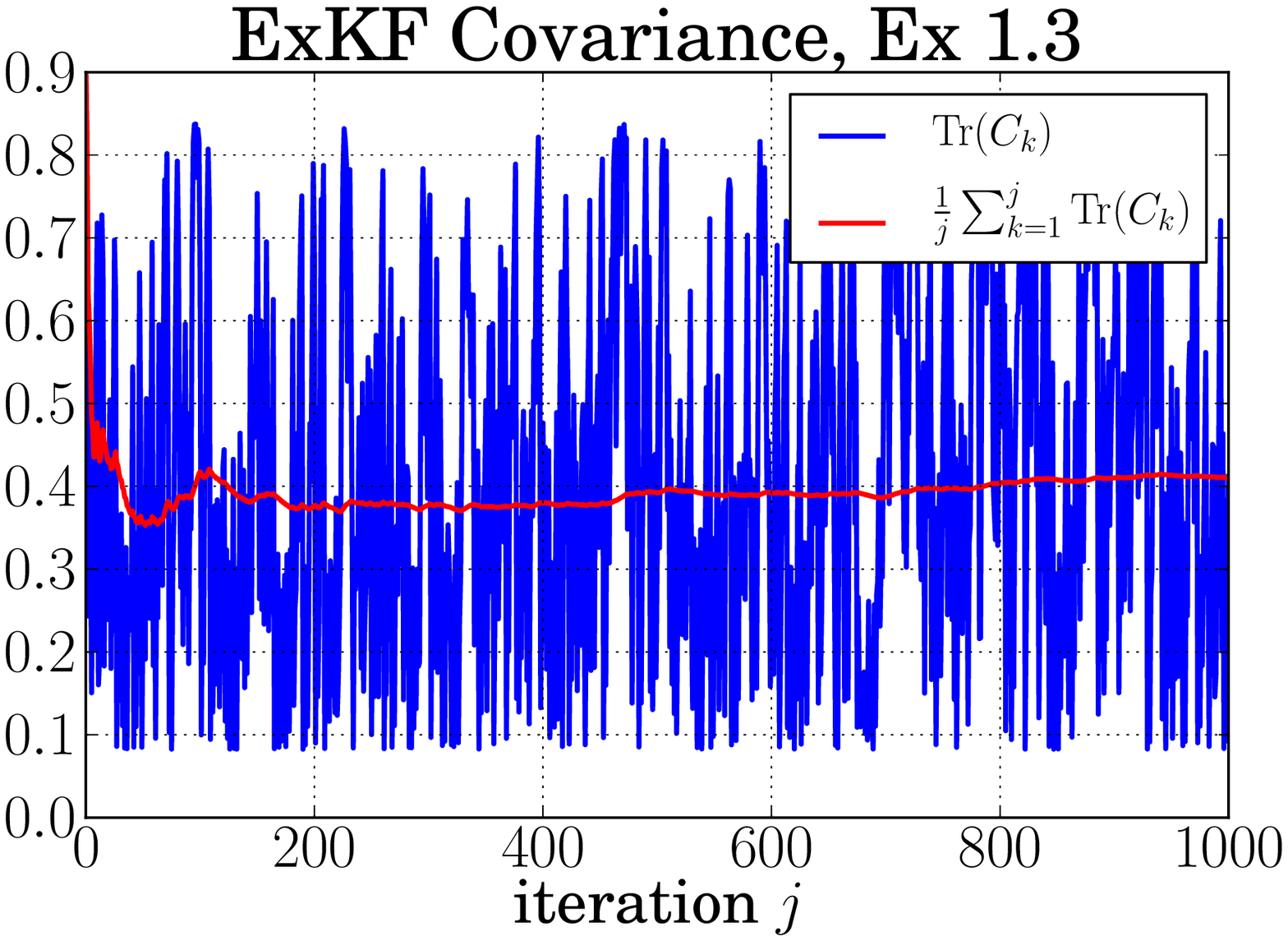}}
\subfigure[Error.]{\includegraphics[scale=0.365]{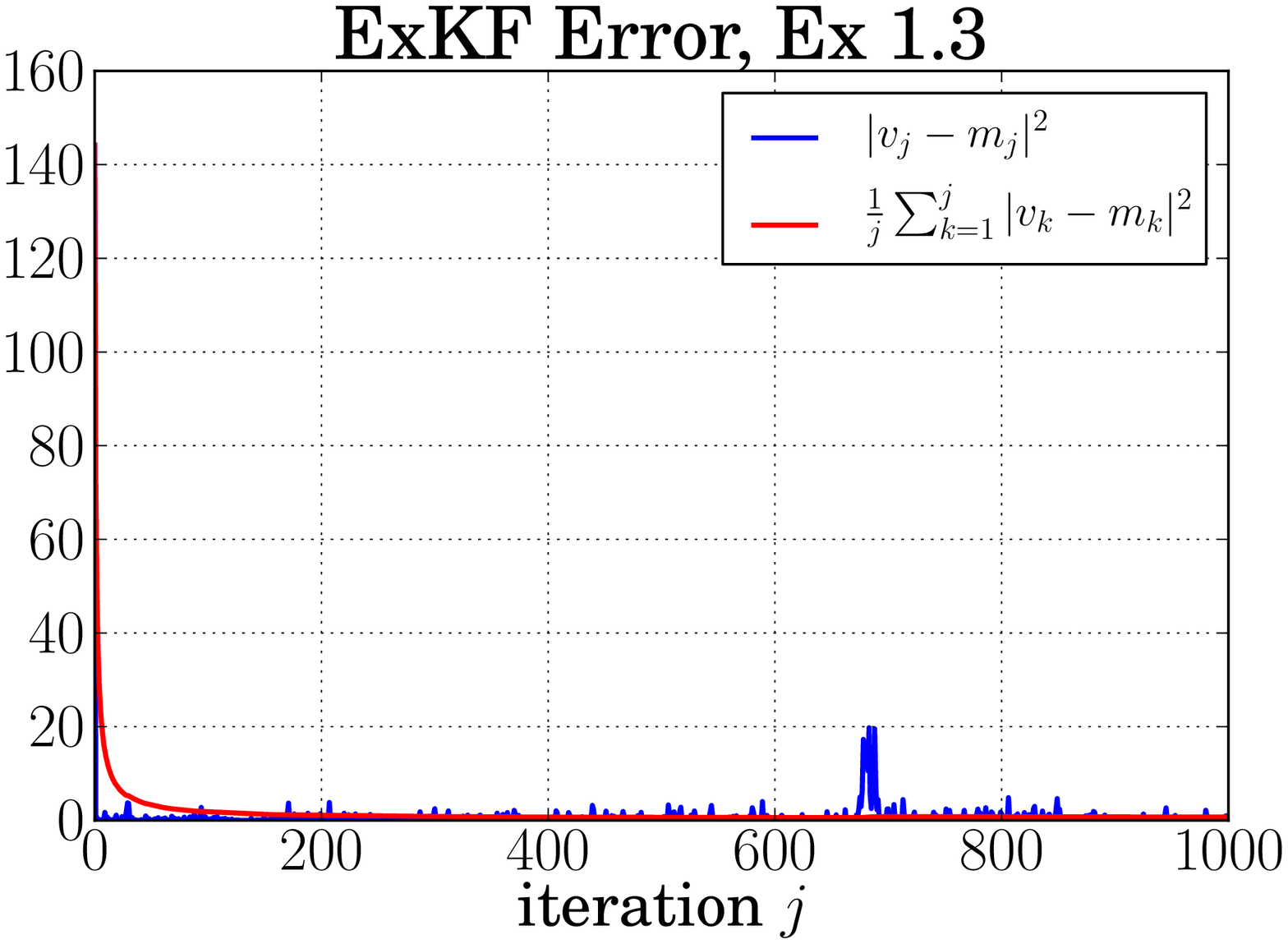}}
\caption{ExKF\index{Kalman filter!extended} on the sin map Example
\ref{ex:ex3} with $\alpha=2.5$, $\sigma=0.3,$ and $\gamma=1$, see also {\tt p11.m} in Section \ref{ssec:p10}.}
\label{fig:ExKF}
\end{figure}

\begin{figure}[h]
\centering
\subfigure[Solution.]{\includegraphics[scale=0.365]{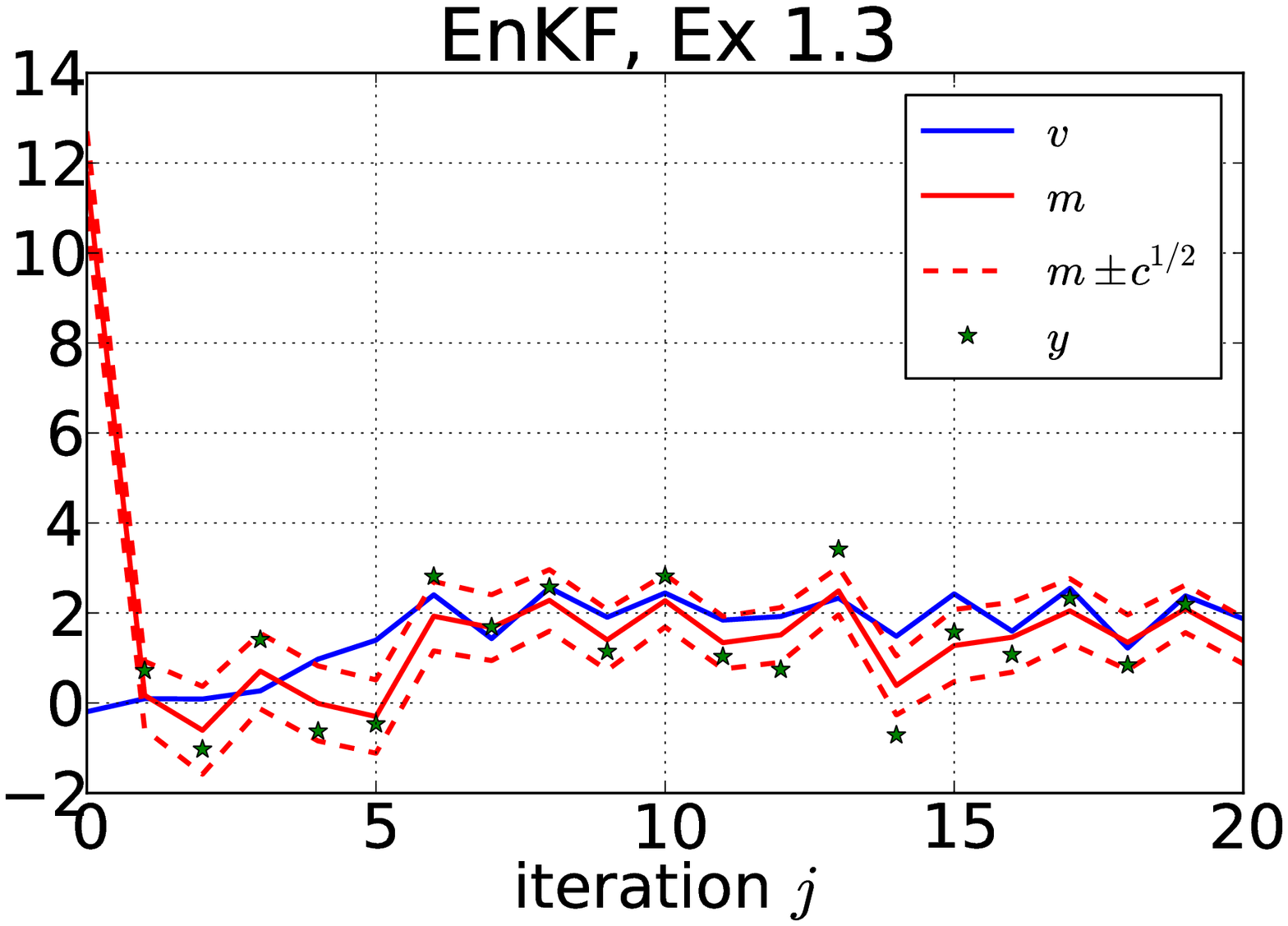}}
\subfigure[Covariance.]{\includegraphics[scale=0.365]{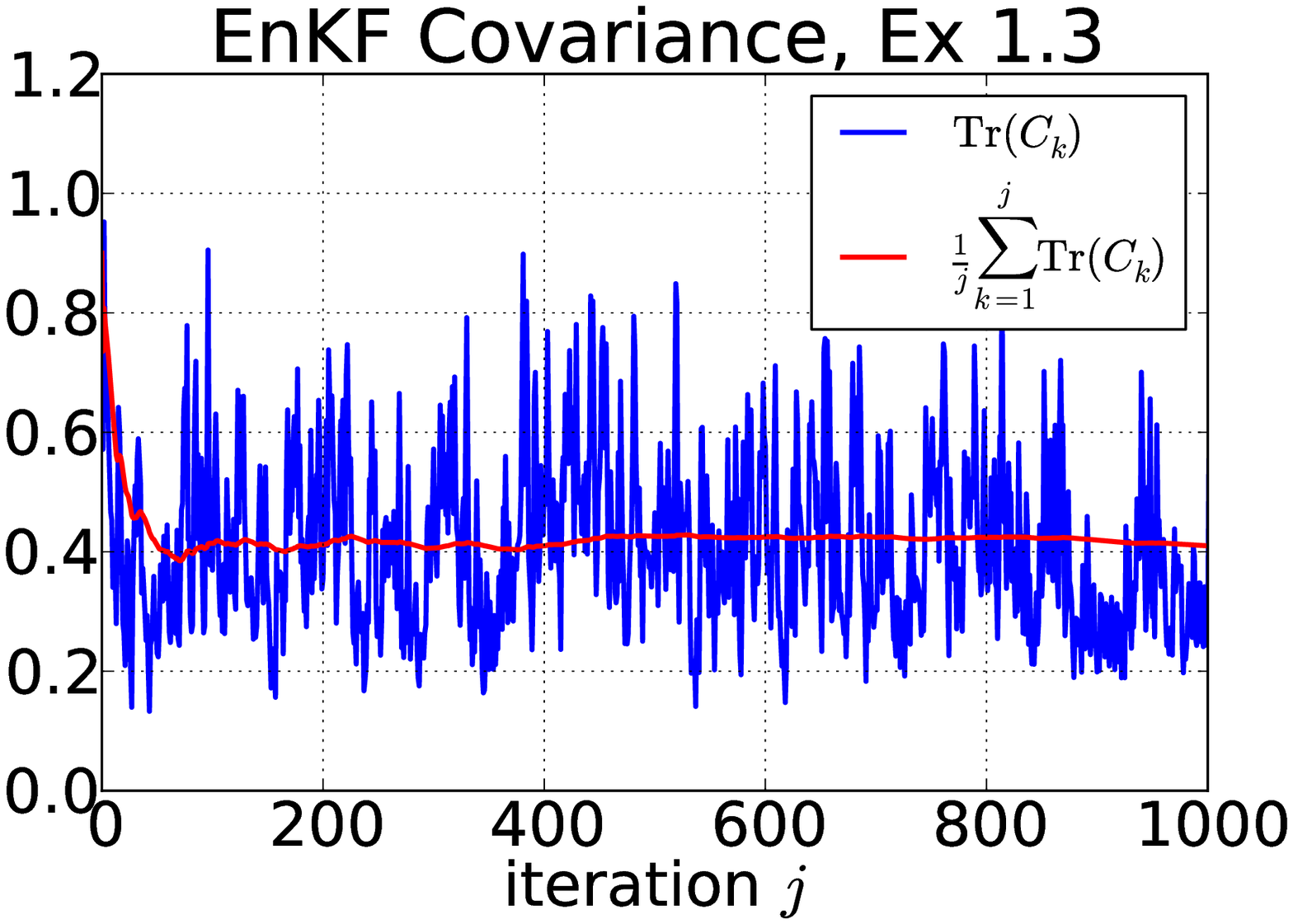}}
\subfigure[Error.]{\includegraphics[scale=0.365]{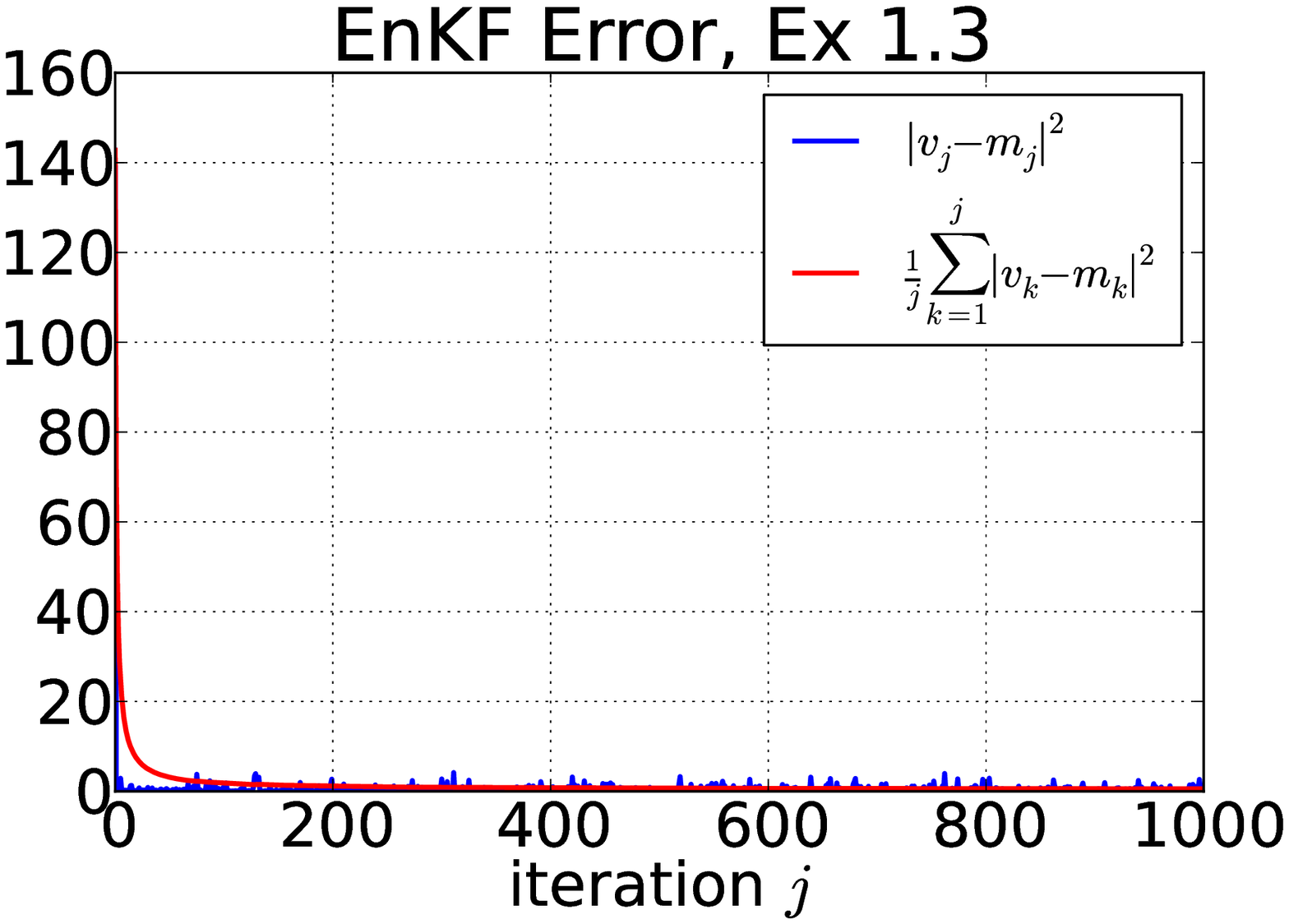}}
\caption{EnKF\index{Kalman filter!ensemble} on the sin map Example \ref{ex:ex3} with 
$\alpha=2.5$, $\sigma=0.3,$ $\gamma=1$ and $N=100$, see also {\tt p12.m} in Section \ref{ssec:p11}.}
\label{fig:EnKF}
\end{figure}

\begin{figure}[h]
\centering
\subfigure[Solution.]{\includegraphics[scale=0.365]{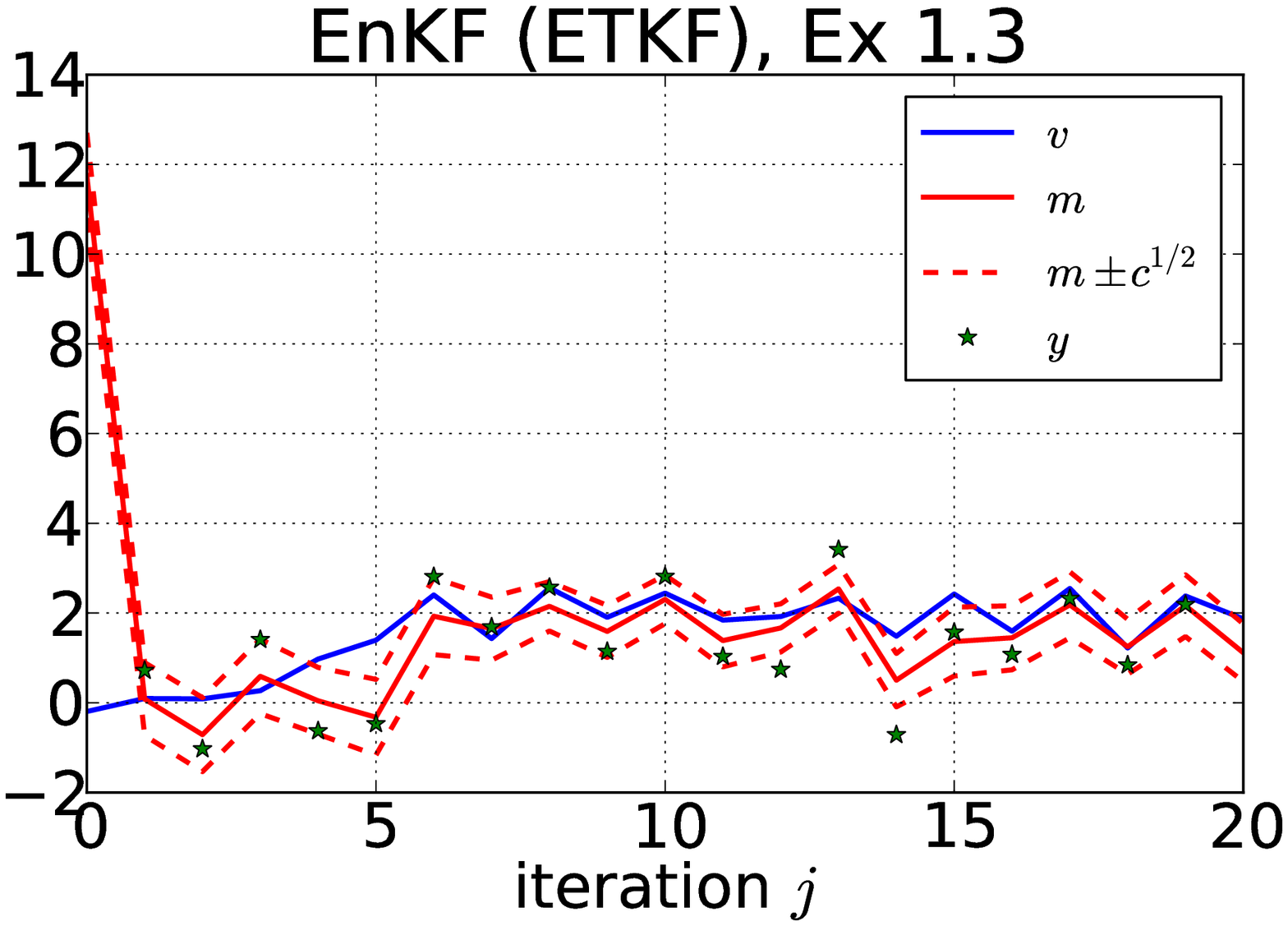}}
\subfigure[Covariance.]{\includegraphics[scale=0.365]{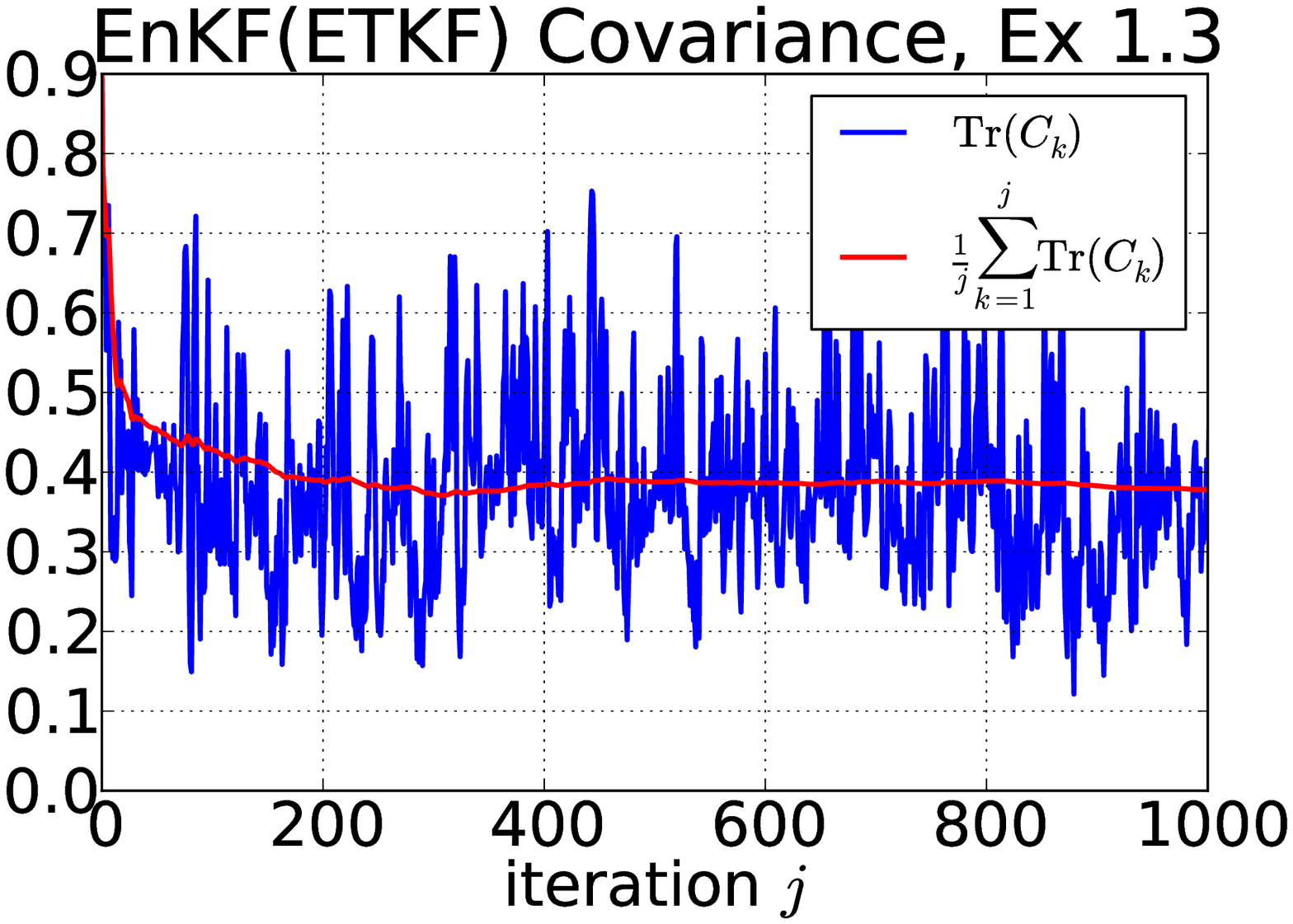}}
\subfigure[Error.]{\includegraphics[scale=0.365]{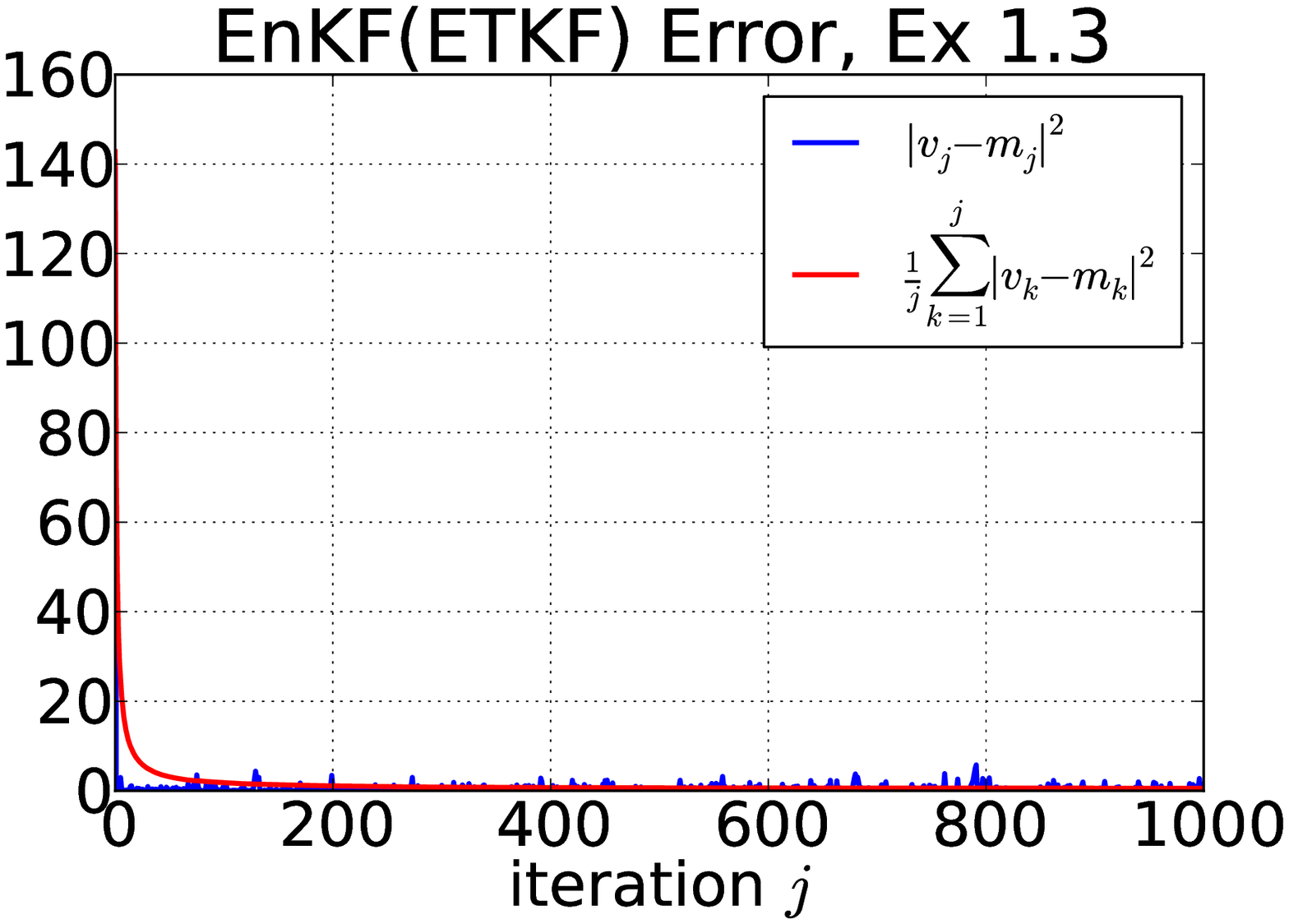}}
\caption{ETKF on the sin map Example \ref{ex:ex3} with 
$\alpha=2.5$, $\sigma=0.3,$ $\gamma=1$ and $N=100$, see also {\tt p13.m} in Section \ref{ssec:p12}.}
\label{fig:ETKF}
\end{figure}

\begin{figure}[h]
\centering
\subfigure[Solution.]{\includegraphics[scale=0.365]{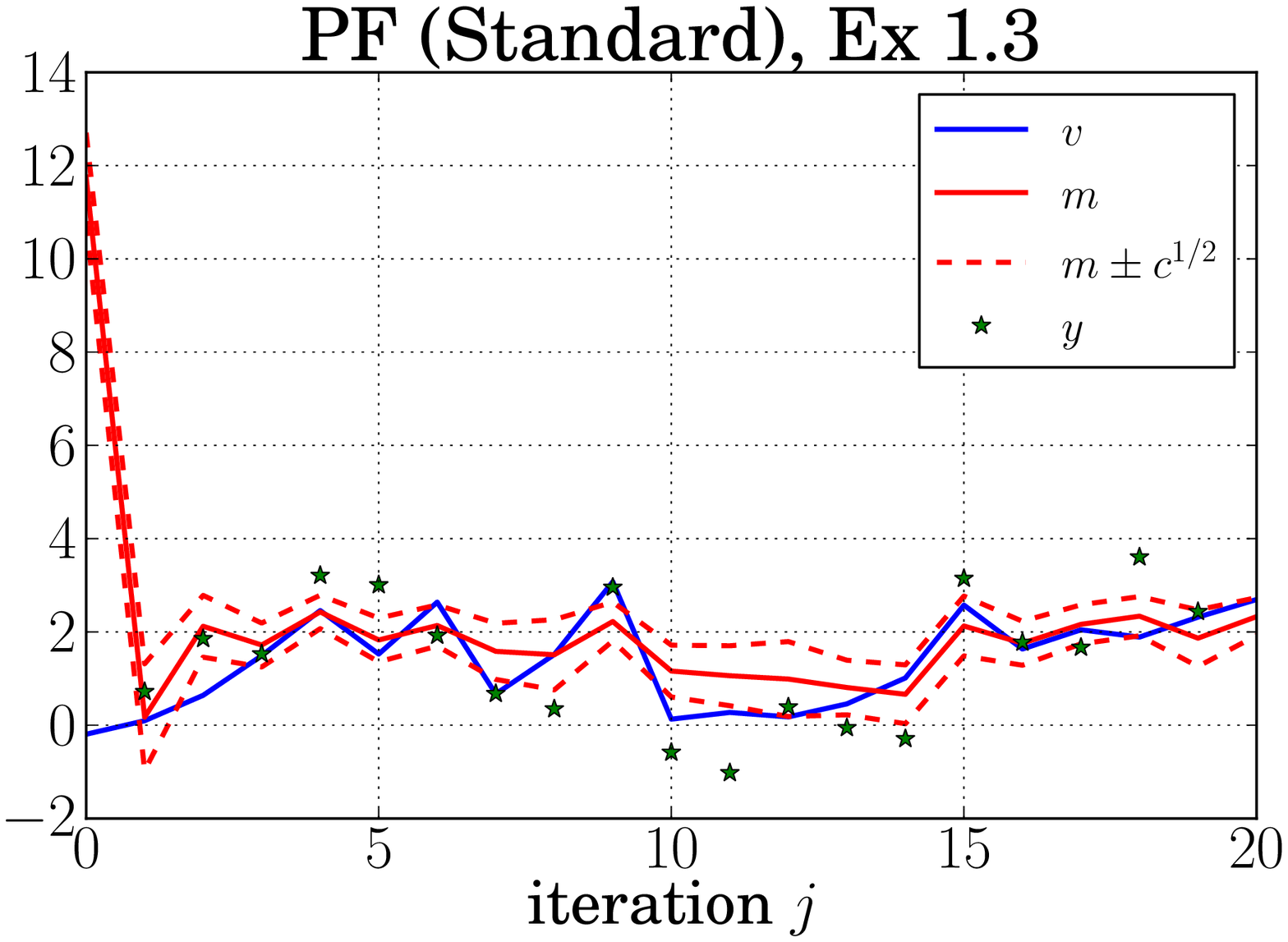}}
\subfigure[Covariance.]{\includegraphics[scale=0.365]{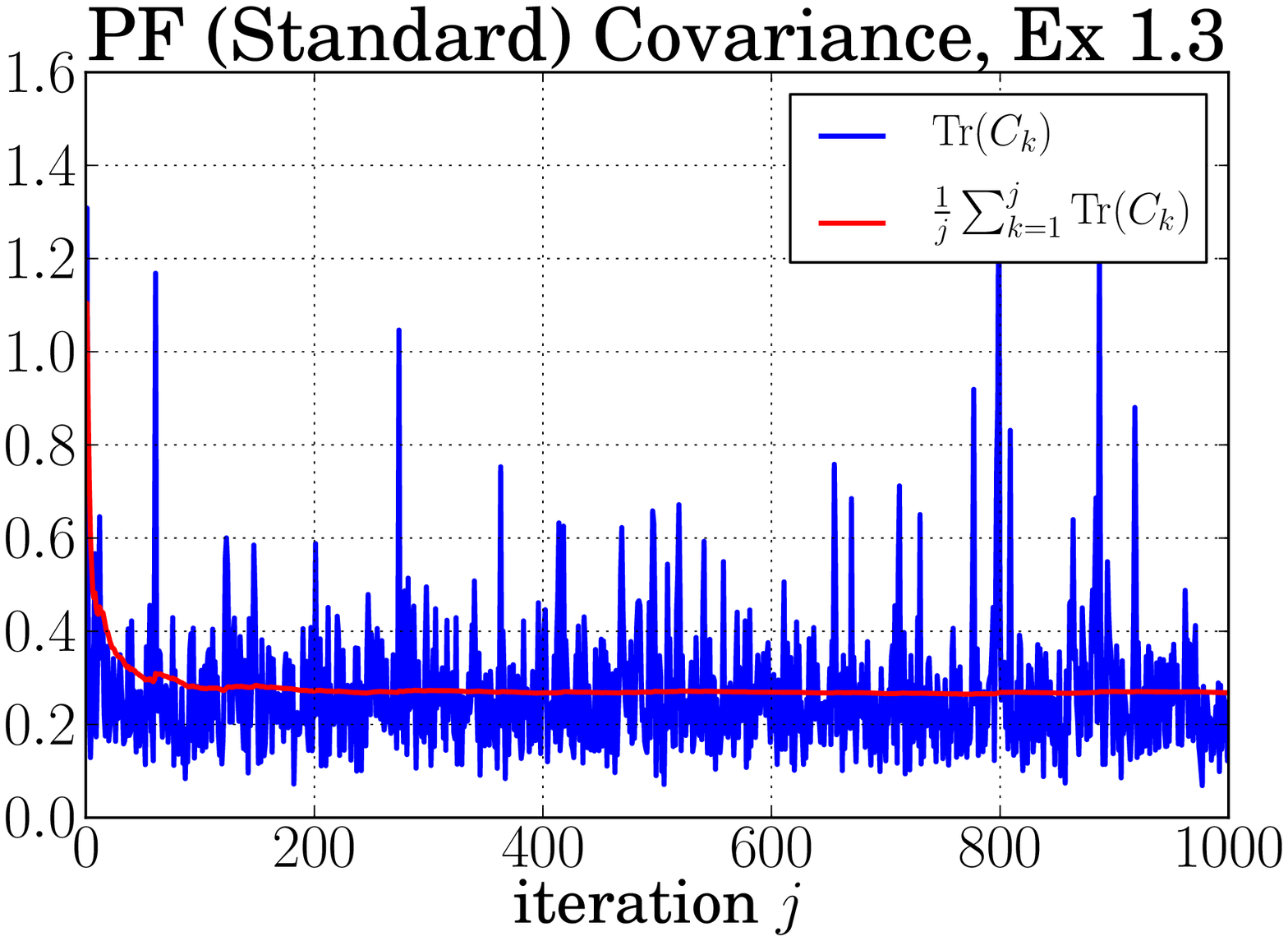}}
\subfigure[Error.]{\includegraphics[scale=0.365]{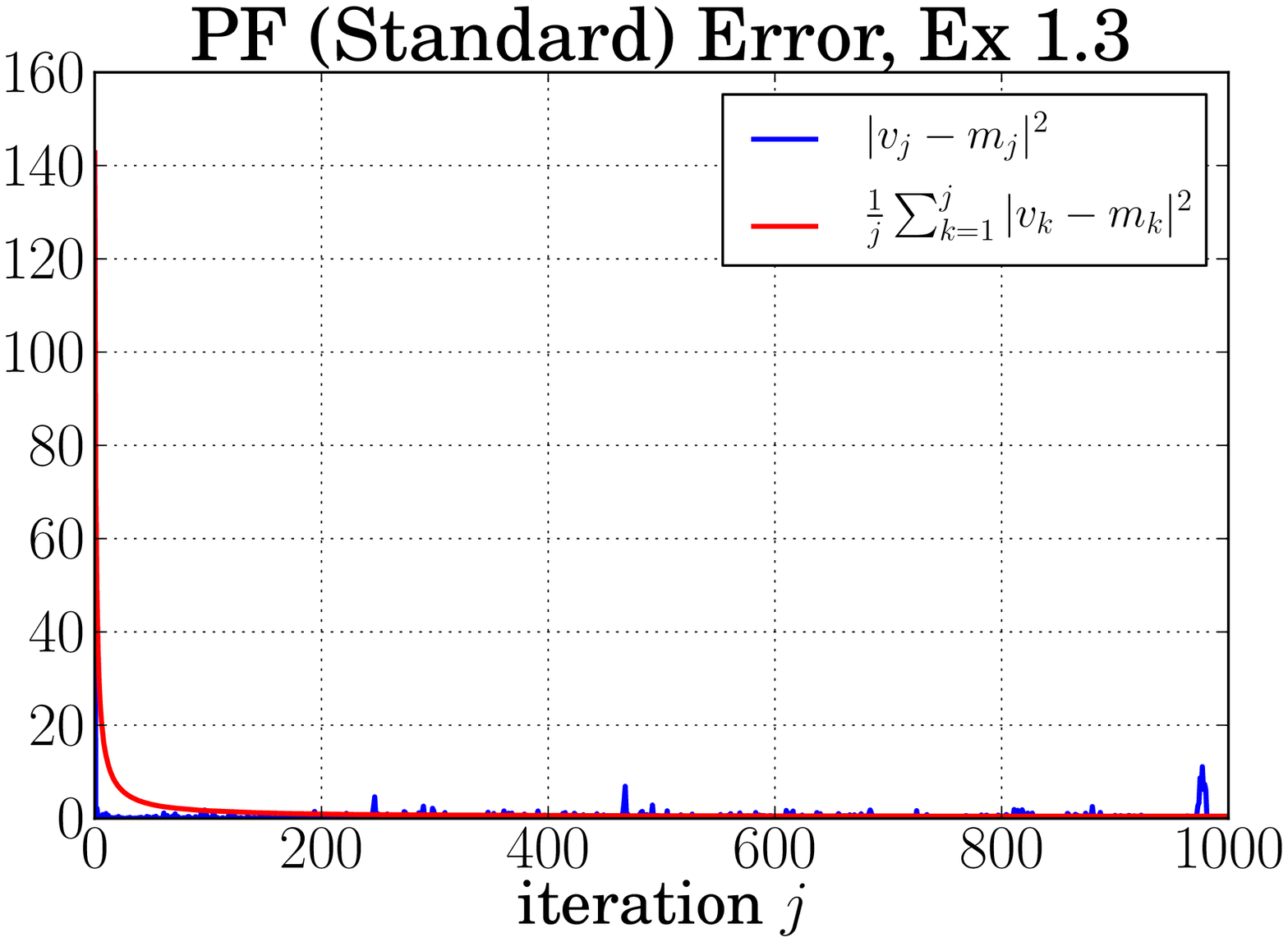}}
\caption{Particle Filter (standard proposal\index{proposal!standard}) 
on the sin map Example \ref{ex:ex3} with 
$\alpha=2.5$, $\sigma=0.3,$ $\gamma=1$ and $N=100$, see also {\tt p14.m} in Section \ref{ssec:p13}.}
\label{fig:PFST}
\end{figure}

\begin{figure}[h]
\centering
\subfigure[Solution.]{\includegraphics[scale=0.365]{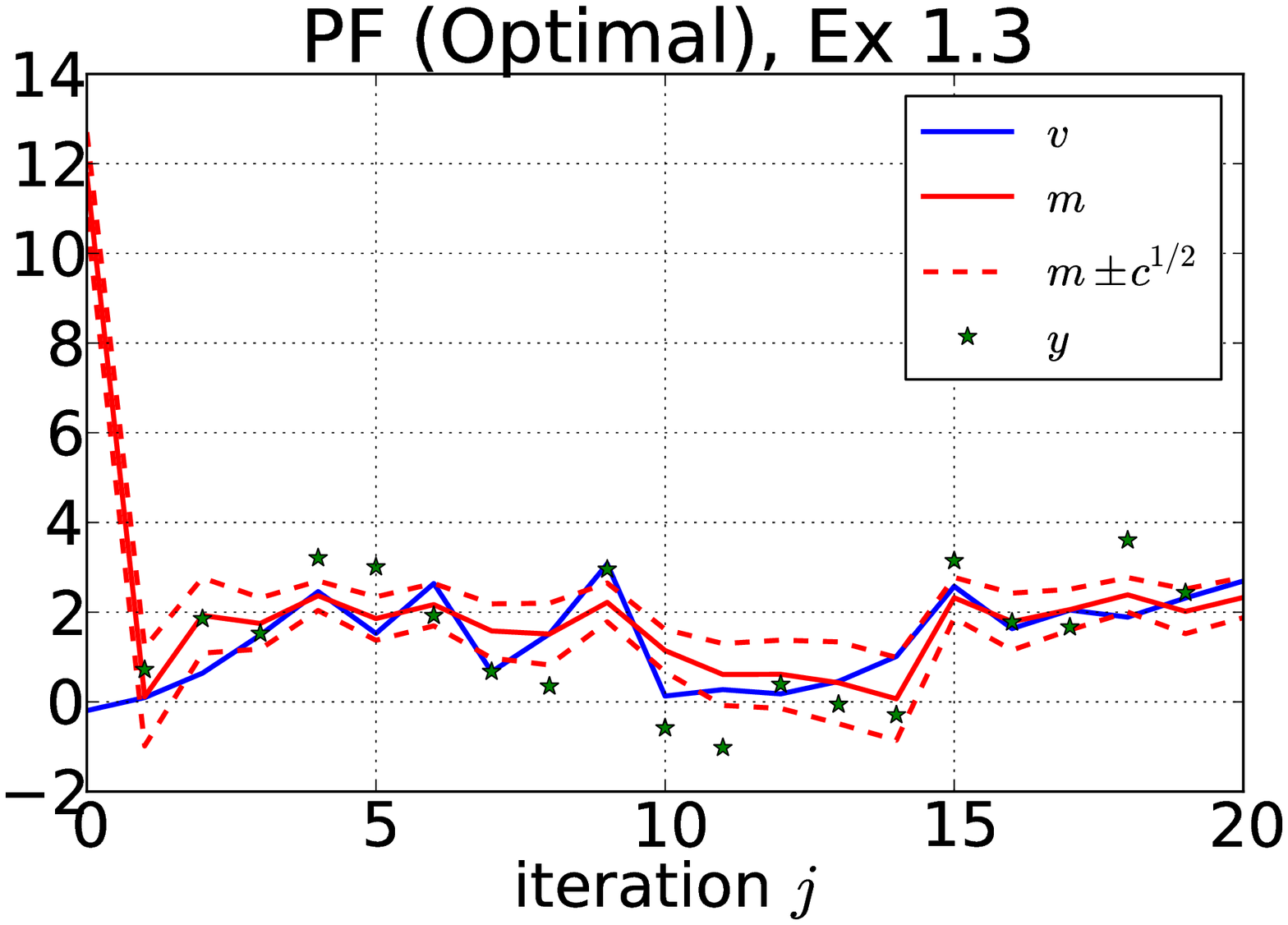}}
\subfigure[Covariance.]{\includegraphics[scale=0.365]{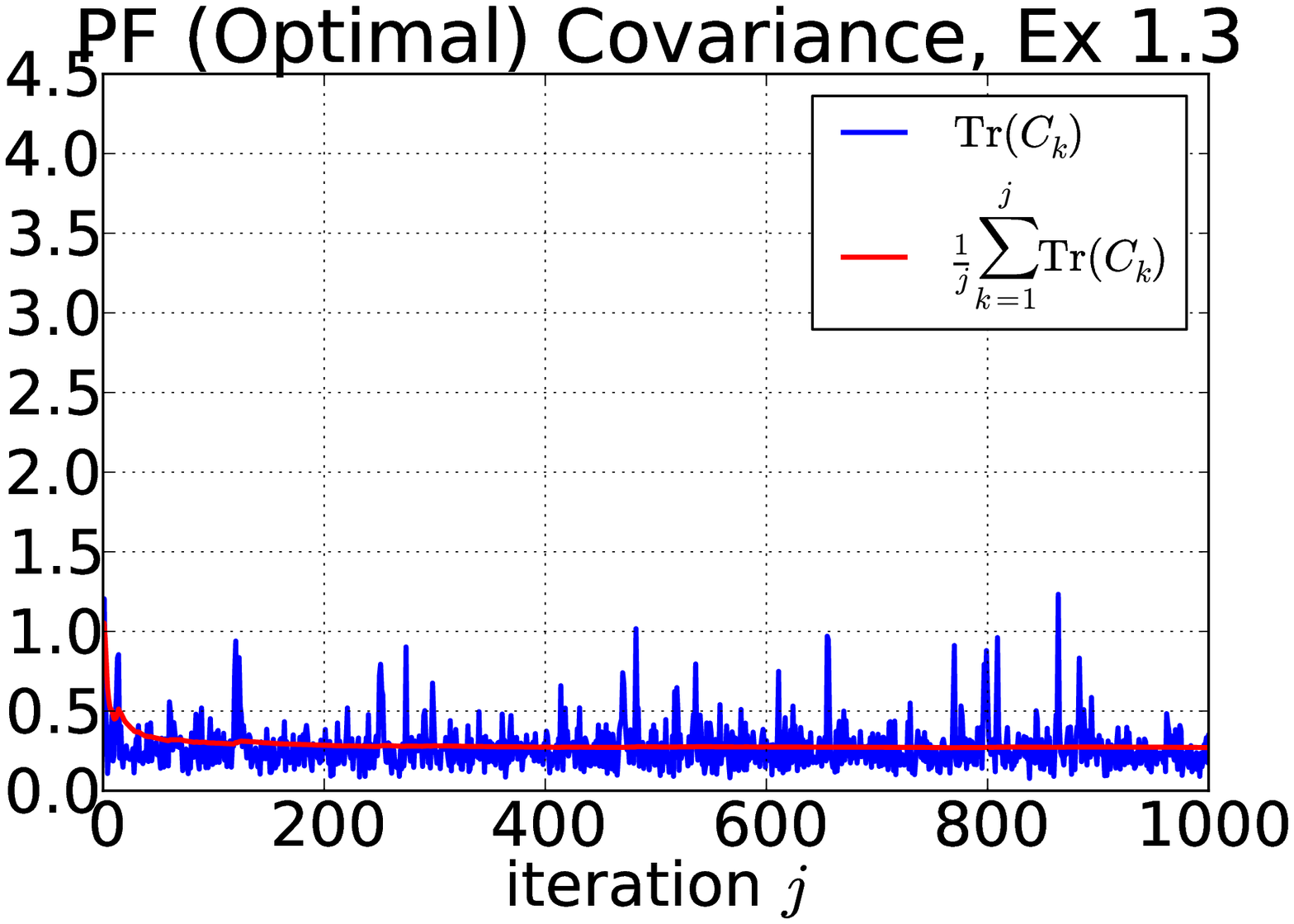}}
\subfigure[Error.]{\includegraphics[scale=0.365]{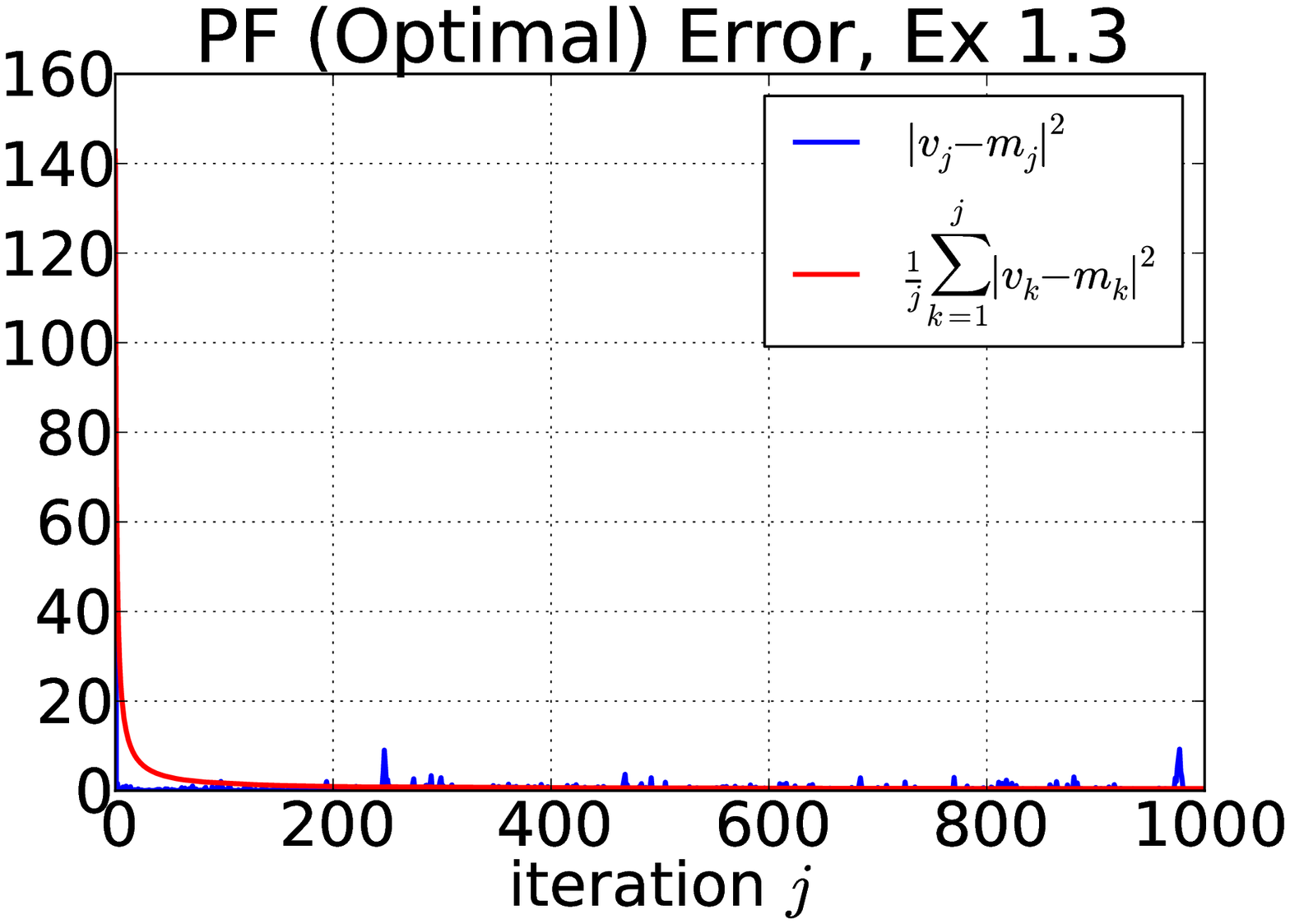}}
\caption{Particle Filter (optimal proposal\index{proposal!optimal}) on the sin map Example \ref{ex:ex3} with 
$\alpha=2.5$, $\sigma=0.3,$ $\gamma=1$ and $N=100$, see also {\tt p15.m} in section \ref{ssec:p14}.}
\label{fig:PFOP}
\end{figure}

\begin{figure}[h]
\centering
\subfigure[Log-scale histograms.]{\includegraphics[scale=0.365]{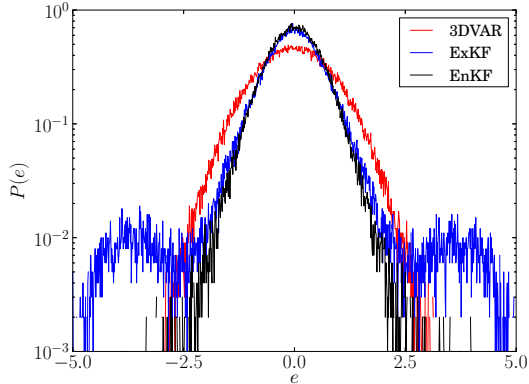}}
\subfigure[Running average root mean square $e$.]{\includegraphics[scale=0.365]{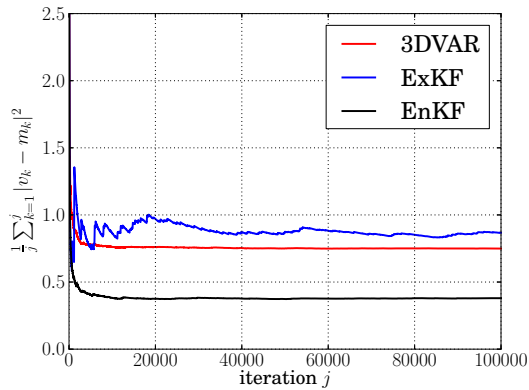}}
\caption{Convergence of $e=m-\vd$ for each filter for the sin map Example \ref{ex:ex3}, corresponding
to solutions from Figs. \ref{fig:3DVAR3}, \ref{fig:ExKF}, \ref{fig:EnKF}.}
\label{fig:error}
\end{figure}

\begin{figure}[h]
\centering
\subfigure[Log-scale histograms.]{\includegraphics[scale=0.365]{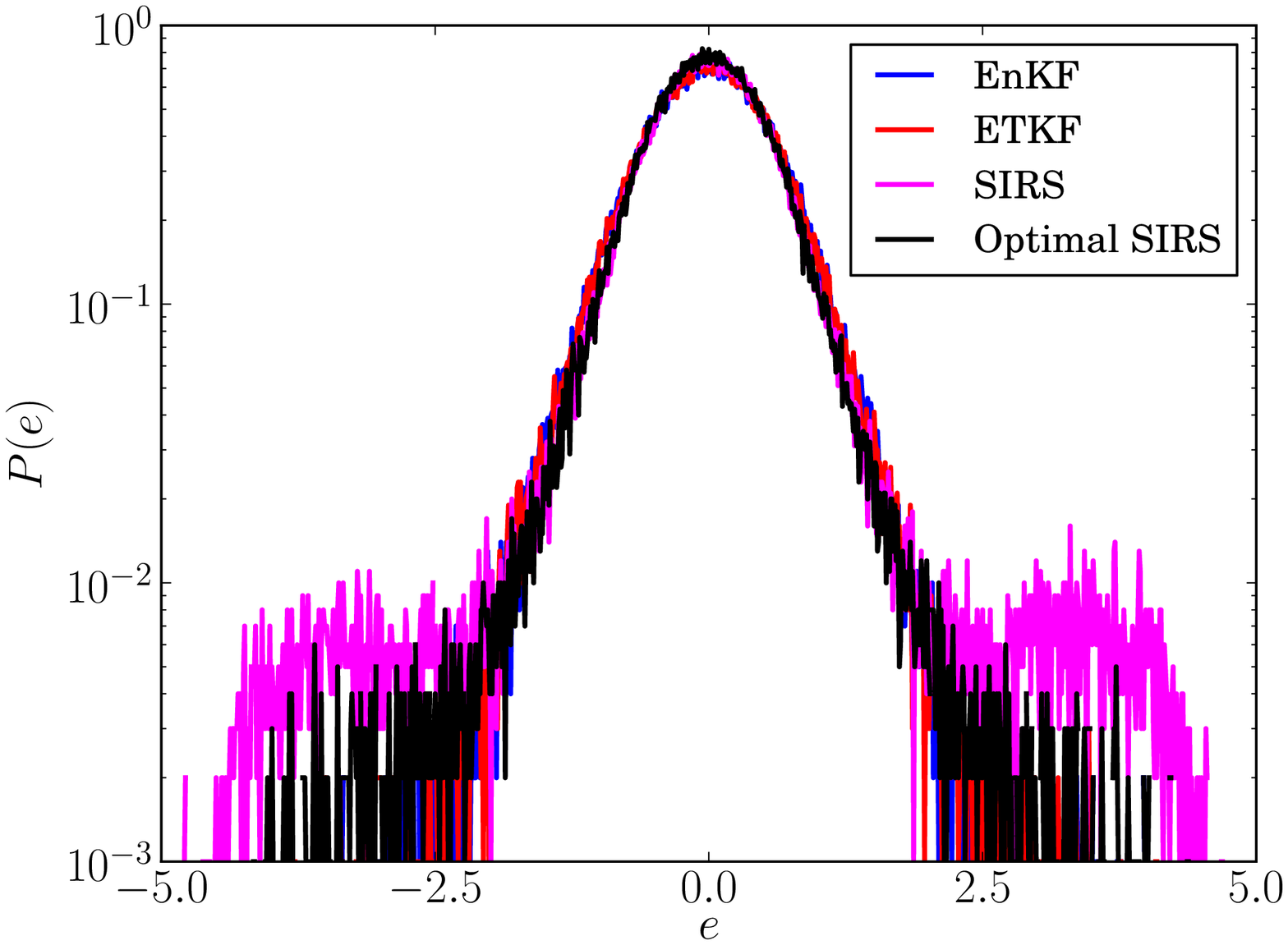}}
\subfigure[Running average root mean square $e$.]{\includegraphics[scale=0.365]{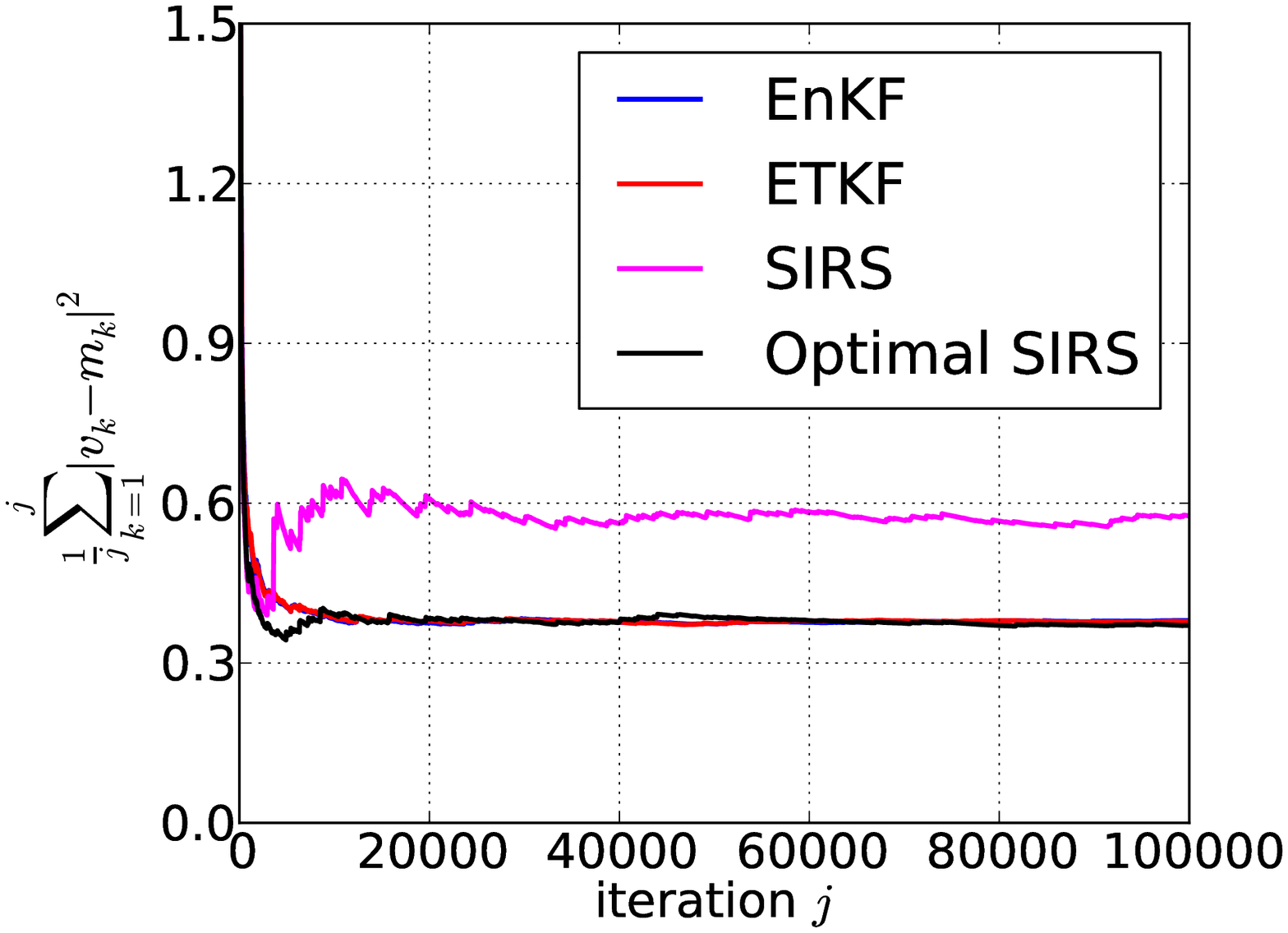}}
\caption{Convergence of $e=m-\vd$ for both versions of EnKF\index{Kalman filter!ensemble} in comparison to the
particle filter\index{particle filter}s for the sin map Ex. 1.3, corresponding
to solutions from Figs. \ref{fig:EnKF}, \ref{fig:ETKF}, \ref{fig:PFST}, and  \ref{fig:PFOP}.}
\label{fig:errorp}
\end{figure}

The first illustration concerns the Kalman filter\index{Kalman filter} applied to the
linear system of Example \ref{ex:ex2} with $A=A_3$. We assume that
$H=(1,0)$ so that we observe only the first component of the system
and the model and observational covariances are $\Sigma=I$ 
and $\Gamma=1$, where
$I$ is the $2 \times 2$ identity.
The problem is initialized with mean $0$ and covariance $10\,I$. 
Figure \ref{fig:KF}a shows the
behaviour of the filter on the unobserved component, showing how the
mean locks onto a small neighbourhood of the truth and how the
one-standard deviation confidence intervals computed from the variance on the
second component also shrink from a large initial value to an asymptotic
small value; this value is determined by the observational noise variance
in the first component. In Figure \ref{fig:KF}b
the trace of the covariance matrix is plotted demonstrating 
that the total covariance matrix asymptotes to a small limiting matrix.
And finally Figure \ref{fig:KF}c shows the error (in the Euclidean norm)
between the filter mean and the truth underlying the data, together with its
running average.
We will employ similar figures (a), (b) and (c) in the examples
which follow in this section.

The next illustration shows the 3DVAR\index{3DVAR} algorithm applied to the
Example \ref{ex:ex4} with $r=2.5.$ We consider noise-free dynamics
and observational variance of $\gamma^2=10^{-2}.$ The fixed
model covariance is chosen to be $c=\gamma^2/\eta$ with $\eta=0.2$. 
The resulting algorithm performs well at tracking the truth with asymptotic 
time-averaged Euclidean error of size roughly $10^{-2}.$
See Figure \ref{fig:3DVAR4}.

The rest of the figures illustrate the behaviour of the various
filters, all applied to the Example
\ref{ex:ex3} with $\alpha=2.5, \sigma=0.3,$ and $\gamma=1$.
In particular, 3DVAR\index{3DVAR}(Figure \ref{fig:3DVAR3}), 
ExKF\index{Kalman filter!extended}(Figure \ref{fig:ExKF}),
EnKF\index{Kalman filter!ensemble} (Figure \ref{fig:EnKF}), 
ETKF (Figure \ref{fig:ETKF}), 
and the particle filter \index{particle filter} with standard (Figure \ref{fig:PFST})
and optimal (Figure \ref{fig:PFOP})
proposals\index{proposal} are all compared on the same example.
The ensemble-based methods all use $100$ ensemble members each (notice
this is much larger than the dimension of the state space which is $n=1$ here,
and so a regime outside of which the ensemble methods would usually be employed
in practice).
For 3DVAR\index{3DVAR}, results from which (for this example)
are only shown in the summary Figure \ref{fig:error}, we take $\eta=0.5$.  

All of the methods perform well at tracking the true signal,
asymptotically in time, recovering from a large initial error.
However they also all exhibit occasional instabilities, and lose
track of the true signal for short periods of time.
From Fig. \ref{fig:ExKF}(c) we can observe 
that the ExKF\index{Kalman filter!extended} has small error for most of
the simulation, but that sporadic large 
excursions are seen in the error.
From Fig. \ref{fig:ETKF}(c) one can observe that ETKF is similarly prone to 
small destabilization and local instability as the 
EnKF\index{Kalman filter!ensemble} with perturbed observations\index{perturbed observations} in Fig. \ref{fig:EnKF}(c). 
 Also, notice from Figure \ref{fig:PFST}(c) that the particle filter\index{particle filter} with
standard proposal\index{proposal!standard} 
is perhaps slightly more prone to destabilization than 
the optimal proposal \index{proposal!optimal} 
in Figure \ref{fig:PFOP}(c), although the difference is minimal. 

The performance of the filters is now compared through a detailed
study of the statistical properties of the error $e = m - v^\dagger$, 
over long simulation times. In particular we compare the histograms of the 
errors, and their large time averages.
Figure \ref{fig:error} compares the errors incurred by the three basic methods
3DVAR \index{3DVAR}, ExKF, and EnKF, demonstrating that the EnKF\index{Kalman filter!ensemble} 
is the most accurate method of the three on average, with 
ExKF\index{Kalman filter!extended} the least accurate on average. 
Notice from Fig. \ref{fig:error}(a) that the error distribution of 
3DVAR\index{3DVAR} is the widest, and both it and EnKF\index{Kalman filter!ensemble} remain consistently 
accurate.  The distribution of ExKF\index{Kalman filter!extended}
is similar to EnKF\index{Kalman filter!ensemble}, except with "fat tails" associated to the 
destabilization intervals seen in Fig. \ref{fig:ExKF}.
 
Figure \ref{fig:errorp} compares the errors incurred by the four more
accurate ensemble-based methods EnKF, ETKF, SIRS, and SIRS(OP).
The error distribution, Fig.  \ref{fig:errorp}(a) of all these filters is similar.
In Fig.  \ref{fig:errorp}(b) one can see that the time-averaged  error is
indistinguishable between EnKF and ETKF. 
Also, the EnKF, ETKF, and SIRS(OP) also remain more or less 
consistently accurate. 
 The distribution of $e$ for SIRS is similar 
to SIRS(OP), except with fat tails associated to the 
destabilization intervals seen in Fig. \ref{fig:PFST}, which leads to the larger 
time-averaged error seen in Fig.  \ref{fig:errorp}(b).  
In this sense, the distribution of $e$ is similar to that for
ExKF\index{Kalman filter!extended}.

\section{Bibliographic Notes} 
\label{ssec:dtnb3}

\begin{itemize}

\item Section \ref{ssec:tkf} The Kalman Filter
has found wide-ranging application to low dimensional engineering
applications where the linear Gaussian model is appropriate,
since its introduction in 1960 \cite{kalman1960new}.
In addition to the original motivation in control of flight vehicles,
it has grown in importance
in the fields of econometric time-series analysis, and
signal processing \cite{harvey1991forecasting}.
It is also important because it plays a key role in
the development of the  approximate Gaussian
filters\index{filter!approximate Gaussian} which are the subject of section \ref{ssec:ngf}.
The idea behind the Kalman filter, to optimally combine model
and data, is arguably one of the most important ideas in
applied mathematics over the last century: the impact of the
paper \cite{kalman1960new} on many applications domains
has been huge.

\item Section \ref{ssec:ngf} All the non-Gaussian Filters 
we discuss are based on modifying the Kalman filter\index{Kalman filter} so that
it may be applied to non-linear problems. The development
of new filters is a very active area of research
and the reader is directed to the book \cite{majda2012filtering},
together with the articles \cite{CMT10},\cite{GHM09} and \cite{VL09} for insight into some of the recent developments with an
applied mathematics perspective.

The 3DVAR\index{3DVAR} algorithm was proposed at the UK
Met Office in 1986 \cite{article:Lorenc1986, article:Lorenc2000}, 
and was subsequently developed by the US National Oceanic and 
Atmospheric Administration \cite{article:Parrish1992}
and by the European Centre for Medium-Range Weather Forecasts (ECMWF)
in \cite{article:Courtier1998}. The perspective of these
papers was one of minimization and, as such, easily incorporates
nonlinear observation operators via the objective functional \eqref{eq:J3},
with a fixed $\hC=\hC_{j+1}$, for the analysis step of filtering;\index{filtering} 
nonlinear observation operators are important in numerous applications, 
including numerical weather forecasting.
In the case of linear observation operators the
objective functional is given by \eqref{eq:dtfa8} with explicit solution
given, in the case $\hC=\hC_{j+1}$, by \eqref{eq:dtfa10}. In fact the method of
{\em optimal interpolation}\index{optimal interpolation} 
predates 3DVAR \index{3DVAR} and takes the linear equations \eqref{eq:dtfa10} as
the starting point, rather than starting from a minimization principle;
it is then very closely related to the method of krigging \index{krigging}
from the geosciences \cite{Tar05}.
The 3DVAR\index{3DVAR} algorithm is important because it is 
prototypical of the many more sophisticated filters which are
now widely used in practice and it is thus natural to study it.

The extended Kalman filter\index{Kalman filter!extended}
was developed in the control theory community and is discussed
at length in \cite{jaz70}. It is not practical to implement
in high dimensions, and low-rank extended Kalman filter\index{Kalman filter!extended}s
are then used instead; see \cite{LS12} for a recent
discussion. 

The ensemble Kalman filter\index{Kalman filter!ensemble}
uses a set of particles to estimate covariance information,
and may be viewed as an approximation of the extended
Kalman filter, designed to be suitable in high dimensions.
See \cite{eve06} for an overview of the methodology, written
by one of its originators, and \cite{EVL96} for an early
example of the power of the method.
We note that the minimization principle \eqref{eq:minn} has
the very desirable property that the samples $\{\hv_{n+1}^{(n)}\}_{n=1}^N$ 
correspond, to samples of the Gaussian distribution
found by Bayes theorem with prior $N(\hm_{j+1},\hC_{j+1})$ 
likelihood  \index{likelihood} $y_{j+1}|v$. This is the idea behind the
{\em randomized maximum likelihood} method described in \cite{orl08},
and widely used in petroleum applications; the idea is discussed
in detail in the context of the EnKF in \cite{KLS13}. 
There has been some analysis of the EnKF in the large sample limit;
see for example \cite{li2008numerical,le2009large,mandel2011convergence}. 
However, the primary power of the method for practitioners is that
it seems to provide useful information for small sample sizes; it is
therefore perhaps a more interesting direction for analysis to study
the behaviour of the algorithm, and determine methodologies to improve it,
for fixed numbers of ensemble members. There is some initial work in
this direction and we describe it below.

Note that the $\Gamma$ appearing in the perturbed observation\index{perturbed observations}
EnKF\index{Kalman filter!ensemble} can be replaced 
by the sample covariance $\tilde{\Gamma}$ of the $\{\eta_{j+1}^{(n)}\}_{n=1}^N$
and this is often done in practice.  
The sample covariance of the updated
ensemble in this case is equal to $(I-\tilde{K}_{j+1}H)\hc_{j+1}$ where 
$\tilde{K}_{j+1}$ is the gain corresponding to 
the sample covariance $\tilde{\Gamma}$.

There are a range of parameters which can be used to
tune the approximate Gaussian filters\index{filter!approximate Gaussian} 
or modifications of those filters.
In practical implementations, especially for high dimensional problems,
the basic forms of the ExKF\index{Kalman filter!extended} and EnKF\index{Kalman filter!ensemble} as described here are prone to poor behaviour and
such tuning is essential \cite{kal03, eve06}. 
In Examples \ref{ex:101} and \ref{ex:202} we have already shown
the role of variance inflation\index{variance inflation} 
for 3DVAR\index{3DVAR} and this type of approach is also 
fruitfully used within ExKF\index{Kalman filter!extended} and EnKF\index{Kalman filter!ensemble}. 
A basic version of variance
inflation\index{variance inflation}
is to replace the estimate $\widehat{C}_{j+1}$ in 
\eqref{eq:dtfa9} by
$\epsilon \widehat{C}+\widehat{C}_{j+1}$
where $\widehat{C}$ is a fixed covariance such as that used in
a 3DVAR \index{3DVAR} method. Introducing $\epsilon \in (0,1)$ 
leads, for positive-definite $\widehat{C}$,
to an operator without a null-space and consequently to
better behaviour. In contrast taking $\epsilon=0$ can lead
to singular model covariances. This observation
is particularly important when the EnKF\index{Kalman filter!ensemble} is used
in high dimensional systems where the number of ensemble members, $N$, is always less
than the dimension $n$ of the state space. In this situation $\widehat{C}_{j+1}$
necessarily has a null-space of dimension at least $n-N$. It can also
be important for the ExKF\index{Kalman filter!extended} where the evolving dynamics can lead, asymptotically in $j$,
to degenerate $\widehat{C}_{j+1}$ with non-trivial null-space. 
Notice also that this form of variance inflation\index{variance inflation}
can be thought of as using 3DVAR\index{3DVAR}-like covariance updates,
in the directions not described by the ensemble covariance.
This can be beneficial in terms of the ideas underlying Theorem 
\ref{t35} where the key idea is that $K$ close to the
identity can help ameliorate growth in the underlying
dynamics. This may also be achieved 
by replacing the estimate $\widehat{C}_{j+1}$ in \eqref{eq:dtfa9}
by $(1+\epsilon)\widehat{C}_{j+1}.$  This is another commonly
used inflation tactic; note, however, that it 
lacks the benefit of rank correction.  It may therefore be combined
with the additive inflation yielding 
$\epsilon_1 \widehat{C}+(1+\epsilon_2)\widehat{C}_{j+1}.$
More details regarding tuning of filters through inflation can be found in
\cite{anderson1999monte,fisher2005equivalence,jaz70,kal03, eve06}.  

Another methodology which is important for practical implementation of the
EnKF\index{Kalman filter!ensemble} is {\bf localization}\index{localization} 
\cite{kal03, eve06}. 
This is used to 
reduce unwanted correlations in $\widehat{C}_{j}$ 
between points which are separated by large distances in space.
 The underlying assumption is that the
 correlation between points decays proportionally to their 
 distance from one another,
 and as such is increasingly corrupted by the sample error
in ensemble methods.  
 The sample covariance is hence modified to remove
correlations between points separated by large distances in space.
This is typically achieved by composing the empirical correlation
matrix with a convolution kernel. 
Localization can have the further benefit of increasing rank, 
as for the first type of variance inflation\index{variance inflation}
described above. 
An early reference illustrating the benefits and possible implementation
of localization \index{localization} is \cite{houtekamer2001sequential}.
An important reference which links this concept firmly with ideas from
dynamical systems is \cite{ott04}.

Following the great success of the ensemble Kalman filter\index{Kalman filter!ensemble} 
algorithm, in a series of papers 
\cite{tippett2003ensemble,bishop2001adaptive,anderson2001ensemble,
whitaker2002ensemble}, the square-root filter framework was
(re)discovered.  The idea goes back to at least 
\cite{andrews1968square}.  We focused the discussion above in section
\ref{ssec:sqrt} on the ETKF, but we note that it is possible to
derive different transformations. 
For example, the singular evolutive interpolated Kalman (SEIK) filter proceeds by first projecting the ensemble into the $(K-1)$-dimensional 
mean-free subspace, 
and then identifying a $(K-1) \times (K-1)$ matrix transformation, 
effectively prescribing a $K \times (K-1)$ matrix transformation $L_j$ 
as opposed to the $K \times K$ rank $(K-1)$ matrix $T_j^{1/2}$ proposed in ETKF.  
The former is unique up to unitary transformation, while the 
latter is unique only up to unitary transformations which have
${\bf 1}$ as eigenvector.  Other alternative transformations may take the forms
$A_j$ or $\tilde{K}_j$ such that $X_j = A_j \widehat{X}_j$ or 
$X_j = (I-\tilde{K}H) \widehat{X}_j$.  These are known as the
ensemble adjustment Kalman filter (EAKF) and the
ensemble square-root filter (ESRF)\index{Kalman filter!ESRF}\index{Kalman filter!square root} respectively.
See \cite{bishop2001adaptive} for details about the 
ETKF, \cite{anderson2001ensemble} for details about the 
EAKF and \cite{whitaker2002ensemble} for details about the 
ESRF \cite{whitaker2002ensemble}.  A review of all three is given in 
\cite{tippett2003ensemble}.  
The similar singular evolutive interpolated Kalman (SEIK) filter was introduced in 
\cite{pham1998singular} and is compared with the other square root filters in
\cite{nerger2012unification}.  Other ensemble-based filters have been developed 
in recent years bridging the ensemble Kalman filter with the particle filter, for example 
\cite{hoteit2008new, hoteit2012particle, reich2012gaussian,slivinski2014hybrid}.

\item Section \ref{ssec:pf}. 
In the linear case, the extended Kalman filter\index{Kalman filter!extended}
of course coincides with the Kalman filter\index{Kalman filter};
furthermore, in this case the perturbed observation\index{perturbed observations} ensemble 
Kalman filter reproduces the true posterior distribution\index{posterior distribution} in the large particle limit \cite{eve06}.
However the filters introduced in section \ref{ssec:ngf}
do not produce the correct posterior distribution when applied
to general nonlinear problems. 
On the contrary, the particle filter\index{particle filter} {\em does} 
recover the true posterior distribution as the number of particles tends to
infinity, as we show in Theorem \ref{th39}. This proof
is adapted from the very clear exposition in \cite{rebeschini2013can}.

For more refined analyses of the convergence of particle
filters see, for example, \cite{crisan2002survey,del2001stability}
 and references therein. As explained in Remarks
\ref{rem:fil} the constant appearing in the convergence results 
may depend exponentially on time if the mixing 
properties of the transition kernel $\rp(dv_j|v_{j-1})$ are poor 
(the undesirable properties of deterministic dynamics illustrate
this). There is also interesting work studying the effect of
the dimension \cite{SBBA08}.  A proof which exploits 
ergodicity\index{ergodic} of the transition kernel, when that is present,
may be found in \cite{del2001stability}; the assumptions 
there on the transition and observation kernels are very strong, 
and are generally not satisfied  in practice, but studies 
indicate that comparable results may hold under less stringent conditions.

For a derivation and discussion of the optimal proposal \index{proposal!optimal},
introduced in section \ref{ssec:opt},
see \cite{doucet2000sequential} and references therein.
We also mention here the implicit filters developed by Chorin
and co-workers \cite{CT09,CT09b,CMT10}.
These involve solving an implicit nonlinear equation
for each particle which includes knowledge of the next
set of observed data. This has some similarities to the method
proposed in \cite{vl10} and both are related to the
optimal proposal \index{proposal!index} mentioned above. 

\item Section \ref{ssec:stab}. 
The stability of the Kalman filter is a well-studied subject
and the book \cite{LR} provides an
excellent overview from the perspective of linear
algebra. For extensions to the extended Kalman filter
see \cite{jaz70}.
Theorem \ref{t35} provides a glimpse into the mechanisms
at play within 3DVAR\index{3DVAR}, and approximate Gaussian filters\index{filter!approximate Gaussian} 
in general, in determining stability and accuracy: 
the incorporation of data can convert
unstable dynamical systems, with positive Lyapunov exponents,
into contractive non-autonomous dynamical systems, thereby
leading, in the case of small observational noise, to
filters which recover the true signal within a small error.
This idea was highlighted in \cite{carrassi2008data}
and first studied rigorously for the 3DVAR\index{3DVAR} method 
applied to the Navier-Stokes equation in \cite{brett2012accuracy};
this work was subsequently generalized to a variety of different models in  
\cite{moodey2012nonlinear,shuk14,LSSS14}. 
It is also of note that these analyses of 3DVAR \index{3DVAR} build heavily on
ideas developed in \cite{hayden2011discrete} for a specialized
form of data assimilation in which the observations \index{observations} are noise-free. 
In the language of the synchronization\index{synchronization}
filter\index{filter!synchronization}
introduced in subsection \ref{ssec:4}, this paper demonstrates the
synchronization property \eqref{eq:sync} for the Navier-Stokes equation
with sufficiently large number of Fourier mode observations, and for the 
Lorenz\index{Lorenz model!'63} '63 model of Example \ref{ex:ex6} observed in only the first component.
The paper \cite{LSSS14} consider similar issues for the Lorenz\index{Lorenz model!'96} '96
model of Example \ref{ex:ex7}.  Similar ideas are studied for the perturbed
observation EnKF\index{Kalman filter!ensemble} in \cite{KLS13}; 
in this case it is necessary to introduce a form a variance inflation to
get a result analogous to Theorem \ref{t35}.
An important step in the theoretical analysis of ensemble Kalman
filter methods is the paper \cite{hunt13} which uses ideas from 
the theory of shadowing\index{shadowing} in dynamical systems; the
work proves that the ETKF\index{Kalman filter!ETKF} variant can shadow
the truth on arbitrarily long time intervals, provided the
dimension of the ensemble is greater than the number of unstable directions
in the system.

In the context of filter stability it is important
to understand the {\em optimality} of the
mean of the true filtering distribution  \index{filtering!distribution}. We observe that all of the filtering
algorithms that we have described produce an estimate of the
probability distribution $\bbP(v_j|Y_j)$ that depends only
on the data $Y_j$. There is a precise sense in which the
true filtering distribution  \index{filtering!distribution} can be used to find a lower bound
on the accuracy that can be achieved by any of these approximate
algorithms. We let $\bbE(v_j|Y_j)$ denote the mean of
$v_j$ under the probability distribution $\bbP(v_j|Y_j)$ and 
let $E_j(Y_j)$ denote any estimate of the state $v_j$ based only
on data $Y_j.$ Now consider all possible random data sets
$Y_j$ generated by the model \eqref{eq:dtf1}, \eqref{eq:dtf2}, noting
that the randomness is generated by the initial condition $v_0$ and
the noises $\{\xi_j,\eta_j\};$ in particular, conditioning on
$Y_j$ to obtain the probability distribution $\bbP(v_j|Y_j)$ can
be thought of as being induced by conditioning on the observational
noise $\{\eta_k\}_{k=1,\dots, j}.$ 
Then $E_j^*(Y_j):=\bbE(v_j|Y_j)$ minimizes the mean-square error with
respect to the random model  \eqref{eq:dtf1}, \eqref{eq:dtf2}
\cite{luenberger1968optimization, jaz70, kalman1960new}:
\begin{equation}
\label{eq:OF}
\bbE \|v_j-E_j^*(Y_j)\|^2 \le \bbE \|v_j-E_j(Y_j)\|^2
\end{equation}
for all $E_j(Y_j).$
Thus the algorithms we have described can do no better
at estimating the state of the system than can be achieved,
in principle, from the conditional mean of the state given
the data $\bbE(v_j|Y_j).$ This lower bound holds on average
over all instances of the model. An alternative way to
view the inequality \eqref{eq:OF} is as a means to
providing upper bounds on the true filter. For example,
under the conditions of Theorem \ref{t35} the righthand
side of \eqref{eq:OF} is, asymptotically as $j \to \infty$,
of size ${\cal O}(\epsilon^2)$; thus we deduce that
$$\limsup_{j \to \infty} \bbE \|v_j-\bbE(v_j|Y_j)\|^2 \le C\epsilon^2.$$
This viewpoint is adopted in \cite{SAS14} where the 3DVAR \index{3DVAR}
 filter
is used to bound the true filtering distribution.\index{filtering!distribution}
This latter optimality property can be viewed as resulting from the
Galerkin\index{Galerkin} orthogonality interpretation of
the error resulting from taking conditional expectation.

We have considered large time-behaviour on the assumption that
the map $\PPsi$ can be implemented exactly.
In situations where the underlying map $\PPsi$ arises from a
differential equation and numerical methods are required, 
large excursions in phase space caused 
by observational noise can cause numerical instabilities in the
integration methods underlying the filters. 
Remark \ref{rem:added} illustrates this fact in the context 
of continuous time. See \cite{gottwald2013mechanism} for a discussion
of this issue.

\item Section \ref{ssec:filti} We mention here the {\it rank\index{histogram!rank} histogram}. This is another consistency 
check on the output of ensemble or particle 
based approximations of the filtering
distribution  \index{filtering!distribution}. The idea is to consider scalar observed quantities 
consisting of generating ordered bins associated to that scalar and then keeping track of the statistics
over time of the data $y_j$ with respect to the bins. 
For example, if one has an approximation of the distribution consisting of $N$ equally-weighted particles, 
then a rank\index{histogram!rank} histogram for the first component of the state
consists of three steps, each carried out at each time $j$.  
First, add a random draw from the observational noise $N(0,\Gamma)$ to each 
particle after the  prediction phase of the algorithm. 
Secondly order the particles according to the value of their first component, generating $N-1$ bins 
between the values of the first component of each particle, and with one extra bin on each end.   
Finally, rank the current observation $y_j$ between $1$ and $N+1$ depending on which bin it
lands in.  Proceeding to do this at each time $j$, a histogram \index{histogram!rank} of the rank 
of the observations is obtained.  
The ``spread'' of the ensemble can be evaluated using this diagnostic.  If the histogram \index{histogram} is uniform,
then the spread is consistent.  If it is concentrated to the center, then the spread is overestimated.
If it is concentrated at the edges, then the spread is underestimated.  
This consistency check on the statistical model used was introduced
in \cite{anderson1996method} and is widely adopted throughout the
data assimilation community.

\end{itemize}

\section{Exercises}
\label{ex:filt}

\begin{enumerate}

\item Consider the Kalman filter in the case where $M=H=I,$
$\Sigma=0$ and $\Gamma>0.$ Prove that the covariance operator $C_j$ converges
to $0$ as $j \to \infty.$ Modify the program {\tt p8.m} so that it
applies to this set-up, in the two dimension setting. Verify what you
have proved regarding the covariance and make a conjecture about the
limiting behaviour of the mean of the Kalman filter. 

\item Consider the 3DVAR algorithm in the case where $\PPsi(v)=v, H=I,$
$\Sigma=0$ and $\Gamma>0.$ Choose $\hc=\alpha \Gamma$. Find an equation
for the error $e_j:=v_j-m_j$ and derive an expression for
$\lim_{j \to \infty}\bbE |e_j|^2$ in terms of $\alpha$ and $\sigma^2:=\bbE |\eta_j|^2.$ 
Modify the program {\tt p9.m} so that it
applies to this set-up, in the one dimensional setting. Verify what you
have proved regarding the limiting behaviour of the $m_j$.

\item Consider the EnKF algorithm in the same setting as the previous example.
Modify program {\tt p12.m} so that it
applies to this set-up, in the one dimensional setting. Study the behaviour
of the sequence $m_j$ found as the mean of the particles $v_j^{(n)}$ over the
ensemble index $n$. 

\item Consider the SIRS algorithm in the same setting as the previous example.
Modify program {\tt p14.m} so that it
applies to this set-up, in the one dimensional setting. Study the behaviour
of the sequence $m_j$ found as the mean of the particles $v_j^{(n)}$ over the
ensemble index $n$. 

\item Make comparative comments regarding the 3DVAR, EnKF and SIRS
algorithms, on the basis of your solutions to the three preceding
exercises.

\item 
In this exercise we study the behaviour of the  mean of the Kalman 
filter\index{Kalman filter} in the case of one dimensional dynamics.
The notation follows the development in subsection \ref{sssec:kfo}.
Consider the case $\sigma=0$ and assume that the data $\{y_j\}_{j \in \N}$ 
is generated from a true signal $\{\vd_j\}_{j \in \Z^+}$ governed
by the equation
$$\vd_{j+1}=\lambda \vd_j,$$
and that the additive observational noise  $\{\eta_j\}_{j \in \N}$ is 
drawn from an i.i.d. \index{i.i.d.}sequence with variance $\gamma^2.$
Define  the error $e_{j}=m_{j}-\vd_{j}$ between the estimated mean and the true signal and use (\ref{eq:kalman_update1d}b) to show that 
\begin{equation} \label{eq:specific}
e_{j+1}=\left(1-\frac{c_{j+1}}{\gamma^{2}}\right)\lambda e_{j}+\frac{c_{j+1}}{\gamma^{2}}\eta_{j+1}.
\end{equation}
Deduce that $e_j$ is Gaussian and that its mean and covariance satisfy the
equations
\begin{equation} \label{eq:specific1}
\mathbb{E}e_{j+1}=\lambda \left(1-\frac{c_{j+1}}{\gamma^{2}}\right)\mathbb{E}e_{j},
\end{equation}
and
\begin{equation} \label{eq:specific2}
\mathbb{E}e^{2}_{j+1}=\lambda^{2}\left(1-\frac{c_{j+1}}{\gamma^{2}}\right)^{2}\mathbb{E}e_{j}^{2}+\frac{c^{2}_{j+1}}{\gamma^{2}}.
\end{equation}
Equation \eqref{eq:specific1} can be solved to obtain 
\begin{equation} \label{eq:specific1a}
\mathbb{E}e_{j}=\lambda^{j}\left[\prod_{i=0}^{j}\left(1-\frac{c_{i+1}}{\gamma^{2}}\right) \right]\mathbb{E}e_{0},
\end{equation}
and, in a similar way, obtain for the solution of \eqref{eq:specific2}:
\begin{equation}\label{eq:specific2a}
\mathbb{E}e^{2}_{j}=\lambda^{2j}\left[\prod_{i=0}^{j-1}\left(1-\frac{c_{i+1}}{\gamma^{2}}\right)^{2} \right]\mathbb{E}e^{2}_{0}+\sum_{i=0}^{j-1}\left \{\left[\prod_{k=i+1}^{j}\left(1-\frac{c_{k}}{\gamma^{2}}\right)\right]\lambda^{2(j-i)}\frac{c^{2}_{i}}{\gamma^{2}}\right \}+\frac{c^{2}_{j}}{\gamma^{2}}
\end{equation}
Using the properties of the variance derived in \ref{sssec:kfo}, 
prove that the mean of the error tends to zero and that the asymptotic
variance is bounded by $\gamma^2$.

\end{enumerate}

\chapter{Discrete Time: {\sc matlab} Programs}\label{sec:matd}

This chapter is dedicated to illustrating the examples, theory and algorithms,
as presented in the previous chapters, through a few short and easy to follow {\sc matlab} programs. 
These programs are provided for two reasons:
(i) for some readers they will form the best route by which to appreciate the details of
the examples, theory and algorithms we describe; (ii) for other readers they 
will be a useful starting point to develop their own codes: 
whilst ours are not necessarily the optimal 
implementations of the algorithms discussed in these notes, they have been
structured to be simple to understand, to modify
and to extend. In particular the code may be readily extended to solve problems more complex than those
described in the Examples \ref{ex:ex1}-\ref{ex:ex7} which we will use for most  
of our illustrations.  
The chapter is divided into three sections, corresponding to 
programs relevant to each of the preceding three chapters.

Before getting into details we highlight a few principles 
that have been adopted in the programs and in 
accompanying text of this chapter.  
First, notation is consistent between programs, and matches the
text in the previous sections of the book as far as possible.  
Second, since many of the elements of the individual programs are repeated,
they will be described in detail only in the text corresponding to 
the program in which they first appear;  the short annotations explaining
them will be repeated within the programs however.
Third, the reader is advised to use
the documentation available at the command line
for {\it any} built-in functions of {\sc {matlab}}; this 
information can be accessed using the {\tt help} command --  
for example the documentation for the command 
{\tt help} can be accessed by typing {\tt help help}.

\section{Chapter \ref{sec:dtf} Programs}

The programs {\tt p1.m} and {\tt p2.m} used to generate the figures in Chapter 
\ref{sec:dtf} are presented in this section. Thus these algorithms
simply solve the dynamical system \eqref{eq:dtf1}, and process the
resulting data.

\subsection{p1.m}
\label{ssec:p1}

The first program {\tt p1.m} illustrates how to obtain sample paths from equations 
\eqref{eq:dtf1} and \eqref{eq:dtf11}.   In particular the program  simulates 
sample paths of the equation
\begin{equation}
u_{j+1}=\alpha \sin(u_{j})+ \xi_{j},
\label{eq:matsin}
\end{equation}
with $\xi_j \sim N(0,\sigma^{2})$ i.i.d. \index{i.i.d.} and $\alpha=2.5$, both for deterministic ($\sigma=0$) and stochastic dynamics \index{stochastic dynamics} ($\sigma \neq 0$) corresponding to Example \ref{ex:ex3}.   In line 5 the variable {\tt J} is defined, which corresponds to the number of forward steps that we will take. The parameters $\alpha$ and $\sigma$
are set in lines 6-7.  The seed\index{seed} for the random number generator is set to {\tt sd}$\in \bbN$ in line 8 using the command {\tt rng(sd)}.  This guarantees the results will be reproduced exactly by running the program with this same {\tt sd}.  Different choices of {\tt sd}$\in \bbN$ will lead to different streams of random numbers used in the program, which may also be desirable in order to observe the effects of different random numbers on the output.  The command 
{\tt sd} will be called in the preamble of all of the programs that follow.  In line 9,
two vectors of length {\tt J} are created named {\tt v} and {\tt vnoise}; after running the program,
these two vectors contain the 
solution for the case of deterministic ($\sigma=0$) and stochastic dynamics \index{stochastic dynamics} ($\sigma=0.25$) respectively. After setting the initial conditions in line 10,
the desired map is iterated, without and with noise, in lines $12-15$. 
Note that the only difference between the forward iterations of {\tt v} and {\tt vnoise} is the presence of the {\tt sigma*randn} term, which corresponds to the generation of a random variable sampled from  $N(0,\sigma^{2})$. 
Lines 17-18 contain code which graphs the trajectories, with and without noise,
to  produce Figure \ref{fig:p1}.  
Figures  \ref{fig:ex1},\ref{fig:ex2} and \ref{fig:ex4} were obtained by simply modifying lines $12-15$ of this program, in order to create sample paths for the corresponding $\PPsi$ for the other three examples; furthermore,
Figure \ref{fig:ex34}a was generated from output of this program and
Figure \ref{fig:ex34}b was generated from output of a modification of
this program.  
\newline
\begin{lstlisting}
clear;set(0,'defaultaxesfontsize',20);format long
%%% p1.m - behaviour of sin map (Ex. 1.3) 
%%% with and without observational noise

J=10000;% number of steps
alpha=2.5;% dynamics determined by alpha
sigma=0.25;% dynamics noise variance is sigma^2
sd=1;rng(sd);% choose random number seed
v=zeros(J+1,1);  vnoise=zeros(J+1,1);% preallocate space
v(1)=1;vnoise(1)=1;% initial conditions

for i=1:J
    v(i+1)=alpha*sin(v(i));
    vnoise(i+1)=alpha*sin(vnoise(i))+sigma*randn;
end

figure(1), plot([0:1:J],v), 
figure(2), plot([0:1:J],vnoise), 



\end{lstlisting}
\newpage

\subsection{p2.m}
\label{ssec:p2}

The second program presented here, {\tt p2.m}, is designed
to visualize the posterior distribution\index{posterior distribution}  
in the case of  one dimensional deterministic dynamics \index{deterministic dynamics}. 
For clarity, the program is separated into three main sections.
The {\tt setup} section in lines 5-10 defines the parameters of the problem.
The model parameter {\tt r} is defined in line 6, and determines the
dynamics of the forward model, in this case
given by the logistic map \eqref{eq:ex4}:
\begin{equation}
v_{j+1}=rv_{j}(1-v_{j}).
\label{eq:matlog}
\end{equation}
The dynamics are taken as deterministic, so the parameter {\tt sigma} does
not feature here.  The parameter {\tt r}$=2$ so that the dynamics are not chaotic,
as the explicit solution given in Example \ref{ex:ex4} shows.
The parameters {\tt m0} and {\tt C0} define the mean and covariance
of the prior distribution
$v_0 \sim N(m_0,C_0)$, whilst {\tt gamma} defines the observational noise 
$\eta_{j}\sim N(0,\gamma^{2})$.

The {\tt truth} section in lines 14-20 generates the true reference trajectory (or, truth) {\tt vt}
in line 18 given by \eqref{eq:matlog}, as well as the observations {\tt y} in line 19 given by 
\begin{equation}
y_{j}=v_{j}+\eta_{j}.
\label{eq:matdat}
\end{equation} 
Note that the index of {\tt y(:,j)} corresponds to observation of {\tt H*v(:,j+1)}.
This is due to the fact that the first index of an array in matlab is {\tt j=1}, while
the initial condition is $v_0$, and the first observation is of $v_1$.  So, effectively the
indices of {\tt y} are correct as corresponding to the text and equation \eqref{eq:matdat}, 
but the indices of {\tt v} are one off.
The memory for these vectors is {\it preallocated}\index{preallocation} in line 14.  This is not necessary
because \MAT would simply {\it dynamically allocate} the memory in its absence, but 
it would slow down the computations due to the necessity of allocating new memory each
time the given array changes size.  Commenting this line out allows
observation of this effect,
which becomes significant when {\tt J} becomes sufficiently large.

The {\tt solution} section after line 24 computes the solution, in this case the 
point-wise representation of the posterior smoothing 
distribution\index{smoothing!distribution}
on the scalar initial condition. 
The point-wise values of initial condition are given by the vector 
{\tt v0} ($v_0$) defined in line 24.
There are many ways to construct such vectors, this convention defines the initial (0.01) and final (0.99) 
values and a uniform step size 0.0005.  It is also possible to use the command 
{\tt v0=linspace(0.01,0.99,1961)}, defining the {\it number} $1961$ of intermediate points, rather than the stepsize 0.0005. 
The corresponding vector of values of {\tt Phidet} ($\Phid$), {\tt Jdet} ($\Jjd$), 
and {\tt Idet} ($\Iid$) 
are computed in lines 32, 29, and 34 for each value of {\tt v0}, as related by the equation
\begin{equation}
\Iid(v_{0};y)=\Jjd(v_{0})+\Phid(v_{0};y),
\label{eq:matlogpostdet}
\end{equation}
where $\Jjd(v_{0})$ is the {\bf background} penalization\index{background!penalization} and 
$\Phid(v_{0};y)$ is the {\bf model-data misfit}\index{model-data misfit} functional given by (\ref{eq:Ip}b) and (\ref{eq:Ip}c) respectively.
The function $\Iid(v_{0};y)$ is the negative log-posterior as given
in Theorem \ref{th112}.
Having obtained  $\Iid(v_{0};y)$
we calculate $\mathbb{P}(v_{0}|y)$ in lines 37-38, using the formula
 \begin{equation}
 \mathbb{P}(v_{0}|y) = \frac{\exp(-\Iid(v_{0};y))}{\int \exp(-\Iid(v_{0};y))}.
\label{eq:matinitpost}
 \end{equation}
 The trajectory $v$ corresponding to the given value of $v_0$ ({\tt v0(i)}) is denoted by {\tt vv}
 and is replaced for each new value of {\tt v0(i)} in lines 29 and 31 since it is only required to compute {\tt Idet}.
 The command {\tt trapz(v0,exp(-Idet))} in line 37 approximates the denominator of the 
 above by the trapezoidal rule, i.e. the summation
\begin{equation}
{\tt trapz(v0,exp(-Idet))} = \sum_{i=1}^{N-1} {\tt (v0(i+1) - v0(i))*(Idet(i+1)+Idet(i))/2}.  
\label{eq:mattrapz}
\end{equation}
The rest of the program deals with plotting our results and in this instance it coincides with the output of Figure \ref{fig:smooth2}b. 
Again simple modifications of this program were used to produce Figures \ref{fig:smooth1}, \ref{fig:smooth2a} and \ref{fig:smooth3}. 
Note that {\tt rng(sd)} in line 8 allows us to use 
the same random numbers every time the file is executed; those
random  numbers are  generated with the seed\index{seed} 
{\tt sd} as described in the previous section \ref{ssec:p1}. Commenting this line out would 
result in the creation of new realizations of the random data $y$,
different from the ones used to obtain Figure \ref{fig:smooth2}b. 
\newline
\begin{lstlisting}
clear; set(0,'defaultaxesfontsize',20); format long
%%% p2.m smoothing problem for the deterministic logistic map (Ex. 1.4)
%% setup

J=1000;% number of steps
r=2;% dynamics determined by r
gamma=0.1;% observational noise variance is gamma^2
C0=0.01;% prior initial condition variance
m0=0.7;% prior initial condition mean
sd=1;rng(sd);% choose random number seed

%% truth

vt=zeros(J+1,1); y=zeros(J,1);% preallocate space to save time
vt(1)=0.1;% truth initial condition
for j=1:J
    % can be replaced by Psi for each problem
    vt(j+1)=r*vt(j)*(1-vt(j));% create truth
    y(j)=vt(j+1)+gamma*randn;% create data
end

%% solution

v0=[0.01:0.0005:0.99];% construct vector of different initial data
Phidet=zeros(length(v0),1);Idet=Phidet;Jdet=Phidet;% preallocate space
vv=zeros(J,1);% preallocate space 
% loop through initial conditions vv0, and compute log posterior I0(vv0) 
for j=1:length(v0) 
   vv(1)=v0(j); Jdet(j)=1/2/C0*(v0(j)-m0)^2;% background penalization 
     for i=1:J
        vv(i+1)=r*vv(i)*(1-vv(i)); 
        Phidet(j)=Phidet(j)+1/2/gamma^2*(y(i)-vv(i+1))^2;% misfit 
     end 
    Idet(j)=Phidet(j)+Jdet(j);
end

constant=trapz(v0,exp(-Idet));% approximate normalizing constant
P=exp(-Idet)/constant;% normalize posterior distribution
prior=normpdf(v0,m0,sqrt(C0));% calculate prior distribution

figure(1),plot(v0,prior,'k','LineWidth',2)
hold on, plot(v0,P,'r--','LineWidth',2), xlabel 'v_0',
legend 'prior' J=10^3  
\end{lstlisting}
\newpage

\section{Chapter \ref{sec:dtsa} Programs}

The programs {\tt p3.m-p7.m}, 
used to generate the figures in Chapter \ref{sec:dtsa}, 
are presented in this section. 
Hence various MCMC\index{MCMC} algorithms used to sample the posterior 
smoothing distribution\index{smoothing!distribution} are given. Furthermore, 
optimization\index{optimization} algorithms used to obtain solution of the 
4DVAR\index{4DVAR}  and w4DVAR\index{4DVAR!weak constraint} 
variational methods\index{variational method} 
are also introduced.  Our general theoretical development of 
MCMC methods in section \ref{ssec:mcmcm}
employs a notation of $u$ for the 
state of the chain and $w$ for the proposal\index{proposal}.
For deterministic dynamics the state is the
initial condition $v_0$; for stochastic
dynamics it is either the signal $v$ or
the pair $(v_0,\xi)$ where $\xi$ is the noise (since this
pair determines the signal).  Where appropriate, the programs 
described here use the letter $v$, and variants on it,
for the state of the Markov chain\index{Markov chain}, 
to keep the connection with the underlying dynamics model.

\subsection{p3.m}
\label{ssec:p3}

The program {\tt p3.m} contains an implementation of the Random Walk Metropolis (RWM)\index{Random Walk Metropolis} MCMC algorithm.
The development follows section \ref{ssec:dd2} where the algorithm is
used to determine the posterior distribution\index{posterior distribution} on the initial condition arising from the deterministic logistic map of Example \ref{ex:ex4} given by \eqref{eq:matlog}.
Note that in this case, since the the underlying dynamics are  
deterministic and hence completely determined by the initial condition, 
the RWM algorithm will provide samples from a probability distribution on $\mathbb{R}$. 
 
As in program {\tt p2.m}, the code is divided into 3 sections:
{\tt setup} where parameters are defined, {\tt truth} where the truth and data are generated, 
and {\tt solution} where the solution is computed, this time by means
of MCMC samples from the posterior smoothing distribution\index{smoothing!distribution}.
The parameters in lines 5-10 and the true solution 
(here taken as only the initial condition, rather than the trajectory it gives rise to) 
{\tt vt} in line 14 are taken to be the same as those used to generate 
Figure \ref{fig:smooth3}.  
The temporary vector {\tt vv} generated in line 19 is the trajectory
corresponding to the truth ({\tt vv(1)=vt} in line 14), 
and used to calculate the observations {\tt y} in line 20.  
The true value {\tt vt} will also be used as the initial sample in the 
Markov chain\index{Markov chain} for this and for all subsequent MCMC programs.
This scenario is, of course, not possible in the case that the data is 
not simulated. However it is useful in the case that the data is simulated, 
as it is here, because it can reduce the burn-in\index{burn-in time} 
time, i.e. the time 
necessary for the current sample in the chain to reach the target distribution,
or the high-probability region of the state-space.
Because we initialize the Markov chain at the truth, 
the value of $\Iid(v^\dagger)$, denoted by the temporary variable {\tt Idet}, 
is required to determine the initial acceptance probability, 
as described below. 
It is computed in lines 15-23 exactly as in lines 25-34 of program {\tt p2.m},
as described around equation \eqref{eq:matlogpostdet}.


In the {\tt solution} section some additional MCMC parameters are defined.  
In line 28 the number of samples is set to {\tt N =}$10^5$.
 For the parameters and specific data used here, this is sufficient for convergence
of the Markov chain\index{Markov chain}.  In line 30 the step-size parameter {\tt beta} is pre-set such that the algorithm
for this particular posterior distribution has a reasonable acceptance probability, or 
ratio of accepted vs. rejected moves.  A general rule of thumb for this is that it should be somewhere around $0.5$, to ensure that the algorithm is not
too correlated because of high rejection rate (acceptance probability
near zero) and that it is not too correlated
because of small moves (acceptance probability near one).
The vector {\tt V} defined in line 29 will save all of the samples.  
This is an example where preallocation\index{preallocation} is {\it very}
important.  Try using the commands {\tt tic} and {\tt toc} before and respectively after the loop in lines
33-50 in order to time the chain both with and without preallocation\index{allocation}.
\footnote{In practice, one may often choose to collect certain statistics from the chain ``on-the-fly'' rather than 
saving every sample, particularly if the state-space is high-dimensional where
the  memory required for each sample is large.}
In line 34 a move is proposed according to the proposal\index{proposal}
equation \eqref{eq:prop1}:
\[
w^{(k)}=v^{(k-1)}+\beta \iota^{(k-1)}
\]
where $v$({\tt v}) is the current state of the chain (initially taken to be equal to the true initial condition $v_{0}$), 
$\iota^{(k-1)}$={\tt randn} is an i.i.d. \index{i.i.d.} standard normal, and {\tt w} represents $w^{(k)}$.  Indices are not used for 
{\tt v} and {\tt w} because they will be overwritten at each iteration. 

The temporary variable {\tt vv} is again used for the trajectory corresponding to 
$w^{(k)}$ as a vehicle to compute the value of
the proposed $\Iid(w^{(k)};y)$, denoted 
in line 42 by {\tt I0prop = J0prop + Phiprop}.
In lines 44-46 the decision to accept or reject the
proposal\index{proposal}
is made based on the acceptance probability 
\[
a(v^{(k-1)},w^{(k)})= 1\wedge\exp(\Iid(v^{(k-1)};y)-\Iid(w^{(k)};y)).
\]
In practice this corresponds to drawing a uniform random number {\tt rand} and
replacing {\tt v} and {\tt Idet} in line 45 with {\tt w} and {\tt I0prop}
if {\tt rand<exp(I0-I0prop)} in line 44.
The variable {\tt bb} is incremented if the proposal\index{proposal} is accepted, so that 
the running ratio of accepted moves {\tt bb} to total steps {\tt n} can be computed
in line 47.  This approximates the average acceptance probability.  
The current sample $v^{(k)}$ is stored in line 48.  Notice that here one could replace
{\tt v} by {\tt V(n-1)} in line 34, and by {\tt V(n)} in line 45, thereby eliminating {\tt v}
and line 48, and letting {\tt w} be the only temporary variable.  However, the present 
construction is favourable because, as mentioned above,  in general one may not wish to save every sample. 

The samples {\tt V} are used in lines 51-53 in order to visualize the posterior distribution. 
In particular, bins of width {\tt dx} are defined in line 51, and the command {\tt hist} is used
in line 52.  The assignment {\tt Z = hist(V,v0)} means first the real-number line is split into $M$ bins with centers 
defined according to {\tt v0(i)} for $i=1,\dots, M$, with the first and last bin corresponding to the negative, 
respectively positive, half-lines.  Second, 
{\tt Z(i)} counts the number of {\tt k} for which {\tt V(k)} 
is in the bin with center determined by {\tt v0(i)}.
Again, {\tt trapz} \eqref{eq:mattrapz} is used to compute the normalizing constant in line 53,
directly within the plotting command.
The choice of the location of the histogram bins 
allows for a direct comparison with the posterior distribution calculated from the 
program {\tt p2.m}
by directly evaluating $\Iid(v;y)$ defined in \eqref{eq:matlogpostdet} 
for different values of initial conditions $v$. 
This output is then compared with the corresponding output of {\tt p2.m} for the same 
parameters in Figure \ref{fig:MCMC1}.

\newpage
\begin{lstlisting}
clear; set(0,'defaultaxesfontsize',20); format long
%%% p3.m MCMC RWM algorithm for logistic map (Ex. 1.4)
%% setup

J=5;% number of steps
r=4;% dynamics determined by alpha
gamma=0.2;% observational noise variance is gamma^2
C0=0.01;% prior initial condition variance
m0=0.5;% prior initial condition mean
sd=10;rng(sd);% choose random number seed

%% truth

vt=0.3;vv(1)=vt;% truth initial condition
Jdet=1/2/C0*(vt-m0)^2;% background penalization
Phidet=0;% initialization model-data misfit functional
for j=1:J
    % can be replaced by Psi for each problem
    vv(j+1)=r*vv(j)*(1-vv(j));% create truth
    y(j)=vv(j+1)+gamma*randn;% create data
    Phidet=Phidet+1/2/gamma^2*(y(j)-vv(j+1))^2;% misfit functional
end
Idet=Jdet+Phidet;% compute log posterior of the truth

%% solution
% Markov Chain Monte Carlo: N forward steps of the
% Markov Chain on R (with truth initial condition)
N=1e5;% number of samples
V=zeros(N,1);% preallocate space to save time
beta=0.05;% step-size of random walker 
v=vt;% truth initial condition (or else update I0) 
n=1; bb=0; rat(1)=0; 
while n<=N
    w=v+sqrt(2*beta)*randn;% propose sample from random walker
    vv(1)=w;
    Jdetprop=1/2/C0*(w-m0)^2;% background penalization
    Phidetprop=0;	
    for i=1:J
            vv(i+1)=r*vv(i)*(1-vv(i));
            Phidetprop=Phidetprop+1/2/gamma^2*(y(i)-vv(i+1))^2;
    end
    Idetprop=Jdetprop+Phidetprop;% compute log posterior of the proposal 	

    if rand<exp(Idet-Idetprop)% accept or reject proposed sample
        v=w; Idet=Idetprop; bb=bb+1;% update the Markov chain
    end
    rat(n)=bb/n;% running rate of acceptance
    V(n)=v;% store the chain
    n=n+1; 
end
dx=0.0005; v0=[0.01:dx:0.99];
Z=hist(V,v0);% construct the posterior histogram 
figure(1), plot(v0,Z/trapz(v0,Z),'k','Linewidth',2)% visualize posterior
\end{lstlisting}
\newpage

\subsection{p4.m}
\label{ssec:p4}

The 
program {\tt p4.m} contains an implementation of the 
Independence Dynamics Sampler\index{Independence Dynamics Sampler} 
for stochastic dynamics \index{stochastic dynamics}, as
introduced in subsection \ref{ssec:sd2}. Thus the 
posterior distribution\index{posterior distribution} 
is on the entire signal $\{v_j\}_{j \in \J}.$ 
The forward model in this case is from Example \ref{ex:ex3}, 
given by \eqref{eq:matsin}.
The smoothing distribution\index{smoothing!distribution} 
$\mathbb{P}(v|Y)$ is therefore over 
the state-space $\mathbb{R}^{J+1}$.

The sections {\tt setup}, {\tt truth}, and {\tt solution} are defined as for 
program {\tt p3.m}, but note that now
the smoothing distribution\index{smoothing!distribution} 
is over the entire path, not just over the initial
condition, because we are considering stochastic dynamics \index{stochastic dynamics}.    
Since the state-space is now the path-space, rather than the initial condition 
as it was in program {\tt p3.m}, the truth {\tt vt}$\in \bbR^{J+1}$ is now a vector.
Its initial condition is taken as a draw from $N(m_0,C_0)$ in line 16, and 
the trajectory is computed in line 20, so that at the end {\tt vt}$\sim \rho_0$.
As in program {\tt p3.m}, $v^\dagger$ ({\tt vt}) will be the chosen
initial condition in the Markov chain\index{Markov chain}
(to ameliorate burn-in issues\index{burn-in}) and so $\PPhi(v^\dagger;y)$
is computed in line 23.  Recall from subsection \ref{ssec:sd2} that only $\PPhi(\cdot;y)$ is required 
to compute the acceptance probability in this algorithm. 

Notice that the collection of samples {\tt V}$\in \bbR^{N\times J+1}$ 
preallocated\index{preallocation} in line 30 
is substantial in this case, illustrating
the memory issue which arises when the dimension of the signal space, 
and number of samples, increase.

The current state of the chain $v^{(k)}$, and the value of $\Phi(v^{(k)};y)$
are again denoted {\tt v} and {\tt Phi}, while the proposal\index{proposal} 
$w^{(k)}$ and the value of $\PPhi(w^{(k)};y)$
are again denoted {\tt w} and {\tt Phiprop},
as in program {\tt p3}.
 As discussed in section \ref{ssec:sd2}, the proposal\index{proposal} 
$w^{(k)}$ is an independent sample from  the prior distribution $\rho_{0}$, 
similarly to $v^\dagger$, and it is constructed in lines 34-39.   
The acceptance probability used in line 40 is now
\begin{equation}
a(v^{(k-1)},w^{(k)})= 1\wedge\exp(\PPhi(v^{(k-1)};y)-\PPhi(w^{(k)};y)),
\label{eq:matacceptind}
\end{equation}

The remainder of the program is structurally the same as {\tt p3.m}.
The outputs of this program are used to plot Figures \ref{fig:smooth_MCMC1}, \ref{fig:smooth_MCMC2}, and \ref{fig:smooth_MCMC3}. 
Note that in the case of Figure \ref{fig:smooth_MCMC3}, we have used $N=10^8$ samples. 

\newpage
\begin{lstlisting}
clear; set(0,'defaultaxesfontsize',20); format long
%%% p4.m MCMC INDEPENDENCE DYNAMICS SAMPLER algorithm 
%%% for sin map (Ex. 1.3) with noise
%% setup

J=10;% number of steps
alpha=2.5;% dynamics determined by alpha
gamma=1;% observational noise variance is gamma^2
sigma=1;% dynamics noise variance is sigma^2
C0=1;% prior initial condition variance
m0=0;% prior initial condition mean
sd=0;rng(sd);% choose random number seed
 
%% truth

vt(1)=m0+sqrt(C0)*randn;% truth initial condition 
Phi=0; 

for j=1:J
    vt(j+1)=alpha*sin(vt(j))+sigma*randn;% create truth
    y(j)=vt(j+1)+gamma*randn;% create data
    % calculate log likelihood of truth, Phi(v;y) from (1.11)
    Phi=Phi+1/2/gamma^2*(y(j)-vt(j+1))^2;
end

%% solution
% Markov Chain Monte Carlo: N forward steps of the
% Markov Chain on R^{J+1} with truth initial condition
N=1e5;% number of samples
V=zeros(N,J+1);% preallocate space to save time
v=vt;% truth initial condition (or else update Phi) 
n=1; bb=0; rat(1)=0; 
while n<=N
    w(1)=sqrt(C0)*randn;% propose sample from the  prior 
    Phiprop=0; 
    for j=1:J
    w(j+1)=alpha*sin(w(j))+sigma*randn;% propose sample from the prior 
    Phiprop=Phiprop+1/2/gamma^2*(y(j)-w(j+1))^2;% compute likelihood
    end
    if rand<exp(Phi-Phiprop)% accept or reject proposed sample
       v=w; Phi=Phiprop; bb=bb+1;% update the Markov chain
    end
   rat(n)=bb/n;% running rate of acceptance
   V(n,:)=v;% store the chain
   n=n+1;
end
% plot acceptance ratio and cumulative sample mean
figure;plot(rat)
figure;plot(cumsum(V(1:N,1))./[1:N]')
xlabel('samples N')
ylabel('(1/N) \Sigma_{n=1}^N v_0^{(n)}')








\end{lstlisting}
\newpage

\subsection{p5.m}
\label{ssec:p5}

The Independence Dynamics Sampler\index{Independence Dynamics Sampler} of subsection
\ref{ssec:p4} may be very inefficient 
as typical random draws from the dynamics may be unlikely to fit the
data as well as the current state, and will then be rejected. 
The fifth 
program {\tt p5.m} gives 
an implementation of 
the pCN \index{pCN}  algorithm from section \ref{ssec:sd2} which is designed to overcome this issue by including the parameter $\beta$ which, if chosen small, allows
for incremental steps in signal space and hence the possibility of
non-negligible acceptance probabilities.
This program is used to generate Figure \ref{fig:pCN}

This program is almost identical to {\tt p4.m}, and so only the points at which it differs will be described.
First, since the acceptance probability is given by
\[
a(v^{(k-1)},w^{(k)})=1\wedge\exp(\PPhi(v^{(k-1)};y)-\PPhi(w^{(k)};y)+G(v^{(k-1)})-G(w^{(k)})),
\] 
the quantity 
\[
G(u)=\sum_{j=0}^{J-1}\Bigl(\frac12|\Sigma^{-\frac12}\PPsi(u_{j})|^2-\langle 
\Sigma^{-\frac12}u_{j+1}, \Sigma^{-\frac12}\PPsi(u_{j})\rangle\Bigr)
\]
will need to be computed, both for $v^{(k)}$ (denoted by {\tt v} in lines 31 and 44) 
where its value is denoted by {\tt G} ($v^{(0)}=v^\dagger$), as well as for $G(v^\dagger)$ is computed in 
line 22), and for $w^{(k)}$ (denoted by {\tt w} in line 36) where its value is denoted by {\tt Gprop} in line 39.


%
As discussed in section \ref{ssec:sd2} the proposal\index{proposal} 
$w^{(k)}$ is given by \eqref{eq:pcnprop}: 
\begin{equation}
w^{(k)}=m+(1-\beta^{2})^{\frac{1}{2}}(v^{(k-1)}-m)+\beta \iota^{(k-1)};
\label{eq:matproppcn}
\end{equation}  
here $\iota^{(k-1)} \sim N(0,C)$ are i.i.d. \index{i.i.d.} and denoted by {\tt iota} in line
35.  $C$ is the covariance of 
the Gaussian measure $\pi_0$ given in Equation \eqref{eq:nuz}
corresponding to the case of trivial dynamics $\PPsi=0$, and $m$ is the mean of $\pi_0$.  
The value of $m$ is given by {\tt m} in line 33.

\newpage
\begin{lstlisting}
clear;set(0,'defaultaxesfontsize',20);format long
%%% p5.m MCMC pCN algorithm for sin map (Ex. 1.3) with noise

%% setup
J=10;% number of steps
alpha=2.5;% dynamics determined by alpha
gamma=1;% observational noise variance is gamma^2
sigma=.1;% dynamics noise variance is sigma^2
C0=1;% prior initial condition variance
m0=0;% prior initial condition mean
sd=0;rng(sd);% Choose random number seed
 
%% truth
vt(1)=m0+sqrt(C0)*randn;% truth initial condition 
G=0;Phi=0;

for j=1:J
    vt(j+1)=alpha*sin(vt(j))+sigma*randn;% create truth
    y(j)=vt(j+1)+gamma*randn;% create data
    % calculate log density from (1.--)
    G=G+1/2/sigma^2*((alpha*sin(vt(j)))^2-2*vt(j+1)*alpha*sin(vt(j)));
    % calculate log likelihood phi(u;y) from (1.11)
    Phi=Phi+1/2/gamma^2*(y(j)-vt(j+1))^2;
end

%% solution
% Markov Chain Monte Carlo: N forward steps
N=1e5;% number of samples
beta=0.02;% step-size of pCN walker
v=vt;% truth initial condition (or update G + Phi)
V=zeros(N,J+1); n=1; bb=0; rat=0;
m=[m0,zeros(1,J)];
while n<=N
    iota=[sqrt(C0)*randn,sigma*randn(1,J)];% Gaussian prior sample
    w=m+sqrt(1-beta^2)*(v-m)+beta*iota;% propose pCN sample
    Gprop=0;Phiprop=0;
    for j=1:J
        Gprop=...
         Gprop+1/2/sigma^2*((alpha*sin(w(j)))^2-2*w(j+1)*alpha*sin(w(j)));
        Phiprop=Phiprop+1/2/gamma^2*(y(j)-w(j+1))^2;
    end  
    if rand<exp(Phi-Phiprop+G-Gprop)% accept or reject proposed sample
        v=w;Phi=Phiprop;G=Gprop;bb=bb+1;% update the Markov chain
    end
    rat(n)=bb/n;% running rate of acceptance
    V(n,:)=v;% store the chain
    n=n+1;
end
% plot acceptance ratio and cumulative sample mean
figure;plot(rat)
figure;plot(cumsum(V(1:N,1))./[1:N]')
xlabel('samples N')
ylabel('(1/N) \Sigma_{n=1}^N v_0^{(n)}')



\end{lstlisting}
\newpage

\subsection{p6.m}
\label{ssec:p5_5}

The pCN \index{pCN} dynamics sampler is now introduced as program {\tt p6.m}.
The Independence Dynamics Sampler\index{Independence Dynamics Sampler} of subsection
\ref{ssec:p4} may be viewed as a special case of this algorithm for 
proposal variance\index{proposal!variance} $\beta=1$.
This proposal\index{proposal} combines the benefits of tuning the step size $\beta$, while
still respecting the prior distribution on the dynamics.  It does so by sampling 
the initial condition and noise $(v_0,\xi)$ rather than the path itself, in lines 
34 and 35, as given by Equation \eqref{eq:matproppcn}.  However, as opposed to 
the pCN \index{pCN} sampler of the previous section, this variable {\tt w} is 
now interpreted as a sample of $(v_0,\xi)$ and is therefore fed into the 
path {\tt vv} itself in line 39.  The acceptance probability is the 
same as the Independence Dynamics Sampler\index{Independence Dynamics Sampler} 
\eqref{eq:matacceptind}, depending 
only on {\tt Phi}.
If the proposal\index{proposal} 
is accepted, both the forcing {\tt u=w} and the path 
{\tt v=vv} are updated in line 44.  Only the path is saved as in the previous
routines, in line 47.

\newpage
\begin{lstlisting}
clear;set(0,'defaultaxesfontsize',20);format long
%%% p6.m MCMC pCN Dynamics algorithm for
%%% sin map (Ex. 1.3) with noise
%% setup

J=10;% number of steps
alpha=2.5;% dynamics determined by alpha
gamma=1;% observational noise variance is gamma^2
sigma=1;% dynamics noise variance is sigma^2
C0=1;% prior initial condition variance
m0=0;% prior initial condition mean
sd=0;rng(sd);% Choose random number seed
 
%% truth
vt(1)=m0+sqrt(C0)*randn;% truth initial condition 
ut(1)=vt(1);
Phi=0;
for j=1:J
    ut(j+1)=sigma*randn;
    vt(j+1)=alpha*sin(vt(j))+ut(j+1);% create truth
    y(j)=vt(j+1)+gamma*randn;% create data
    % calculate log likelihood phi(u;y) from (1.11)
    Phi=Phi+1/2/gamma^2*(y(j)-vt(j+1))^2;
end

%% solution
% Markov Chain Monte Carlo: N forward steps
N=1e5;% number of samples
beta=0.2;% step-size of pCN walker 
u=ut;v=vt;% truth initial condition (or update Phi)
V=zeros(N,J+1); n=1; bb=0; rat=0;m=[m0,zeros(1,J)];
while n<=N 
    iota=[sqrt(C0)*randn,sigma*randn(1,J)];% Gaussian prior sample        
    w=m+sqrt(1-beta^2)*(u-m)+beta*iota;% propose pCN sample 
    vv(1)=w(1);
    Phiprop=0;    
    for j=1:J
        vv(j+1)=alpha*sin(vv(j))+w(j+1);% create path
        Phiprop=Phiprop+1/2/gamma^2*(y(j)-vv(j+1))^2;
    end
    
    if rand<exp(Phi-Phiprop)% accept or reject proposed sample
        u=w;v=vv;Phi=Phiprop;bb=bb+1;% update the Markov chain
    end    
       rat(n)=bb/n;% running rate of acceptance
       V(n,:)=v;% store the chain
       n=n+1;
end
% plot acceptance ratio and cumulative sample mean
figure;plot(rat)
figure;plot(cumsum(V(1:N,1))./[1:N]')
xlabel('samples N')
ylabel('(1/N) \Sigma_{n=1}^N v_0^{(n)}')




\end{lstlisting}

\newpage

\subsection{p7.m}
\label{ssec:p6}

The next 
program {\tt p7.m} contains an implementation 
of the weak constrained variational algorithm w4DVAR\index{4DVAR!weak constraint} 
discussed in section \ref{ssec:vm}. 
This program is written as a function, whilst all previous programs were
written as scripts. This choice was made for {\tt p7.m}
so that the \MAT built-in function 
{\tt fminsearch} can be used for optimization\index{optimization} in the {\tt solution} section,
and the program can still be self-contained.  
To use this built-in function it is necessary to define an {\it auxiliary} objective 
function {\tt I} to be optimized.
The function {\tt fminsearch} 
can be used within a script, but the auxiliary function would then have to be written
separately, so we cannot avoid functions altogether unless we write the 
optimization\index{optimization} 
algorithm by  hand.  We avoid the latter in order not to divert the focus of this text 
from the data assimilation problem, and algorithms to solve it, to the problem of  
how to optimize an objective function.

Again the forward model is that given 
by Example \ref{eq:ex3}, namely \eqref{eq:matsin}. 
The {\tt setup} and {\tt truth} sections are similar to the previous programs, except that $G$, for example, need not be computed here.
The auxiliary objective function {\tt I} in this case is $\Ii(\cdot;y)$ from equation 
\eqref{eq:dtf4} given by
\begin{equation}\label{eq:matdtf4}
\Ii(\cdot;y)=\Jj(\cdot)+\PPhi(\cdot;y),
\end{equation}
where 
\begin{equation}\Jj(u):=\frac12\bigl|C_0^{-\frac12}(u_0-m_0)\bigr|^2+\sum_{j=0}^{J-1}\frac12\bigl|\Sigma^{-\frac12}\bigl(u_{j+1}-\PPsi(u_j)\bigr)\bigr|^2,
\label{eq:matJJ}
\end{equation}
and
\begin{equation}\label{eq:matdtf3}
\PPhi(u;y)=\sum_{j=0}^{J-1}\frac12\bigl|\Gamma^{-\frac12}\bigl(y_{j+1}-h(u_{j+1})\bigr)\bigr|^2.
\end{equation}
It is defined in lines 38-45.  The auxiliary objective function
takes as inputs ({\tt u,y,sigma,gamma,\allowbreak alpha,m0,}
{\tt C0,J}), and gives
output {\tt out}$=\Ii(u;y)$ where $u\in R^{J+1}$ (given all the other parameters in its definition -- 
the issue of identifying the input to be optimized over is discussed also below).
 
 The initial guess for the optimization\index{optimization} algorithm {\tt uu} 
is taken as a standard normal random vector over $\bbR^{J+1}$ in line 27.
In line 24, a standard normal random matrix of size $100^2$ is drawn and thrown
away.  This is so one can easily change the input, e.g. to {\tt randn(z)} for {\tt z}$\in \bbN$, 
and induce different random initial vectors {\tt uu} for the optimization\index{optimization} algorithm, while
keeping the data fixed by the random number seed\index{seed} {\tt sd} set in line 12.  
The truth {\tt vt} may be used as initial guess by uncommenting line 28.
In particular, if the output of the minimization procedure is different 
for different initial conditions, then it is
possible that the objective function $\Ii(\cdot;y)$ has multiple minima, 
and hence the posterior distribution\index{posterior distribution} 
$\mathbb{P}(\cdot|y)$ is multi-modal. 
As we have already seen in Figure \ref{fig:4DVAR} 
this is certainly true even in the case of scalar deterministic dynamics \index{deterministic dynamics}, 
when the underlying map gives rise to a chaotic flow.  

The \MAT optimization\index{optimization} function {\tt fminsearch} is called in line 32.
The {\it function handle} command {\tt @(u)I(u, $\cdots$)} is used to 
tell {\tt fminsearch} that the objective function {\tt I} is to be considered a function of {\tt u},
even though it may take other parameter values as well 
(in this case, {\tt y,sigma,gamma,alpha,m0,C0}, and {\tt J}).  
The outputs of {\tt fminsearch} are the value {\tt vmap} such
that {\tt I(vmap)} is minimum, the value {\tt fval =  I(vmap)}, and the
{\tt exit flag} which takes the value 1 if the algorithm has converged.
The reader is encouraged to use the {\tt help} command for more details
on this and other \MAT functions used in the notes.
The results of this minimization procedure are plotted in lines 34-35
together with the true value $\vd$ as well as the data $y$.  In Figure 
\ref{fig:w4DVAR} such results are presented, including two minima 
which were found with different initial conditions. 
\newline
\begin{lstlisting}
function this=p7
clear;set(0,'defaultaxesfontsize',20);format long
%%% p7.m weak 4DVAR for sin map (Ex. 1.3)
%% setup

J=5;% number of steps
alpha=2.5;% dynamics determined by alpha
gamma=1e0;% observational noise variance is gamma^2
sigma=1;% dynamics noise variance is sigma^2
C0=1;% prior initial condition variance
m0=0;% prior initial condition mean
sd=1;rng(sd);% choose random number seed

%% truth

vt(1)=sqrt(C0)*randn;% truth initial condition
for j=1:J
    vt(j+1)=alpha*sin(vt(j))+sigma*randn;% create truth
    y(j)=vt(j+1)+gamma*randn;% create data
end

%% solution

    randn(100);% try uncommenting or changing the argument for different 
               % initial conditions -- if the result is not the same,
               % there may be multimodality (e.g. 1 & 100).                         
    uu=randn(1,J+1);% initial guess       
    %uu=vt;     % truth initial guess option              

% solve with blackbox
% exitflag=1 ==> convergence
[vmap,fval,exitflag]=fminsearch(@(u)I(u,y,sigma,gamma,alpha,m0,C0,J),uu)

figure;plot([0:J],vmap,'Linewidth',2);hold;plot([0:J],vt,'r','Linewidth',2)
plot([1:J],y,'g','Linewidth',2);hold;xlabel('j');legend('MAP','truth','y')

%% auxiliary objective function definition
function out=I(u,y,sigma,gamma,alpha,m0,C0,J)

Phi=0;JJ=1/2/C0*(u(1)-m0)^2;
for j=1:J
    JJ=JJ+1/2/sigma^2*(u(j+1)-alpha*sin(u(j)))^2;
    Phi=Phi+1/2/gamma^2*(y(j)-u(j+1))^2;
end
out=Phi+JJ;




\end{lstlisting}
\newpage

\section{Chapter \ref{sec:dtfa} Programs}
\label{sec:c4ch3}

The programs {\tt p8.m-p15.m}, 
used to generate the figures in Chapter \ref{sec:dtfa}, 
are presented in this section. 
Various filtering\index{filtering!algorithm} 
algorithms used to sample the posterior 
filtering distribution\index{filtering!distribution} 
are given, involving both 
Gaussian approximation and particle approximation.
Since these algorithms are run for very large times (large $J$),
they will only be divided in two sections, {\tt setup} in which the parameters are defined, 
and {\tt solution} in which {\it both} the truth and observations are generated, {\it and} 
the online assimilation of the current observation into the filter 
solution is performed.  The generation of truth can be separated into
a {\tt truth} section as in the previous sections, 
but two loops of length $J$ would be required, and 
loops are inefficient in \MAT, so the present 
format is preferred.  The programs in this section are all very similar, 
and their output is also similar,
giving rise to Figures \ref{fig:KF}-\ref{fig:errorp}.  With the exception of {\tt p8.m} and {\tt p9.m},
the forward model is given by Example \ref{eq:ex3} \eqref{eq:matsin}, and the output is
identical, given for {\tt p10.m} through {\tt p15.m} in 
Figures \ref{fig:3DVAR3}-\ref{fig:EnKF} and \ref{fig:ETKF}-\ref{fig:PFOP}.  
Figures \ref{fig:error} and \ref{fig:errorp} compare the filters
from the other Figures.   {\tt p8.m} features a two-dimensional linear forward model, 
and {\tt p9.m} features the forward model from Example \ref{eq:ex4} \eqref{eq:matlog}.
At the end of each program, the outputs are used to plot the mean and the covariance  
as well as the mean square error of the filter as  functions of the iteration number $j$.

\subsection{p8.m}
\label{ssec:p8}

The first 
filtering\index{filtering!program} 
program is {\tt p8.m} which contains an implementation of the Kalman Filter
applied to Example \ref{ex:ex2}: 
\[
v_{j+1}=Av_{j}+\xi_{j}, \quad \text{with} \quad A=\left( \begin{array}{cc}
0 & 1 \\
-1 & 0 \end{array} \right)
\]
and observed data given by 
$$y_{j+1}=Hv_{j+1}+\eta_{j+1}$$
with $H=(1,0)$ and Gaussian noise.  Thus only the first component of $v_j$ is observed. 

The parameters and initial condition are defined in the {\tt setup} section, lines 3-19. 
The vectors {\tt v}, {\tt m} $\in \bbR^{N\times J}$, {\tt y}$\in \bbR^J$, and {\tt c} $\in \bbR^{N\times N \times J}$ 
are preallocated to hold the truth, mean, observations, and covariance over the {\tt J} observation times
defined in line 5.
In particular, notice that the true initial condition is drawn from $N(m_0,C_0)$ in line 16, 
where $m_0=0$ and $C_0=1$ are defined in lines 10-11.  The initial {\it estimate} of the distribution
is defined in lines 17-18 as $N(m_0,C_0)$, where $m_0\sim N(0,100 I)$ and $C_0 \leftarrow 100C_0$
so that the code may test the ability of the filter 
to lock onto the true distribution, asymptotically in $j$, 
given a poor initial estimate.  That is to say, the values of $(m_0,C_0)$ are {\it changed}
such that the initial condition is {\it not} drawn from this distribution.

The main {\tt solution} loop then follows in lines 21-34. 
The truth {\tt v} and the data that are being assimilated {\tt y} are 
sequentially generated within the loop, in lines 24-25. 
The filter prediction step, in lines 27-28, consists of computing 
the predictive mean and covariance $\hm_{j}$ and $\hc_{j}$
as  defined in \eqref{eq:dtfa2} and \eqref{eq:dtfa3} respectively:
\[
\hm_{j+1}=Am_{j}  ,\quad \hc_{j+1}=AC_{j}A^{T}+\Sigma.
\]
Notice that 
indices are not used for the transient variables {\tt mhat} and {\tt chat} representing
$\hm_{j}$ and $\hc_{j}$ because they will not be saved from one iteration to the next. 
 In lines 30-33 we implement the analysis formulae for the Kalman filter\index{Kalman filter}
from Corollary \ref{c32}.  
In particular, the innovation\index{innovation}
between the observation of the 
predicted mean and the actual observation, as introduced in
Corollary \ref{c32}, is first computed in line 30
\begin{equation}
d_{j} = y_{j} - H \hm_{j}.
\label{eq:matinno}
\end{equation}
Again {\tt d}, which represents $d_j$, 
does not have any index for the same reason as above.  
Next, the Kalman gain \index{Kalman gain}defined in Corollary \ref{c32} is computed in line 31
\begin{equation}
K_j = \hc_j H^T ( H \hc_j H^T + \Gamma)^{-1}.
\label{eq:matkalg}
\end{equation}
Once again index $j$ is not used for the transient variable {\tt K} representing $K_j$.  
Notice the "forward slash"
{\tt /} is used to compute {\tt B/A=B A}$^{-1}$.  This is an internal function of \MAT which will 
analyze the matrices {\tt B} and {\tt A} to determine an ``optimal''
method for inversion, given their structure. 
The update given in Corollary \ref{c32} is completed in lines 30-32 with the equations
\begin{equation}
m_j = \hm_j + K_j d_j \quad {\rm and} \quad C_j =  (I - K_j H) \hc_j.
\label{eq:matkalup}
\end{equation}

Finally, in lines 36-50 the outputs of the program are used to plot the mean and the covariance  
as well as the mean square error of the filter as  functions of the iteration number $j$, as shown
in Figure \ref{fig:KF}.

\newpage

\begin{lstlisting}
clear;set(0,'defaultaxesfontsize',20);format long
%%% p8.m Kalman Filter, Ex. 1.2
%% setup

J=1e3;% number of steps
N=2;% dimension of state
I=eye(N);% identity operator
gamma=1;% observational noise variance is gamma^2*I
sigma=1;% dynamics noise variance is sigma^2*I
C0=eye(2);% prior initial condition variance
m0=[0;0];% prior initial condition mean
sd=10;rng(sd);% choose random number seed
A=[0 1;-1 0];% dynamics determined by A

m=zeros(N,J);v=m;y=zeros(J,1);c=zeros(N,N,J);% pre-allocate
v(:,1)=m0+sqrtm(C0)*randn(N,1);% initial truth
m(:,1)=10*randn(N,1);% initial mean/estimate
c(:,:,1)=100*C0;% initial covariance
H=[1,0];% observation operator

%% solution % assimilate!

for j=1:J   
    v(:,j+1)=A*v(:,j) + sigma*randn(N,1);% truth
    y(j)=H*v(:,j+1)+gamma*randn;% observation
    
    mhat=A*m(:,j);% estimator predict
    chat=A*c(:,:,j)*A'+sigma^2*I;% covariance predict  
    
    d=y(j)-H*mhat;% innovation
    K=(chat*H')/(H*chat*H'+gamma^2);% Kalman gain
    m(:,j+1)=mhat+K*d;% estimator update
    c(:,:,j+1)=(I-K*H)*chat;% covariance update    
end

figure;js=21;plot([0:js-1],v(2,1:js));hold;plot([0:js-1],m(2,1:js),'m');
plot([0:js-1],m(2,1:js)+reshape(sqrt(c(2,2,1:js)),1,js),'r--');
plot([0:js-1],m(2,1:js)-reshape(sqrt(c(2,2,1:js)),1,js),'r--');
hold;grid;xlabel('iteration, j');
title('Kalman Filter, Ex. 1.2');

figure;plot([0:J],reshape(c(1,1,:)+c(2,2,:),J+1,1));hold
plot([0:J],cumsum(reshape(c(1,1,:)+c(2,2,:),J+1,1))./[1:J+1]','m', ...
'Linewidth',2); grid; hold;xlabel('iteration, j');axis([1 1000 0 50]);
title('Kalman Filter Covariance, Ex. 1.2');

figure;plot([0:J],sum((v-m).^2));hold;
plot([0:J],cumsum(sum((v-m).^2))./[1:J+1],'m','Linewidth',2);grid
hold;xlabel('iteration, j');axis([1 1000 0 50]);
title('Kalman Filter Error, Ex. 1.2')









\end{lstlisting}

\newpage

\subsection{p9.m}
\label{ssec:p82}

The 
program 
{\tt p9.m} 
contains an implementation of the 3DVAR\index{3DVAR} method applied to the chaotic 
logistic map of Example \ref{ex:ex4} \eqref{eq:matlog} for $r=4$. 
As in the previous section, 
the parameters and initial condition are defined  in the {\tt setup} section, lines 3-16.
In particular, notice that the truth initial condition {\tt v(1)} and initial mean {\tt m(1)}, are now initialized 
in lines 12-13 with a {\it uniform} random number using the command {\tt rand}, so that they are in the interval $[0,1]$ 
where the model is well-defined.  Indeed the 
solution will eventually become unbounded if initial conditions are chosen outside this interval.  
With this in mind, we set the 
dynamics noise  {\tt sigma = 0} in line 8, i.e. deterministic dynamics, so that the true dynamics themselves 
do not go unbounded.

The analysis step  \index{filtering!analysis step} of 3DVAR \index{3DVAR} consists of 
minimizing
$$\Jtd(v)=\frac12|\Gamma^{-\frac12}(y_{j+1}-Hv)|^2+\frac12|\hc^{-\frac12}(v-\PPsi(m_j))|^2.$$
In this one-dimensional case we set $\Gamma=\gamma^2$, $\hC=\sigma^2$
and define $\eta^2=\gamma^2/\eta^2.$
The stabilization parameter $\eta$ ({\tt eta}) from Example \ref{ex:101} 
is set in line 14, representing the ratio in uncertainty in the data to
that of  the model; equivalently it measures trust in the model over 
the observations.  The choice $\eta=0$ means the model is irrelevant 
in the minimization step \eqref{eq:dtfa8} of 3DVAR \index{3DVAR}, in the
observed space --the synchronization 
filter\index{synchronization}\index{filter!synchronization}.
Since, in the example, the signal space and observation space both have
dimension equal to one, the choice $\eta=0$ simply corresponds to
using only the data.
In contrast the choice 
$\eta=\infty$ ignores the observations  and uses only the model.

The 3DVAR set-up gives rise to the constant scalar covariance 
{\tt C} and resultant constant scalar gain {\tt K}; this should not 
be confused with the changing $K_j$ in \eqref{eq:matkalg}, 
temporarily defined by {\tt K} in line 31 of {\tt p8.m}.
The main {\tt solution} loop follows in lines 20-33.  
Up to the different forward model, lines 21-22, 24, 26, and 27 of this program are 
identical to lines 24-25, 27, 30, and 32 of {\tt p8.m} described in 
section \ref{ssec:p8}.
The only other difference is that the covariance updates are not here 
because of the constant covariance assumption underlying the 3DVAR \index{3DVAR} algorithm.

The 3DVAR\index{3DVAR} filter 
may in principle generate estimated mean {\tt mhat} outside $[0,1]$, because of the noise in the
data. In order to flag potential unbounded trajectories of the filter,
which in principle could arise because of this, 
an extra stopping criteria is included in lines 29-32. 
To illustrate this try setting {\tt sigma}$\neq0$ in line 8.  Then 
the signal will eventually become unbounded, regardless of
how small the noise variance is chosen.  In this case the estimate will surely blowup while
tracking the unbounded signal.  Otherwise, if $\eta$ is chosen appropriately so as to 
stabilize the filter it is extremely unlikely that the estimate will ever blowup.
Finally, similarly to {\tt p8.m}, in the last lines of the program
we use the outputs of the program in order to produce Figure \ref{fig:3DVAR4}, 
namely plotting the mean and the covariance  as well as the mean square error of the filter 
as  functions of the iteration number $j$.

\newpage

\begin{lstlisting}
clear;set(0,'defaultaxesfontsize',20);format long
%%% p9.m 3DVAR Filter, deterministic logistic map (Ex. 1.4)
%% setup

J=1e3;% number of steps
r=4;% dynamics determined by r
gamma=1e-1;% observational noise variance is gamma^2
sigma=0;% dynamics noise variance is sigma^2
sd=10;rng(sd);% choose random number seed

m=zeros(J,1);v=m;y=m;% pre-allocate
v(1)=rand;% initial truth, in [0,1]
m(1)=rand;% initial mean/estimate, in [0,1]
eta=2e-1;% stabilization coefficient 0 < eta << 1
C=gamma^2/eta;H=1;% covariance and observation operator
K=(C*H')/(H*C*H'+gamma^2);% Kalman gain

%% solution % assimilate!

for j=1:J    
    v(j+1)=r*v(j)*(1-v(j)) + sigma*randn;% truth
    y(j)=H*v(j+1)+gamma*randn;% observation

    mhat=r*m(j)*(1-m(j));% estimator predict
    
    d=y(j)-H*mhat;% innovation
    m(j+1)=mhat+K*d;% estimator update
    
    if norm(mhat)>1e5
        disp('blowup!')
        break
    end
end
js=21;% plot truth, mean, standard deviation, observations
figure;plot([0:js-1],v(1:js));hold;plot([0:js-1],m(1:js),'m');
plot([0:js-1],m(1:js)+sqrt(C),'r--');plot([1:js-1],y(1:js-1),'kx');
plot([0:js-1],m(1:js)-sqrt(C),'r--');hold;grid;xlabel('iteration, j');
title('3DVAR Filter, Ex. 1.4')

figure;plot([0:J],C*[0:J].^0);hold
plot([0:J],C*[0:J].^0,'m','Linewidth',2);grid
hold;xlabel('iteration, j');title('3DVAR Filter Covariance, Ex. 1.4');

figure;plot([0:J],(v-m).^2);hold;
plot([0:J],cumsum((v-m).^2)./[1:J+1]','m','Linewidth',2);grid
hold;xlabel('iteration, j');
title('3DVAR Filter Error, Ex. 1.4')







\end{lstlisting}
\newpage


\subsection{p10.m}
\label{ssec:p9}

A variation of program {\tt p9.m} is given by {\tt p10.m}, where 
the  3DVAR\index{3DVAR} filter is implemented for Example \ref{ex:ex3} given by \eqref{eq:matsin}.
Indeed the remaining programs of this section will all be for the same Example \ref{ex:ex3} 
so this will not be mentioned again. 
In this case, the initial condition is again taken as a draw from the prior $N(m_0,C_0)$ as in {\tt p7.m},
and the initial mean estimate is again {\it changed} to $m_0\sim N(0,100 I)$ 
so that the code may test the ability of the filter 
to lock onto the signal given a poor initial estimate.
Furthermore, for this problem there is no need to 
introduce the stopping criteria present in the case of {\tt p9.m} since the underlying deterministic 
dynamics are dissipative.
The output of this program is shown in Figure \ref{fig:3DVAR3}.

\newpage

\begin{lstlisting}
clear;set(0,'defaultaxesfontsize',20);format long
%%% p10.m 3DVAR Filter, sin map (Ex. 1.3)
%% setup

J=1e3;% number of steps
alpha=2.5;% dynamics determined by alpha
gamma=1;% observational noise variance is gamma^2
sigma=3e-1;% dynamics noise variance is sigma^2
C0=9e-2;% prior initial condition variance
m0=0;% prior initial condition mean
sd=1;rng(sd);% choose random number seed

m=zeros(J,1);v=m;y=m;% pre-allocate
v(1)=m0+sqrt(C0)*randn;% initial truth
m(1)=10*randn;% initial mean/estimate
eta=2e-1;% stabilization coefficient 0 < eta << 1
c=gamma^2/eta;H=1;% covariance and observation operator
K=(c*H')/(H*c*H'+gamma^2);% Kalman gain

%% solution % assimilate!

for j=1:J    
    v(j+1)=alpha*sin(v(j)) + sigma*randn;% truth
    y(j)=H*v(j+1)+gamma*randn;% observation

    mhat=alpha*sin(m(j));% estimator predict
    
    d=y(j)-H*mhat;% innovation
    m(j+1)=mhat+K*d;% estimator update
    
end

js=21;% plot truth, mean, standard deviation, observations
figure;plot([0:js-1],v(1:js));hold;plot([0:js-1],m(1:js),'m');
plot([0:js-1],m(1:js)+sqrt(c),'r--');plot([1:js-1],y(1:js-1),'kx');
plot([0:js-1],m(1:js)-sqrt(c),'r--');hold;grid;xlabel('iteration, j');
title('3DVAR Filter, Ex. 1.3')

figure;plot([0:J],c*[0:J].^0);hold
plot([0:J],c*[0:J].^0,'m','Linewidth',2);grid
hold;xlabel('iteration, j');
title('3DVAR Filter Covariance, Ex. 1.3');

figure;plot([0:J],(v-m).^2);hold;
plot([0:J],cumsum((v-m).^2)./[1:J+1]','m','Linewidth',2);grid
hold;xlabel('iteration, j');
title('3DVAR Filter Error, Ex. 1.3')








\end{lstlisting}
\newpage

\subsection{p11.m}
\label{ssec:p10}

The next 
program is {\tt p11.m}. 
This program comprises an implementation of the 
extended Kalman Filter.
It is very similar in structure to {\tt p8.m}, except with a different 
forward model. Since the dynamics are scalar, the observation operator 
is defined by setting {\tt H} to take value $1$ in line 16.
The predicting covariance $\hc_j$ is not independent of the mean
as it is for the linear problem {\tt p8.m}.
Instead, as described in section \ref{ssec:exkf}, it is determined via the 
{\it linearization} of the forward map around $m_j$, in line 26:
\[
\hc_{j+1}=\left( \alpha\cos(m_{j}) \right) C_{j} \left( \alpha\cos(m_{j}) \right).
\]
As in {\tt p8.m} we {\it change} the prior to a poor initial estimate 
of the distribution to study if, and how, the filter locks onto a neighbourhood
of the true signal, despite poor initialization, for large $j.$
This initialization is in lines 15-16, 
where $m_0\sim N(0,100 I)$ and $C_0 \leftarrow 10C_0$.
Subsequent filtering\index{filtering!program} 
programs use an identical initialization, with
the same rationale as in this case. We will not state this again.
The output of this program is shown in Figure \ref{fig:ExKF}.

\newpage

\begin{lstlisting}
clear;set(0,'defaultaxesfontsize',20);format long
%%% p11.m  Extended Kalman Filter, sin map (Ex. 1.3)
%% setup

J=1e3;% number of steps
alpha=2.5;% dynamics determined by alpha
gamma=1;% observational noise variance is gamma^2
sigma=3e-1;% dynamics noise variance is sigma^2
C0=9e-2;% prior initial condition variance
m0=0;% prior initial condition mean
sd=1;rng(sd);% choose random number seed

m=zeros(J,1);v=m;y=m;c=m;% pre-allocate
v(1)=m0+sqrt(C0)*randn;% initial truth
m(1)=10*randn;% initial mean/estimate
c(1)=10*C0;H=1;% initial covariance and observation operator

%% solution % assimilate!

for j=1:J 
    
    v(j+1)=alpha*sin(v(j)) + sigma*randn;% truth
    y(j)=H*v(j+1)+gamma*randn;% observation

    mhat=alpha*sin(m(j));% estimator predict
    chat=alpha*cos(m(j))*c(j)*alpha*cos(m(j))+sigma^2;% covariance predict  
    
    d=y(j)-H*mhat;% innovation
    K=(chat*H')/(H*chat*H'+gamma^2);% Kalman gain
    m(j+1)=mhat+K*d;% estimator update
    c(j+1)=(1-K*H)*chat;% covariance update
    
end

js=21;% plot truth, mean, standard deviation, observations
figure;plot([0:js-1],v(1:js));hold;plot([0:js-1],m(1:js),'m');
plot([0:js-1],m(1:js)+sqrt(c(1:js)),'r--');plot([1:js-1],y(1:js-1),'kx');
plot([0:js-1],m(1:js)-sqrt(c(1:js)),'r--');hold;grid;
xlabel('iteration, j');title('ExKF, Ex. 1.3')

figure;plot([0:J],c);hold
plot([0:J],cumsum(c)./[1:J+1]','m','Linewidth',2);grid
hold;xlabel('iteration, j');
title('ExKF Covariance, Ex. 1.3');

figure;plot([0:J],(v-m).^2);hold;
plot([0:J],cumsum((v-m).^2)./[1:J+1]','m','Linewidth',2);grid
hold;xlabel('iteration, j');
title('ExKF Error, Ex. 1.3')









\end{lstlisting}

\newpage

\subsection{p12.m}
\label{ssec:p11}

The program {\tt p12.m} contains an implementation of the ensemble Kalman Filter, with perturbed observations\index{perturbed observations}, 
as described in section \ref{ssec:enkf}.
The structure of this program is again very similar to {\tt p8.m} and {\tt p11.m}, 
except now an ensemble of particles, of size {\tt N} defined in line 12, 
is retained as an approximation of the 
filtering distribution\index{filtering!distribution}.
The ensemble $\{v^{(n)}\}_{n=1}^N$ represented by the matrix 
{\tt U} is then constructed out of draws from this Gaussian in line 18, 
and the mean $m'_0$ is reset to the ensemble sample mean.

In line 27 the predicting ensemble $\{\widehat{v}_j^{(n)}\}_{n=1}^N$ represented
by the matrix {\tt Uhat} is computed from a realization of 
the forward map applied to each ensemble member.
This is then used to compute the ensemble sample mean $\hm_j$ ({\tt mhat})
and covariance $\hc_j$ ({\tt chat}).
There is now an ensemble of "innovations" 
with a new i.i.d. \index{i.i.d.} realization $y^{(n)}_j \sim N(y_j,\Gamma)$ for each ensemble member, 
computed in line 31(not to be confused with the actual innovation \index{innovation} as defined in
Equation \eqref{eq:matinno}) 
\[
d^{(n)}_j = y_j^{(n)} - H \hv^{(n)}_j. 
\]
The Kalman gain \index{Kalman gain}$K_j$ ({\tt K}) is computed using  \eqref{eq:matkalg}, 
very similarly to {\tt p8.m} and {\tt p11.m},
and the ensemble of updates are computed in line 33:
\[
v^{(n)}_j = \hv^{(n)}_j + K_j d^{(n)}_j.
\]
The output of this program is shown in Figure \ref{fig:EnKF}.
Furthermore, long simulations of length $J=10^5$ were performed for this and the previous two
programs {\tt p10.m} and {\tt p11.m} and their errors are compared in Figure \ref{fig:error}.

\newpage
\begin{lstlisting}
clear;set(0,'defaultaxesfontsize',20);format long
%%% p12.m Ensemble Kalman Filter (PO), sin map (Ex. 1.3)
%% setup

J=1e5;% number of steps
alpha=2.5;% dynamics determined by alpha
gamma=1;% observational noise variance is gamma^2
sigma=3e-1;% dynamics noise variance is sigma^2
C0=9e-2;% prior initial condition variance
m0=0;% prior initial condition mean
sd=1;rng(sd);% choose random number seed
N=10;% number of ensemble members

m=zeros(J,1);v=m;y=m;c=m;U=zeros(J,N);% pre-allocate
v(1)=m0+sqrt(C0)*randn;% initial truth
m(1)=10*randn;% initial mean/estimate
c(1)=10*C0;H=1;% initial covariance and observation operator
U(1,:)=m(1)+sqrt(c(1))*randn(1,N);m(1)=sum(U(1,:))/N;% initial ensemble

%% solution % assimilate!

for j=1:J   
    
    v(j+1)=alpha*sin(v(j)) + sigma*randn;% truth
    y(j)=H*v(j+1)+gamma*randn;% observation

    Uhat=alpha*sin(U(j,:))+sigma*randn(1,N);% ensemble predict
    mhat=sum(Uhat)/N;% estimator predict
    chat=(Uhat-mhat)*(Uhat-mhat)'/(N-1);% covariance predict  

    d=y(j)+gamma*randn(1,N)-H*Uhat;% innovation
    K=(chat*H')/(H*chat*H'+gamma^2);% Kalman gain
    U(j+1,:)=Uhat+K*d;% ensemble update
    m(j+1)=sum(U(j+1,:))/N;% estimator update
    c(j+1)=(U(j+1,:)-m(j+1))*(U(j+1,:)-m(j+1))'/(N-1);% covariance update    
    
end

js=21;% plot truth, mean, standard deviation, observations
figure;plot([0:js-1],v(1:js));hold;plot([0:js-1],m(1:js),'m');
plot([0:js-1],m(1:js)+sqrt(c(1:js)),'r--');plot([1:js-1],y(1:js-1),'kx');
plot([0:js-1],m(1:js)-sqrt(c(1:js)),'r--');hold;grid;
xlabel('iteration, j');title('EnKF, Ex. 1.3')

figure;plot([0:J],c);hold
plot([0:J],cumsum(c)./[1:J+1]','m','Linewidth',2);grid
hold;xlabel('iteration, j');
title('EnKF Covariance, Ex. 1.3');

figure;plot([0:J],(v-m).^2);hold;
plot([0:J],cumsum((v-m).^2)./[1:J+1]','m','Linewidth',2);grid
hold;xlabel('iteration, j');
title('EnKF Error, Ex. 1.3')







\end{lstlisting}
\newpage

\subsection{p13.m}
\label{ssec:p12}

The program {\tt p13.m} contains a particular square-root filter 
implementation of the ensemble Kalman filter \index{Kalman filter!ensemble},
namely the ETKF filter\index{Kalman filter!ensemble transform} 
described in detail in section 
\ref{ssec:sqrt}.  The program thus is very similar to {\tt p12.m} for 
the EnKF\index{Kalman filter!ensemble} with perturbed observations.
In particular, the filtering distribution  \index{filtering!distribution} of the state is again approximated by an ensemble of particles.
The predicting ensemble $\{\widehat{v}_j^{(n)}\}_{n=1}^N$ ({\tt Uhat}), 
mean $\hm_j$({\tt mhat}), and covariance $\hc_j$ ({\tt chat}) are
computed exactly as in {\tt p12.m}.  However, this time the covariance is kept in factorized
form $\widehat{X}_j \widehat{X}_j^\top= \hc_j$ in lines 29-30, with factors denoted {\tt Xhat}.  
The transformation matrix is computed in line 31
\[
T_j=\left ( I_N + \widehat{X}_j^\top H^\top \Gamma^{-1} H \widehat{X}_j\right)^{-\frac12},
\]
and $X_j = \widehat{X}_j T_j$ ({\tt X}) is computed in line 32,
from which the covariance $C_j = X_j X_j^\top$ is reconstructed in line 38.
A single innovation \index{innovation}$d_j$ is computed in line 34 and 
a single updated mean $m_j$ is then computed in line 36 using the Kalman gain \index{Kalman gain} 
$K_j$ \eqref{eq:matkalg} computed in line 35. 
 This is the same as in the Kalman Filter\index{Kalman filter}
 and extended Kalman filter (ExKF)\index{Kalman filter!extended} of {\tt p8.m} and {\tt p11.m}, 
 in contrast to the EnKF\index{Kalman filter!ensemble} 
with perturbed observations
\index{perturbed observations} appearing in {\tt p12.m}.
The ensemble is then updated to {\tt U} in line 37 using the formula
\[
v_j^{(n)} = m_j + X_j^{(n)}\sqrt{N-1},
\]
where $X_j^{(n)}$ is the $n^{th}$ column of $X_j$.

Notice that the operator which is factorized and inverted has dimension $N$,
which in this case is large in comparison to the state and observation dimensions.
This is of course natural for computing sample statistics but in the
context of the one dimensional examples considered here makes 
{\tt p13.m} run far more slowly than {\tt p12.m}. However in 
many applications the signal state-space dimension is 
the largest, with the observation dimension
coming next, and the ensemble size being far smaller than either of these.  
In this context the ETKF has become a very popular method. So its 
relative inefficiency, compared for example with the perturbed observations
Kalman filter, should not be given too much weight in the overall
evaluation of the method. 
Results illustrating the algorithm are shown in Figure \ref{fig:ETKF}.

\newpage

\begin{lstlisting}
clear;set(0,'defaultaxesfontsize',20);format long
%%% p13.m Ensemble Kalman Filter (ETKF), sin map (Ex. 1.3)
%% setup

J=1e3;% number of steps
alpha=2.5;% dynamics determined by alpha
gamma=1;% observational noise variance is gamma^2
sigma=3e-1;% dynamics noise variance is sigma^2
C0=9e-2;% prior initial condition variance
m0=0;% prior initial condition mean
sd=1;rng(sd);% choose random number seed
N=10;% number of ensemble members

m=zeros(J,1);v=m;y=m;c=m;U=zeros(J,N);% pre-allocate
v(1)=m0+sqrt(C0)*randn;% initial truth
m(1)=10*randn;% initial mean/estimate
c(1)=10*C0;H=1;% initial covariance and observation operator
U(1,:)=m(1)+sqrt(c(1))*randn(1,N);m(1)=sum(U(1,:))/N;% initial ensemble

%% solution % assimilate!

for j=1:J    
    
    v(j+1)=alpha*sin(v(j)) + sigma*randn;% truth
    y(j)=H*v(j+1)+gamma*randn;% observation
    
    Uhat=alpha*sin(U(j,:))+sigma*randn(1,N);% ensemble predict
    mhat=sum(Uhat)/N;% estimator predict
    Xhat=(Uhat-mhat)/sqrt(N-1);% centered ensemble
    chat=Xhat*Xhat';% covariance predict
    T=sqrtm(inv(eye(N)+Xhat'*H'*H*Xhat/gamma^2));% sqrt transform
    X=Xhat*T;% transformed centered ensemble
    
    d=y(j)-H*mhat;randn(1,N);% innovation
    K=(chat*H')/(H*chat*H'+gamma^2);% Kalman gain
    m(j+1)=mhat+K*d;% estimator update
    U(j+1,:)=m(j+1)+X*sqrt(N-1);% ensemble update
    c(j+1)=X*X';% covariance update
    
end

js=21;% plot truth, mean, standard deviation, observations
figure;plot([0:js-1],v(1:js));hold;plot([0:js-1],m(1:js),'m');
plot([0:js-1],m(1:js)+sqrt(c(1:js)),'r--');plot([1:js-1],y(1:js-1),'kx');
plot([0:js-1],m(1:js)-sqrt(c(1:js)),'r--');hold;grid;
xlabel('iteration, j');title('EnKF(ETKF), Ex. 1.3');

figure;plot([0:J],(v-m).^2);hold;
plot([0:J],cumsum((v-m).^2)./[1:J+1]','m','Linewidth',2);grid
plot([0:J],cumsum(c)./[1:J+1]','r--','Linewidth',2);
hold;xlabel('iteration, j');
title('EnKF(ETKF) Error, Ex. 1.3')





\end{lstlisting}

\newpage

\subsection{p14.m}
\label{ssec:p13}

The program {\tt p14.m} is an implementation of the standard 
SIRS 
filter from subsection \ref{ssec:bsf}.  The 
{\tt setup} section is almost identical to the EnKF\index{Kalman filter!ensemble} methods, because those
methods also rely on particle approximations of the filtering distribution  \index{filtering!distribution}.  
However, the particle filters consistently estimate quite general distributions,
whilst the EnKF\index{Kalman filter!ensemble} is only provably accurate for Gaussian distributions. 
The truth and data generation and ensemble prediction in lines 24-27 
are the same as in {\tt p12.m} and {\tt p13.m}.  The way this prediction in line 27
is phrased in section \ref{ssec:bsf} is $\widehat{v}_{j+1}^{(n)} \sim \bbP(\cdot | v_{j}^{(n)})$.
An ensemble of "innovation" terms $\{d_j^{(n)}\}_{n=1}^N$ are again required
again, but all using the {\it same} observation, as computed in line 28.
Assuming $w_{j}^{(n)}=1/N$,
then  
\[
\widehat{w}_{j}^{(n)} \propto \bbP(y_j|v_j^{(n)}) \propto \exp \left \{ {-\frac{1}{2} \left |d_j^{(n)}\right |_\Gamma^2} \right \},
\]  
where $d_j^{(n)}$ is the innovation\index{innovation} of the $n^{th}$ particle, as given in \eqref{eq:waits}.
The vector of un-normalized weights $\{\widehat{w}_j^{(n)}\}_{n=1}^N$ ({\tt what})
are computed in line 29 and normalized to $\{w_{j}^{(n)}\}_{n=1}^N$ ({\tt w}) 
in line 30.  
Lines 32-39 implement the resampling step.  
First, the cumulative distribution function 
of the weights $W \in [0,1]^N$ ({\tt ws})
is computed in line 32.  
Notice $W$ has the properties that $W_1 = w_{j}^{(1)}$, $W_{n}\leq W_{n+1}$, 
and $W_N=1$. 
Then $N$ uniform random numbers $\{u^{(n)}\}_{n=1}^N$
are drawn.  
For each $u^{(n)}$, let $n^*$ be such that $W_{n^*-1}\leq u^{(n)} < W_{n^*}$.
This $n^*$ ({\tt ix}) is found in line 34 using the {\tt find} function, which can identify
the first or last element in an array to exceed zero (see {\tt help} file):
{\tt ix = find ( ws > rand, 1, 'first' )}. 
This corresponds to drawing the $(n^*)^{th}$ element from the discrete measure
defined by $\{w_{j}^{(n)}\}_{n=1}^N$.  
The $n^{th}$ particle 
$v_j^{(n)}$ ({\tt U(j+1,n)}) is set to be equal to $\widehat{v}_j^{(n^*)}$ ({\tt Uhat(ix)})
in line 37.
The sample mean and covariance are then computed in lines 41-42.
The rest of the program follows the others, generating the output displayed in 
Figure \ref{fig:PFST}.

\newpage

\begin{lstlisting}
clear;set(0,'defaultaxesfontsize',20);format long
%%% p14.m Particle Filter (SIRS), sin map (Ex. 1.3)
%% setup

J=1e3;% number of steps
alpha=2.5;% dynamics determined by alpha
gamma=1;% observational noise variance is gamma^2
sigma=3e-1;% dynamics noise variance is sigma^2
C0=9e-2;% prior initial condition variance
m0=0;% prior initial condition mean
sd=1;rng(sd);% choose random number seed
N=100;% number of ensemble members

m=zeros(J,1);v=m;y=m;c=m;U=zeros(J,N);% pre-allocate
v(1)=m0+sqrt(C0)*randn;% initial truth
m(1)=10*randn;% initial mean/estimate
c(1)=10*C0;H=1;% initial covariance and observation operator
U(1,:)=m(1)+sqrt(c(1))*randn(1,N);m(1)=sum(U(1,:))/N;% initial ensemble

%% solution % Assimilate!
for j=1:J     
    v(j+1)=alpha*sin(v(j)) + sigma*randn;% truth
    y(j)=H*v(j+1)+gamma*randn;% observation
 
    Uhat=alpha*sin(U(j,:))+sigma*randn(1,N);% ensemble predict
    d=y(j)-H*Uhat;% ensemble innovation  
    what=exp(-1/2*(1/gamma^2*d.^2));% weight update
    w=what/sum(what);% normalize predict weights   
 
    ws=cumsum(w);% resample: compute cdf of weights
    for n=1:N
        ix=find(ws>rand,1,'first');% resample: draw rand \sim U[0,1] and 
        % find the index of the particle corresponding to the first time 
        % the cdf of the weights exceeds rand.
        U(j+1,n)=Uhat(ix);% resample: reset the nth particle to the one 
        % with the given index above
    end
    
    m(j+1)=sum(U(j+1,:))/N;% estimator update
    c(j+1)=(U(j+1,:)-m(j+1))*(U(j+1,:)-m(j+1))'/N;% covariance update     
end

js=21;% plot truth, mean, standard deviation, observations
figure;plot([0:js-1],v(1:js));hold;plot([0:js-1],m(1:js),'m');
plot([0:js-1],m(1:js)+sqrt(c(1:js)),'r--');plot([1:js-1],y(1:js-1),'kx');
plot([0:js-1],m(1:js)-sqrt(c(1:js)),'r--');hold;grid;
xlabel('iteration, j');title('Particle Filter (Standard), Ex. 1.3');

figure;plot([0:J],(v-m).^2);hold;
plot([0:J],cumsum((v-m).^2)./[1:J+1]','m','Linewidth',2);grid
hold;xlabel('iteration, j');
title('Particle Filter (Standard) Error, Ex. 1.3')





\end{lstlisting}

\newpage

\subsection{p15.m}
\label{ssec:p14}

The program {\tt p15.m} is an implementation of the 
SIRS(OP) algorithm from subsection \ref{ssec:opt}.
The {\tt setup} section and truth and observation generation are again the 
same as in the previous programs.  The difference between this program
and {\tt p14.m} arises arises because 
the importance sampling proposal\index{proposal} kernel $Q_j$ with density
$\bbP(v_{j+1}|v_{j},y_{j+1})$ used to propose each $\widehat{v}_{j+1}^{(n)}$ 
given each particular $v_{j}^{(n)}$; in particular $Q_j$ depends on the next 
data point whereas the kernel $P$ used in {\tt p14.m} has density 
$\bbP(v_{j+1}|v_{j})$ which is independent of $y_{j+1}.$

Observe that if $v_{j}^{(n)}$ and $y_{j+1}$ are both fixed, then $\bbP\left (v_{j+1}|v_{j}^{(n)},y_{j+1}\right )$
is the density of the Gaussian with mean $m'^{(v)}$ and covariance $\Sigma'$ given by
\[
m'^{(n)} = \Sigma' \left (\Sigma^{-1} \PPsi\left (v_{j}^{(n)}\right ) + H^\top \Gamma^{-1} y_{j+1} \right ), \quad  \left(\Sigma'\right )^{-1} = \Sigma^{-1} + H^\top \Gamma^{-1} H.
\]
Therefore, $\Sigma'$ ({\tt Sig}) and the ensemble of means $\left\{m'^{(n)}\right \}_{n=1}^{N}$ (vector {\tt em})
are computed in lines 27 and 28 and 
used to sample $\widehat{v}_{j+1}^{(n)} \sim N(m'^{(n)},\Sigma')$ in line 29 for all of 
$\left\{\widehat{v}_{j+1}^{(n)}\right\}_{n=1}^N$ ({\tt Uhat}).  

Now the weights are therefore updated by \eqref{eq:opwts} rather than \eqref{eq:waits}, 
i.e. assuming $w_{j}^{(n)}=1/N$, then  
\[
\widehat{w}_{j+1}^{(n)} \propto \bbP\left (y_{j+1}|v_{j}^{(n)}\right ) \propto 
\exp \left \{{-\frac12 \left |y_{j+1} - \PPsi\left(v_{j}^{(n)}\right )\right |_{\Gamma+\Sigma}^2}\right \}.
\]  
 This is computed in lines 31-32, using another auxiliary "innovation" vector {\tt d} 
 in line 31.  
 Lines 35-45 are again identical to lines 32-42 of program {\tt p14.m}, performing 
 the resampling step and computing sample mean and covariance.
 



The output of this program was used to produce Figure \ref{fig:PFOP} similar to the other filtering algorithms  \index{filtering!algorithm}.
Furthermore, long simulations of length $J=10^5$ were performed for this and the previous three
programs {\tt p12.m}, {\tt p13.m} and {\tt p14.m} and their errors are compared in Figure \ref{fig:errorp}, 
similarly to Figure \ref{fig:error} comparing the basic filters {\tt p10.m}, {\tt p11.m}, and {\tt p12.m}.

\newpage

\begin{lstlisting}
clear;set(0,'defaultaxesfontsize',20);format long
%%% p15.m Particle Filter (SIRS, OP), sin map (Ex. 1.3)
%% setup

J=1e3;% number of steps
alpha=2.5;% dynamics determined by alpha
gamma=1;% observational noise variance is gamma^2
sigma=3e-1;% dynamics noise variance is sigma^2
C0=9e-2;% prior initial condition variance
m0=0;% prior initial condition mean
sd=1;rng(sd);% choose random number seed
N=100;% number of ensemble members

m=zeros(J,1);v=m;y=m;c=m;U=zeros(J,N);% pre-allocate
v(1)=m0+sqrt(C0)*randn;% initial truth
m(1)=10*randn;% initial mean/estimate
c(1)=10*C0;H=1;% initial covariance and observation operator
U(1,:)=m(1)+sqrt(c(1))*randn(1,N);m(1)=sum(U(1,:))/N;% initial ensemble

%% solution % Assimilate!
for j=1:J       
    v(j+1)=alpha*sin(v(j)) + sigma*randn;% truth
    y(j)=H*v(j+1)+gamma*randn;% observation
      
    Sig=inv(inv(sigma^2)+H'*inv(gamma^2)*H);% optimal proposal covariance
    em=Sig*(inv(sigma^2)*alpha*sin(U(j,:))+H'*inv(gamma^2)*y(j));% mean
    Uhat=em+sqrt(Sig)*randn(1,N);% ensemble optimally importance sampled     
    
    d=y(j)-H*alpha*sin(U(j,:));% ensemble innovation   
    what=exp(-1/2/(sigma^2+gamma^2)*d.^2);% weight update
    w=what/sum(what);% normalize predict weights 
    
    ws=cumsum(w);% resample: compute cdf of weights
    for n=1:N
        ix=find(ws>rand,1,'first');% resample: draw rand \sim U[0,1] and 
        % find the index of the particle corresponding to the first time 
        % the cdf of the weights exceeds rand.
        U(j+1,n)=Uhat(ix);% resample: reset the nth particle to the one 
        % with the given index above
    end
    
    m(j+1)=sum(U(j+1,:))/N;% estimator update
    c(j+1)=(U(j+1,:)-m(j+1))*(U(j+1,:)-m(j+1))'/N;% covariance update
end

js=21;%plot truth, mean, standard deviation, observations
figure;plot([0:js-1],v(1:js));hold;plot([0:js-1],m(1:js),'m');
plot([0:js-1],m(1:js)+sqrt(c(1:js)),'r--');plot([1:js-1],y(1:js-1),'kx');
plot([0:js-1],m(1:js)-sqrt(c(1:js)),'r--');hold;grid;
xlabel('iteration, j');title('Particle Filter (Optimal), Ex. 1.3');






\end{lstlisting}
\newpage

\section{ODE Programs}

The programs {\tt p16.m} and {\tt p17.m} are 
used to simulate and plot the Lorenz '63 and '96 models\index{Lorenz model}
from Examples
\ref{ex:ex6} and \ref{ex:ex7}, respectively.  These programs are 
both \MAT functions, similar to the program {\tt p7.m} presented in
Section \ref{ssec:p6}.  The reason for using functions and not
scripts is that the black box \MAT
built-in function {\tt ode45} can be used for the time integration (see {\tt help}
page for details regarding this function).  Therefore, each has an {\it auxiliary}
function defining the right-hand side of the given ODE, which is passed via 
{\it function handle} to {\tt ode45}.

\subsection{p16.m}
\label{ssec:p16}

The first of the ODE programs, {\tt p16.m}, integrates the Lorenz '63
model \ref{ex:ex6}.  The setup section of the program, on lines 4-11, 
defines the parameters of the model and the initial conditions.  
In particular, a random Gaussian initial condition is chosen in line 9, 
and a small perturbation to its first ($x$) component is introduced in
line 10.  The trajectories are computed on lines 13-14 using the built-in 
function {\tt ode45}.
Notice that the auxiliary function {\tt lorenz63}, defined on line 29, takes 
as arguments $(t,y)$, prescribed through the definition of the function 
handle {\tt @(t,y)}, while $(\alpha, b, r)$ are given as fixed parameters 
{\tt (a,b,r)}, defining the particular instance of the function.  
The argument $t$ is intended for defining {\it non-autonomous} ODE, and is 
spurious here as it is an {\it autonomous} ODE and therefore $t$ does not
appear on the right-hand side.  It is nonetheless included for completeness,
and causes no harm.  The Euclidean norm of the error is computed in line
16, and the results are plotted similarly to previous programs in lines 18-25.
This program is used to plot Figs. \ref{fig:ex6_1} and \ref{fig:ex6_2}.

\subsection{p17.m}
\label{ssec:p17}

The second of the ODE programs, {\tt p17.m}, integrates the {\tt J=40} dimensional
Lorenz '96 model \ref{ex:ex7}.  This program is almost identical to the previous one,
where a small perturbation of the random Gaussian initial condition 
defined on line 9 is introduced on lines 10-11.  The major difference is the 
function passed to {\tt ode45} on lines 14-15, which now defines the
right-hand side of the Lorenz '96 model given by sub-function {\tt lorenz96} 
on line 30.  Again the system is
autonomous, and the spurious $t$ variable is included for completeness.
A few of the 40 degrees of freedom are plotted along with the error in lines
19-27.
This program is used to plot Figs. \ref{fig:ex7_2} and \ref{fig:ex7_1}

\newpage
\begin{lstlisting}
function this=p16
clear;set(0,'defaultaxesfontsize',20);format long
%%% p16.m Lorenz '63 (Ex. 2.6)
%% setup

a=10;b=8/3;r=28;% define parameters
sd=1;rng(sd);% choose random number seed

initial=randn(3,1);% choose initial condition
initial1=initial + [0.0001;0;0];% choose perturbed initial condition

%% calculate the trajectories with blackbox  
[t1,y]=ode45(@(t,y) lorenz63(t,y,a,b,r), [0 100], initial);
[t,y1]=ode45(@(t,y) lorenz63(t,y,a,b,r), t1, initial1);

error=sqrt(sum((y-y1).^2,2));% calculate error

%% plot results

figure(1), semilogy(t,error,'k')
axis([0 100 10^-6 10^2])
set(gca,'YTick',[10^-6 10^-4 10^-2 10^0 10^2])

figure(2), plot(t,y(:,1),'k')
axis([0 100 -20 20])


%% auxiliary dynamics function definition
function rhs=lorenz63(t,y,a,b,r)

rhs(1,1)=a*(y(2)-y(1));
rhs(2,1)=-a*y(1)-y(2)-y(1)*y(3);
rhs(3,1)=y(1)*y(2)-b*y(3)-b*(r+a);


\end{lstlisting}
\newpage

\begin{lstlisting}
function this=p17
clear;set(0,'defaultaxesfontsize',20);format long
%%% p17.m Lorenz '96 (Ex. 2.7)
%% setup

J=40;F=8;% define parameters
sd=1;rng(sd);% choose random number seed

initial=randn(J,1);% choose initial condition
initial1=initial;
initial1(1)=initial(1)+0.0001;% choose perturbed initial condition

%% calculate the trajectories with blackbox  
[t1,y]=ode45(@(t,y) lorenz96(t,y,F), [0 100], initial);
[t,y1]=ode45(@(t,y) lorenz96(t,y,F), t1, initial1);

error=sqrt(sum((y-y1).^2,2));% calculate error

%% plot results

figure(1), plot(t,y(:,1),'k')
figure(2), plot(y(:,1),y(:,J),'k')
figure(3), plot(y(:,1),y(:,J-1),'k')

figure(4), semilogy(t,error,'k')
axis([0 100 10^-6 10^2])
set(gca,'YTick',[10^-6 10^-4 10^-2 10^0 10^2])

%% auxiliary dynamics function definition
function rhs=lorenz96(t,y,F)

rhs=[y(end);y(1:end-1)].*([y(2:end);y(1)] - ...
    [y(end-1:end);y(1:end-2)]) - y + F*y.^0;



\end{lstlisting}

%
%
%
%



\bibliographystyle{abbrv}
\bibliography{mybib}

\printindex


\end{document}